\documentclass[a4paper,11pt]{article}
\usepackage{amsfonts}
\usepackage{mathrsfs}
\usepackage{amsmath,amssymb,amsthm}
\usepackage{mathtools}
\usepackage{latexsym}
\usepackage{indentfirst}
\usepackage[top=1in, bottom=1in, left=1in, right=1in]{geometry}
\usepackage{color}
\usepackage{footmisc}
\usepackage{endnotes}
\usepackage{algpseudocode}
\usepackage{algorithm}
\usepackage{float}
\usepackage{graphicx}
\usepackage{caption}
\usepackage{subcaption}
\usepackage{epstopdf}
\usepackage{enumitem}
\usepackage{comment}
\usepackage{longtable}
\usepackage{multirow}
\usepackage[font={footnotesize}]{caption}
\usepackage{authblk}
\usepackage{bookmark}
\usepackage[numbers]{natbib}
\usepackage[toc,page,titletoc]{appendix}
\usepackage[capitalise,nameinlink]{cleveref}
\usepackage{arydshln}

\let\footnote=\endnote

\newcommand{\GG}[1]{}

\DeclareMathOperator*{\pr}{Pr}
\DeclareMathOperator*{\Cov}{Cov} \DeclareMathOperator*{\Var}{Var}
\DeclareMathOperator*{\EE}{E}

\DeclareMathOperator*{\diag}{diag}

\DeclareMathOperator*{\gp}{GP}
\DeclareMathOperator*{\kl}{KL}

\DeclareMathOperator*{\tr}{tr}
\DeclareMathOperator*{\op}{op}

\DeclareMathOperator*{\tv}{TV}
\DeclareMathOperator*{\numer}{\mathsf{N}}
\DeclareMathOperator*{\denom}{\mathsf{D}}

\def\underkappa{\underline\kappa}
\def\overkappa{\overline\kappa}
\newcommand{\ud}{\mathrm{d}}

\newtheorem{theorem}{Theorem}
\newtheorem{lemma}{Lemma}
\newtheorem{corollary}{Corollary}
\newtheorem{prop}{Proposition}

\theoremstyle{definition}

\allowdisplaybreaks

\DeclareMathOperator{\Acal}{\mathcal{A}}
\DeclareMathOperator{\Bcal}{\mathcal{B}}
\DeclareMathOperator{\Ccal}{\mathcal{C}}
\DeclareMathOperator{\Dcal}{\mathcal{D}}
\DeclareMathOperator{\Ecal}{\mathcal{E}}
\DeclareMathOperator{\Fcal}{\mathcal{F}}

\DeclareMathOperator{\Hcal}{\mathcal{H}}
\DeclareMathOperator{\Ical}{\mathcal{I}}

\DeclareMathOperator{\Kcal}{\mathcal{K}}
\DeclareMathOperator{\Lcal}{\mathcal{L}}

\DeclareMathOperator{\Ncal}{\mathcal{N}}

\DeclareMathOperator{\Scal}{\mathcal{S}}

\DeclareMathOperator{\Wcal}{\mathcal{W}}

\def\RR{\mathbb{R}}
\def\ZZ{\mathbb{Z}}
\def\CC{\mathbb{C}}
\def\NN{\mathbb{N}}
\def\ee{\mathrm{e}}

\def\rr{\mathrm{r}}
\def\vv{\mathrm{v}}
\def\bbm{\mathrm{m}}

\def\sP{\mathsf{P}}
\def\cbeta{b}

\begin{document}
\title{\bf Bayesian Fixed-domain Asymptotics for Covariance Parameters in a Gaussian Process Model}

\author[1]{Cheng Li \thanks{stalic@nus.edu.sg}}
\affil[1]{Department of Statistics and Data Science, National University of Singapore}

\date{}
\maketitle

\begin{abstract}
Gaussian process models typically contain finite dimensional parameters in the covariance function that need to be estimated from the data. We study the Bayesian fixed-domain asymptotics for the covariance parameters in a universal kriging model with an isotropic Mat\'ern covariance function, which has many applications in spatial statistics. We show that when the dimension of domain is less than or equal to three, the joint posterior distribution of the microergodic parameter and the range parameter can be factored independently into the product of their marginal posteriors under fixed-domain asymptotics. The posterior of the microergodic parameter is asymptotically close in total variation distance to a normal distribution with shrinking variance, while the posterior distribution of the range parameter does not converge to any point mass distribution in general. Our theory allows an unbounded prior support for the range parameter and flexible designs of sampling points. We further study the asymptotic efficiency and convergence rates in posterior prediction for the Bayesian kriging predictor with covariance parameters randomly drawn from their posterior distribution. In the special case of one-dimensional Ornstein-Uhlenbeck process, we derive explicitly the limiting posterior of the range parameter and the posterior convergence rate for asymptotic efficiency in posterior prediction. We verify these asymptotic results in numerical experiments.
\end{abstract}

\section{Introduction} \label{sec:intro}

Gaussian processes (GP) have been widely used in spatial statistics, computer experiments, machine learning, and many other fields. In this paper, we consider the observation from the following spatial Gaussian process regression model, known as the \textit{universal kriging model} (Chapter 3 Section 3.4, \citet{Cre93}):
\begin{align} \label{eq:obs.model}
Y(s_i) & = \bbm(s_i)^\top \beta +  X(s_i),\quad \text{for } i=1,\ldots,n.
\end{align}
In the model \eqref{eq:obs.model}, $\Scal_n=\{s_1,\ldots,s_n\}$ is a sequence of distinct sampling points in the fixed domain $\Scal=[0,T]^d$, and $0<T<\infty$ is a known constant and the dimension $d\in \{1,2,3\}$. Such a dimension $d$ is of primary interest in spatial statistics. Here $\bbm(\cdot)=(\bbm_1(\cdot),\ldots,\bbm_p(\cdot))^\top$ is a $p$-dimensional vector of linearly independent and known deterministic functions defined on $\Scal$, and $\beta\in \RR^p$ is the regression coefficient vector. In applications, $\bbm_1,\ldots,\bbm_p$ can include the constant function $1$, and hence $\beta$ can include an intercept term. In the model \eqref{eq:obs.model}, $X(\cdot)$ is a mean-zero Gaussian stochastic process $X=\left\{X(s):s\in \Scal\right\}$. We assume that the covariance function of $X$ is the isotropic Mat\'ern covariance function given by
\begin{align} \label{eq:MaternCov}
\Cov(X(s),X(t)) &= \sigma^2 K_{\alpha,\nu}(s-t) =\sigma^2\frac{2^{1-\nu}}{\Gamma(\nu)}\left(\alpha \|s-t\|\right)^{\nu} \Kcal_{\nu}\left(\alpha\|s-t\|\right),
\end{align}
for any $s,t\in \Scal$, where $\nu>0$ is the smoothness parameter, $\sigma^2>0$ is the variance (or partial sill) parameter, and $\alpha>0$ is the inverse range (or length-scale) parameter, $\Kcal_{\nu}(\cdot)$ is the modified Bessel function of the second kind (\citet{Kre12}), and $\|\cdot\|$ is the Euclidean norm. The Mat\'ern covariance function is popular in applications of spatial statistics and computer experiments because the smoothness parameter $\nu$ provides flexibility in controlling the smoothness of sample paths (\citet{Stein99a}). The observed data from the model \eqref{eq:obs.model} are $Y_n=(Y(s_1),\ldots,Y(s_n))^\top$. Parameter estimation and prediction of $Y(\cdot)$ at a new spatial location (known as kriging) is based on $Y_n$. For simplicity, we call $\alpha$ the range parameter in the rest of the paper.

In Bayesian inference on GP models (\citet{HanSte93}, \citet{DeO97}), it is common practice to assign prior distributions on the regression coefficient $\beta$ and the covariance parameters $(\sigma^2,\alpha)$, and the prediction of $Y(s^*)$ at a new location $s^*$ is based on the posterior distribution of $(\beta,\sigma^2,\alpha)$. There is abundant literature in Bayesian spatial statistics on speeding up the costly GP posterior computation for spatial datasets with a large sample size $n$ (\citet{Banetal08}, \citet{SanHua12}, \citet{Datetal16}, \citet{Guhetal17}, \citet{Heatonetal18}, etc.) However, there is a clear lack of theoretical understanding of the asymptotic properties of the Bayesian posterior distributions of covariance parameters $(\sigma^2,\alpha)$. This theory is important because in Bayesian inference, instead of taken as fixed values, the covariance parameters $(\sigma^2,\alpha)$ are randomly drawn from their posterior using sampling algorithms such as Markov chain Monte Carlo (MCMC), which eventually affect the posterior prediction performance of the GP model.

To illustrate our motivation, we fit a Bayesian universal kriging model in \eqref{eq:obs.model} to the sea surface temperature (SST) data. The data is obtained from National Oceanographic Data Centres (NODC) World Ocean Database (\href{https://www.ncei.noaa.gov/products/world-ocean-database}{https://www.ncei.noaa.gov/products/world-ocean-database}) and the entire data corresponds to sea surface temperature measured by remote sensing satellites on 16th August 2016. The data we test come from the Pacific Ocean between $45^\circ$--$48^\circ$ north latitudes and $150^\circ$--$153^\circ$ west longitudes. The original dataset is high-resolution on a $0.025^\circ\times 0.025^\circ$ fine grid. We choose subsets of size $\{400, 800, 1200, 1600, 2000\}$ on equispaced grids. For the regressors $\bbm(\cdot)$, we include all $p=10$ monomials of the latitude and longitude up to degree 3, since on average SST is lower at higher latitudes. We set $\nu=1/2$, and assign a flat prior $\pi(\beta)\propto 1$ on $\beta$, an inverse gamma prior with shape and rate parameters both equal to 2 on $\sigma^2$ and an independent $\text{Uniform}(0.01,300)$ prior on $\alpha$. The marginal posterior densities of $\alpha$ and $\theta=\sigma^2\alpha$ are shown in Figure \ref{fig:sst} below. As the sample size increases, the marginal posterior density of the parameter $\theta=\sigma^2\alpha$ seems to contract faster with $n$ than that of the range parameter $\alpha$. Even with sample size $n=2000$, the posterior of $\alpha$ still has a relatively large uncertainty. It is natural to ask the following questions: (i) Do the posteriors of $(\sigma^2,\alpha)$ (or $(\theta,\alpha)$) converge, and if so, at what rates? (ii) How does the posterior uncertainty in $(\sigma^2,\alpha)$ affect the posterior prediction of the response $Y(\cdot)$ at a new location?

\begin{figure}
\centering
\includegraphics[width=0.8\textwidth]{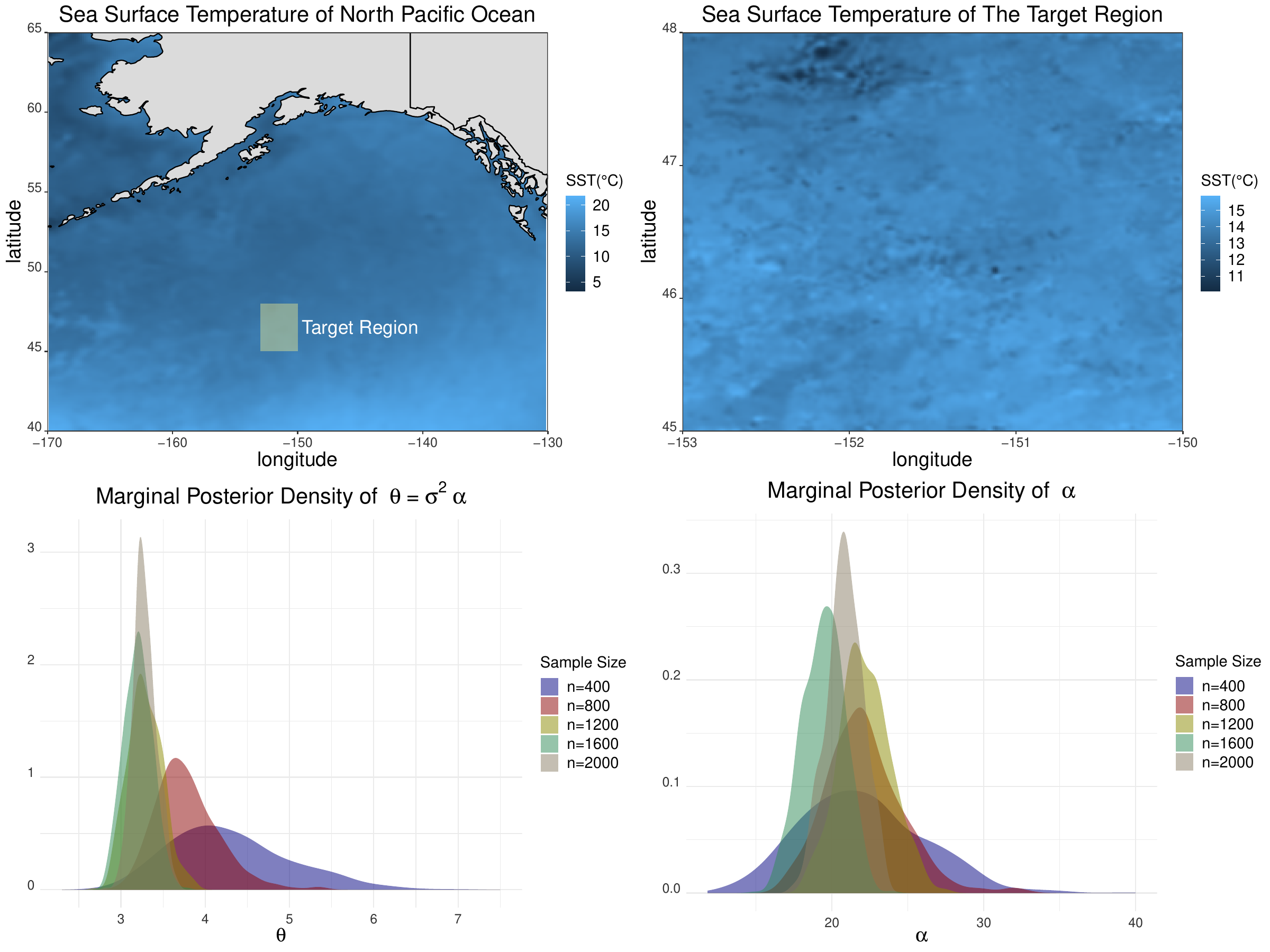}
\caption{Example of the Sea Surface Temperature (SST) data. Top left: The SST data in North Pacific Ocean and the target region of our sampled data. Top right: The SST data in the target region. Bottom left: The marginal posterior densities of $\theta=\sigma^2\alpha$ for sample sizes $n=400,800,1200,1600,2000$.  Bottom right: The marginal posterior densities of $\alpha$ for sample sizes $n=400,800,1200,1600,2000$. The posterior densities are based on 2000 MCMC draws.}
\label{fig:sst}
\end{figure}

We provide an answer to (i) by studying the limiting posterior distributions of the covariance parameters $(\sigma^2,\alpha)$ in the Mat\'ern covariance function in \eqref{eq:MaternCov}, under the \textit{fixed-domain asymptotics (or infill asymptotics)} framework (\citet{Stein88}, \citet{Stein99a}, \citet{Zhang04}). We further answer (ii) and show that the randomness in  $(\sigma^2,\alpha)$ in general does not affect the posterior prediction performance. To the best of our knowledge, this paper is the first theoretical work on the fixed-domain asymptotics for the Bayesian posterior distribution of the finite dimensional parameters in Gaussian process covariance functions. In the following, we explain the reasons we adopt the fixed-domain asymptotics regime and the main technical challenges.

\subsection{Why fixed-domain asymptotics?} \label{subsec:reasons}

In the fixed-domain asymptotics regime, the domain $\Scal$ remains fixed and bounded regardless of the increasing sampling size $n$. This implies that as $n$ goes to infinity, the sampling points $\Scal_n$ become increasingly dense in the domain $\Scal$, leading to increasingly stronger dependence between adjacent observations in $Y_n$. Besides the fixed-domain asymptotics regime, there are also increasing-domain asymptotics (\citet{MarMar84}) and mixed-domain asymptotics (\citet{Chaetal17}), in which the domain is assumed to increase as $n$ goes to infinity and therefore the minimum distance between two adjacent sampling points is either not decreasing or decreasing slowly with $n$.

Compared to these alternatives, the fixed-domain setup has several advantages. First and foremost, a fixed domain matches up with the reality in many spatial applications. The advances in remote sensing technology make it possible to collected spatial data in larger volume and higher resolution in a given region (\citet{Sunetal18}). The motivating example above of the SST data from NODC has millions of observations with high-resolution on the $0.025^\circ\times 0.025^\circ$ fine grid (about $2\sim 4$km range). Second, since the model \eqref{eq:obs.model} has a stationary Mat\'ern covariance function, this stationarity assumption of GP is more likely to hold on a fixed domain rather than an expanding domain. Therefore, the fixed-domain asymptotics regime is more suitable for interpolation of spatial processes; see Section 3.3 of \citet{Stein99a} for a cogent argument. Third, \citet{ZhaZim05} has shown that the fixed-domain asymptotics has better parameter estimation performance than the increasing-domain asymptotics.

\subsection{What are the main difficulties in Bayesian fixed-domain asymptotics?} \label{subsec:difficulties}
Theoretically, the increasingly stronger spatial dependence among the observed data $Y_n$ in fixed-domain asymptotics leads to a lack of consistent estimation for the covariance parameters $(\sigma^2,\alpha)$ (\citet{Zhang04}) and therefore poses significant challenges to theory development. When the dimension of sampling points $d=1,2,3$, a well known fixed-domain asymptotics result \citet{Zhang04} says that it is only possible to consistently estimate the \textit{microergodic parameter} $\theta=\sigma^2\alpha^{2\nu}$ in an isotropic Mat\'ern covariance function, but not the individual variance parameter $\sigma^2$ and the range parameter $\alpha$. The microergodic parameter is defined to be the parameter that uniquely determines the Gaussian measure induced by a Gaussian process, such that different values of microergodic parameter will lead to mutually orthogonal Gaussian measures; see Section 6.2 of \citet{Stein99a} for a detailed explanation on this definition. On the other hand, both the variance and range parameters $(\sigma^2,\alpha)$ can be consistently estimated if $d\geq 5$, with the case of $d=4$ still open (\citet{And10}). Nevertheless, the cases with $d=1,2,3$ are of primary interest in spatial and spatiotemporal applications and will be our main focus.

The standard Bayesian asymptotic theory consists of results such as posterior consistency, posterior convergence rates, and the Bernstein-von Mises (BvM) theorem (\citet{GhoVan17}). For parametric models, the BvM theorem typically relies on the local asymptotic normality (LAN) condition and the existence of uniformly consistent tests; see for example, Chapter 10 in \citet{Van98}. Since no consistent frequentist estimator exists for $(\sigma^2,\alpha)$ under fixed-domain asymptotics, one cannot expect to establish posterior consistency for $(\sigma^2,\alpha)$. Instead, we will consider the microergodic parameter $\theta=\sigma^2\alpha^{2\nu}$ which can be consistently estimated, and reparametrize the covariance function \eqref{eq:MaternCov} by $(\theta,\alpha)$. \citet{Cro76} is an early work on the asymptotic normality of maximum likelihood estimator (MLE) in the presence of dependent observations and nuisance parameters. We will establish the LAN condition for the microergodic parameter $\theta$, \textit{uniformly} over a wide range of values of the ``nuisance" range parameter $\alpha$. Such a uniform LAN condition based on data with increasingly stronger dependence is new in the literature and differs significantly from the LAN in classic parametric models with independent or weakly dependent data. The asymptotic normality for microergodic parameter $\theta$ is crucial and guarantees the posterior prediction performance of $Y(\cdot)$ at a new location.

For Bayesian inference on the GP covariance parameters, the only theoretical work we are aware of is \citet{ShaRup12}, who have worked under the increasing-domain asymptotics regime and have established that the joint posterior of all parameters in the tapered covariance functions converges to a limiting normal distribution. This is similar to the classic BvM theorem since the dependence among data does not get stronger under increasing-domain asymptotics. A key assumption in \citet{ShaRup12} is that the observed covariance matrix have lower and upper bounded eigenvalues, which no longer holds under fixed-domain asymptotics.

We define some universal notation. Let $\RR^+=(0,+\infty)$. For two positive sequences $a_n$ and $b_n$, we use $a_n\preceq b_n$ and $b_n\succeq a_n$ to denote the relation $\limsup_{n\to\infty} a_n/b_n<+\infty$, and $a_n\asymp b_n$ to denote the relation $a_n\preceq b_n$ and $a_n\succeq b_n$. For any integers $k,m$, we let $I_k$ be the $k\times k$ identity matrix, $0_k$ and $1_k$ be the $k$-dimensional column vectors of all zeros and all ones, $0_{k\times m}$ be the $k\times m$ zero matrix. For any generic matrix $A$, $cA$ denotes the matrix of $A$ with all entries multiplied by the number $c$, and $|A|$ denotes the determinant of $A$. If $A$ is positive semidefinite, then $\lambda_{\min}(A)$ and $\lambda_{\max}(A)$ denote the smallest and largest eigenvalues of $A$. Let $\Ncal(\mu,\Sigma)$ be the normal distribution with mean $\mu$ and covariance matrix $\Sigma$. Sometimes to highlight the random variable $Z \sim \Ncal(\mu,\Sigma)$, we also write  $\Ncal(z;\mu,\Sigma)$ and the normal measure as $\Ncal(\ud z;\mu,\Sigma)$.

The remainder of the paper is organized as follows. In Section \ref{sec:main.bvm} we introduce the basic model setup and present the main theorems on limiting posterior distribution of covariance parameters under fixed-domain asymptotics. Section \ref{sec:PAE} presents the theory on asymptotic efficiency in posterior prediction. Section \ref{sec:simulation} presents some empirical results from simulation study to verify the main theory. Section \ref{sec:discussion} includes some discussion on further extensions. The technical proofs of all theorems, propositions, lemmas, corollaries and additional simulation results are in the Supplementary Material.

\section{Limiting Posterior Distribution for Covariance Parameters} \label{sec:main.bvm}
\subsection{Bayesian Model Setup}\label{subsec:model}
We consider the Bayesian estimation of $(\beta,\sigma^2,\alpha)$ in the model \eqref{eq:obs.model} based on the observed data $Y_n$. Throughout the paper, we assume that the domain dimension satisfies $d\in\{1,2,3\}$, and that the smoothness parameter $\nu>0$ is fixed and known. Estimation of the smoothness parameter $\nu$ is an important research topic with some recent developments in frequentist literature (\citet{Loh15}, \citet{Lohetal20}), but is beyond the scope of the current paper. We let the true parameter values in the Mat\'ern covariance function that generates $X$ be $(\sigma^2_0,\alpha_0)$ and let the true regression coefficient vector be $\beta_0$. We use the notation $X\sim \gp(0,\sigma_0^2K_{\alpha_0,\nu})$ and hence $Y\sim \gp(\bbm^\top\beta_0, \sigma_0^2K_{\alpha_0,\nu})$.

Let $Y_n=(Y(s_1),\ldots,Y(s_n))^\top$. Let $M_n$ be the $n\times p$ matrix by stacking the row vectors $\bbm(s_i)^\top$ for $i=1,\ldots,n$. Throughout the paper, we assume that $M_n$ is a rank-$p$ matrix without loss of generality, since all our results are asymptotic with $n\to\infty$. Let $R_{\alpha}$ be the implied $n\times n$ Mat\'ern correlation matrix on $\Scal_n$ indexed by $\alpha$, whose $(i,j)$th entry is $R_{\alpha,ij}=K_{\alpha,\nu}(s_i-s_j)$, for $i,j\in\{1,\ldots,n\}$. We omit the dependence of $R_{\alpha}$ on $\nu$. The covariance matrix of $X_n$ is then $\sigma^2 R_{\alpha}$. Therefore, the model \eqref{eq:obs.model} can be equivalently written as $Y_n=M_n\beta + X_n$. The log-likelihood function based on $Y_n$ is
\begin{align} \label{eq:loglik}
\Lcal_n(\beta,\sigma^2,\alpha) &= -\frac{n}{2}\log \sigma^2 - \frac{1}{2}\log |R_{\alpha}| - \frac{1}{2\sigma^2} (Y_n-M_n\beta)^\top R_{\alpha}^{-1} (Y_n-M_n\beta).
\end{align}

We study the Bayesian posterior distribution based on the log-likelihood \eqref{eq:loglik}.  We follow the common practice in Bayesian spatial modeling literature (\citet{Banetal08}, \citet{SanHua12}, \citet{Datetal16}, \citet{Guhetal17}, \citet{Heatonetal18}, \citet{Peretal21}, etc.) and assign the conjugate normal prior on $\beta$, given by
\begin{align}\label{eq:beta.prior}
\beta~|~\sigma^2,\alpha\sim \Ncal\big(0_p,\sigma^2\Omega_{\beta}^{-1}\big),
\end{align}
which uses a rescaling with $\sigma^2$, and the prior precision matrix $\Omega_{\beta}\in \RR^{p\times p}$ is assumed to be symmetric positive semidefinite. Here we can set the prior mean to be $0_p$ without any loss of generality. This is because if the prior is $\beta|\sigma^2,\alpha\sim \Ncal(\mu_{\beta},\sigma^2\Omega_{\beta}^{-1})$ and the prior mean is $\mu_{\beta}\neq 0_p$, we can always define a new response variable $Y'(s)=Y(s)-\bbm(s)^\top \mu_{\beta}$, the new regression coefficient vector $\beta'=\beta-\mu_{\beta}$, and rewrite the original model \eqref{eq:obs.model} as $Y'(s)=\bbm(s)^\top \beta' + X(s) $ for $s\in \Scal$, where $Y'(s)$ is still fully observable on $\Scal_n$ given that $\bbm(\cdot)$ is observable and $\mu_{\beta}$ is known. Furthermore, we allow the precision matrix $\Omega_{\beta}$ to be arbitrarily small, leading to a prior of $\beta$ with arbitrarily large variance. In particular, all our later theory covers the extreme case of improper noninformative prior $\pi(\beta|\sigma^2,\alpha)\propto 1$ (\citet{Beretal01}, \citet{Guetal18}), which corresponds to $\Omega_{\beta}=0_{p\times p}$. The joint posterior density of $(\beta,\sigma^2,\alpha)$ is then $\pi(\beta,\sigma^2,\alpha|Y_n)\propto \exp\{\Lcal_n(\beta,\sigma^2,\alpha)\}\pi(\beta|\sigma^2,\alpha)\pi(\sigma^2,\alpha)$. Since $\pi(\beta|\sigma^2,\alpha)$ is the normal prior density, it is straightforward to obtain the conditional posterior of $\beta$:
\begin{align} \label{eq:beta.post}
& \beta~|~\sigma^2,\alpha,Y_n\sim \Ncal\left(\widetilde\beta_{\alpha}, \sigma^2\big(M_n^\top R_{\alpha}^{-1} M_n + \Omega_{\beta}\big)^{-1}\right),
\end{align}
where $\widetilde\beta_{\alpha} = \big(M_n^\top R_{\alpha}^{-1} M_n + \Omega_{\beta}\big)^{-1}M_n^\top R_{\alpha}^{-1} Y_n$ and the subscript is to highlight its dependence on $\alpha$ but not $\sigma^2$. We can further integrate out $\beta$ and obtain the marginal posterior density of the covariance parameters $(\sigma^2,\alpha)$. We write $\pi(\sigma^2,\alpha|Y_n)\propto \exp\{\Lcal_n(\sigma^2,\alpha)\} \pi(\sigma^2,\alpha)$, where the \textit{restricted log-likelihood} $\Lcal_n(\sigma^2,\alpha)$ is given by
\begin{align} \label{eq:loglik2}
\Lcal_n(\sigma^2,\alpha) &= -\frac{1}{2\sigma^2} Y_n^\top  \left[R_{\alpha}^{-1} - R_{\alpha}^{-1} M_n \big(M_n^\top R_{\alpha}^{-1} M_n + \Omega_{\beta}\big)^{-1} M_n^\top R_{\alpha}^{-1}  \right] Y_n   \nonumber \\
&\quad -\frac{n-p}{2}\log \sigma^2 - \frac{1}{2}\log |R_{\alpha}| - \frac{1}{2}\log \big|M_n^\top R_{\alpha}^{-1} M_n + \Omega_{\beta}\big|  .
\end{align}

In spatial statistical theory, it is well known (\citet{Zhang04}) that the parameters $(\sigma^2,\alpha)$ cannot be consistently estimated under fixed-domain asymptotics. The main reason is that for two Gaussian measures $\gp(0,\sigma_j^2K_{\alpha_j,\nu})$ ($j=1,2$) on the space of sample paths on the domain $\Scal=[0,T]^d$ and $d\in \{1,2,3\}$, they are equivalent (or mutually absolutely continuous) as long as $\sigma_1^2\alpha_1^{2\nu}=\sigma_2^2\alpha_2^{2\nu}$, and they are orthogonal otherwise. As a result, one cannot tell from a finite sample which parameter values $(\sigma_j^2,\alpha_j)$ ($j=1,2$) are correct. Empirically, this phenomenon has been also observed (\citet{And10}, \citet{Fugetal19}). Despite the lack of consistent estimator for $(\sigma^2,\alpha)$, the microergodic parameter $\theta=\sigma^2\alpha^{2\nu}$ can still be consistently estimated (\citet{Zhang04}). For a fixed $\alpha>0$, we maximize $\Lcal_n(\sigma^2,\alpha)$ with respect to $\sigma^2$ (and so $\theta$) to derive the \textit{restricted maximum likelihood estimator (REML)}, given by
\begin{align}\label{tildetheta1}
\widetilde \sigma^2_{\alpha} & = \frac{1}{n-p}Y_n^\top  \left[R_{\alpha}^{-1} - R_{\alpha}^{-1} M_n \big(M_n^\top R_{\alpha}^{-1} M_n + \Omega_{\beta}\big)^{-1} M_n^\top R_{\alpha}^{-1}  \right] Y_n , ~~\widetilde\theta_{\alpha} = \alpha^{2\nu} \widetilde\sigma^2_{\alpha} .
\end{align}
In \eqref{tildetheta1}, we have slightly extended the meaning of REML such that we can account for general prior precision matrix $\Omega_{\beta}$, including the special case of $\Omega_{\beta}=0_{p\times p}$ where $\beta$ can be viewed as normal random effects of $\bbm(\cdot)$, such that $\widetilde \sigma^2_{\alpha}$ (and $\widetilde\theta_{\alpha}$) can be viewed as the conventional REML of $\sigma^2$ (and $\theta$) in random effects models. We can plug in $\widetilde \theta_{\alpha}$ in \eqref{eq:loglik} to obtain the \emph{profile restricted log-likelihood of $\alpha$} (up to an additive constant), which plays an important role in our theory:
\begin{align}\label{def:prologlik}
\widetilde \Lcal_n(\alpha) & \equiv \Lcal_n(\alpha^{-2\nu}\widetilde \theta_{\alpha}, \alpha) \nonumber \\
&= -\frac{n-p}{2}\log  \left\{\frac{1}{n-p}Y_n^\top  \left[R_{\alpha}^{-1} - R_{\alpha}^{-1} M_n \big(M_n^\top R_{\alpha}^{-1} M_n + \Omega_{\beta}\big)^{-1} M_n^\top R_{\alpha}^{-1}  \right] Y_n \right\} \nonumber \\
&\quad - \frac{1}{2}\log \left|R_{\alpha}\right| - \frac{1}{2}\log \big|M_n^\top R_{\alpha}^{-1} M_n + \Omega_{\beta}\big| - \frac{n-p}{2} .
\end{align}
The frequentist asymptotic normality for the MLE of $\theta$ has been studied for the model \eqref{eq:obs.model} without the regression term, i.e., $Y(\cdot)\equiv X(\cdot)\sim \gp(0,\sigma^2 K_{\alpha,\nu})$. For this simplified model, \citet{Ying91} first studied the special case of $d=1$ and $\nu=1/2$, followed by \citet{Zhang04}, \citet{Duetal09}, \citet{WangLoh11}, and \citet{KauSha13} for a general $\nu>0$. If $\alpha\in [\alpha_1,\alpha_2]$ for some constants $0<\alpha_1<\alpha_2<\infty$, the MLE of $\theta$, denoted by $\widehat\theta$, satisfies that $\sqrt{n}(\widehat\theta-\theta_0)\overset{\Dcal}{\rightarrow} \mathcal{N}(0,2\theta_0^2)$ as $n\to\infty$ under fixed-domain asymptotics, where $\theta_0=\sigma_0^2\alpha_0^{2\nu}$ is the true value, and $\overset{\Dcal}{\rightarrow} $ is the convergence in distribution.

We study the fixed-domain asymptotic limit for the Bayesian posterior distribution of $(\sigma^2,\alpha)$ based on the log-likelihood \eqref{eq:loglik2}. We reparametrize the model using $(\theta,\alpha)$, with $\theta=\sigma^2\alpha^{2\nu}$ being the microergodic parameter. This reparametrization has been suggested in \citet{Stein99a} (p.175) and also used in recent Bayesian GP works such as \citet{Fugetal19}. For the consistency of notation, we will still maintain the parametrization of $(\sigma^2,\alpha)$ for the log-likelihood functions and quantities related to the probability distributions, such as $P_{(\beta,\sigma^2,\alpha)}$ for the probability distribution of $\gp(m^\top \beta,\sigma^2 K_{\alpha,\nu})$. The change of variable from $\sigma^2$ to $\theta=\sigma^2\alpha^{2\nu}$ is often clear from the context. We assign prior distributions on $(\theta,\alpha)$ and write the joint prior density as $\pi(\theta,\alpha)=\pi(\theta|\alpha)\pi(\alpha)$. The joint posterior density of $(\theta,\alpha)$ is given by
\begin{align}\label{jointpost1}
\pi(\theta,\alpha| Y_n) &= \frac{\exp\left\{ \Lcal_n(\theta/\alpha^{2\nu},\alpha) \right\} \pi(\theta|\alpha)\pi(\alpha)}{\int_0^{\infty}\int_0^{\infty} \exp\left\{\Lcal_n(\theta'/\alpha^{'2\nu},\alpha') \right\} \pi(\theta'|\alpha')\pi(\alpha')\ud \alpha' \ud \theta'}.
\end{align}
We will use $\Pi(\ud \theta, \ud \alpha |Y_n)$ to denote the posterior probability measure with the density in \eqref{jointpost1}.

\subsection{Main Results} \label{subsec:main.res}
We first present the limiting posterior distribution of $\theta$ conditional on a fixed $\alpha>0$. Let $L_2(\Scal)$ be the space of square integrable functions on $\Scal$ and $\|f\|_2$ be the $L_2(\Scal)$ norm of $f$ for any $f\in L_2(\Scal)$. Let $\mathsf{j}=(j_1,\ldots,j_d)$ with $j_1,\ldots,j_d\in \NN$, $|j|=\sum_{i=1}^d j_i$, and $\mathsf{D}^{\mathsf{j}}$ be the partial differentiation operator of order $\mathsf{j}$. For $k>0$, define the Sobolev space $\Wcal_2^k(\Scal)=\Big\{f\in L_2(\Scal): \|f\|^2_{\Wcal_2^k(\Scal)} = \sum_{\mathsf{j}\in \NN^d: |\mathsf{j}|\leq k} \left\|\mathsf{D}^{\mathsf{j}}f\right\|_2^2 <\infty\Big\}$. We make the following assumptions.

\begin{enumerate}[label=(A.\arabic*)]
\item \label{assump.m.func} $\bbm_j\in  \Wcal_{2}^{\nu+d/2}(\Scal)$ for each $j=1,\ldots,p$. $M_n$ is a rank-$p$ matrix for all $\Scal_n$ with $n\geq p$.
\item \label{prior.1} The prior of $\beta$ given $(\sigma^2,\alpha)$ is $\Ncal(0_p,\sigma^2\Omega_{\beta}^{-1})$ for a symmetric positive semidefinite matrix $\Omega_{\beta}$. The conditional prior density of $\theta$ given $\alpha$, $\pi(\theta|\alpha)$, is a proper prior density that is continuously differentiable in $\theta$, continuous in $\alpha$, and finite everywhere for all $\theta\in \RR^+$ and $\alpha\in \RR^+$. $\pi(\theta|\alpha)$ does not depend on $n$. $\pi(\theta_0|\alpha)>0$ for all $\alpha>0$.
\end{enumerate}
Assumption \ref{assump.m.func} is the regularity assumption on the regression functions $\bbm_1,\ldots,\bbm_p$. By Theorem 10.35 of \citet{Wen05}, $\Wcal_{2}^{\nu+d/2}(\Scal)$ is norm equivalent to the reproducing kernel Hilbert space (RKHS) associated with the Mat\'ern kernel $\sigma_0^2 K_{\alpha_0,\nu}$. As a result, Assumption \ref{assump.m.func} implies that $\bbm_1(\cdot),\ldots,\bbm_p(\cdot)$ are smoother functions than the sample paths from $\gp(0,\sigma^2 K_{\alpha,\nu})$ for any $(\sigma^2,\alpha)\in \RR^+ \times \RR^+$; see for example, Corollary 4.15 of \citet{Kanetal18}. Such a smoothness assumption is necessary. Otherwise, if $\bbm_1,\ldots,\bbm_p$ are rougher functions than the sample path of $X$, their roughness will overwhelm the information contained in the smoother $\gp(0,\sigma_0^2K_{\alpha_0,\nu})$, and one cannot expect to estimate any covariance parameter consistently, including $\theta$. As argued in p.12 of \citet{Stein99a}, $\bbm_1(\cdot),\ldots,\bbm_p(\cdot)$ in applications are often highly regular functions such as monomials, which are infinitely differentiable on $\Scal$ and therefore, satisfy Assumption \ref{assump.m.func}. Assumption \ref{prior.1} on $\pi(\theta|\alpha)$ is mild and satisfied in most applications.

For two probability measures $P_1,P_2$, let $\|P_1(\cdot)-P_2(\cdot)\|_{\tv}=\sup_{\Acal}|P_1(\Acal)-P_2(\Acal)|$, where the supremum is taken over all measurable sets $\Acal$.
\begin{theorem}[Limiting Distribution for Conditional Posterior] \label{thm:bvm1:theta}
Suppose that $\alpha>0$ is fixed and does not depend on $n$. Under Assumptions \ref{assump.m.func} and \ref{prior.1}, the REML $\widetilde\theta_{\alpha}$ defined in \eqref{tildetheta1} is asymptotically normal, with $\sqrt{n}\big(\widetilde \theta_{\alpha}-\theta_0\big)\overset{\Dcal}{\rightarrow} \Ncal(0,2\theta_0^2)$ as $n\to\infty$. Furthermore, the conditional posterior distribution of $\theta$ given $\alpha>0$ satisfies that
\begin{align}\label{limit:dist:fix}
\left\|\Pi(\ud\theta|Y_n,\alpha) - \Ncal \left(\ud\theta \big| \widetilde\theta_{\alpha}, 2\theta_0^2/n \right) \right\|_{\tv} \preceq n^{-1/2}\log^3 n \rightarrow 0,
\end{align}
as $n\to\infty$ almost surely $P_{(\beta_0,\sigma_0^2,\alpha_0)}$, where $\widetilde\theta_{\alpha}$ is given in \eqref{tildetheta1}, and $\Pi(\cdot|Y_n,\alpha)$ is the conditional posterior probability measure of $\theta$ given a fixed $\alpha>0$ with the density
\begin{align}\label{post:theta:rho}
\pi(\theta|Y_n,\alpha)=\frac{\exp\left\{\Lcal_n(\theta/\alpha^{2\nu},\alpha)\right\}\pi(\theta|\alpha)}{\int_0^{\infty} \exp\left\{\Lcal_n(\theta'/\alpha^{2\nu},\alpha)\right\}\pi(\theta'|\alpha)\ud \theta'}.
\end{align}
\end{theorem}
Theorem \ref{thm:bvm1:theta} shows that under fixed-domain asymptotics, the REML $\widetilde\theta_{\alpha}$ is asymptotically normal, and the conditional posterior $\pi(\theta|Y_n,\alpha)$ is asymptotically close the normal distribution $\mathcal{N}(\widetilde\theta_{\alpha},2\theta_0^2/n)$ in total variation distance. Some comments are in order.

First, to the best of our knowledge, Theorem \ref{thm:bvm1:theta} is the first in the literature to establish both frequentist and Bayeisan asymptotic normality for the microergodic parameter $\theta$ for any $\nu>0$ and $d\in\{1,2,3\}$ in the universal kriging model \eqref{eq:obs.model} with regression terms $\bbm(\cdot)^\top \beta$. Most of the existing frequentist fixed-domain asymptotic theory has considered either only the GP model with mean zero and no regression terms (\citet{Zhang04}, \citet{Duetal09}, \citet{And10}, \citet{WangLoh11}, \citet{KauSha13}, \citet{BacLag20}), or only for some particular values of $\nu$ (such as $\nu=1/2$ in \citet{Ying91}, \citet{Ying93}, \citet{Chenetal00}, \citet{Chaetal14}, \citet{Veletal17}, \citet{Bacetal19}, and $\nu=3/2$ in \citet{Loh05}). Theorem 3 of \citet{Ying91} has shown the asymptotic normality for the MLE of $\theta$ in the GP model with regression terms, but only for the special case of $\nu=1/2$ and $d=1$, and their proof techniques cannot be generalized to any $\nu>0$ and $d>1$. Our proof is based on the general RKHS theory and spectral analysis of isotropic Mat\'ern covariance functions; see Section S1 of the Supplementary Material. Theorem 3 of \citet{Ying91} almost needs that $\bbm_1,\ldots,\bbm_p \in \Wcal_2^1([0,1])$ in the special case of $\nu=1/2$ and $d=1$, i.e.,  they are bounded functions with square integrable derivatives (following the comments after their Theorem 3), which is exactly the same as the space $\Wcal_2^{\nu+d/2}(\Scal)$ assumed in our Assumption \ref{assump.m.func}.

Second, if the model \eqref{eq:obs.model} does not have regression terms $\bbm(\cdot)^\top \beta$, i.e., if $p=0$ and we observe $Y_n=X_n$ directly from $\gp(0,\sigma_0^2K_{\alpha_0,\nu})$, then the REML $\widetilde\theta_{\alpha}$ in \eqref{tildetheta1} coincides with the MLE of $\theta$, and $2\theta_0^2$ is also the asymptotic variance of this MLE (\citet{WangLoh11}, \citet{KauSha13}).

Third, the posterior convergence of \eqref{limit:dist:fix} Theorem \ref{thm:bvm1:theta} has a similar format to the classic BvM theorem in regular parametric models for independent data, such as Theorem 8.2 in \citep{LehCas98} and Theorem 10.1 in \citep{Van98}, where the limiting normal distribution is centered at the MLE with variance equal to the asymptotic variance of MLE. However, the classic BvM theorem usually relies on the LAN condition and the existence of uniformly consistent tests (Theorem 10.1 in \citep{Van98}) which can be readily verified for models with independent and weakly dependent data. The main technical challenge for proving Theorem \ref{thm:bvm1:theta} is to establish the LAN condition for data with increasingly stronger dependence under fixed-domain asymptotics. We need the asymptotic normality of the REML $\widetilde\theta_{\alpha}$ at a given range parameter $\alpha>0$ which can be different from the true $\alpha_0$. Our proof leverages the spectral analysis of Mat\'ern covariance functions (see Section S1.4 in the Supplementary Material (Li 2020)), which has also been used in the previous works for the MLE of $\theta$ for GP with mean zero (\citep{Duetal09}, \citep{WangLoh11}, and \citep{KauSha13}), though they have not considered the model with regression terms as ours. Finally, we provide an explicit convergence rate $n^{-1/2}\log^3 n$ for the convergence in total variation distance. The $\log^3 n$ term is mainly used to ensure the strong mode of almost sure convergence.

In most spatial applications, the range parameter $\alpha$ is unknown and assigned a prior $\pi(\alpha)$. Next, we present a much stronger theorem for the limit of the joint posterior distribution of $(\theta,\alpha)\in \RR^+ \times \RR^+$. The consistency of the REML of $\theta$ and the nonexistence of consistent frequentist estimator for $\alpha$ indicates that the posterior of $\theta$ should converge to a normal limit, while the posterior of $\alpha$ does not necessarily converge to any fixed value under fixed-domain asymptotics. We prove this idea rigorously.

We define two small positive constants $\underkappa$ and $\overkappa$ that depend on the smoothness $\nu>0$ and the dimension $d$ ($d\in \{1,2,3\}$), together with two deterministic sequences $\underline \alpha_n$ and $\overline \alpha_n$:
\begin{align}\label{eq:2kappa}
& \underkappa = \frac{1}{2} \min\left\{\frac{0.9}{(2d+0.94)(8\nu+3d-0.9)},~~ \frac{1}{4(3\nu+d)}, ~~0.01\right\}, \quad \underline \alpha_n= n^{-\underkappa}, \nonumber \\
&\overkappa = \frac{1}{2} \min\left\{\frac{0.9}{(2d+0.94)(8\nu+5d+0.9)},~~ \frac{1}{2(2\nu+d)},~~0.01\right\},  \quad \overline \alpha_n = n^{\overkappa}.
\end{align}
The choices of $\overkappa$ and $\underkappa$ in \eqref{eq:2kappa} are not unique and can be replaced by other sufficiently small positive numbers; see Lemma S.20 in the Supplementary Material. By definition, $\underline \alpha_n \to 0$ and $\overline \alpha_n\to +\infty$ as $n\to\infty$, and both are in slow polynomial rates. A key result below is that uniformly for all $\alpha$ in the slowly expanding interval $[\underline \alpha_n,\overline \alpha_n]$, the difference between $\widetilde \theta_{\alpha}$ and $\widetilde \theta_{\alpha_0}$ converges to zero at a faster rate than $n^{-1/2}$.
\begin{lemma}[Monotonicity and Uniform Convergence of $\widetilde \theta_{\alpha}$] \label{lem:dimension reduction}
Suppose that Assumption \ref{assump.m.func} holds. Then for the REML $\widetilde \theta_{\alpha}$ defined in \eqref{tildetheta1},
\begin{itemize}[leftmargin=5mm]
\item[(i)] $\widetilde \theta_{\alpha}$ is a non-decreasing function of $\alpha$ for all $\alpha \in \RR^+$;
\item[(ii)] There exists a large integer $N_1$ and a positive constant $\tau\in (0,1/2)$ that only depend on $\nu,d,T,\beta_0,\theta_0,\alpha_0$ and the $\Wcal_2^{\nu+d/2}(\Scal)$ norms of $\bbm_1(\cdot),\ldots,\bbm_p(\cdot)$, such that for all $n>N_1$,
\begin{align}
&\pr \left( \sup_{\alpha \in [\underline {\alpha}_n, \overline {\alpha}_n]} \sqrt{n} \left|\widetilde \theta_{\alpha} - \widetilde \theta_{\alpha_0} \right| \leq \theta_0 n^{-\tau} \right) \geq 1- \exp(-2\log^2 n), \nonumber
\end{align}
where $\pr(\cdot)$ denotes the probability under the true probability measure $P_{(\beta_0,\sigma_0^2,\alpha_0)}$.
\end{itemize}
\end{lemma}
Lemma \ref{lem:dimension reduction} involves a new discovery in Part (i) that the REML $\widetilde \theta_{\alpha}$ is monotone in $\alpha$ for the universal kriging model \eqref{eq:obs.model}. The monotonicity of $\widetilde \theta_{\alpha}$ for the universal kriging model \eqref{eq:obs.model} has significantly extended the previous work of \citet{KauSha13} which only considered the MLE of $\theta$ for GP with mean zero. Previously \citet{WangLoh11} has shown that $\big|\widetilde \theta_{\alpha} - \widetilde \theta_{\alpha_0} \big|$ can be small, but only for a \textit{fixed and known} value of range parameter $\alpha$ and only for GP with mean zero. In Part (ii) of Lemma \ref{lem:dimension reduction}, we make a novel utilization of the monotonicity of $\widetilde\theta_{\alpha}$ in $\alpha$, and prove in Lemma \ref{lem:dimension reduction} that the difference $\big|\widetilde \theta_{\alpha} - \widetilde \theta_{\alpha_0} \big|$ can be uniformly small over an expanding interval $[\underline {\alpha}_n, \overline {\alpha}_n]$ for the more general model \eqref{eq:obs.model} with regression terms. Even though the REML $\widetilde \theta_{\alpha} $ defined in \eqref{tildetheta1} is in fact a stochastic process indexed by $\alpha$, our techniques using the monotonicity property of $\widetilde \theta_{\alpha}$ have the advantage of completely circumventing any empirical process argument. Our proof of Part (ii) also develops a much strengthened concentration inequality for $\widetilde\theta_{\alpha}$ using more detailed spectral analysis of Mat\'ern covariance functions than \citet{WangLoh11}.

The non-decreasing property of $\widetilde \theta_{\alpha}$ in \eqref{tildetheta1} is crucial for both establishing the uniform convergence of $\widetilde \theta_{\alpha}$ on the interval $[\underline {\alpha}_n, \overline {\alpha}_n]$ and understanding the asymptotic behavior of the joint posterior $\pi(\theta,\alpha|Y_n)$. Based on the uniform convergence in Lemma \ref{lem:dimension reduction}, a heuristic argument to extend the limiting conditional posterior in Theorem \ref{thm:bvm1:theta} to the joint posterior $\pi(\theta,\alpha|Y_n)$ is as follows: For each $\alpha\in [\underline {\alpha}_n, \overline {\alpha}_n]$, the conditional posterior $\pi(\theta|Y_n,\alpha)$ can be approximated by the normal distribution $\mathcal{N}(\widetilde\theta_{\alpha},2\theta_0^2/n)$. Since the center $\widetilde\theta_{\alpha}$ only differs from $\widetilde\theta_{\alpha_0}$ by a higher order term $O(n^{-1/2-\tau})$, this normal distribution can be further approximated by $\mathcal{N}(\widetilde\theta_{\alpha_0},2\theta_0^2/n)$, whose mean parameter only depends on the data $Y_n$ but not $\alpha$. Hence, the limiting distribution of $\theta$ is approximately independent of $\alpha$.

To solidify this idea, we need additional prior conditions such that the posterior probabilities outside the interval $[\underline \alpha_n,\overline \alpha_n]$ can be made small, such that the convergence to the normal limit inside $[\underline \alpha_n,\overline \alpha_n]$ is dominant in driving the asymptotics of the joint posterior $\pi(\theta,\alpha|Y_n)$. We specify the following general assumptions on the prior densities $\pi(\theta|\alpha)$ and $\pi(\alpha)$.
\begin{enumerate}[label=(A.\arabic*)]
\setcounter{enumi}{2}
\item \label{prior.2} There exist positive constants $C_{\pi,1}$, $C_{\pi,2}$, and $C_{\pi,3}$ that can depend on $\nu,d,T,\alpha_0,\theta_0$, such that $0<C_{\pi,1}+C_{\pi,2}<1/2$, $0<C_{\pi,3}<1$, and for $\underline\alpha_n$ and $\overline\alpha_n$ defined in \eqref{eq:2kappa}, for all sufficiently large $n$,
\begin{align}
&\sup_{\alpha\in [\underline\alpha_n, \overline\alpha_n]} \sup_{\theta\in (\theta_0/2,2\theta_0)}
\left|\frac{\partial \log \pi(\theta|\alpha)}{\partial \theta} \right| \leq n^{C_{\pi,1}},
\label{A2.1} \\
&\sup_{\alpha\in [\underline\alpha_n, \overline\alpha_n]} \sup_{\theta\in (\theta_0/2,2\theta_0)} \frac{\pi(\theta|\alpha)}{\pi(\theta_0|\alpha)} \leq n^{C_{\pi,2}}, \label{A2.2} \\
&\inf_{\alpha\in [\underline\alpha_n, \overline\alpha_n]} \log \pi(\theta_0|\alpha) \geq -n^{C_{\pi,3}}. \label{A2.3}
\end{align}
\item \label{prior.3} The marginal prior $\pi(\alpha)$ is a proper and continuous density function on $\RR^+$. $\pi(\alpha)$ does not depend on $n$. $\pi(\alpha_0)>0$. $\int_0^{\infty}\pi(\theta_0|\alpha)\pi(\alpha) \ud \alpha <\infty$. There exist positive constants $\underline{c_{\pi}}<(\nu+d/2)\underkappa$ and $\overline{c_{\pi}}<(\nu+d/2)\overkappa$ for $\underkappa$ and $\overkappa$ defined in \eqref{eq:2kappa}, such that for $\underline\alpha_n$ and $\overline\alpha_n$ defined in \eqref{eq:2kappa}, and for all sufficiently large $n$,
\begin{align}
&\max\left\{\int_0^{\underline\alpha_n} \alpha^{-n(\nu+d/2)}\pi(\alpha) \ud \alpha,  \int_0^{\underline\alpha_n} \alpha^{-n(\nu+d/2)} \pi(\theta_0|\alpha)\pi(\alpha) \ud \alpha \right\}\leq \exp\left(\underline{c_{\pi}} n\log n\right), \label{A3.1} \\
&\max\left\{\int_{\overline\alpha_n}^{\infty} \alpha^{n(\nu+d/2)} \pi(\alpha) \ud \alpha, \int_{\overline\alpha_n}^{\infty} \alpha^{n(\nu+d/2)} \pi(\theta_0|\alpha)\pi(\alpha) \ud \alpha \right\} \leq \exp\left(\overline{c_{\pi}} n\log n\right). \label{A3.2}
\end{align}
\end{enumerate}
We will discuss these two assumptions in greater detail after presenting our main theorem for the joint posterior of $(\theta,\alpha)$.
\begin{theorem}[Limiting Distributions for Joint and Marginal Posteriors] \label{thm:bvm2:joint}
Under Assumptions \ref{assump.m.func}, \ref{prior.1}, \ref{prior.2}, and \ref{prior.3}, the posterior distributions of $\theta$ and $\alpha$ are asymptotically independent, in the sense that the joint posterior distribution of $(\theta,\alpha)$ satisfies
\begin{align}\label{joint:theta1}
\left\|\Pi(\ud \theta, \ud \alpha |Y_n) - \mathcal{N}\left(\ud \theta \big| \widetilde\theta_{\alpha_0}, 2\theta_0^2/n \right) \times \widetilde \Pi(\ud \alpha|Y_n)\right\|_{\tv} \rightarrow 0,
\end{align}
as $n\to\infty$ almost surely $P_{(\beta_0,\sigma_0^2,\alpha_0)}$, where $\widetilde \Pi(\ud \alpha|Y_n)$ is the profile posterior distribution with density $\widetilde \pi(\alpha|Y_n)$ given by
\begin{align}\label{profile:post1}
\widetilde \pi(\alpha|Y_n) &= \frac{\exp\big\{\widetilde \Lcal_n(\alpha)\big\}\pi(\alpha|\theta_0)}{\int_0^{\infty} \exp\big\{\widetilde \Lcal_n(\alpha')\big\}\pi(\alpha'|\theta_0)\ud \alpha'},
\end{align}
where the profile restricted log-likelihood $\widetilde \Lcal_n(\alpha)$ is given in \eqref{def:prologlik} and $\pi(\alpha|\theta_0)$ is the conditional prior density of $\alpha$ given $\theta=\theta_0$. Furthermore, this profile posterior density $\widetilde \pi(\alpha|Y_n)$ is well defined for any given $n\geq p$ almost surely $P_{(\beta_0,\sigma_0^2,\alpha_0)}$. As a result, the total variation distance between $\Pi(\ud\theta|Y_n)$ and $\mathcal{N}\big(\ud\theta \big| \widetilde\theta_{\alpha_0},2\theta_0^2/n\big)$ converges to zero, and the total variation distance between $\Pi(\ud\alpha|Y_n)$ and $\widetilde \Pi(\ud\alpha|Y_n)$ converges to zero, as $n\to\infty$ almost surely $P_{(\beta_0,\sigma_0^2,\alpha_0)}$.
\end{theorem}
Theorem \ref{thm:bvm2:joint} provides a clear description of the limiting behavior of the joint posterior of $(\theta,\alpha)$ in the universal kriging model \eqref{eq:obs.model}. Under fixed-domain asymptotics, the microergodic parameter $\theta$ and the range parameter $\alpha$ have \textit{asymptotically independent} posterior distributions. The posterior of $\theta$ is centered at the REML $\widetilde\theta_{\alpha_0}$ and the variance is the same $2\theta_0^2/n$ as the asymptotic variance of REML $\widetilde\theta_{\alpha_0}$ in Theorem \ref{thm:bvm1:theta}. In fact, according to Part (ii) of Lemma \ref{lem:dimension reduction}, the center $\widetilde\theta_{\alpha_0}$ can be replaced by $\widetilde\theta_{\alpha_1}$ for any fixed $\alpha_1>0$, since $\alpha_1$ will be eventually covered by the slowly expanding interval $[\underline \alpha_n,\overline \alpha_n]$, and the difference between $\widetilde \theta_{\alpha_0}$ and $\widetilde \theta_{\alpha_1}$ is negligible compared to the limiting normal standard deviation $\sqrt{2\theta_0^2/n}$.

The posterior convergence of microergodic parameter $\theta$ with a varying range parameter $\alpha$ shows that we can consistently estimate the equivalent class of Gaussian measures using the Bayesian procedure even if the range parameter $\alpha$ has possibly large posterior uncertainty. An important consequence is that based on a random draw of parameters $(\theta,\alpha)$ from the posterior, the predictive variance at a new location is asymptotically close to the predictive variance based on the true parameters $(\theta_0,\alpha_0)$. We will elaborate this in Section \ref{sec:PAE}.

Theorem \ref{thm:bvm2:joint} has three advantages in its generality. First, the theorem works for the universal kriging model with regression terms $\bbm(\cdot)^\top \beta$. Second, it allows an unbounded prior support for $\alpha$, which is not available in previous frequentist fixed-domain asymptotics literature. Third, the theorem does not require any assumption on the design points $\Scal_n$. In other words, the asymptotic factorization and normality works for \textit{arbitrary} design of the sampling points $\Scal_n$, not even requiring $\Scal_n$ to be dense in the domain $\Scal$. Theorem \ref{thm:bvm2:joint} also shows that the marginal posterior density of $\alpha$ can be approximated by the more abstract profile posterior with density $\widetilde \pi(\alpha|Y_n)$, which is based on the profile restricted likelihood of $\alpha$. Using the result in \citet{Guetal18}, we can show that this profile posterior is always well defined. On the other hand, without further assumptions on $\Scal_n$, it is not likely that the form of the profile posterior density $\widetilde \pi(\alpha|Y_n)$ can be simplified. In general, this profile posterior of $\alpha$ does not necessarily converge to any point mass. In Theorem \ref{thm:OU1} below, for a special case of 1-dimensional Ornstein-Uhlenbeck process (Mat\'ern with $\nu=1/2$) observed on an equispaced grid without regression terms, we approximate $\widetilde \pi(\alpha|Y_n)$ using an explicit density of $\alpha$ that asymptotically does not contract to any fixed value with high probability. Such non-converging property of $\pi(\alpha|Y_n)$ explains the seemingly slow convergence of posterior of $\alpha$ in our SST data example in Section \ref{sec:intro}. We also demonstrate this phenomenon using simulation examples in Section \ref{sec:simulation}.

The difficulty in the estimation of range parameter $\alpha$ is a well-known problem in the GP literature (\citet{KenOha01}). Gaussian processes with different values of $\alpha$ but the same microergodic parameter $\theta$ in the Mat\'ern covariance function \eqref{eq:MaternCov} can have similar sample paths (\citet{Fugetal19}), making it difficult to infer an appropriate value for $\alpha$ from the data. \citet{Zhang04} and many others have observed that for a fixed value of $\theta>0$, $\Lcal_n(\theta/\alpha^{2\nu},\alpha)$ has a long right tail in $1/\alpha$ that creates problem for finding the MLE of $\alpha$. The sampling distribution of the MLE of $\alpha$ does not show any sign of convergence as $n\to\infty$. For Bayesian inference, \citet{Guetal18} identifies prior conditions using the objective priors in \citet{Beretal01} for robust estimation of $1/\alpha$ in finite samples. Though we do not study point estimation of $\alpha$, our technical proofs have derived some new properties for the profile posterior $\widetilde \pi(\alpha|Y_n)$, which could be of independent interest for Mat\'ern covariance functions; see Section S2 of the Supplementary Material for details.

Theorem \ref{thm:bvm2:joint} works for the domain dimension $d\in\{1,2,3\}$. For completeness, we also derive a similar theorem for the limiting joint posterior distribution when $d\geq 5$ under additional assumptions; see Section S3.4 of the Supplementary Material.

\subsection{On the Prior Assumptions} \label{subsec:prior.assumption}

We discuss the two technical prior assumptions \ref{prior.2} and \ref{prior.3}. The inequalities \eqref{A2.1} and \eqref{A2.2} in \ref{prior.2} require that the conditional prior $\pi(\theta|\alpha)$ does not vary too dramatically in a neighborhood of $\theta_0$ and in the slowly expanding interval $[\underline \alpha_n,\overline \alpha_n]$. The interval $(\theta_0/2,2\theta_0)$ in principle can be replaced by any neighborhood of the true parameter $\theta_0$, such as $(\theta_0-\delta_0,\theta_0+\delta_0)$ for some $0<\delta_0<\theta_0$. The inequality \eqref{A2.3} in \ref{prior.2} requires that the prior assigns a minimum of $\exp(-n^{C_{\pi,3}})$ prior mass on the true parameter $\theta_0$ uniformly over all $\alpha\in [\underline \alpha_n,\overline \alpha_n]$. Such minimal prior mass assumption is often necessary for achieving the basic posterior consistency in Bayesian models (\citet{GhoVan17}). In particular, we can verify Assumption \ref{prior.2} for the following examples of the prior $\pi(\theta|\alpha)$, some of which are commonly used in applications.
\begin{prop} \label{prop:prior2}
Suppose that the prior $\pi(\theta|\alpha)$ does not depend on the sample size $n$. Then Assumption \ref{prior.2} holds in either one of the following cases:
\begin{itemize}[leftmargin=5mm]
\item[(i)] $\pi(\theta|\alpha)=\pi(\theta)$ is independent of $\alpha$. $\pi(\theta)$ has continuous first derivative on $\RR^+$ and $\pi(\theta)>0$ for all $\theta\in \RR^+$.
\item[(ii)] $\pi(\alpha)$ is supported on a compact interval $[\alpha_1,\alpha_2]$, with constant lower and upper bounds $0<\alpha_1<\alpha_2<\infty$. $\pi(\theta|\alpha)$ is positive for all $(\theta,\alpha)\in \RR^+ \times \RR^+$, continuous in $\alpha\in \RR^+$, and has continuous first derivative with respect to $\theta$ on $\RR^+$ for all $\alpha\in \RR^+$.
\item[(iii)]  The prior of $\sigma^2$ is independent of $\alpha$ and belongs to the broad distribution family of the generalized beta of the second kind (or the Feller-Pareto family, \citet{Arn15}), with the density $\pi(\sigma^2)=\frac{\Gamma(\gamma_1+\gamma_2)}{\Gamma(\gamma_1)\Gamma(\gamma_2)}\frac{(\sigma^2/b)^{\gamma_2/\gamma-1}}{b\gamma [1+(\sigma^2/b)^{1/\gamma}]^{\gamma_1+\gamma_2}}$ with parameters $b>0,\gamma>0,\gamma_1>0,\gamma_2>0$.
\end{itemize}
\end{prop}

Proposition \ref{prop:prior2} shows that Assumption \ref{prior.2} about $\pi(\theta|\alpha)$ is satisfied by a wide range of prior distributions on $\theta$ with continuously differentiable densities. Case (i) says that \ref{prior.2} holds as long as the priors of $\theta$ and $\alpha$ are independent. Case (ii) says that \ref{prior.2} holds as long as the support of the prior of $\alpha$ is bounded away from zero and infinity. Compactly supported priors for the range parameter $\alpha$ have been widely used in Bayesian spatial statistics literature; see for example, \citet{Banetal08}, \citet{Sanetal11}, \citet{Datetal16}, \citet{Guhetal17}, etc. Case (iii) provides the example in which an independent prior is assigned on the variance parameter $\sigma^2$ instead of on $\theta$. The generalized beta of the second kind (or Feller-Pareto family, \citet{Bra02}, \citet{Arn15}) has polynomially decaying tails at both $\sigma^2\to 0+$ and $\sigma^2\to +\infty$. This family covers a wide range of continuous distributions on $(0,+\infty)$ including the half-Student's $t$ distributions, the $F$ distributions, the log-logistic distributions, the Burr distributions, and many others (\citet{Arn15}). Case (iii) mainly illustrates that if $\pi(\alpha)$ has a full support on $[0,+\infty)$, then $\pi(\theta|\alpha)$ cannot decay too fast in the two tails. For example, if $\pi(\theta|\alpha)$ has exponentially decaying tails at either $\theta\to 0+$ and $\theta\to +\infty$, then \ref{prior.2} is not satisfied when $\pi(\alpha)$ has a full support on $[0,+\infty)$. Fortunately, most spatial applications use a compactly supported prior for $\alpha$, and \ref{prior.2} is satisfied as in Case (ii).

Next, we discuss Assumption \ref{prior.3}, which imposes some technical conditions on the tail behavior of $\pi(\alpha)$ as $\alpha\to 0+$ and $\alpha\to +\infty$.

\begin{prop} \label{prop:prior3}
Let $\underkappa,\overkappa,\underline\alpha_n,\overline\alpha_n$ be defined in \eqref{eq:2kappa}. If a nonnegative function $p(\alpha)$ for $\alpha>0$ satisfies either one of the following conditions:
\begin{itemize}[leftmargin=5mm]
\item[(i)] $p(\alpha)\leq \exp\left(-\alpha^{\delta_1}\right)$ for all $\alpha > \overline\alpha_n$, for some constant $\delta_1>1/\overkappa$ and for all sufficiently large $n$;
\item[(ii)] $p(\alpha)\leq n^{\delta_3} \exp\left(-n^{\delta_2} \alpha\right)$ for all $\alpha > \overline\alpha_n$, for some constant $1-\overkappa< \delta_2\leq \delta_3 <\infty$ and all sufficiently large $n$;
\end{itemize}
then there exists a constant $0<\overline{c_{\pi}}<(\nu+d/2)\overkappa$ such that for all sufficiently large $n$,
\begin{align} \label{ineq:g.right}
& \int_{\overline\alpha_n}^{\infty} \alpha^{n(\nu+d/2)} p(\alpha) \ud \alpha \leq \exp(\overline{c_{\pi}} n\log n),
\end{align}
Similarly, if a nonnegative function $p(\alpha)$ for $\alpha>0$ satisfies either one of the following conditions:
\begin{itemize}[leftmargin=5mm]
\item[(i)] $p(\alpha)\leq \exp\left(-\alpha^{-\delta_1}\right)$ for all $0<\alpha < \underline\alpha_n$, for some constant $\delta_1>1/\underkappa $ and for all sufficiently large $n$;
\item[(ii)] $p(\alpha)\leq n^{\delta_3} \exp\left(-n^{\delta_2}/\alpha\right)$ for all $0< \alpha < \underline\alpha_n$, for some constant $1-\underkappa < \delta_2\leq \delta_3 <\infty$ and all sufficiently large $n$;
\end{itemize}
then there exists a constant $0<\underline{c_{\pi}}<(\nu+d/2)\underkappa$ such that for all sufficiently large $n$,
\begin{align} \label{ineq:g.left}
& \int_0^{\underline\alpha_n} \alpha^{-n(\nu+d/2)} p(\alpha) \ud \alpha \leq \exp(\underline{c_{\pi}} n\log n).
\end{align}
\end{prop}

Whilst having formulated Proposition \eqref{prop:prior3} for a generic function $p(\alpha)$, we have in mind to apply it to the priors $\pi(\alpha)$ and $\pi(\theta_0|\alpha)\pi(\alpha)$ in \eqref{A3.1} and \eqref{A3.2} in Assumption \ref{prior.3}. Since $\int_0^{\infty} \pi(\theta_0|\alpha)\pi(\alpha) \ud \alpha<\infty$ as in \ref{prior.3}, the tail conditions on $\pi(\theta_0|\alpha)\pi(\alpha)$ are the same as the tail conditions on $\pi(\alpha|\theta_0)$. Two types of tail decaying conditions are given in Proposition \ref{prop:prior3}. In the first case, the tail of $\pi(\alpha)$ or $\pi(\alpha|\theta_0)$ decays at the exponential power rate $\exp(-\alpha^{\delta_1})$ in the right tail (or $\exp(-\alpha^{-\delta_1})$ in the left tail), with some lower conditions on $\delta_1$ depending on the values of $\overkappa$ (or $\underkappa$). This condition requires that $\pi(\alpha)$ and $\pi(\alpha|\theta_0)$ decay very fast in the right (or left) tail. One example of $\pi(\alpha)$ is that $\alpha^{1/\min(\underkappa,\overkappa)}$ follows the inverse Gaussian distribution, since the inverse Gaussian distribution has exponentially decaying tails at zero and infinity. In the second case of Proposition \ref{prop:prior3}, we allow the tails of $\pi(\alpha)$ and $\pi(\alpha|\theta_0)$ to be upper bounded by some exponential rate in $\alpha$ that depends on $n$. These tail decaying conditions in Proposition \ref{prop:prior3} and Assumption \ref{prior.3} can ensure that the convergence to a normal limit will be dominant in the joint posterior of $(\theta,\alpha)$.

We remark that the tail conditions in \ref{prior.3} are often stronger than necessary in practice. This is partly because we have made \emph{no assumption} on the design of the sampling points $\Scal_n$. Even when $\Scal_n$ is highly unevenly distributed in $\Scal$ or is not dense in the full space of $\Scal$, Theorem \ref{thm:bvm2:joint} still holds true under \ref{prior.3}, which allows the prior $\pi(\alpha)$ to have a full support in $[0,+\infty)$. If one is willing to impose more assumptions on $\Scal_n$, for example, the maximum distance between two adjacent points decreases at a certain rate to zero, then it is possible to relax the tail conditions in \ref{prior.3}. Furthermore, such assumptions on the sampling design $\Scal_n$ may also improve how fast the total variation distance between the joint posterior distribution $\Pi(\ud\theta,\ud\alpha|Y_n)$ and its limiting distribution in Theorem \ref{thm:bvm2:joint} converges to zero. For a general smoothness parameter $\nu$, analyzing the effect of design $\Scal_n$ inevitably requires more sophisticated matrix theory for the properties of the Mat\'ern correlation matrix $R_{\alpha}$ and the related quantities $Y_n^\top R_{\alpha}^{-1} Y_n$ and $|R_{\alpha}|$ as $\alpha\to 0+$ and $\alpha \to +\infty$, since these two terms determine the properties of the profile restricted log-likelihood function \eqref{def:prologlik}. We will see in Theorem \ref{thm:OU1} below that in a special case when the sampling points are from an equispaced grid, the tail conditions in \ref{prior.3} can be significantly weakened and the conclusion of Theorem \ref{thm:bvm2:joint} can hold for a broader class of priors on $\alpha$.

\subsection{Limiting Posterior Distribution for 1-Dimensional Ornstein-Uhlenbeck Process} \label{subsec:OU1}

For a concrete example of Theorem \ref{thm:bvm2:joint}, we consider the special case of $d=1$, $\Scal=[0,1]$, and $\nu=1/2$ in the Mat\'ern covariance function. The covariance function becomes $\Cov(X(s),X(t))=\sigma^2\exp(-\alpha|s-t|)$ for $s,t\in [0,1]$, which is also known as the exponential covariance function. The resulted stochastic process $X$ is the 1-dimensional Ornstein-Uhlenbeck process (\citet{RasWil06}). We assume that the sampling points in $\Scal_n$ are on the equispaced grid with $s_i=i/n$ for $i=1,\ldots,n$. For the regression terms, we consider two different cases:
\begin{itemize}[leftmargin=5mm]
\item[(i)] Model \eqref{eq:obs.model} without the regression term $\bbm(\cdot)^\top \beta$, i.e., $p=0$, $Y(s)=X(s)$ for any $s\in [0,1]$, which implies that $Y_n\sim \Ncal(0,\sigma_0^2 R_{\alpha_0})$;
\item[(ii)] Model \eqref{eq:obs.model} with a constant regression term, i.e., $p=1$, $\bbm_1(\cdot)\equiv 1$, $\beta\in \RR$, $Y(s)=\beta + X(s)$ for any $s\in [0,1]$, which implies that $Y_n\sim \Ncal(1_n\beta_0,\sigma_0^2 R_{\alpha_0})$, where $1_n$ denotes the $n$-dimensional column vector of all 1's.
\end{itemize}
For Case (i), we derive an explicit formula for the limiting posterior of $\alpha$ and relax the condition on the tail of $\pi(\alpha)$ in the new Assumption \ref{prior.3OU}. For Case (ii), we show that the posterior of $\beta$ does not converge to the true parameter $\beta_0$ as $n\to\infty$.

For the model in Case (i), the frequentist MLE of $(\theta,\alpha)$ under fixed-domain asymptotics has been extensively studied in \citet{Ying91}, \citet{Ying93}, \citet{Chenetal00}, \citet{Duetal09}, etc. Since $s_i=i/n$ for $i=1,\ldots,n$, the inverse matrix $R_{\alpha}^{-1}$ is given by
\begin{align}
(R_{\alpha}^{-1})_{ii} &= \left\{
\begin{array}{ll}
 (1-\ee^{-2\alpha/n})^{-1}, & ~~ i=1,n\\
 (1+\ee^{-2\alpha/n})/(1-\ee^{-2\alpha/n}), & ~~ i=2,\ldots,n-1,
\end{array}
\right. \nonumber \\
(R_{\alpha}^{-1})_{i,i+1} &= (R_{\alpha}^{-1})_{i+1,i} = - \ee^{-\alpha/n} (1-\ee^{-2\alpha/n})^{-1},~~ i=1,\ldots,n-1, \nonumber
\end{align}
and all other entries of $R_{\alpha}$ are zero. Furthermore, the determinant of $R_{\alpha}$ is  $|R_{\alpha}| = (1-\ee^{-2\alpha/n})^{n-1}$. Since the model does not contain $\beta$, the profile restricted log-likelihood in \eqref{def:prologlik} has the explicit form
\begin{align}
&\widetilde \Lcal_n(\alpha) = -\frac{n}{2}\log \left(A_1 \ee^{-2\alpha/n}-2A_2 \ee^{-\alpha/n} +A_3\right) + \frac{1}{2}\log (1-\ee^{-2\alpha/n}),\label{eq:prologlik2} \\
&\text{where } \quad
A_1=\sum_{i=2}^{n-1} Y(s_i)^2, \quad A_2=\sum_{i=1}^{n-1}Y(s_i)Y(s_{i+1}), \quad A_3=\sum_{i=1}^n Y(s_i)^2. \label{eq:3A}
\end{align}
For the prior of $\alpha$, instead of Assumption \ref{prior.3}, we use a weaker alternative assumption.
\begin{enumerate}[label=(A.4')]
\item \label{prior.3OU} The marginal prior $\pi(\alpha)$ is a proper and continuous density on $\RR^+$. $\pi(\alpha)$ does not depend on $n$. $\pi(\alpha_0)>0$.
    $\int_0^{\infty} \pi(\theta_0|\alpha)\pi(\alpha) \ud \alpha<\infty$. $\int_0^{\infty} \sqrt{\alpha} \pi(\theta_0|\alpha) \pi(\alpha)\ud\alpha<\infty$. $\int_0^{\infty} \sqrt{\alpha} \pi(\alpha)\ud\alpha<\infty$. Furthermore, for $\underline\alpha_n$ and $\overline\alpha_n$ defined in \eqref{eq:2kappa}, the following relations hold as $n\to\infty$:
\begin{align}
& \sqrt{n} \int_0^{\underline\alpha_n} \sqrt{\alpha} \pi(\alpha) \ud \alpha \rightarrow 0, \qquad \sqrt{n} \int_{\overline\alpha_n}^{\infty} \sqrt{\alpha} \pi(\alpha) \ud \alpha \rightarrow 0. \label{A3.1.OU}
\end{align}
\end{enumerate}
Assumption \ref{prior.3OU} is considerably weaker than Assumption \ref{prior.3}. Assumption \ref{prior.3OU} only requires that $\pi(\alpha)$ and $\pi(\theta_0|\alpha)\pi(\alpha)$ (or equivalently, $\pi(\alpha|\theta_0)$) to have polynomially decaying tails at zero and infinity, compared to the exponential power tails as in Proposition \ref{prop:prior3}. With appropriate choice of hyperparameters, $\pi(\alpha)$ in \ref{prior.3OU} can be taken as gamma, inverse gamma, inverse Gaussian, or the family of generalized beta of the second kind defined in Proposition \ref{prop:prior2}; see the beginning of Section S5 in the Supplementary Material for detailed discussion on the choice of hyperparameters.

\begin{theorem} \label{thm:OU1}
Consider the model \eqref{eq:obs.model} with $p=0$, $d=1$, $\Scal=[0,1]$, $\nu=1/2$, and observations $Y_n$ on the equispaced grid $s_i=i/n$ for $i=1,\ldots,n$. Suppose that Assumptions \ref{prior.1}, \ref{prior.2}, and \ref{prior.3OU} hold. Then
\begin{align}
&\left\|\Pi(\ud\theta,\ud\alpha|Y_n) - \mathcal{N} \left(\ud\theta \big| \widetilde\theta_{\alpha_0}, 2\theta_0^2/n\right) \times \widetilde \Pi(\ud\alpha|Y_n) \right\|_{\tv} \rightarrow 0, \label{eq:OU.joint1} \\
&\left\|\Pi(\ud\theta,\ud\alpha|Y_n) - \mathcal{N} \left(\ud\theta \big| \widetilde\theta_{\alpha_0}, 2\theta_0^2/n\right) \times \Pi_*(\ud\alpha|Y_n) \right\|_{\tv} \rightarrow 0, \label{eq:OU.joint2}
\end{align}
as $n\to\infty$ in $P_{(\sigma_0^2,\alpha_0)}$-probability, where $\widetilde\theta_{\alpha_0}=n^{-1} \alpha_0^{2\nu} Y_n^\top R_{\alpha_0}^{-1} Y_n$, the profile posterior distribution $\widetilde \Pi(\ud\alpha|Y_n)$ has the density $\widetilde \pi(\alpha|Y_n) \propto \exp\big\{\widetilde \Lcal_n(\alpha)\big\} \cdot \pi(\alpha|\theta_0)$ with $\widetilde \Lcal_n(\alpha)$ given in \eqref{eq:prologlik2}, and the distribution $\Pi_*(\ud\alpha|Y_n)$ has the density
\begin{align*}
& \pi_*(\alpha|Y_n)\propto \sqrt{\alpha} \exp\left\{-\frac{(\alpha-u_*)^2}{2v_*}\right\} \cdot \pi(\alpha|\theta_0), ~~\text{ for all } \alpha\in \RR^+, \\
\text{where } ~~ & u_* = \frac{n(A_1-A_2)}{A_1},\quad v_* = \frac{n(A_1-2A_2+A_3)}{A_1},
\end{align*}
and $A_1,A_2,A_3$ are defined in \eqref{eq:3A}.
Furthermore, $|u_*|\preceq 1$, $v_*>0$ and $v_*\asymp 1$ as $n\to\infty$ in $P_{(\sigma_0^2,\alpha_0)}$-probability. Therefore, $\pi(\alpha|Y_n)$ does not converge to any point mass distribution as $n\to\infty$ in $P_{(\sigma_0^2,\alpha_0)}$-probability.
\end{theorem}

Theorem \ref{thm:OU1} provides a concrete form for the limiting joint posterior distribution of $(\theta,\alpha)$ in the 1-dimensional Ornstein-Uhlenbeck process under fixed-domain asymptotics. Since the model does not contain $\beta$, we write $P_{(\sigma_0^2,\alpha_0)}$ instead of $P_{(\beta_0,\sigma_0^2,\alpha_0)}$ in Theorem \ref{thm:OU1}. Compared to Theorem \ref{thm:bvm2:joint}, Theorem \ref{thm:OU1} shows the same limiting distribution under the weaker \ref{prior.3OU}. Furthermore, Theorem \ref{thm:OU1} simplifies the profile posterior density $\widetilde \pi(\alpha)$ to a more explicit form $\pi_*(\alpha|Y_n)$, which is a \emph{polynomially tilted normal density} (\citet{BocGre14}) times the conditional prior density $\pi(\alpha|\theta_0)$. The ``normal" part of $\pi_*(\alpha|Y_n)$ is centered at $u_*$ with scale $v_*$. Both center $u_*$ and the scale $v_*$ are of constant order in $P_{(\sigma_0^2,\alpha_0)}$-probability. Moreover, \ref{prior.1} and \ref{prior.3OU} ensure that $\pi(\alpha|\theta_0)$ is positive for all $\alpha\in \RR^+$. Therefore, the limiting distribution $\pi_*(\alpha|Y_n)$ has a continuous and positive density with a non-shrinking variance on $\RR^+$. If $\pi(\alpha|\theta_0)$ does not depend on $n$, then as a result of the convergence in total variation distance in \eqref{eq:OU.joint2}, the marginal posterior $\pi(\alpha|Y_n)$ also cannot converge to any point mass distribution as $n\to\infty$. Therefore, the posterior of $\alpha$ does not converge to the true parameter $\alpha_0$. This Bayesian asymptotic result matches with the frequentist theory in \citet{Zhang04} that there exists no consistent estimator for $\alpha$ under fixed-domain asymptotics.

Next we consider Case (ii). To simplify the expressions, we assume the noninformative prior $\pi(\beta|\sigma^2,\alpha)\propto 1$ which corresponds to $\Omega_{\beta}=0_{p\times p}$ in Assumption \ref{prior.1}. We notice that in Case (ii), $m_1(\cdot)\equiv 1$ and it is infinitely differentiable on $[0,1]$ with all derivatives equal to zero. Hence it lies in $\Wcal_2^{\nu+d/2}([0,1])$ for any $\nu>0$ and $d\in\{1,2,3\}$, and Assumption \ref{assump.m.func} is satisfied. We have the following corollary from Theorem \ref{thm:bvm2:joint}.
\begin{corollary} \label{cor:OU2}
Consider the model \eqref{eq:obs.model} with $p=1$, $\bbm_1(\cdot)\equiv 1$, $\pi(\beta|\sigma^2,\alpha)\propto 1$, $d=1$, $\Scal=[0,1]$, $\nu=1/2$, and observations $Y_n$ on the equispaced grid $s_i=i/n$ for $i=1,\ldots,n$. Suppose that Assumptions \ref{prior.1}, \ref{prior.2}, and \ref{prior.3} hold. Then
\begin{align}
& \beta ~|~ Y_n,\theta,\alpha \sim \Ncal\left(\frac{B_2-B_1\ee^{-\alpha/n}}{(n-2)(1-\ee^{-\alpha/n})+2}, \frac{\theta\left(1+\ee^{-\alpha/n}\right)}{\left[(n-2)(1-\ee^{-\alpha/n})+2\right]\alpha} \right), \label{eq:OU2.beta}  \\
&\left\|\Pi(\ud\theta,\ud\alpha|Y_n) - \mathcal{N} \left(\ud\theta \big| \widetilde\theta_{\alpha_0}, 2\theta_0^2/n\right) \times \widetilde \Pi(\ud\alpha|Y_n) \right\|_{\tv} \rightarrow 0, \label{eq:OU2.joint},
\end{align}
as $n\to\infty$ almost surely $P_{(\beta_0,\sigma_0^2,\alpha_0)}$, where the profile posterior distribution $\widetilde\Pi(\ud\alpha|Y_n)$ has the density $\widetilde \pi(\alpha|Y_n)\propto \exp\big\{\widetilde \Lcal_n(\alpha)\big\}\cdot \pi(\alpha|\theta_0)$, and the formulas of $\widetilde\theta_{\alpha}$ and $\widetilde \Lcal_n(\alpha)$ are given by
\begin{align} 
&\widetilde\theta_{\alpha} = \frac{\alpha(1-\ee^{-2\alpha/n})^{-1}}{n-1}\left\{ \left(A_1 \ee^{-2\alpha/n}-2A_2 \ee^{-\alpha/n} +A_3\right) - \frac{ (1-\ee^{-\alpha/n})(B_2-B_1\ee^{-\alpha/n})^2}{(n-2)(1-\ee^{-\alpha/n})+2}\right\} , \nonumber \\
&\widetilde \Lcal_n(\alpha)  = -\frac{n-1}{2}\log \left\{\left(A_1 \ee^{-2\alpha/n}-2A_2 \ee^{-\alpha/n} +A_3\right) - \frac{ (1-\ee^{-\alpha/n})(B_2-B_1\ee^{-\alpha/n})^2}{(n-2)(1-\ee^{-\alpha/n})+2} \right\} \nonumber \\
&\qquad \qquad + \frac{1}{2}\log \frac{1+\ee^{-\alpha/n}}{(n-2)(1-\ee^{-\alpha/n})+2}, \nonumber
\end{align}
where $B_1 = \sum_{i=2}^{n-1} Y(s_i)$, $B_2  = \sum_{i=1}^{n} Y(s_i)$, and $A_1,A_2,A_3$ are as defined in \eqref{eq:3A}.

Furthermore, for any $\eta \in (0,1/4)$, there exists constants $\epsilon_0>0$, $\delta_0 \in (0,1)$ and a large integer $N_2$, such that $\pr\left(\Pi(|\beta-\beta_0|>\epsilon_0|Y_n)>\delta_0\right) > 1-\eta$ for all $n>N_2$. Therefore, the posterior distribution of $\beta$ is inconsistent for the true parameter $\beta_0$.
\end{corollary}

Corollary \ref{cor:OU2} provides a concrete example that the posterior of $\beta$ is not consistent under fixed-domain asymptotics. In fact, this can be seen from the conditional posterior variance of $\beta$ given in \eqref{eq:OU2.beta}. For a fixed $\alpha$, this variance is close to $2\theta_0/[\alpha(\alpha+2)]$ as $n\to\infty$ since $\theta$ drawn from the posterior is close to $\theta_0$. Therefore, the posterior variance of $\beta$ does not vanish as $n\to\infty$. We expect that this is also true for general $\bbm(\cdot)$ functions, since one cannot expect to consistently estimate the regression coefficients $\beta$ only based on a single sample path $Y(\cdot)$. This echoes the frequentist result that the MLE of $\beta$ is inconsistent under fixed-domain asymptotics; see for example, Lemma 5 of \citet{GuAnd18}.

\subsection{Relation to Previous Bayesian Results} \label{subsec:relations}

\noindent \textbf{Relation to previous BvM results.} In the presence of nuisance parameters, \citet{Shen02} and \citet{BicKle12} have developed general machinery for proving BvM results in the presence of possibly nonparametric nuisance parameters. They assume that the model depends on an identifiable parameter and a nuisance parameter. \citet{BicKle12} first establish a LAN result for each value of the identifiable parameter inside a neighborhood of the ``least-favorable submodel", which is a contracting neighborhood of the nuisance parameter around the minimizer of the Kullback-Leibler divergence. Then their Theorem 4.2 gives the integral LAN property with integration over the nuisance parameter. They further proposes a rate free BvM theorem in their Corollary 5.2 that allows a non-contracting posterior for the nuisance parameter, which can be related to the posterior distribution of $\alpha$ in our GP model.

Despite the similarity, we adopt a more direct proof technique for the GP model with isotropic Mat\'ern covariance function, instead of checking the condition on Hellinger distance in \citet{BicKle12} for uniform tests. There are several additional challenges. First, the likelihood function in our GP model cannot be written in an independent product form. The design of the sampling points $\Scal_n$ is arbitrary, making $R_{\alpha}^{-1}$ and $|R_{\alpha}|$ completely intractable. This determines that the LAN condition in our model is fundamentally different from that for independent or weakly dependent data considered in \citet{BicKle12}. We instead use the tools of RKHS theory and spectral analysis to establish the LAN condition for $\theta$. We integrate out $\theta$ for each given $\alpha$ and obtain the profile posterior distribution of $\alpha$ as in \eqref{profile:post1}. Second, our LAN condition holds uniformly over all $\alpha\in[\underline\alpha_n, \overline\alpha_n]$, but we still need to handle those $\alpha$ outside $[\underline\alpha_n, \overline\alpha_n]$. We derive sufficient tail conditions on $\pi(\alpha)$ such that the posterior probability outside $[\underline\alpha_n, \overline\alpha_n]$ vanishes as $n\to\infty$. This involves detailed analysis on the properties of the profile posterior distribution in \eqref{profile:post1}; see Section S2 of the Supplementary Material.

In the broader sense, our work contributes a new example to the literature of limiting posterior distribution for nonregular models; see for example, \citet{CheHon04}, \citet{KleKna12}, \citet{BocGre14}, \citet{Junetal15}, \citet{Chenetal18}, etc.

\vspace{2mm}

\noindent \textbf{Relation to partially identified models.}
Our theorems for the covariance parameters can also be related to the Bayesian literature of \emph{partially identified models}. Such models have been studied extensively in statistics and econometrics literature, but only for independent and weakly dependent data (\citep{Man03}, \citep{Tam10}, \citep{Gus15}). In partially identified models, the probability distribution of the data is compatible with a set of different parameter values. This parameter set is referred to as the \emph{identification region}. As a result, consistent point estimator for the true parameter does not exist, though one can still consistently estimate the identification region. The asymptotic property of posterior distributions in partially identified models have been studied in \citep{MonSch12}, \citep{Gus14}, \citep{Jiang17}, \citep{Chenetal18}, \citep{JiangLi19}, etc. However, the Bayesian theory from these works only applies to independent data and weakly dependent data, and does not apply to our GP model. Depending on the assumptions, the limiting posterior of the nuisance parameter can either only depend on the prior (\citep{MonSch12}), or depend on the prior and some asymptotically deterministic function of the identifiable part of the parameter vector (\citep{Jiang17}).

Our paper contributes a new example to the Bayesian partial identification literature. Consider the model \eqref{eq:obs.model} with isotropic Mat\'ern covariance function $\sigma^2 K_{\alpha,\nu}$ and without regression terms, i.e., $Y(\cdot)=X(\cdot)$. Under fixed-domain asymptotics, the distribution of $Y_n$ is asymptotically compatible with any parameters on the curve $\Gamma_{\theta_0} = \{(\sigma^2,\alpha)\in \RR^+ \times \RR^+: \sigma^2 \alpha^{2\nu}=\theta_0 \}$, which is the identification region in our problem. Different from \citet{MonSch12}, our Theorem \ref{thm:bvm2:joint} shows that both the prior and the data $Y_n$ play important roles in the posterior of $\alpha$. The data $Y_n$ influences the posterior through the profile restricted likelihood function. Different from \citet{Jiang17}, Theorem \ref{thm:OU1} shows that the influence from $Y_n$ is always stochastic instead of asymptotically deterministic, as the polynomially tilted normal distribution $\pi_*(\alpha|Y_n)$ has a scale $v_*$ dependent on $Y_n$ and not converging to any point limit asymptotically.

\section{Asymptotic Efficiency and Convergence Rate of Posterior Prediction} \label{sec:PAE}

The limiting theorems in Section \ref{sec:main.bvm} shows that the posterior of the microergodic parameter $\theta$ in the Mat\'ern covariance function satisfies the same $n^{-1/2}$-convergence to a normal limit. This result has an important implication for the Bayesian GP (or kriging) prediction with covarinace parameters randomly drawn from the posterior distribution at a new location $s^*\in \Scal \backslash \Scal_n$, i.e., $s^*$ is an arbitrary point in $\Scal$ but different from the sampling points $\Scal_n$. We first show that for the general model \eqref{eq:obs.model}, the Bayesian GP predictive variance is almost equal to the one with a known $\theta_0$. Then we discuss the detailed posterior asymptotic efficiency for the model without regression terms and the convergence rates for the model with regression terms. We also present results both for a fixed $\alpha$ and for a range of $\alpha$ values.

Consider the linear prediction (or kriging) of $Y(s^*)$ using the data $Y_n$. Let $r_{\alpha}(s^*)=(K_{\alpha,\nu}(s_1-s^*),\ldots,K_{\alpha,\nu}(s_n-s^*))^\top$ be the correlation vector between $s^*$ and $\{s_1,\ldots,s_n\}$. Then under a possibly misspecified model $Y\sim \gp\left(\bbm^\top \beta,\sigma^2 K_{\alpha,\nu}\right)$, the best linear unbiased predictor (BLUP) for $Y(s^*)$ using (Section 1.5 of \citet{Stein99a}) is
\begin{align}\label{eq:BLUP}
\widehat Y(s^*;\beta,\alpha) & =
\bbm(s^*)^\top \beta + r_{\alpha}(s^*)^\top R_{\alpha}^{-1} \left(Y_n - M_n \beta\right) .
\end{align}
This kriging predictor only depends on $(\beta,\alpha)$ but not $\sigma^2$. Now under the Bayesian setup, we randomly draw $(\beta,\sigma^2,\alpha)$ from the posterior $\Pi(\cdot|Y_n)$ to predict $Y(s^*)$. We denote the predicted variable as $\widetilde Y(s^*)$. Using the Gaussian process predictive distribution, we have
\begin{align*}
\widetilde Y(s^*)|Y_n,\beta,\sigma^2,\alpha &\sim \Ncal \left(\widehat Y(s^*;\beta,\alpha), ~\sigma^2\left\{1 - r_{\alpha}(s^*)^\top R_{\alpha}^{-1} r_{\alpha}(s^*)\right\} \right).
\end{align*}
We can integrate out $\beta$ using \eqref{eq:beta.post} to derive that
\begin{align}\label{eq:varBLUP.true}
&\widetilde Y(s^*)|Y_n,\sigma^2,\alpha \sim \Ncal \Big(\widehat Y(s^*;\alpha), \vv_n(s^*;\sigma^2,\alpha) \Big), \\
\text{where } &\widehat Y(s^*;\alpha) = r_{\alpha}(s^*)^\top R_{\alpha}^{-1} Y_n + b_{\alpha}(s^*)^\top \big(M_n^\top R_{\alpha}^{-1} M_n + \Omega_{\beta}\big)^{-1} M_n^\top R_{\alpha}^{-1}Y_n, \nonumber \\
&{\vv}_n(s^*;\sigma^2,\alpha) = \sigma^2\left\{1 - r_{\alpha}(s^*)^\top R_{\alpha}^{-1} r_{\alpha}(s^*)\right\} + \sigma^2 b_{\alpha}(s^*)^\top \big(M_n^\top R_{\alpha}^{-1} M_n + \Omega_{\beta}\big)^{-1} b_{\alpha}(s^*), \nonumber \\
\text{and } & b_{\alpha}(s^*)=\bbm(s^*) - M_n^\top R_{\alpha}^{-1} r_{\alpha}(s^*), \text{ for any } s^*\in \Scal.  \nonumber
\end{align}
The detailed derivation of \eqref{eq:varBLUP.true} is in Section S6.1 of the Supplementary Material. This normal predictive distribution is the same as in Equation (2.4) of \citet{HanSte93} which is for the special case of $\Omega_{\beta}=0_{p\times p}$. The predictive variance of $\widetilde Y(s^*)$, $\vv_n(s^*;\sigma^2,\alpha)$ in \eqref{eq:varBLUP.true}, is the main focus of this section, because it directly quantifies the Bayesian uncertainty of GP prediction.

We first show that if $(\sigma^2,\alpha)$ is randomly drawn from the posterior $\Pi(\cdot|Y_n)$, then the GP predictive variance ${\vv}_n(s^*;\sigma^2,\alpha)$ is almost equal to ${\vv}_n(s^*;\theta_0/\alpha^{2\nu},\alpha)$, i.e., as if the true microergodic parameter $\theta_0$ were known. We notice that ${\vv}_n(s^*;\sigma^2,\alpha)$ is random due to the randomness in the posterior distribution of $(\sigma^2,\alpha)$.
\begin{theorem}\label{thm:pae.micro}
(Posterior asymptotic efficiency compared to the half oracle model)
\begin{itemize}
\item[(i)] Under Assumptions \ref{assump.m.func} and \ref{prior.1}, for any fixed $\alpha>0$, as $n\to\infty$, almost surely $P_{(\beta_0,\sigma_0^2,\alpha_0)}$,
\begin{align*}
& \Pi\left[ \sup_{s^* \in \Scal \backslash \Scal_n} \left|\frac{{\vv}_n(s^*;\sigma^2,\alpha)}{{\vv}_n(s^*;\theta_0/\alpha^{2\nu},\alpha)} - 1 \right| > 7 n^{-1/2} \log n  \Bigg |Y_n ,\alpha \right] \rightarrow 0 .
\end{align*}
\item[(ii)] Under Assumptions \ref{assump.m.func}, \ref{prior.1}, \ref{prior.2} and \ref{prior.3}, as $n\to\infty$, almost surely $P_{(\beta_0,\sigma_0^2,\alpha_0)}$,
\begin{align*}
& \Pi\left[ \sup_{s^* \in \Scal \backslash \Scal_n} \left|\frac{{\vv}_n(s^*;\sigma^2,\alpha)}{{\vv}_n(s^*;\theta_0/\alpha^{2\nu},\alpha)} - 1 \right| > 7 n^{-1/2} \log n  \Bigg |Y_n \right] \rightarrow 0 .
\end{align*}
\end{itemize}
\end{theorem}
Theorem \ref{thm:pae.micro} shows that the GP predictive variance at an arbitrary new location $s^*$ evaluated under the measure $P_{(\beta,\sigma^2,\alpha)}$ is asymptotically equal to the predictive MSE evaluated under the measure $P_{(\beta_0,\theta_0/\alpha^{2\nu}, \alpha)}$. Part (i) and Part (ii) are the direct consequence of Theorem \ref{thm:bvm1:theta} for the posterior of $\theta$ given $\alpha$ and Theorem \ref{thm:bvm2:joint} for the joint posterior of $(\theta,\alpha)$, respectively. We also give the explicit convergence rate $n^{-1/2}\log n$, in which the $\log n$ factor is to ensure the almost sure convergence. Theorem \ref{thm:pae.micro} shows that the prediction performance from a random draw of $(\theta,\alpha)$ from the posterior is as good as the ``half oracle" model with the true microergodic parameter $\theta_0$ and the same range parameter $\alpha$. It is half oracle because Theorem \ref{thm:pae.micro} has not yet set the range parameter at the true $\alpha_0$ and compared with ${\vv}_n(s^*;\theta_0/\alpha_0^{2\nu},\alpha_0)$. On the other hand, Theorem \ref{thm:pae.micro} only requires the same conditions as Theorem \ref{thm:bvm2:joint}.

In the following, we will compare ${\vv}_n(s^*;\sigma^2,\alpha)$ with ${\vv}_n(s^*;\theta_0/\alpha_0^{2\nu},\alpha_0)$, the predictive variance from the full oracle model where both $\theta$ and $\alpha$ are set at their true values. We first study a simplified model without regression terms and prove the asymptotic efficiency in posterior prediction with respect to the full oracle model, and then consider the general model \eqref{eq:obs.model} and show the same optimal posterior convergence rates as the full oracle model.

\subsection{Posterior Asymptotic Efficiency Without Regression Terms} \label{subsec:pae.no.reg}
In this subsection, we consider a special case of the model \eqref{eq:obs.model} where the regression term $m(\cdot)^\top \beta$ is absent and the model simplifies to
\begin{align}\label{eq:obs.model.2}
Y(s)=X(s), \quad \text{for any }s\in \Scal, \quad X\sim \gp(0,\sigma^2 K_{\alpha,\nu}).
\end{align}
We observe $Y_n\sim \Ncal(0,\sigma_0^2 R_{\alpha_0})$ at the sampling points $\Scal_n$. This is equivalent to setting $p=0$. For this model, we prove the strong result that ${\vv}_n(s^*;\sigma^2,\alpha)$ with $(\sigma^2,\alpha)$ randomly drawn from the posterior is asymptotically equal to ${\vv}_n(s^*;\theta_0/\alpha_0^{2\nu},\alpha_0)$ and quantify the convergence rate. We need the following dense assumption.
\begin{enumerate}[label=(A.\arabic*)]
\setcounter{enumi}{4}
\item \label{assump.dense1} The sequence of $\Scal_n=\{s_1,\ldots,s_n\}$ is getting dense in $\Scal=[0,T]^d$ as $n\to\infty$, in the sense that $\sup_{s^*\in \Scal} \min_{1\leq i\leq n} \|s^*-s_i\| \to 0$ as $n\to\infty$.
\end{enumerate}
The sets $\Scal_1,\Scal_2,\ldots$ are increasingly dense in the fixed domain $\Scal$, so that we can predict at any new location accurately. But we do not require the sequence $\Scal_1,\Scal_2,\ldots$ to be nested.

In the model \eqref{eq:obs.model.2}, the BLUP of $Y(s^*)$ is $\widehat Y(s^*;\alpha)=r_{\alpha}(s^*)^\top R_{\alpha}^{-1} Y_n$, and $\vv_n(s^*;\sigma^2,\alpha)=\sigma^2\left\{1 - r_{\alpha}(s^*)^\top R_{\alpha}^{-1} r_{\alpha}(s^*)\right\}$ from \eqref{eq:varBLUP.true}. We notice that in this case, another interpretation of $\vv_n(s^*;\sigma^2,\alpha)$ is the GP prediction mean squared error of the BLUP $\widehat Y(s^*;\alpha)$ (\citet{KauSha13}). That is, if we let $e_n(s^*;\alpha)=\widehat Y(s^*;\alpha) - Y(s^*)$, then $\vv_n(s^*;\sigma^2,\alpha)={\EE}_{(\sigma^2,\alpha)}\left\{e_n(s^*;\alpha)^2\right\}$. The optimal ``oracle" predictive MSE using the true parameters is $\vv_n(s^*;\sigma_0^2,\alpha_0)={\EE}_{(\sigma_0^2,\alpha_0)}\left\{e_n(s^*;\alpha_0)^2\right\}$. Under the true model $\gp(0,\sigma_0^2K_{\alpha_0,\nu})$, the predictive MSE based on a misspecified $\alpha$ is
\begin{align}
{\EE}_{(\sigma_0^2,\alpha_0)}\left\{e_n(s^*;\alpha)^2\right\} & = \sigma_0^2 \Big\{1-2r_{\alpha}(s^*)^\top R_{\alpha}^{-1}r_{\alpha_0}(s^*)
+ r_{\alpha}(s^*)^\top R_{\alpha}^{-1}R_{\alpha_0}R_{\alpha}^{-1}r_{\alpha}(s^*)\Big\}. \nonumber
\end{align}

We are interested in whether ${\EE}_{(\sigma^2,\alpha)}\left\{e_n(s^*;\alpha)^2\right\}$, the predictive MSE under the true measure ${\EE}_{(\sigma_0^2,\alpha_0)}\left\{e_n(s^*;\alpha)^2\right\}$, and the oracle predictive MSE ${\EE}_{(\sigma_0^2,\alpha_0)}\left\{e_n(s^*;\alpha)^2\right\}$ are close to each other. In a series of works \citet{Stein88}, \citet{Stein90a}, \citet{Stein90b}, \citet{Stein93}, \citet{Stein97} and \citet{Stein99b}, Stein has systematically studied the GP prediction problem and shown that if an incorrect Gaussian process model is used for prediction, the predictive variance at $s^*$ is asymptotically equal to the predictive variance at $s^*$ using the incorrect model but evaluated under the true Gaussian process model, as long as the two Gaussian measures are compatible (or mutually absolutely continuous). For our GP model with mean-zero and isotropic Mat\'ern covariance function with $d\in\{1,2,3\}$, the compatibility of the incorrect model $\gp(0,\sigma^2 K_{\alpha,\nu})$ and the true model $\gp(0,\sigma_0^2 K_{\alpha_0,\nu})$ simplifies to the equivalence condition $\sigma^2\alpha^{2\nu}=\theta_0=\sigma_0^2\alpha_0^{2\nu}$, i.e., they have the same microergodic parameter $\theta_0$. If the equivalence condition holds, then \citet{Stein88}, \citet{Stein90a}, and \citet{Stein90b} have shown that for the model without regression terms \eqref{eq:obs.model.2}, as $n\to\infty$,
\begin{align} \label{eq:ae.original}
\sup_{s^*\in \Scal \backslash \Scal_n} \left| \frac{{\EE}_{(\sigma^2,\alpha)}\left\{e_n(s^*;\alpha)^2\right\}}{{\EE}_{(\sigma_0^2,\alpha_0)}\left\{e_n(s^*;\alpha)^2\right\} } - 1 \right| \to 0,
\sup_{s^*\in \Scal \backslash \Scal_n} \left| \frac{{\EE}_{(\sigma^2,\alpha)}\left\{e_n(s^*;\alpha)^2\right\}}{{\EE}_{(\sigma_0^2,\alpha_0)}\left\{e_n(s^*;\alpha_0)^2\right\} } - 1 \right| \to 0,
\end{align}
which is called \emph{asymptotic efficiency in linear prediction}. The first convergence shows that for the BLUP \eqref{eq:BLUP}, the predictive MSEs are almost the same under either the incorrect Gaussian measure $P_{(\sigma^2,\alpha)}$ or the true Gaussian measure $P_{(\sigma_0^2,\alpha_0)}$. The second convergence shows that the predictive MSEs obtained from the incorrect model $\gp(0,\sigma^2 K_{\alpha,\nu})$ is asymptotically equal to the optimal predictive MSE from the true model $\gp(0,\sigma_0^2 K_{\alpha_0,\nu})$.

Using the weakened conditions in \citet{Stein93}, Theorem 4 of \citet{KauSha13} shows that in the model \eqref{eq:obs.model.2}, for a given $\alpha>0$, the prediction based on the MLE of $\sigma^2$ for a \textit{fixed} $\alpha>0$ satisfies that
\begin{align*}
\sup_{s^*\in \Scal \backslash \Scal_n} \left| \frac{{\EE}_{(\widetilde \sigma^2_{\alpha},\alpha)}\left\{e_n(s^*;\alpha)^2\right\}}{{\EE}_{(\sigma_0^2,\alpha_0)}\left\{e_n(s^*;\alpha)^2\right\} } - 1 \right| \to 0,
\end{align*}
as $n\to\infty$ almost surely $P_{(\sigma_0^2,\alpha_0)}$, where $\widetilde \sigma^2_{\alpha} = n^{-1} Y_n^\top R_{\alpha}^{-1} Y_n$ is the MLE of $\sigma^2$.

Motivated by these works, we establish the Bayesian version of \eqref{eq:ae.original}, called \emph{asymptotic efficiency in posterior prediction}, which is the posterior asymptotic efficiency compared to the full oracle model. In Bayesian inference, we randomly draw $(\sigma^2,\alpha)$ from the joint posterior distribution, and compute the predictive MSE at a new location $s^*\in \Scal \backslash \Scal_n$ using the Gaussian measure $P_{(\sigma^2,\alpha)}$.

For a given $\alpha>0$, we define the following sequence $\varsigma_n(\alpha)$ which will be useful
\begin{align} \label{eq:vs.alpha}
\varsigma_n(\alpha) & = \max\left\{ \sup_{s^* \in \Scal \backslash \Scal_n}
\left|\tfrac{{\EE}_{(\theta_0/\alpha^{2\nu},\alpha)}\big\{e_n(s^*;\alpha)^2\big\}}
    {{\EE}_{(\sigma_0^2,\alpha_0)}\big\{e_n(s^*;\alpha)^2\big\}} - 1 \right|,
    \sup_{s^* \in \Scal \backslash \Scal_n}
    \left|\tfrac{{\EE}_{(\theta_0/\alpha^{2\nu},\alpha)}\big\{e_n(s^*;\alpha)^2\big\}}
    {{\EE}_{(\sigma_0^2,\alpha_0)}\big\{e_n(s^*;\alpha_0)^2\big\}} - 1 \right|
    \right\} .
\end{align}
For a given $\alpha>0$, as $n\to\infty$, Theorem 3.1 of \citet{Stein90a} shows that the first rate in $\varsigma_n(\alpha)$ in \eqref{eq:vs.alpha} converges to zero, and Theorem 1 of \citet{Stein90b} further implies that the second rate in $\varsigma_n(\alpha)$ in \eqref{eq:vs.alpha} converges to zero. To handle a random range parameter $\alpha$, we need the following uniform convergence condition.
\begin{enumerate}[label=(A.\arabic*)]
\setcounter{enumi}{5}
\item \label{assump.Mn} There exists a positive deterministic sequence $\varsigma_n \to 0$ as $n\to \infty$, such that $\sup_{\alpha \in [\underline\alpha_n, \overline\alpha_n]} \varsigma_n(\alpha) \leq \varsigma_n$ for the sequence $\varsigma_n(\alpha)$ defined in \eqref{eq:vs.alpha}.
\end{enumerate}

We have the following theorem for the prediction MSE in the model \eqref{eq:obs.model.2}.
\begin{theorem}\label{thm:pae.main}
(Posterior asymptotic efficiency compared to the full oracle model under \eqref{eq:obs.model.2})
\begin{itemize}[leftmargin=5mm]
\item[(i)] (For a fixed $\alpha$) Under Assumptions \ref{prior.1} and \ref{assump.dense1}, as $n\to\infty$, almost surely $P_{(\sigma_0^2,\alpha_0)}$,
\begin{align*}
& \Pi\left[ \sup_{s^* \in \Scal \backslash \Scal_n}  \left|\frac{{\EE}_{(\sigma^2,\alpha)}\big\{e_n(s^*;\alpha)^2\big\}}{{\EE}_{(\sigma_0^2,\alpha_0)}\big\{e_n(s^*;\alpha)^2\big\}} - 1 \right| > \max\left\{16n^{-1/2}\log n, 2\varsigma_n(\alpha)\right\} \Bigg |Y_n,\alpha \right] \rightarrow 0 ,\\
& \Pi\left[ \sup_{s^* \in \Scal \backslash \Scal_n}  \left|\frac{{\EE}_{(\sigma^2,\alpha)}\big\{e_n(s^*;\alpha)^2\big\}}{{\EE}_{(\sigma_0^2,\alpha_0)}\big\{e_n(s^*;\alpha_0)^2\big\}} - 1 \right| > \max\left\{16n^{-1/2}\log n, 2\varsigma_n(\alpha)\right\} \Bigg |Y_n,\alpha \right] \rightarrow 0,
\end{align*}
where $\varsigma_n(\alpha)$ is given in \eqref{eq:vs.alpha};
\item[(ii)] (For random $\alpha$) Under Assumptions \ref{prior.1}, \ref{prior.2}, \ref{prior.3}, \ref{assump.dense1} and \ref{assump.Mn}, as $n\to\infty$, almost surely $P_{(\sigma_0^2,\alpha_0)}$,
\begin{align}
& \Pi\left[ \sup_{s^* \in \Scal \backslash \Scal_n}  \left|\frac{{\EE}_{(\sigma^2,\alpha)}\big\{e_n(s^*;\alpha)^2\big\}}{{\EE}_{(\sigma_0^2,\alpha_0)}\big\{e_n(s^*;\alpha)^2\big\}} - 1 \right| > \max\left(16n^{-1/2}\log n, 2\varsigma_n \right) \Bigg |Y_n \right] \rightarrow 0 , \nonumber  \\
& \Pi\left[ \sup_{s^* \in \Scal \backslash \Scal_n}  \left|\frac{{\EE}_{(\sigma^2,\alpha)}\big\{e_n(s^*;\alpha)^2\big\}}{{\EE}_{(\sigma_0^2,\alpha_0)}\big\{e_n(s^*;\alpha_0)^2\big\}} - 1 \right| > \max\left(16n^{-1/2}\log n, 2\varsigma_n \right) \Bigg |Y_n \right] \rightarrow 0,
\end{align}
where $\varsigma_n$ is given in Assumption \ref{assump.Mn}.
\end{itemize}
\end{theorem}
We emphasize again that ${\EE}_{(\sigma^2,\alpha)}\big\{e_n(s^*;\alpha)^2\big\}={\vv}_n(s^*;\sigma^2,\alpha)$ and ${\EE}_{(\sigma_0^2,\alpha_0)}\big\{e_n(s^*;\alpha_0)^2\big\}={\vv}_n(s^*;\sigma_0^2,\alpha_0)$ for the model \eqref{eq:obs.model.2} without regression terms.
Part (i) of Theorem \ref{thm:pae.main} establishes two posterior convergence results. The first convergence is about the ratio of the predictive MSEs using a misspecified range parameter $\alpha$ evaluated under the measure $P_{(\sigma^2,\alpha)}$ and the true measure $P_{(\sigma_0^2,\alpha_0)}$, which implies that these two predictive MSEs are asymptotically equal. The second convergence is about the ratio of the predictive MSEs using the incorrect model $P_{(\sigma^2,\alpha)}$ and the full oracle optimal predictive MSE using the true model $P_{(\sigma_0^2,\alpha_0)}$. This implies that the predictive MSE computed with random parameters $(\theta,\alpha)$ drawn from the posterior can asymptotically recover the exact full oracle optimal predictive MSE. Both convergence rates depend on two parts: one is the posterior convergence rate of $\theta$ to $\theta_0$, which is as fast as $n^{-1/2}\log n$; the other is the convergence rate from the convergence of the two ratios in the definition of $\varsigma_n(\alpha)$ in \eqref{eq:vs.alpha}, which has been shown before by \citet{Stein90a} and \citet{Stein90b}.

Part (ii) of Theorem \ref{thm:pae.main} is similar to Part (i) with the same interpretation of asymptotic efficiency, except that $\alpha$ is also random and $(\sigma^2,\alpha)$ is drawn from their joint posterior. Furthermore, Assumption \ref{assump.Mn} is used to guarantee the uniform convergence over the majority of $\alpha$ values in the interval $[\underline\alpha_n, \overline\alpha_n]$. Part (ii) shows that the predictive MSE computed from randomly drawn $(\sigma^2,\alpha)$ from the posterior is asymptotically equal to the oracle optimal predictive MSE with the true parameters.

We emphasize that the posterior asymptotic efficiency in Theorem \ref{thm:pae.main} automatically implies that ${\EE}_{(\sigma^2,\alpha)}\big\{e_n(s^*;\alpha)^2\big\}$ with $(\sigma^2,\alpha)$ randomly drawn from the posterior must always converge at exactly the same rate to zero as ${\EE}_{(\sigma_0^2,\alpha_0)}\big\{e_n(s^*;\alpha_0)^2\big\}$, regardless of how fast ${\EE}_{(\sigma_0^2,\alpha_0)}\big\{e_n(s^*;\alpha_0)^2\big\}$ converges to zero. Therefore, the posterior asymptotic efficiency is stronger than posterior convergence rate results.

To clarify the rate $\varsigma_n$ in Assumption \ref{assump.Mn}, we revisit the 1-dimensional Ornstein-Uhlenbeck process in Case (i) in Section \ref{subsec:OU1} and derive an explicit form for $\varsigma_n$.
\begin{theorem} \label{thm:suprate.OU}
For the case of $d=1$, $\nu=1/2$, $\Scal=[0,1]$, and equispaced grid $s_i=i/n$, for $i=1,\ldots,n$, Assumption \ref{assump.Mn} is satisfied with $\varsigma_n= 3n^{-1/2+(\overline\kappa+\underline\kappa/2)}$, where $\overkappa$ and $\underkappa$ are defined in \eqref{eq:2kappa}. As a result, under Assumptions \ref{prior.1}, \ref{prior.2}, \ref{prior.3}, \ref{assump.dense1}, as $n\to\infty$, almost surely $P_{(\sigma_0^2,\alpha_0)}$,
\begin{align}
& \Pi\left[ \sup_{s^* \in \Scal \backslash \Scal_n}  \left|\frac{{\EE}_{(\sigma^2,\alpha)}\big\{e_n(s^*;\alpha)^2\big\}}{{\EE}_{(\sigma_0^2,\alpha_0)}\big\{e_n(s^*;\alpha)^2\big\}} - 1 \right| > 6n^{-1/2+(\overline\kappa+\underline\kappa/2)} \Bigg |Y_n \right] \rightarrow 0 , \nonumber  \\
& \Pi\left[ \sup_{s^* \in \Scal \backslash \Scal_n}  \left|\frac{{\EE}_{(\sigma^2,\alpha)}\big\{e_n(s^*;\alpha)^2\big\}}{{\EE}_{(\sigma_0^2,\alpha_0)}\big\{e_n(s^*;\alpha_0)^2\big\}} - 1 \right| > 6n^{-1/2+(\overline\kappa+\underline\kappa/2)} \Bigg |Y_n \right] \rightarrow 0. \nonumber
\end{align}
\end{theorem}
To prove Theorem \ref{thm:suprate.OU}, we use the result in \citet{Stein90b} and relate the rate $\varsigma_n$ in Assumption \ref{assump.Mn} to the convergence rate of the finite sample version of the symmetrized Kullback-Leibler divergence between two equivalent Gaussian measures towards its limit. Since $\overkappa$ and $\underkappa$ are both small positive numbers as given in \eqref{eq:2kappa}, the two posterior convergence rates for asymptotic efficiency in Theorem \ref{thm:suprate.OU} are both close to the rate $n^{-1/2}$.

\subsection{Optimal Rates for GP Predictive Variance with Regression Terms} \label{subsec:pae.reg}
We now consider the general universal kriging model \eqref{eq:obs.model} with the regression term $\bbm(\cdot)^\top \beta$. Like Theorem \ref{thm:pae.main}, we also need a similar assumption to Assumption \ref{assump.Mn}.
\begin{enumerate}[label=(A.6')]
\item \label{assump.Mn2} There exists a positive deterministic sequence $\tilde\varsigma_n\to 0$ as $n\to\infty$, such that
\begin{align}\label{eq:varsigma.p}
\sup_{\alpha\in [\underline\alpha_n,\overline\alpha_n]} \sup_{s^* \in \Scal\backslash\Scal_n} \left|\frac{(\theta_0/\alpha^{2\nu})[1-r_{\alpha}(s^*)^\top R_{\alpha}^{-1}r_{\alpha}(s^*)]}
    {\sigma_0^2[1-r_{\alpha_0}(s^*)^\top R_{\alpha_0}^{-1}r_{\alpha_0}(s^*)]} - 1 \right| \leq \tilde \varsigma_n .
\end{align}
\end{enumerate}
Because the relative error in \eqref{eq:varsigma.p} is exactly the second relative error in the definition of $\varsigma_n(\alpha)$ in \eqref{eq:vs.alpha}, Assumption \ref{assump.Mn2} is weaker than and implied by Assumption \ref{assump.Mn}. Therefore, by Theorem \ref{thm:suprate.OU}, we can take $\tilde \varsigma_n = 3n^{-1/2+(\overline\kappa+\underline\kappa/2)}$ for 1-dimensional Ornstein-Uhlenbeck process in Assumption \ref{assump.Mn2}.

To quantify the convergence rate of ${\vv}_n(s^*;\sigma^2,\alpha)$, we follow the literature on kriging and define the \textit{fill distance} given a set of design points $\Scal_n=\{s_1,\ldots,s_n\}$ as
\begin{align} \label{eq:fill.dist}
h_{\Scal_n} & = \sup_{s\in \Scal} \min_{s_i\in\Scal_n} \|s-s_i\|.
\end{align}
The fill distance quantifies the space-filling properties of $\Scal_n$. The convergence rates of kriging in Model \eqref{eq:obs.model} can often be expressed as a function of $h_{\Scal_n}$ (\citet{Wen05}, \citet{Wangetal19}, \citet{TuoWang20}, \citet{Wynetal21}). Then we have the following theorem on the posterior convergence rate of Bayesian GP predictive variance.

\begin{theorem} \label{thm:vn.rate}
Suppose that Assumptions \ref{assump.m.func}, \ref{prior.1}, \ref{prior.2}, \ref{prior.3}, \ref{assump.dense1}, and \ref{assump.Mn2} hold. Let $C_{\bbm}=\sum_{j=1}^p\|\bbm_j\|_{\Wcal_2^{\nu+d/2}(\Scal)}^2$. For an index set $\Ical\subseteq \{1,\ldots,n\}$, let $|\Ical|$ be its cardinality and $M_{\Ical}$ be the submatrix of $M_n$ with row indexes in $\Ical$. Assume that for each $\Scal_n$, $\underline \lambda(M_n,p)=$ $\max_{\Ical\subseteq \{1,\ldots,n\},|\Ical|=p} \lambda_{\min} \left(M_{\Ical}^\top M_{\Ical}\right) / p >0 $. Then for any $\eta,\delta\in (0,1)$, there exist large constants $C_{\vv,1}>0,C_{\vv,2}>0$ that depend on $\sigma_0^2,\alpha_0,\nu,d,T$, and a large constant $C_{\vv,3}>0$ and large integer $N_3$ that depend on $\eta,\delta,\sigma_0^2,\alpha_0,\nu,d,T$, such that for all $n>N_3$,
\begin{align} \label{eq:post.vn.opt.rate}
&\sup_{s^*\in \Scal} {\vv}_n(s^*;\sigma_0^2,\alpha_0) \leq C_{\vv,1} \left[C_{\bbm} \sigma_0^2 \underline \lambda(M_n,p)^{-1} + 1 \right] h_{\Scal_n}^{2\nu}, \quad \text{and } \nonumber \\
& \Pr\left(\Pi\left[\sup_{s^* \in \Scal} {\vv}_n(s^*;\sigma^2,\alpha) \leq C_{\vv,2} \left[C_{\vv,3} C_{\bbm} \underline \lambda(M_n,p)^{-1} + 1 \right] h_{\Scal_n}^{2\nu}~ \Big|~ Y_n\right] > 1-\delta \right) > 1-\eta .
\end{align}
\end{theorem}
Theorem \ref{thm:vn.rate} essentially shows that with $(\sigma^2,\alpha)$ randomly drawn from the posterior distribution $\Pi(\cdot|Y_n)$, the Bayesian GP predictive variance ${\vv}_n(s^*;\sigma^2,\alpha)$ converges to zero at almost the same rate as the oracle predictive variance ${\vv}_n(s^*;\sigma_0^2,\alpha_0)$ using the true parameters $(\sigma_0^2,\alpha_0)$, as $n\to\infty$ in $P_{(\sigma_0^2,\alpha_0)}$-probability. Given that the posterior support of $(\sigma^2,\alpha)$ is unbounded, $\sup_{s^*\in \Scal} {\vv}_n(s^*;\sigma^2,\alpha)$ with ${\vv}_n(s^*;\sigma^2,\alpha)$ defined in \eqref{eq:varBLUP.true} could be potentially very large if $\sigma^2$ is large. However, our Theorem \ref{thm:vn.rate} shows that the posterior convergence rate can still be controlled even with $(\sigma^2,\alpha)$ randomly drawn from the posterior with unbounded support. The proof of Theorem \ref{thm:vn.rate} crucially depends on the limiting posterior distribution of $(\theta,\alpha)$ proved in Theorem \ref{thm:bvm2:joint}.

The convergence rates of GP predictive error have been extensively studied in the frequentist literature (\citet{YakSzi85}, \citet{Stein90a}, \citet{Wangetal19}, \citet{TuoWang20}, etc.) \citet{WuSch93} has shown that the squared $L_2$ kriging prediction error for the GP with a Mat\'ern covariance function, fixed covariance parameters, and no regression terms is $O(h_{\Scal_n}^{2\nu})$ for sufficiently small $h_{\Scal_n}$. \citet{Rit00} and \citet{TuoWang20} have proved that for the GP with isotropic Mat\'ern $\sigma_0^2 K_{\alpha_0,\nu}$ and no regression terms, the \textit{optimal} convergence rate of squared $L_2$ kriging prediction error is $n^{-2\nu/d}$, which is also a lower bound and not improvable. This optimal rate $n^{-2\nu/d}$ can be attained when $\Scal_n$ has the quasi-uniform design, such as a regular grid in $\Scal$, such that $h_{\Scal_n} \asymp n^{-1/d}$; see Table 1 of \citet{TuoWang20}. If $C_{\bbm} \underline \lambda(M_n,p)^{-1}$ in Theorem \ref{thm:vn.rate} is of constant order, then Theorem \ref{thm:vn.rate} provides the upper bound of the order $h_{\Scal_n}^{2\nu/d} \asymp n^{-2\nu/d}$ for the Bayesian GP predictive variance ${\vv}_n(s^*;\sigma^2,\alpha)$ with a quasi-uniform design $\Scal_n$, which matches up with the optimal rate of squared $L_2$ kriging prediction error.

The multiplicative factor $C_{\bbm} \underline\lambda(M_n,p)^{-1}$ in the upper bounds in Theorem \ref{thm:vn.rate} is due to the regression terms $\bbm(\cdot)^\top \beta$. The same factor also appears in the frequentist kriging error bound in Theorem 2 of \citet{Wangetal19} under a fixed covariance functions. By Assumption \ref{assump.m.func}, $C_{\bbm}$ is already a constant. In many applications, the term $\underline\lambda(M_n,p)$ is bounded from below by constant for fixed $p$ as $n\to\infty$, for example, when $\Scal_n$ is either some regular grid in $\Scal$ or drawn from some sampling distribution (\citet{Wangetal19}). Then Theorem \ref{thm:vn.rate} leads to the optimal convergence rate for the posterior predictive variance with randomly drawn $(\sigma^2,\alpha)$.

In the special case of $p=1$, $\bbm_1(\cdot)\equiv 1$, and Mat\'ern with $d=1$ and $\nu=1/2$, \citet{PutYou01} has shown the stronger frequentist asymptotic efficiency in linear prediction. Therefore, one can possibly establish the Bayesian posterior asymptotic efficiency similar to Theorem \ref{thm:pae.main} for this special case. However, posterior asymptotic efficiency for the general universal kriging model \eqref{eq:obs.model} with $p>1$ regression functions, a general smoothness parameter $\nu>0$ and $d\in\{1,2,3\}$ is technically very challenging and likely to involve more demanding assumptions on the functions $\bbm_1(\cdot),\ldots,\bbm_p(\cdot)$ and the sampling design of $\Scal_n$. While we leave this general problem for future research, we provide some empirical evidence of this posterior asymptotic efficiency in the simulation study in Section S7 of the Supplementary Material.

Our results on convergence rates are not directly comparable with the previous literature on Bayesian Gaussian process regression, such as \citet{VarZan08a}, \citet{VarZan09}, \citet{VarZan11}, \citet{YanTok15}, etc., since our model assumes a random sample path $Y(\cdot)$ from a GP instead of a deterministic true function, and our model does not contain the additional measurement error as in these works.

\section{Simulation Study} \label{sec:simulation}
We verify our limiting theorems and posterior asymptotic efficiency using several numerical examples. In this section, we consider the 1 and 2-dimensional Ornstein-Uhlenbeck process with $\nu=1/2$ in the isotropic Mat\'ern covariance function without the regression terms $\bbm(\cdot)^\top \beta$. We provide additional simulation results for the model with regression terms $\bbm(\cdot)^\top \beta$ for $\nu=1/2,1/4,3/2$ and dimension $d=1,2$ in Section S7 of the Supplementary Material.

In the model without regression terms, we have $Y(s)=X(s)$ for $s\in \Scal$, $d=1,2$, and $X(\cdot)\sim \gp(0,\sigma_0^2 K_{\alpha_0,1/2})$. The main purpose is to verify Theorems \ref{thm:bvm2:joint} and \ref{thm:OU1}. The true covariance parameters are $\sigma_0^2=2$, $\alpha_0=1$, and $\theta_0=2$. We assign independent gamma priors to $\theta$ and $\alpha$, with the same shape parameter 1.1 and rate parameter 0.1. This prior satisfies Assumptions \ref{prior.1}, \ref{prior.2}, and the right tail condition (the second relation of \eqref{A3.1.OU}) in \ref{prior.3OU}, but does not satisfy the left tail condition (the first relation of \eqref{A3.1.OU}) in \ref{prior.3OU}; see Proposition \ref{prop:prior2}. We will see that empirically this prior still yields convergent results.

We consider two cases with dimensions $d=1$ and $d=2$. For the $d=1$ case, we set $\Scal=[0,1]$ and the sampling points of $\Scal_n$ to be the grid $s_i=\tfrac{2i-1}{2n}$ ($i=1,\ldots,n$), for $n=25,50,100,200,400$. For the $d=2$ case, we set $\Scal=[0,1]^2$ and the sampling points of $\Scal_n$ to be the regular grid $\left(\tfrac{2i-1}{2m},\tfrac{2j-1}{2m}\right)$ ($i,j=1,\ldots,m$), for $m=10,20,30$ and $n=m^2$. Then we draw $Y_n$ from the mean zero Gaussian process with the $\nu=1/2$ Mat\'ern covariance function observed on $\Scal_n$. We use the random walk Metropolis algorithm (RWM) to draw $5000$ samples after $1000$ burnins from the joint posterior $\Pi(\ud \theta,\ud \alpha|Y_n)$ and the limiting posterior $\mathcal{N}\big(\ud \theta \big| \widetilde\theta_{\alpha_0}, 2\theta_0^2/n \big) \times \widetilde \Pi(\ud \alpha|Y_n)$ in Theorem \ref{thm:bvm2:joint}, respectively. For the $d=1$ case, we further use RWM to draw 5000 samples from the limiting posterior $\mathcal{N}\big(\ud \theta \big| \widetilde\theta_{\alpha_0}, 2\theta_0^2/n \big) \times \Pi_*(\ud \alpha|Y_n)$ in Theorem \ref{thm:OU1}.

We compare the true posterior distribution with the limiting posterior distributions using two criteria: (a) the closeness of our limiting distributions in Theorems \ref{thm:bvm2:joint} and \ref{thm:OU1} to the true posterior, and (b) the convergence of the two asymptotic efficiency measures in \eqref{eq:ae.original} with $(\theta,\alpha)$ drawn from the joint posterior. For (a), since it is difficult to evaluate the total variation distance between two 2-dimensional posterior distributions based on finite posterior samples, we instead compute the Wasserstein-2 ($W_2$) distance between the marginal posteriors for $\theta$ and $\alpha$, respectively.  The $W_2$ distance between two 1-dimensional distributions $F_1$ and $F_2$ has the simple expression $W_2(F_1,F_2)^2=\int_0^1 \big[F_1^{-1}(u)-F_2^{-1}(u)\big]^2\ud u$, where $F_1^{-1}$ and $F_2^{-1}$ are the corresponding quantile functions. With finite samples from $F_1$ and $F_2$, $W_2(F_1,F_2)$ can be accurately estimated by replacing $F_1^{-1}$ and $F_2^{-1}$ with the empirical quantile functions (\citet{Lietal17}). In our simulation study, we replace $F_1$ and $F_2$ with $\Pi(\ud\theta|Y_n)$ and $\mathcal{N}\big(\ud \theta \big| \widetilde\theta_{\alpha_0}, 2\theta_0^2/n \big)$ for $\theta$, and $\Pi(\ud\alpha|Y_n)$ and $\widetilde \Pi(\ud \alpha|Y_n)$ for $\alpha$, respectively. For the $d=1$ case, we also compute the $W_2$ distance between $\Pi(\ud\alpha|Y_n)$ and $\Pi_*(\ud \alpha|Y_n)$. The convergence in $W_2$ distance is equivalent to the weak convergence plus the convergence in the second moment (\citet{Vil08}). Therefore, it provides useful empirical evidence for convergence in the posterior means and variances of $\theta$ and $\alpha$. Theoretically, \citet{ChaWal20} has shown that the Wasserstein distance provides an upper bound for the total variation distance between two kernel smoothed densities from discrete draws.

For the $d=1$ case, Table \ref{tab:W2.dim1} reports the estimated posterior means under the true posterior $\Pi(\cdot|Y_n)$, the limiting posterior $\widetilde\Pi(\cdot|Y_n)$ in Theorem \ref{thm:bvm2:joint}, the limiting posterior $\Pi_*(\cdot|Y_n)$ in Theorem \ref{thm:OU1}, and the $W_2$ distances between the marginal posteriors. The posterior mean estimates of the microergodic $\theta$ are accurate for the true value $\theta_0=2$ and the posterior variance decreases as $n$ increases. As expected, the posterior mean estimates of $\alpha$ are not consistent for the true $\alpha_0=1$, and show no sign of convergence for all three distributions. For the approximation accuracy, we can see that the $W_2$ distance between the true marginal posterior of $\theta$ and the normal limit in our theorem decreases quickly to zero as $n$ increases. Furthermore, the $W_2$ distances between the true marginal posterior of $\alpha$ and the two approximations, the profile posterior $\widetilde\Pi(\ud\alpha|Y_n)$ and the polynomially tilted normal distribution $\Pi_*(\ud\alpha|Y_n)$ in Theorem \ref{thm:OU1} also show clear decreasing trends towards zero as $n$ increases. These empirical observations have verified our limiting distributions in Theorems \ref{thm:bvm2:joint} and \ref{thm:OU1} for the 1-dimensional Ornstein-Uhlenbeck process.

\begin{table}[ht]
\caption{Parameter estimation and Wasserstein-2 distances between the true posterior and the limiting posteriors in Theorems \ref{thm:bvm2:joint} and \ref{thm:OU1} for the model with $\nu=1/2$, $d=1$ and without regression terms. $\EE(\cdot|Y_n)$, $\Var(\cdot|Y_n)$, $\widetilde \EE(\cdot|Y_n)$, $\widetilde \Var(\cdot|Y_n)$, $\EE_*(\cdot|Y_n)$, and $\Var_*(\cdot|Y_n)$ are the posterior means and variances under the true posterior, the limiting posterior in Theorem \ref{thm:bvm2:joint}, and the limiting posterior in Theorem \ref{thm:OU1}. The true parameter values are $\theta_0=2$ and $\alpha_0=1$. All numbers are averaged over 100 macro replications. The standard errors are in the parentheses.}
\label{tab:W2.dim1}
\centering
{\footnotesize
\begin{tabular}{c|ccccc}
\hline
$d=1$ & $n=25$ & $n=50$ & $n=100$ & $n=200$ & $n=400$ \\
\hline
$\EE(\theta|Y_n)$ & 2.6795 (0.0763) & 2.1932 (0.0434) & 2.1467 (0.0269) & 2.0740 (0.0202) & 2.0320 (0.0139) \\
$\Var(\theta|Y_n)$ & 0.9825 (0.0557) & 0.2441 (0.0096) & 0.1031 (0.0026) & 0.0455 (0.0010)  & 0.0212 (0.0003) \\
\hdashline
$\widetilde \EE(\theta|Y_n)$ & 2.0404 (0.0592) &  1.9357 (0.0391) & 2.0214 (0.0193) & 2.0130 (0.0251) & 2.0028 (0.0136) \\
$\widetilde \Var(\theta|Y_n)$ & 0.3197 (0.0007) & 0.1599 (0.0003) & 0.0798 (0.0002) & 0.0399 (0.0001) & 0.0200 (0.0000) \\
\hline
$\EE(\alpha|Y_n)$ &  3.1924 (0.2459) & 2.9803 (0.2527) & 2.7392 (0.2049) & 2.9947 (0.2819) & 2.5075 (0.2044) \\
$\Var(\alpha|Y_n)$ & 5.3673 (0.8032) & 4.0441 (0.6657) & 2.9987 (0.4264) & 3.7074 (0.6484) & 2.5080 (0.3876) \\
\hdashline
$\widetilde \EE(\alpha|Y_n)$ &  2.9717 (0.2246)& 2.8767 (0.2389) & 2.6941 (0.2001) & 2.9534 (0.2791) & 2.5012 (0.2044) \\
$\widetilde \Var(\alpha|Y_n)$ & 4.5474 (0.6732) & 3.7045 (0.5762) & 2.9094 (0.4093) & 3.6840 (0.6396) & 2.4664 (0.3818) \\
\hdashline
$\EE_*(\alpha|Y_n)$ & 2.5267 (0.1789) & 2.6534 (0.2135) & 2.5873 (0.1874) & 2.9105 (0.2723) & 2.4933 (0.2044) \\
$\Var_*(\alpha|Y_n)$ & 2.5207 (0.3018) & 2.7894 (0.3862) & 2.5783 (0.3414) & 3.3733 (0.5548) & 2.4291 (0.3660) \\
\hline
\end{tabular}
\begin{tabular}{c|ccccc}
\hline
$d=1$ & $n=25$ & $n=50$ & $n=100$ & $n=200$ & $n=400$ \\
\hline
\multirow{2}{*}{$W_2\left(\Pi(\ud\theta|Y_n),\Ncal\left(\ud\theta\Big|\widetilde\theta_{\alpha_0}, \tfrac{2\theta_0^2}{n}\right) \right)$} & 0.8051  & 0.3000  & 0.1449  & 0.0706  & 0.0335  \\
& (0.0326) & (0.0101) & (0.0042) & (0.0024) & (0.0010) \\
\multirow{2}{*}{$W_2(\Pi(\ud\alpha|Y_n),\widetilde \Pi(\ud\alpha|Y_n))$} & 0.3175  & 0.1807  & 0.1260  & 0.1303  & 0.1073  \\
& (0.0290) & (0.0183) & (0.0086) & (0.0099) & (0.0077) \\
\multirow{2}{*}{$W_2(\Pi(\ud\alpha|Y_n), \Pi_*(\ud\alpha|Y_n))$} & 0.8972  & 0.4259  & 0.2131  & 0.1583  & 0.1095  \\
& (0.0874) & (0.0504) & (0.0211) & (0.0160) & (0.0075) \\
\hline
\end{tabular}
}
\end{table}

For the $d=2$ case, the results are summarized in Table \ref{tab:W2.dim2}, showing similar trends to those from the $d=1$ case. The posterior mean estimates of $\theta$ are accurate with standard errors decreasing with $n$. The posterior mean estimates of $\alpha$ happen to be close to $\alpha_0=1$, though both the true posterior variance and the asymptotic posterior variance remain above 0.4 as $n$ increases. The $W_2$ distance between the true marginal posteriors and the limiting posteriors in Theorem \ref{thm:bvm2:joint} converges to zero as $n$ increases. This has verified the limiting distribution in Theorem \ref{thm:bvm2:joint} for the 2-dimensional process.

\begin{table}[ht]
\caption{Parameter estimation and Wasserstein-2 distances between the true posterior and the limiting posteriors in Theorem \ref{thm:bvm2:joint} for the model with $\nu=1/2$, $d=2$ and without regression terms. $\EE(\cdot|Y_n)$, $\Var(\cdot|Y_n)$, $\widetilde \EE(\cdot|Y_n)$, and $\widetilde \Var(\cdot|Y_n)$ are the posterior means and variances under the true posterior and the limiting posterior in Theorem \ref{thm:bvm2:joint}. The true parameter values are $\theta_0=2$ and $\alpha_0=1$. All numbers are averaged over 100 macro replications. The standard errors are in the parentheses.}
\label{tab:W2.dim2}
\centering
{
\footnotesize
\begin{tabular}{c|ccc}
\hline
$d=2$ & $n=10^2$ & $n=20^2$ & $n=30^2$ \\
\hline
$\EE(\theta|Y_n)$ & 2.0211 (0.0258) & 2.0152 (0.0135)& 1.9959 (0.0097)  \\
$\Var(\theta|Y_n)$ & 0.0835 (0.0022) & 0.0203 (0.0003) & 0.0089 (0.0001) \\
\hdashline
$\widetilde \EE(\theta|Y_n)$ & 2.0150 (0.0262) & 2.0110 (0.0134) & 1.9939 (0.0096)  \\
$\widetilde \Var(\theta|Y_n)$ & 0.0798 (0.0002) & 0.0200 (0.0000) & 0.0089 (0.0001) \\
\hdashline
$\EE(\alpha|Y_n)$ & 1.0936 (0.0479) & 1.1317 (0.0456) & 1.0909 (0.0397)  \\
$\Var(\alpha|Y_n)$ & 0.5054 (0.0392) & 0.4864 (0.0352) &  0.4500 (0.0266) \\
\hdashline
$\widetilde \EE(\alpha|Y_n)$ & 1.1094 (0.0486) & 1.1392 (0.0459) &  1.0941 (0.0397)  \\
$\widetilde \Var(\alpha|Y_n)$ &  0.5131 (0.0406) & 0.4796 (0.0348) & 0.4385 (0.0261) \\
\hline
$W_2\left(\Pi(\ud\theta|Y_n),\Ncal\left(\ud\theta\Big|\widetilde\theta_{\alpha_0}, \tfrac{2\theta_0^2}{n}\right) \right)$ & 0.0652 (0.0024) & 0.0185 (0.0008) & 0.0090 (0.0003) \\
$W_2(\Pi(\ud\alpha|Y_n),\widetilde \Pi(\ud\alpha|Y_n))$ & 0.0547 (0.0030) & 0.0514 (0.0024)  & 0.0505 (0.0021) \\
\hline
\end{tabular}
}
\end{table}

Figure \ref{fig:contour1} illustrates the convergence of posterior densities for the $d=1$ case. With $n=50$, there exists noticeable difference between the true posterior and the limiting posteriors. But their difference gradually disappears as $n$ increases. Furthermore, as $n$ increases, the posterior shrinks along the $\theta$ direction, but remains spread out in the $\alpha$ direction. The ``ridge" of the joint posterior is the REML $\widetilde\theta_{\alpha}$, which increases with $\alpha$ as proved in Lemma \ref{lem:dimension reduction}, but becomes flatter as $n$ increases, indicating the convergence from $\widetilde\theta_{\alpha}$ to $\theta_0=2$ over all values of $\alpha$.

\begin{figure}
\centering
\includegraphics[width=0.84\textwidth]{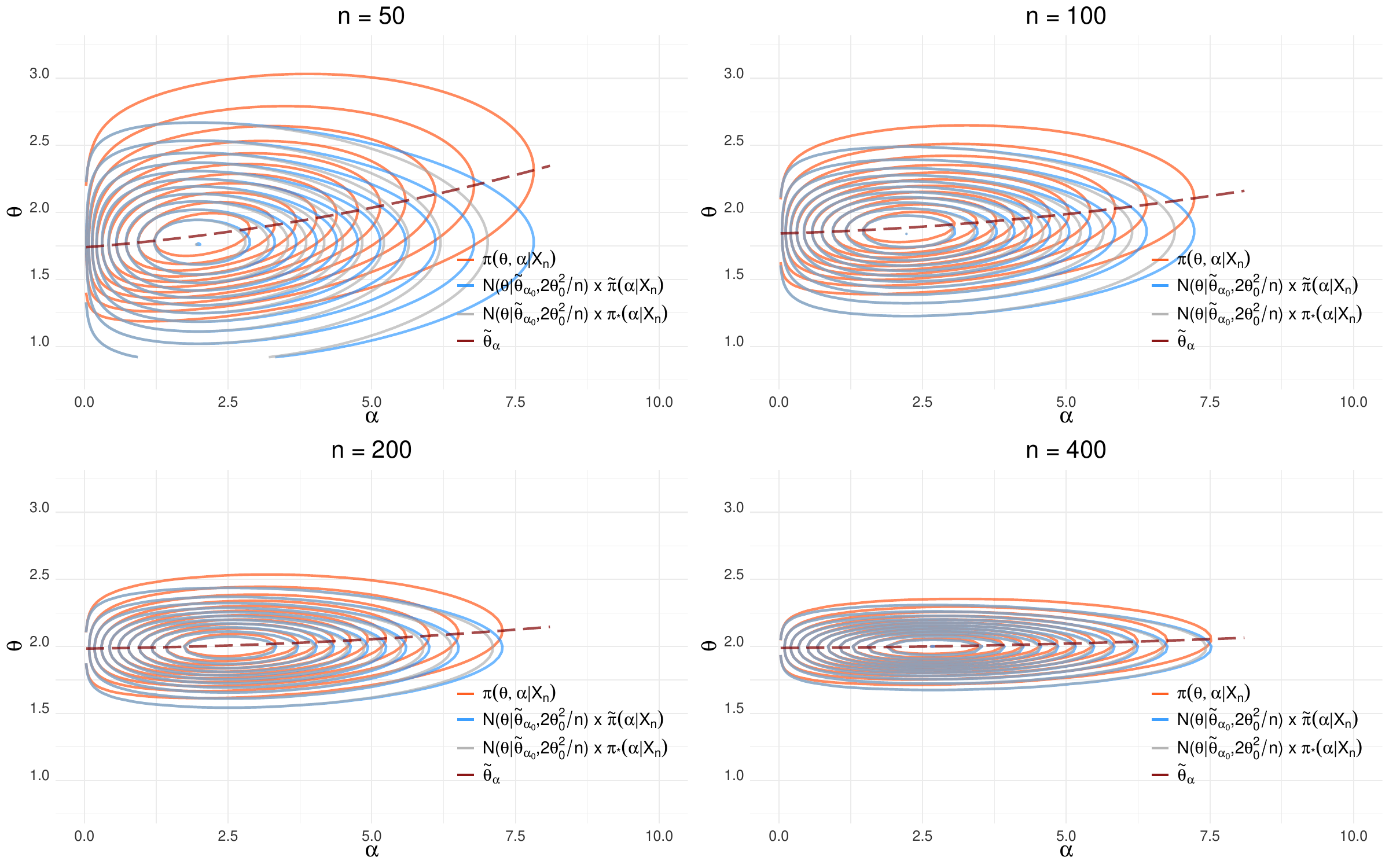}
\caption{Contour plots of the true joint posterior density $\pi(\theta,\alpha|Y_n)$ (in red), the limiting posterior density $\mathcal{N}(\theta|\widetilde \theta_{\alpha_0}, 2\theta_0^2/n)\times \widetilde \pi(\alpha|Y_n)$ in Theorem \ref{thm:OU1} Eq. \eqref{eq:OU.joint1} (in blue), and the limiting posterior density $\mathcal{N}(\theta|\widetilde \theta_{\alpha_0}, 2\theta_0^2/n)\times \pi_*(\alpha|Y_n)$ in Theorem \ref{thm:OU1} Eq. \eqref{eq:OU.joint2} (in grey), for the 1-d Ornstein-Uhlenbeck process with sample size $n=50,100,200,400$ in the model without regression terms. The dashed line is the ``ridge" REML $\widetilde\theta_{\alpha}$ given in \eqref{tildetheta1}. The true parameter values are $\theta_0=2$ and $\alpha_0=1$.}
\label{fig:contour1}
\end{figure}
For the posterior asymptotic efficiency in (b), we compute the two asymptotic efficiency measures in \eqref{eq:ae.original} and Theorems \ref{thm:pae.main} and \ref{thm:suprate.OU} empirically, using the posterior samples of $(\theta,\alpha)$. To approximate the supremums, we take the maximum of the ratios that depend on the random $(\sigma^2,\alpha)$ drawn from the posterior:
\begin{align}\label{eq:ratiodiff.table}
\mathsf{r}_{1n}(s^*) = \left| \tfrac{{\EE}_{(\sigma^2,\alpha)}\left\{e_n(s^*;\alpha)^2\right\}}{{\EE}_{(\sigma_0^2,\alpha_0)}\left\{e_n(s^*;\alpha)^2\right\} } - 1 \right| \text{ and } \mathsf{r}_{2n}(s^*) = \left| \tfrac{{\EE}_{(\sigma^2,\alpha)}\left\{e_n(s^*;\alpha)^2\right\}}{{\EE}_{(\sigma_0^2,\alpha_0)}\left\{e_n(s^*;\alpha_0)^2\right\} } - 1 \right|
\end{align}
over a large number of testing points $s^*$ from the Latin hypercube design. We use $1000$ testing points in $\Scal=[0,1]$ for the $d=1$ case, and $2500$ testing points in $\Scal=[0,1]^2$ for the $d=2$ case. Let the testing set be $\Scal^*$. We report the estimated posterior mean $\EE[\max_{s^*\in \Scal^*}\mathsf{r}_{1n}(s^*)|Y_n]$ and $\EE[\max_{s^*\in \Scal^*}\mathsf{r}_{2n}(s^*)|Y_n]$. The results are summarized in Table \ref{tab:mse}. The simulation results show that the posterior means of the two ratios in \eqref{eq:ratiodiff.table} decrease as $n$ increases, and their standard errors also decrease. This is observed for both 1 and 2-dimensional domains.

\begin{table}[ht]
\caption{The posterior means of the two ratios of predictive MSEs defined in \eqref{eq:ratiodiff.table} maximized over $2500$ testing points $s^*$ for the model with $\nu=1/2$ and without regression terms, averaged over 100 macro replications. The standard errors are in the parentheses.}
\label{tab:mse}
\centering
{
\footnotesize
\begin{tabular}{c|ccccc}
\hline
$d=1$ & $n=25$ & $n=50$ & $n=100$ & $n=200$ & $n=400$ \\
\hline
\multirow{2}{*}{$\EE\big[\max\limits_{s^*\in\Scal^*}\mathsf{r}_{1n}(s^*)|Y_n\big]$} & 0.5129 & 0.2804 & 0.1796 & 0.1232 & 0.0823 \\
& (0.0442) & (0.0197) & (0.0125) & (0.0082) & (0.0055)  \\
\multirow{2}{*}{$\EE\big[\max\limits_{s^*\in\Scal^*}\mathsf{r}_{2n}(s^*)|Y_n\big]$} & 0.4958 & 0.2626 & 0.1741 & 0.1188 & 0.0804 \\
& (0.0447) & (0.0198) & (0.0126) & (0.0082) & (0.0055)  \\
\hline
\hline
$d=2$ & $n=10^2$ & $n=20^2$ & $n=30^2$ & \\
\hline
\multirow{2}{*}{$\EE\big[\max\limits_{s^*\in\Scal^*}\mathsf{r}_{1n}(s^*)|Y_n\big]$} & 0.1887 & 0.0736 & 0.0702 & \\
& (0.0104) & (0.0051) & (0.0041) &   \\
\multirow{2}{*}{$\EE\big[\max\limits_{s^*\in\Scal^*}\mathsf{r}_{2n}(s^*)|Y_n\big]$} & 0.1827 & 0.0718 &  0.0705 &  \\
& (0.0101) & (0.0050) & (0.0041) &   \\
\hline
\end{tabular}
}
\end{table}

\section{Discussion} \label{sec:discussion}

Our theory has answered the two questions from the SST example in Section \ref{sec:intro}. For Question (i), Theorems \ref{thm:bvm2:joint} and \ref{thm:OU1} in Section \ref{sec:main.bvm} show that the posterior of the microergodic parameter $\theta$ converges to a normal limit at the parametric rate, while the posterior of the range parameter $\alpha$ does not converge to any point mass in general. For Question (ii), Theorems \ref{thm:pae.micro}, \ref{thm:pae.main}, \ref{thm:suprate.OU} and \ref{thm:vn.rate} in Section \ref{sec:PAE} show that the predictive performance based on the covariance parameters randomly drawn from their posterior distribution is asymptotically as good as the oracle predictive performance based on the true covariance parameters.

We discuss several future directions based on the current work. In many spatial applications, one may also add a measurement error term to the model, such that $Y(s_i)=\bbm(\cdot)^\top\beta + X(s_i)+\varepsilon(s_i)$ for $i=1,\ldots,n$ with a noise process $\{\varepsilon(s):s\in \Scal\}$ that is independent of $X$. Often it is assumed that $\varepsilon(s)\sim \mathcal{N}(0,\tau^2)$ for all $s\in \Scal$. The parameter $\tau^2$ is the \textit{nugget} parameter (\citep{Cre93}). From the frequentist fixed-domain asymptotic theory, it is already known (\citep{Stein90c}) that the presence of nugget parameter $\tau^2$ will significantly change the convergence rate of the microergodic parameter $\theta$, due to the convolution with Gaussian noise. For example, as shown in \citep{Chenetal00} for the 1-dimensional Ornstein-Uhlenbeck process ($\nu=1/2$) on an equispaced grid, the convergence rate of the MLE of $\theta$ deteriorates from $n^{-1/2}$ to $n^{-1/4}$, though both $\theta$ and the nugget $\tau^2$ can still be consistently estimated; see also the recent development in \citep{Tanetal19}. Therefore, in the Bayesian setting, we expect that the limiting posterior distribution of $(\theta,\alpha,\tau^2)$ will be dramatically different from those in Theorems \ref{thm:bvm2:joint} and \ref{thm:OU1}.

In the proof of Lemma \ref{lem:dimension reduction} and Theorem \ref{thm:bvm2:joint}, we have derived many useful properties of the spectral density of Mat\'ern covariance functions. These derivations can be possibly extended to the tapered Mat\'ern covariance functions (\citep{Duetal09}, \citep{WangLoh11}) and the generalized Wendland (GW) covariance functions (\citep{Gne02}), whose spectral densities also have polynomially decaying tails (\citep{Kauetal08}, \citep{Bevetal19}). As shown in Lemma 1 of \citep{Bevetal19}, for the GP model with mean zero, the MLE of the GW microergodic parameter also has the monotonicity property. Therefore, with suitable modification, we expect that our technical proofs can be generalized to a broader class of covariance functions whose spectral densities share similar tail behavior to Mat\'ern.

We have only considered the isotropic Mat\'ern covariance functions. For anisotropic Mat\'ern covariance functions, the existing fixed-domain asymptotic theory is very limited. Only a few special cases such as $\nu=1/2$ (\citep{Ying93}), $\nu=3/2$ (\citep{Loh05}), and $d>4$ (\citep{And10}) have been studied, while the theory for the anisotropic Mat\'ern with a general $\nu>0$ and $d=1,2,3$ remains unknown. We leave these directions for future research.

\vspace{3mm}

\noindent \textbf{Acknowledgements}
The author sincerely thanks the Associate Editor and two anonymous referees for valuable comments that have significantly improved the paper. The author thanks Michael L. Stein, Wei-Liem Loh, Wenxin Jiang, Sanvesh Srivastava, and Yichen Zhu for helpful discussion. The author was supported by the Singapore Ministry of Education Academic Research Funds Tier 1 Grants R-155-000-201-114 and A-0004822-00-00.

\vspace{1cm}
\appendix

\begin{center}
{\bf \Large Supplementary Material to ``Bayesian Fixed-domain Asymptotics for Covariance Parameters in a Gaussian Process Model"}
\end{center}

\setcounter{equation}{0}
\setcounter{lemma}{0}
\setcounter{section}{0}
\setcounter{table}{0}
\setcounter{figure}{0}
\setcounter{theorem}{0}
\renewcommand{\theequation}{S.\arabic{equation}}
\renewcommand{\thelemma}{S.\arabic{lemma}}
\renewcommand{\thetheorem}{S.\arabic{theorem}}

\makeatletter
\renewcommand{\thefigure}{S.\@arabic\c@figure}
\makeatother

\renewcommand{\thetable}{S.\arabic{table}}

\renewcommand\appendixname{ }

\renewcommand\thesection{S\arabic{section}}

The Supplementary Material includes more simulation results and all technical proofs of the theorems, lemmas, propositions, and corollaries in the main text. The contents are organized as follows. \\

Section \ref{supsec:lem.dimred.main} provides the proof of the monotonicity and uniform convergence of REML in Lemma \ref{lem:dimension reduction} of the main text, as well as auxiliary results on RKHS theory and spectral analysis of Mat\'ern covariance functions. Section \ref{supsec:proflik} includes technical lemmas for the profile likelihood function. Section \ref{sec:auxiliary} presents the proof of Theorem \ref{thm:bvm1:theta} and Theorem \ref{thm:bvm2:joint} of the main text, as well as the theory for $d\geq 5$. Section \ref{supsec:2prop} presents the proof of Propositions \ref{prop:prior2} and \ref{prop:prior3} of the main text. Section \ref{supsec:OU1} presents the proof of Theorem \ref{thm:OU1} and Corollary \ref{cor:OU2}. Section \ref{supsec:pae} presents the proof of all theorems in Section \ref{sec:PAE} of the main text, including Theorems \ref{thm:pae.micro}, \ref{thm:pae.main}, \ref{thm:suprate.OU}, and \ref{thm:vn.rate}. Section \ref{sec:add.sim} includes the additional simulation results for the model with regression terms for $\nu=1/2,1/4,3/2$ in both $d=1$ and $d=2$ cases. To keep consistency, every lemma in the Supplementary Material is immediately followed by its proof. \\

We first define some universal notation that will used throughout the proofs. Let $\RR^+=(0,+\infty)$ and $\ZZ^+$ be the set of all positive integers. For any $x=(x_1,\ldots,x_d)^\top \in \RR^d$, we let $\|x\|=\sqrt{\sum_{i=1}^d x_i^2}$, $\|x\|_1=\sum_{i=1}^d |x_i|$, and $\|x\|_{\infty}=\max(x_1,\ldots,x_d)$. For two positive sequences $a_n$ and $b_n$, we use $a_n\prec b_n$ and $b_n\succ a_n$ to denote the relation $\lim_{n\to\infty} a_n/b_n=0$, $a_n\preceq b_n$ and $b_n\succeq a_n$ to denote the relation $\limsup_{n\to\infty} a_n/b_n<+\infty$, and $a_n\asymp b_n$ to denote the relation $a_n\preceq b_n$ and $a_n\succeq b_n$. For any integers $k,m$, we let $I_k$ be the $k\times k$ identity matrix, $0_k$ and $1_k$ be the $k$-dimensional column vectors of all zeros and all ones, $0_{k\times m}$ be the $k\times m$ zero matrix. For any generic matrix $A$, $cA$ denotes the matrix of $A$ with all entries multiplied by the number $c$, and $|A|$ denotes the determinant of $A$. For a square matrix $A$, $\tr(A)$ denotes the trace of $A$. If $A$ is symmetric positive semidefinite, then $\lambda_{\min}(A)$ and $\lambda_{\max}(A)$ denote the smallest and largest eigenvalues of $A$, and $A^{1/2}$ denotes a symmetric positive semidefinite square root of $A$. For two symmetric positive semidefinite matrix $A$ and $B$, we use $A\leq B$ and $B\geq A$ to denote the relation that $B-A$ is symmetric positive semidefinite, and use $A<B$ and $B>A$ to denote the relation that $B-A$ is symmetric positive definite. For any matrix $A$, $\|A\|_{\op}=\sqrt{\lambda_{\max}(A^\top A)}$ denotes the operator norm of $A$. Let $\Ncal(\mu,\Sigma)$ be the normal distribution with mean $\mu$ and covariance matrix $\Sigma$. Sometimes to highlight the random variable $Z \sim \Ncal(\mu,\Sigma)$, we also write  $\Ncal(z;\mu,\Sigma)$ and the normal measure as $\Ncal(\ud z;\mu,\Sigma)$. $\pr(\cdot)$ denotes the probability under true probability measure $P_{(\beta_0,\sigma_0^2,\alpha_0)}$. The convergence in distribution is denoted by $\overset{\Dcal}{\rightarrow}$. The acronym i.i.d. stands for ``independent and identically distributed".

\section{Proof of Monotonicity and Uniform Convergence in Lemma \ref{lem:dimension reduction}} \label{supsec:lem.dimred.main}
This section is organized as follows.

Subsection \ref{supsec:mon.Matern} contains Lemmas \ref{lem:ABinv}, \ref{lem:2posdef}, \ref{lem:alpha.monotone.matrix}, and \ref{lem:theta1.monotone} for showing the monotonicity of REML $\widetilde\theta_{\alpha}$ in Part (i) of Lemma \ref{lem:dimension reduction} in the main text. The main proof is given in the strengthened Lemma \ref{lem:theta1.monotone}.

Subsection \ref{supsec:unif.conv} contains Lemmas \ref{lem:REML.decomp}, \ref{lem:theta2.bound.2ends}, \ref{lem:theta3.bound.2ends}, \ref{lem:theta1.bound.2ends}, \ref{lem:sup.theta1}, and \ref{lem:theta.alpha0}, for showing the uniform convergence of REML $\widetilde\theta_{\alpha}$ in Part (ii) of Lemma \ref{lem:dimension reduction} in the main text. We start with a decomposition of the REML $\widetilde\theta_{\alpha}$ in Lemma \ref{lem:REML.decomp}, and then provide detailed concentration inequalities for each terms in Lemmas \ref{lem:theta2.bound.2ends}, \ref{lem:theta3.bound.2ends}, and \ref{lem:theta1.bound.2ends}. The uniform convergence is proved in Lemma \ref{lem:sup.theta1}. Lemma \ref{lem:theta.alpha0} includes the proof of asymptotic normality of the REML $\widetilde\theta_{\alpha}$ in Theorem \ref{thm:bvm1:theta}, as well as a concentration error bound for $\widetilde\theta_{\alpha_0}$, which will be used as a crucial result in the proof of Theorem \ref{thm:bvm1:theta} in Section \ref{sec:auxiliary}.

Subsection \ref{supsec:rkhs} introduces the RKHS theory with the technical Lemmas \ref{lem:matern.sobolev}, \ref{lem:rkhs.quadratic}, and \ref{lem:rkhs.ordering}. They are used for proving Lemma \ref{lem:theta3.bound.2ends} and also later for proving Theorem \ref{thm:vn.rate}.

Subsection \ref{supsec:spectral} includes the spectral analysis of Mat\'ern covariance function, with the technical Lemmas \ref{lem:URU}, \ref{lem:specden_lambda}, \ref{lem:fxi}, \ref{lem:qngn}, \ref{lem:zetabound}, and \ref{lem:weight.bound}. Lemma \ref{lem:weight.bound} is used for proving the concentration inequality in Lemma \ref{lem:sup.theta1}. We also cite the two-sided chi-square concentration inequality from \citet{LauMas00} in Lemma \ref{lem:LauMas00} and the Hanson-Wright inequality from \citet{Hsuetal12} in Lemma \ref{lem:Hsuetal12}.

\vspace{5mm}

We assume Assumptions \ref{assump.m.func} throughout this section. We recall that the universal kriging model \eqref{eq:obs.model} in the main text implies that the underlying true model is $Y_n=M_n \beta_0 + X_n$ with $X_n\sim \Ncal(0_n,\sigma_0^2 R_{\alpha_0})$, where $R_{\alpha}$ is the $n\times n$ Mat\'ern correlation matrix on $\Scal_n=\{s_1,\ldots,s_n\}$ indexed by $\alpha$ with the $(i,j)$th entry $R_{\alpha,ij}=K_{\alpha,\nu}(s_i-s_j)$, for $i,j\in\{1,\ldots,n\}$. The REML $\widetilde\theta_{\alpha}$ is defined as
\begin{align} \label{tildetheta2}
\widetilde \theta_{\alpha} & = \frac{\alpha^{2\nu} Y_n^\top  \left[R_{\alpha}^{-1} - R_{\alpha}^{-1}M_n \big(M_n^\top R_{\alpha}^{-1} M_n + \Omega_{\beta}\big)^{-1} M_n^\top R_{\alpha}^{-1} \right]  Y_n }{n-p} .
\end{align}

We emphasize that all the proofs below apply to any symmetric positive semidefinite matrix $\Omega_{\beta}$, including the special case $\Omega_{\beta}=0_{p\times p}$ corresponding to the noninformative improper prior $\pi(\beta)\propto 1$.

\subsection{Proof of Monotonicity in Part (i) of Lemma \ref{lem:dimension reduction}} \label{supsec:mon.Matern}
\begin{lemma}\label{lem:ABinv}
Suppose that $A_1,A_2\in \RR^{n\times n}$ are two symmetric positive definite matrices and $A_2-A_1$ is also positive (semi)definite. Then $A_1^{-1}-A_2^{-1}$ is symmetric positive (semi)definite.
\end{lemma}

\begin{proof}[Proof of Lemma \ref{lem:ABinv}]
The lemma follows from Theorem 7.7.3 and Corollary 7.7.4 in \citet{HorJoh85}.
\end{proof}

\vspace{5mm}

\begin{lemma}\label{lem:2posdef}
Suppose that $A_1,A_2\in \RR^{n\times n}$ are two symmetric positive definite matrices and $A_2-A_1$ is also positive definite. Then for any $p\times p$ symmetric positive semidefinite matrix $\Omega$ and any full-rank $n\times p$ matrix $G$, the matrix
\begin{align}\label{eq:2A.proj}
\Delta A = \left[A_2-A_2G(G^\top A_2 G +\Omega)^{-1}G^\top A_2\right] - \left[A_1-A_1G(G^\top A_1 G+\Omega)^{-1}G^\top A_1\right].
\end{align}
is symmetric positive semidefinite.
\end{lemma}

\begin{proof}[Proof of Lemma \ref{lem:2posdef}]
For any $t>0$, we let $\Omega_t = \Omega + tI_p$. Then $\Omega_t$ is symmetric positive definite and hence invertible.

By the Sherman-Morrison-Woodbury formula, we have that for $i=1,2$,
\begin{align}\label{eq:A.wood}
& A_i - A_i G(G^\top A_i G +\Omega_t)^{-1}G^\top A_i = \left(A_i^{-1} + G\Omega_t^{-1} G^\top \right)^{-1}.
\end{align}
Since $A_2-A_1$ is symmetric positive definite, by Lemma \ref{lem:ABinv}, we have that $A_1^{-1}-A_2^{-1}$ is symmetric positive definite. But $A_1^{-1}-A_2^{-1} = \left(A_1^{-1} + G\Omega_t^{-1} G^\top\right) - \left(A_2^{-1} + G\Omega_t^{-1} G^\top\right)$ and $A_i^{-1} + G\Omega_t^{-1} G^\top$ for both $i=1,2$ are also symmetric positive definite. Therefore, we apply Lemma \ref{lem:ABinv} again to $A_i^{-1} + G\Omega_t^{-1} G^\top$ for $i=1,2$ to conclude that
$$\left(A_2^{-1} + G\Omega_t^{-1} G^\top\right)^{-1} - \left(A_1^{-1} + G\Omega_t^{-1} G^\top\right)^{-1}$$
is a symmetric positive definite matrix. This together with \eqref{eq:A.wood} implies that
\begin{align}\label{eq:2A.proj1}
&\left(A_2^{-1} + G\Omega_t^{-1} G^\top\right)^{-1} - \left(A_1^{-1} + G\Omega_t^{-1} G^\top\right)^{-1} \nonumber \\
={}& \left[A_2-A_2G(G^\top A_2 G +\Omega_t)^{-1}G^\top A_2\right] - \left[A_1-A_1G(G^\top A_1 G+\Omega_t)^{-1}G^\top A_1\right] \nonumber \\
={}& \left[A_2-A_2G(G^\top A_2 G +\Omega + tI_p)^{-1}G^\top A_2\right] - \left[A_1-A_1G(G^\top A_1 G+\Omega + tI_p)^{-1}G^\top A_1\right]
\end{align}
is symmetric positive definite. The eigenvalues of the last matrix in \eqref{eq:2A.proj1} are continuous functions of $t$. We take $t\to 0+$ and conclude that all eigenvalues of the matrix
$$\left[A_2-A_2G(G^\top A_2 G +\Omega)^{-1}G^\top A_2\right] - \left[A_1-A_1G(G^\top A_1 G+\Omega )^{-1}G^\top A_1\right]$$ are nonnegative. Therefore, this matrix is symmetric positive semidefinite.
\end{proof}

\vspace{8mm}

\begin{lemma}\label{lem:alpha.monotone.matrix}
For all $d\in \ZZ^+$, $\nu\in \RR^+$, for any $0<\alpha_1<\alpha_2<\infty$, the two matrices $\alpha_2^{2\nu} R_{\alpha_2}^{-1} - \alpha_1^{2\nu} R_{\alpha_1}^{-1}$ and $\alpha_2^{d} R_{\alpha_2} - \alpha_1^{d} R_{\alpha_1}$ are always positive definite as long as the $n$ points $\{s_1,\ldots,s_n\}$ are distinct in the domain $\Scal = [0,T]^d$.
\end{lemma}

\begin{proof}[Proof of Lemma \ref{lem:alpha.monotone.matrix}]
We first define the matrix $\Omega^{\dagger} = \alpha_1^{-2\nu}R_{\alpha_1} - \alpha_2^{-2\nu}R_{\alpha_2}$. Then the entries of $\Omega^{\dagger}$ can be expressed in terms of a function $\widetilde K_{\Omega^{\dagger}}: \RR^d\to \RR$, with
\begin{align*}
\Omega^{\dagger}_{ij} &= \widetilde K_{\Omega^{\dagger}}(s_i-s_j) = \alpha_1^{-2\nu} K_{\alpha_1,\nu}(s_i-s_j) - \alpha_2^{-2\nu} K_{\alpha_2,\nu}(s_i-s_j),
\end{align*}
for $i,j\in \{1,\ldots,n\}$. The matrix $\Omega^{\dagger}$ is positive definite if $\widetilde K_{\Omega^{\dagger}}$ is a positive definite function.

From \eqref{f.specden} in Section \ref{supsec:spectral}, for the isotropic Mat\'ern covariance function $\sigma^2 K_{\alpha,\nu}$ defined in \eqref{eq:MaternCov} of the main text, its spectral density is
\begin{align}
f_{\sigma,\alpha}(\omega) &= \frac{\Gamma(\nu+d/2)}{\Gamma(\nu)}\cdot \frac{\sigma^2\alpha^{2\nu}}{\pi^{d/2}\left(\alpha^2+\|\omega\|^2\right)^{\nu+d/2}}, \nonumber
\end{align}
for any $\omega \in \RR^d$. Therefore, we can compute the spectral density of $\widetilde K_{\Omega^{\dagger}}$:
\begin{align} \label{eq:fOmega1}
f_{\Omega^{\dagger}}(\omega) &= \frac{1}{(2\pi)^d} \int_{\RR^d} \ee^{-\imath \omega^\top x} \widetilde K_{\Omega^{\dagger}} (x) \ud x \nonumber \\
&= \frac{1}{(2\pi)^d} \left\{ \alpha_1^{-2\nu}\int_{\RR^d} \ee^{-\imath \omega^\top x}K_{\alpha_1,\nu}(x)\ud x -  \alpha_2^{-2\nu}\int_{\RR^d} \ee^{-\imath \omega^\top x}K_{\alpha_2,\nu}(x)\ud x \right\} \nonumber \\
&= \frac{\Gamma(\nu+d/2)}{\pi^{d/2}\Gamma(\nu)} \left\{ \alpha_1^{-2\nu} \cdot \frac{\alpha_1^{2\nu}}{\left(\alpha_1^2 + \|\omega\|^2\right)^{\nu+d/2}} -  \alpha_2^{-2\nu} \cdot \frac{\alpha_2^{2\nu}}{\left(\alpha_2^2 + \|\omega\|^2\right)^{\nu+d/2}} \right\} \nonumber \\
&= \frac{\Gamma(\nu+d/2)}{\pi^{d/2}\Gamma(\nu)} \left\{\frac{1}{\left(\alpha_1^2 + \|\omega\|^2\right)^{\nu+d/2}} -  \frac{1}{\left(\alpha_2^2 + \|\omega\|^2\right)^{\nu+d/2}} \right\}\nonumber \\
&>0, \text{ for all } \omega \in \RR^d,
\end{align}
where the last step follows because $0<\alpha_1<\alpha_2$. This has shown that $\widetilde K_{\Omega^{\dagger}}$ is indeed a positive definite function. Therefore, $\Omega^{\dagger}=\alpha_1^{-2\nu}R_{\alpha_1} - \alpha_2^{-2\nu}R_{\alpha_2}$ is a positive definite matrix. Since $\{s_1,\ldots,s_n\}$ are distinct, both $R_{\alpha_1}$ and $R_{\alpha_2}$ are positive definite matrices. By Lemma \ref{lem:ABinv}, $\alpha_2^{2\nu}R_{\alpha_2}^{-1} - \alpha_1^{2\nu}R_{\alpha_1}^{-1} $ is a positive definite matrix.

Next, we define the matrix $\Omega^{\ddagger} = \alpha_2^{d}R_{\alpha_2} - \alpha_1^{d}R_{\alpha_1}$. Then the entries of $\Omega^{\ddagger}$ can be expressed in terms of a function $\widetilde K_{\Omega^{\ddagger}}: \RR^d\to \RR$, with
\begin{align*}
\Omega^{\ddagger}_{ij} &= \widetilde K_{\Omega^{\ddagger}}(x_i-x_j) = \alpha_2^{d} K_{\alpha_2,\nu}(x_i-x_j) - \alpha_1^{d} K_{\alpha_1,\nu}(x_i-x_j),
\end{align*}
for $i,j\in \{1,\ldots,n\}$. The matrix $\Omega^{\ddagger}$ is positive definite if $\widetilde K_{\Omega^{\ddagger}}$ is a positive definite function. We compute the spectral density of $\widetilde K_{\Omega^{\ddagger}}$:
\begin{align} \label{eq:fOmega2}
f_{\Omega^{\ddagger}}(\omega) &= \frac{1}{(2\pi)^d} \int_{\RR^d} \ee^{-\imath \omega^\top x} \widetilde K_{\Omega^{\ddagger}} (x) \ud x \nonumber \\
&= \frac{1}{(2\pi)^d} \left\{ \alpha_2^{d}\int_{\RR^d} \ee^{-\imath \omega^\top x}K_{\alpha_2,\nu}(x)\ud x -  \alpha_1^{d}\int_{\RR^d} \ee^{-\imath \omega^\top x}K_{\alpha_1,\nu}(x)\ud x \right\} \nonumber \\
&= \frac{\Gamma(\nu+d/2)}{\pi^{d/2}\Gamma(\nu)} \left\{ \alpha_2^{d} \cdot \frac{\alpha_2^{2\nu}}{\left(\alpha_2^2 + \|\omega\|^2\right)^{\nu+d/2}} -  \alpha_1^{d} \cdot \frac{\alpha_1^{2\nu}}{\left(\alpha_1^2 + \|\omega\|^2\right)^{\nu+d/2}} \right\} \nonumber \\
&= \frac{\Gamma(\nu+d/2)}{\pi^{d/2}\Gamma(\nu)} \left\{\frac{1}{\left(1 + \alpha_2^{-2} \|\omega\|^2\right)^{\nu+d/2}} -  \frac{1}{\left(1+\alpha_1^{-2} \|\omega\|^2\right)^{\nu+d/2}} \right\}\nonumber \\
&>0, \text{ for all } \omega \in \RR^d,
\end{align}
where the last step follows because $0<\alpha_1<\alpha_2$. This has shown that $\widetilde K_{\Omega^{\ddagger}}$ is indeed a positive definite function. Therefore, $\Omega^{\ddagger} = \alpha_2^{d}R_{\alpha_2} - \alpha_1^{d}R_{\alpha_1}$ is a positive definite matrix.
\end{proof}

\vspace{8mm}

We restate and strengthen the monotonicity in Part (i) of Lemma \ref{lem:dimension reduction} in the main text as the following lemma.
\begin{lemma}[Monotonicity of $\widetilde \theta_{\alpha}$ in Lemma \ref{lem:dimension reduction} in the Main Text] \label{lem:theta1.monotone}
Both $\widetilde \theta_{\alpha}$ defined in \eqref{tildetheta2} and $\widetilde \theta_{\alpha}^{(1)}$ defined in \eqref{tildetheta2.1} are non-decreasing functions in $\alpha$ for all $\alpha\in \RR^+$, all $d\in \ZZ^+$, all $\nu\in \RR^+$, for any symmetric positive semidefinite matrix $\Omega_{\beta}$.
\end{lemma}

\begin{proof}[Proof of Lemma \ref{lem:theta1.monotone}]
We first show that $\widetilde \theta_{\alpha}$ is a non-decreasing function in $\alpha$. We notice that $M_n$ is full-rank by Assumption \ref{assump.m.func} and $\Omega_{\beta}$ is positive semidefinite. Consider two generic values $0<\alpha_1<\alpha_2$. By Lemma \ref{lem:alpha.monotone.matrix}, we have that $\alpha_2^{2\nu}R_{\alpha_2}^{-1} - \alpha_1^{2\nu}R_{\alpha_1}^{-1}$ is positive definite.

Therefore, in Lemma \ref{lem:2posdef}, we can set $A_1=\alpha_1^{2\nu} R_{\alpha_1}^{-1}$, $A_2=\alpha_2^{2\nu} R_{\alpha_2}^{-1}$, $G=M_n$, $\Omega=\alpha_1^{2\nu}\Omega_{\beta}$, then the conclusion of Lemma \ref{lem:2posdef} implies that the matrix $\Delta A$ should be positive semidefinite, which implies that
\begin{align}\label{eq:theta1.diff.1}
0_{n\times n} \overset{(i)}{\leq} \Delta A & = \left[\alpha_2^{2\nu} R_{\alpha_2}^{-1} - \alpha_2^{2\nu} R_{\alpha_2}^{-1} M_n\big(\alpha_2^{2\nu} M_n^\top R_{\alpha_2}^{-1}M_n + \alpha_1^{2\nu} \Omega_{\beta}\big)^{-1} M_n^\top \big(\alpha_2^{2\nu} R_{\alpha_2}^{-1}\big) \right] \nonumber \\
& \quad - \left[\alpha_1^{2\nu} R_{\alpha_1}^{-1} - \alpha_1^{2\nu} R_{\alpha_1}^{-1} M_n\big(\alpha_1^{2\nu} M_n^\top R_{\alpha_1}^{-1}M_n + \alpha_1^{2\nu} \Omega_{\beta}\big)^{-1} M_n^\top \big(\alpha_1^{2\nu} R_{\alpha_1}^{-1}\big) \right] \nonumber \\
&\overset{(ii)}{\leq} \left[\alpha_2^{2\nu} R_{\alpha_2}^{-1} - \alpha_2^{2\nu} R_{\alpha_2}^{-1} M_n\big(\alpha_2^{2\nu} M_n^\top R_{\alpha_2}^{-1}M_n + \alpha_2^{2\nu} \Omega_{\beta}\big)^{-1} M_n^\top \big(\alpha_2^{2\nu} R_{\alpha_2}^{-1}\big) \right] \nonumber \\
& \quad - \left[\alpha_1^{2\nu} R_{\alpha_1}^{-1} - \alpha_1^{2\nu} R_{\alpha_1}^{-1} M_n\big(\alpha_1^{2\nu} M_n^\top R_{\alpha_1}^{-1}M_n + \alpha_1^{2\nu} \Omega_{\beta}\big)^{-1} M_n^\top \big(\alpha_1^{2\nu} R_{\alpha_1}^{-1}\big) \right] \nonumber \\
&= \alpha_2^{2\nu}\left[ R_{\alpha_2}^{-1} - R_{\alpha_2}^{-1} M_n\big( M_n^\top R_{\alpha_2}^{-1}M_n +  \Omega_{\beta}\big)^{-1} M_n^\top  R_{\alpha_2}^{-1} \right] \nonumber \\
& \quad - \alpha_1^{2\nu}\left[ R_{\alpha_1}^{-1} -  R_{\alpha_1}^{-1} M_n\big( M_n^\top R_{\alpha_1}^{-1}M_n + \Omega_{\beta}\big)^{-1} M_n^\top  R_{\alpha_1}^{-1}\right] ,
\end{align}
where the $\leq$ relation in the inequalities (i) and (ii) of \eqref{eq:theta1.diff.1} means that if $A\leq B$ for two positive semidefinite matrices $A,B$, then $B-A$ is positive semidefinite; (i) follows from Lemma \ref{lem:2posdef}, and (ii) follows from replacing $\alpha_1^{2\nu}\Omega_{\beta}$ inside the first inverse by $\alpha_2^{2\nu}\Omega_{\beta}$. This implies that the right-hand side of \eqref{eq:theta1.diff.1} is positive semidefinite. Therefore, together with the form of $\widetilde\theta_{\alpha}$ in \eqref{tildetheta2}, we have proved that if $0<\alpha_1<\alpha_2$, then
\begin{align}
0& \leq  \alpha_2^{2\nu} Y_n^{\top} \left[ R_{\alpha_2}^{-1} - R_{\alpha_2}^{-1} M_n\big( M_n^\top R_{\alpha_2}^{-1}M_n +  \Omega_{\beta}\big)^{-1} M_n^\top  R_{\alpha_2}^{-1} \right]Y_n / (n-p) \nonumber \\
& \quad - \alpha_1^{2\nu} Y_n^{\top} \left[ R_{\alpha_1}^{-1} -  R_{\alpha_1}^{-1} M_n\big( M_n^\top R_{\alpha_1}^{-1}M_n + \Omega_{\beta}\big)^{-1} M_n^\top  R_{\alpha_1}^{-1}\right]  Y_n / (n-p) \nonumber \\
&= \widetilde\theta_{\alpha_2} - \widetilde\theta_{\alpha_1},
\end{align}
so $\widetilde\theta_{\alpha_1} \leq \widetilde\theta_{\alpha_2}$, i.e., $\widetilde\theta_{\alpha}$ is a non-decreasing function in $\alpha$.

For $\widetilde \theta_{\alpha}^{(1)} = \alpha^{2\nu} X_n^\top  R_{\alpha}^{-1} X_n /(n-p)$ from \eqref{tildetheta2.1}, since $\alpha_2^{2\nu}R_{\alpha_2}^{-1} - \alpha_1^{2\nu}R_{\alpha_1}^{-1}$ is positive definite by Lemma \ref{lem:alpha.monotone.matrix}, we have that for any $X_n\in \RR^n$, $\widetilde \theta_{\alpha_2}^{(1)} \geq \widetilde \theta_{\alpha_1}^{(1)}$, i.e., $\widetilde\theta_{\alpha}^{(1)}$ is a non-decreasing function in $\alpha$.
\end{proof}

\vspace{8mm}

\subsection{Proof of Uniform Convergence in Part (ii) of Lemma \ref{lem:dimension reduction}} \label{supsec:unif.conv}
We prove Part (ii) of Lemma \ref{lem:dimension reduction} in this subsection. We first restate the important quantities of $\underkappa,\overkappa,\underline\alpha_n,\overline\alpha_n$ as in \eqref{eq:2kappa} of the main text. We also define the constant $\tau\in (0,1/2)$:
\begin{align}\label{eq:2kappa.re}
& \underline \kappa = \frac{1}{2} \min\Bigg\{\frac{0.9}{(2d+0.94)(8\nu+3d-0.9)},~ \frac{1}{4(3\nu+d)}, ~0.01\Bigg\}, \quad \underline \alpha_n = n^{-\underline\kappa}, \nonumber \\
&\overline \kappa = \frac{1}{2} \min\Bigg\{\frac{0.9}{(2d+0.94)(8\nu+5d+0.9)},~~ \frac{1}{2(2\nu+d)},~~0.01\Bigg\}, \quad \overline \alpha_n = n^{-\overline\kappa},  \nonumber \\
&\tau = \frac{1}{2} \min\Bigg\{ \frac{0.9}{4d+1.88}- (4\nu+5d+0.45)\overkappa, ~~ \frac{15}{98}-5.95\overkappa, ~~ \frac{1}{2}-(2\nu+d)\overkappa, ~~ \frac{1}{2}-5\overkappa,  \nonumber \\
&\qquad \qquad \frac{0.9}{4d+1.88} - (4\nu+1.5d -0.45)\underkappa,~~  \frac{15}{98} - 4.05\underkappa, ~~ \frac{1}{2} -2(3\nu+d)\underkappa, ~~ \frac{1}{2} - 5\underkappa \Bigg\}.
\end{align}

\begin{lemma}\label{lem:REML.decomp}
For all $d\in \ZZ^+, \nu\in \RR^+, \alpha\in \RR^+$, the REML $\widetilde\theta_{\alpha}$ in \eqref{tildetheta2} can be decomposed into three terms:
\begin{align}\label{tildetheta2.1}
\widetilde \theta_{\alpha} & = \widetilde \theta_{\alpha}^{(1)} - \widetilde \theta_{\alpha}^{(2)} + \widetilde \theta_{\alpha}^{(3)}, \nonumber \\
\widetilde \theta_{\alpha}^{(1)} &= \frac{\alpha^{2\nu} X_n^\top  R_{\alpha}^{-1} X_n }{n-p}, \nonumber \\
\widetilde \theta_{\alpha}^{(2)} &= \frac{\alpha^{2\nu} X_n^\top R_{\alpha}^{-1}M_n \big(M_n^\top R_{\alpha}^{-1} M_n\big)^{-1} M_n^\top R_{\alpha}^{-1}  X_n }{n-p}, \nonumber \\
\widetilde \theta_{\alpha}^{(3)} &= \frac{\alpha^{2\nu} Y_n^\top  R_{\alpha}^{-1}M_n\left[ \big(M_n^\top R_{\alpha}^{-1} M_n \big)^{-1} - \big(M_n^\top R_{\alpha}^{-1} M_n + \Omega_{\beta}\big)^{-1} \right] M_n^\top R_{\alpha}^{-1} Y_n}{n-p} .
\end{align}
Furthermore,
$$0\leq \widetilde \theta_{\alpha}^{(2)}\leq \widetilde \theta_{\alpha}^{(1)},\qquad \widetilde \theta_{\alpha}^{(3)}\geq 0 .$$
\end{lemma}

\begin{proof}[Proof of Lemma \ref{lem:REML.decomp}]
The universal kriging model \eqref{eq:obs.model} implies that $Y_n=M_n \beta_0 + X_n$ with $X_n\sim \Ncal(0_n,\sigma_0^2 R_{\alpha_0})$. Therefore, the REML $\widetilde\theta_{\alpha}$ defined in \eqref{tildetheta2} can be rewritten as
\begin{align} \label{tildetheta33}
\widetilde \theta_{\alpha} & = \frac{\alpha^{2\nu} Y_n^\top  \left[R_{\alpha}^{-1} - R_{\alpha}^{-1}M_n \big(M_n^\top R_{\alpha}^{-1} M_n + \Omega_{\beta}\big)^{-1} M_n^\top R_{\alpha}^{-1} \right]  Y_n }{n-p} \nonumber \\
& = \frac{\alpha^{2\nu} (M_n \beta_0 + X_n)^\top  \left[R_{\alpha}^{-1} - R_{\alpha}^{-1}M_n \big(M_n^\top R_{\alpha}^{-1} M_n \big)^{-1} M_n^\top R_{\alpha}^{-1} \right]  (M_n \beta_0 + X_n)}{n-p} \nonumber \\
&\quad + \frac{\alpha^{2\nu} Y_n^\top  R_{\alpha}^{-1}M_n\left[ \big(M_n^\top R_{\alpha}^{-1} M_n \big)^{-1} - \big(M_n^\top R_{\alpha}^{-1} M_n + \Omega_{\beta}\big)^{-1} \right] M_n^\top R_{\alpha}^{-1} Y_n}{n-p} \nonumber \\
&\stackrel{(i)}{=} \frac{\alpha^{2\nu} X_n^\top  R_{\alpha}^{-1} X_n }{n-p} - \frac{\alpha^{2\nu} X_n^\top R_{\alpha}^{-1}M_n \big(M_n^\top R_{\alpha}^{-1} M_n \big)^{-1} M_n^\top R_{\alpha}^{-1}  X_n }{n-p} \nonumber \\
&\quad + \frac{\alpha^{2\nu} Y_n^\top  R_{\alpha}^{-1}M_n\left[ \big(M_n^\top R_{\alpha}^{-1} M_n \big)^{-1} - \big(M_n^\top R_{\alpha}^{-1} M_n + \Omega_{\beta}\big)^{-1} \right] M_n^\top R_{\alpha}^{-1} Y_n}{n-p} \nonumber \\
&= \widetilde \theta_{\alpha}^{(1)} - \widetilde \theta_{\alpha}^{(2)} + \widetilde \theta_{\alpha}^{(3)},
\end{align}
where in (i), we use the relation $\left[R_{\alpha}^{-1} - R_{\alpha}^{-1}M_n \big(M_n^\top R_{\alpha}^{-1} M_n \big)^{-1} M_n^\top R_{\alpha}^{-1} \right] M_n=0_{n\times p}$.

Since for any $\alpha>0$,
\begin{align*}
& R_{\alpha}^{-1} - R_{\alpha}^{-1}M_n \big(M_n^\top R_{\alpha}^{-1} M_n \big)^{-1} M_n^\top R_{\alpha}^{-1} \\
={}& R_{\alpha}^{-1/2} \left[ I_n - R_{\alpha}^{-1/2}M_n \big(M_n^\top R_{\alpha}^{-1} M_n \big)^{-1} M_n^\top R_{\alpha}^{-1/2}\right]  R_{\alpha}^{-1/2},
\end{align*}
where $I_n - R_{\alpha}^{-1/2}M_n \big(M_n^\top R_{\alpha}^{-1} M_n \big)^{-1} M_n^\top R_{\alpha}^{-1/2}$ is an idempotent matrix, it follows that $0\leq \widetilde \theta_{\alpha}^{(2)}\leq \widetilde \theta_{\alpha}^{(1)}$.

Since $\Omega_{\beta}$ is symmetric positive semidefinite, by Lemma \ref{lem:ABinv}, $\big(M_n^\top R_{\alpha}^{-1} M_n \big)^{-1} - \big(M_n^\top R_{\alpha}^{-1} M_n + \Omega_{\beta}\big)^{-1}$ is positive semidefinite. Therefore, $\widetilde \theta_{\alpha}^{(3)}\geq 0$ for any $\alpha>0$.
\end{proof}

\vspace{8mm}

\begin{lemma}\label{lem:theta2.bound.2ends}
For $\widetilde\theta_{\alpha}^{(2)}$ defined in \eqref{tildetheta2.1}, for $d\in \ZZ^+$ and $\nu\in \RR^+$, there exists a large integer $N_1'$ that only depends on $\nu,d,T,\theta_0,\alpha_0$, such that for all $n>N_1'$ and $\underline\alpha_n,\overline\alpha_n,\tau$ defined in \eqref{eq:2kappa.re},
\begin{align}
&\pr \left(\sqrt{n}\widetilde\theta_{\alpha_0}^{(2)} >  \theta_0 n^{-\tau}/16 \right) \leq \exp(-16\log^2 n), \label{eq:theta.alpha0.2} \\
&\pr \left(\sqrt{n}\widetilde\theta_{\underline\alpha_n}^{(2)} >  \theta_0 n^{-\tau}/16 \right) \leq \exp(-16\log^2 n), \label{eq:theta.alpha.u2} \\
&\pr \left(\sqrt{n}\widetilde\theta_{\overline\alpha_n}^{(2)} > \theta_0  n^{-\tau}/16 \right) \leq \exp(-16\log^2 n) . \label{eq:theta.alpha.o2}
\end{align}
\end{lemma}

\begin{proof}[Proof of Lemma \ref{lem:theta2.bound.2ends}]
\vspace{2mm}

We first prove \eqref{eq:theta.alpha.o2} below. Then the proofs of \eqref{eq:theta.alpha0.2} and \eqref{eq:theta.alpha.u2}  follow similarly.
\vspace{2mm}

For $\sqrt{n}\widetilde \theta_{\overline \alpha_n}^{(2)}$, we notice that by Lemma \ref{lem:specden_lambda}, $\lambda_{i,n}(\overline \alpha_n)\geq (\alpha_0/\overline\alpha_n)^{2\nu+d}$ for all $i=1,\ldots,n$, so $\lambda_{\max}\left(\Lambda_{\overline\alpha_n}^{-1}\right) \leq (\overline\alpha_n/\alpha_0)^{2\nu+d}$.

Using Lemma \ref{lem:URU}, we have $\overline\alpha_n^{2\nu}R_{\overline\alpha_n}^{-1}=\theta_0  U_{\overline\alpha_n} \Lambda_{\overline\alpha_n}^{-1} U_{\overline\alpha_n}^\top$. For any $\alpha>0$, we define $Z_n(\alpha)=(Z_{1,n}(\alpha),\ldots,Z_{n,n}(\alpha))^\top = U_{\alpha}^\top X_n$. Since $X_n\sim \Ncal(0_n,\sigma_0^2 R_{\alpha_0})$, by Lemma \ref{lem:URU}, we have $Z_n(\alpha)\sim \Ncal(0_n,I_n)$ for any $\alpha>0$. We can then write $X_n=U_{\overline\alpha_n}^{-\top}Z_n(\overline\alpha_n)$. Since $\Omega_{\beta}$ is positive semidefinite, we can upper bound $\widetilde\theta_{\overline\alpha_n}^{(2)}$ by
\begin{align} \label{eq:theta.tilde.2.1}
\widetilde \theta_{\overline\alpha_n}^{(2)} &= (n-p)^{-1} \overline\alpha_n^{2\nu} X_n^\top R_{\overline\alpha_n}^{-1} M_n \big(M_n^\top R_{\overline\alpha_n}^{-1} M_n + \Omega_{\beta}\big)^{-1} M_n^\top R_{\overline\alpha_n}^{-1}  X_n \nonumber \\
&\leq (n-p)^{-1} \overline\alpha_n^{2\nu} X_n^\top R_{\overline\alpha_n}^{-1} M_n \big(M_n^\top R_{\overline\alpha_n}^{-1} M_n \big)^{-1} M_n^\top R_{\overline\alpha_n}^{-1}  X_n \nonumber \\
&= (n-p)^{-1} \theta_0 X_n^\top \left(\overline\alpha_n^{2\nu}R_{\overline\alpha_n}^{-1}\right) M_n \big[M_n^\top \left(\overline\alpha_n^{2\nu}R_{\overline\alpha_n}^{-1}\right) M_n \big]^{-1} M_n^\top \left(\overline\alpha_n^{2\nu}R_{\overline\alpha_n}^{-1}\right)  X_n \nonumber \\
&\leq \frac{2\theta_0}{n} Z_n(\overline\alpha_n)^\top \Lambda_{\overline\alpha_n}^{-1} U_{\overline\alpha_n}^\top M_n \big(M_n^\top U_{\overline\alpha_n} \Lambda_{\overline\alpha_n}^{-1} U_{\overline\alpha_n}^\top  M_n \big)^{-1} M_n^\top U_{\overline\alpha_n} \Lambda_{\overline\alpha_n}^{-1} Z_n(\overline\alpha_n) \nonumber \\
&= \frac{2\theta_0}{n} Z_n(\overline\alpha_n)^\top \Lambda_{\overline\alpha_n}^{-1/2} H_{\overline\alpha_n} \Lambda_{\overline\alpha_n}^{-1/2} Z_n(\overline\alpha_n) ,
\end{align}
where $H_{\overline\alpha_n}=\Lambda_{\overline\alpha_n}^{-1/2} U_{\overline\alpha_n}^\top M_n \big(M_n^\top U_{\overline\alpha_n} \Lambda_{\overline\alpha_n}^{-1} U_{\overline\alpha_n}^\top  M_n \big)^{-1} M_n^\top U_{\overline\alpha_n} \Lambda_{\overline\alpha_n}^{-1/2}$ is an $n\times n$ idempotent matrix of rank $p$ (i.e., $H_{\overline\alpha_n}^2=H_{\overline\alpha_n}$), since $\text{rank}(M_n)=p\ll n$ as $n\to\infty$. Hence $\tr(H_{\overline\alpha_n})=p$.

We are going to apply the Hanson-Wright inequality in Lemma \ref{lem:Hsuetal12} to \eqref{eq:theta.tilde.2.1}, with $Z=Z_n(\overline\alpha_n)$, $z=16\log^2 n$, and $\Sigma = \Lambda_{\overline\alpha_n}^{-1/2} H_{\overline\alpha_n} \Lambda_{\overline\alpha_n}^{-1/2}$. For this purpose, we need to find upper bounds for $\tr(\Sigma)$, $\tr(\Sigma^2)$, and $\|\Sigma\|_{\op}$ in Lemma \ref{lem:Hsuetal12}. We first notice that for two generic $n\times n$ symmetric positive semidefinite matrices $A$ and $B$,
$$\tr(BA) = \tr(AB) = \tr(B^{1/2}AB^{1/2}) \leq \tr\{B^{1/2}(\lambda_{\max}(A) I)B^{1/2}\} \leq \lambda_{\max}(A)\tr(B).$$
Therefore, using $\lambda_{\max}(\Lambda_{\overline\alpha_n}^{-1})\leq (\overline\alpha_n/\alpha_0)^{2\nu+d}$, we apply the inequality above repeatedly to obtain that
\begin{align}\label{eq:hw-3sigma}
\tr\left(\Lambda_{\overline\alpha_n}^{-1/2} H_{\overline\alpha_n} \Lambda_{\overline\alpha_n}^{-1/2}\right) & = \tr\left(\Lambda_{\overline\alpha_n}^{-1} H_{\overline\alpha_n}\right) \leq \lambda_{\max}\left(\Lambda_{\overline\alpha_n}^{-1}\right) \tr\left(H_{\overline\alpha_n}\right) \leq p(\overline\alpha_n/\alpha_0)^{2\nu+d}, \nonumber \\
\tr\left[\left(\Lambda_{\overline\alpha_n}^{-1/2} H_{\overline\alpha_n} \Lambda_{\overline\alpha_n}^{-1/2}\right)^2\right] & = \tr\left(\Lambda_{\overline\alpha_n}^{-1} H_{\overline\alpha_n} \Lambda_{\overline\alpha_n}^{-1}H_{\overline\alpha_n}\right) \leq \lambda_{\max}\left(\Lambda_{\overline\alpha_n}^{-1}\right)\cdot \tr\left(H_{\overline\alpha_n} \Lambda_{\overline\alpha_n}^{-1}H_{\overline\alpha_n}\right) \nonumber \\
&\leq \lambda_{\max}\left(\Lambda_{\overline\alpha_n}^{-1}\right)\cdot \tr\left(\Lambda_{\overline\alpha_n}^{-1}H_{\overline\alpha_n}^2 \right) \leq \lambda_{\max}\left(\Lambda_{\overline\alpha_n}^{-1}\right)^2 \cdot \tr\left(H_{\overline\alpha_n}^2 \right) \nonumber\\
&= \lambda_{\max}\left(\Lambda_{\overline\alpha_n}^{-1}\right)^2 \cdot \tr\left(H_{\overline\alpha_n}\right) = p(\overline\alpha_n/\alpha_0)^{2(2\nu+d)}, \nonumber \\
\left\|\Lambda_{\overline\alpha_n}^{-1/2} H_{\overline\alpha_n} \Lambda_{\overline\alpha_n}^{-1/2}\right\|_{\op} & \leq \left[\lambda_{\max}\left\{\left(\Lambda_{\overline\alpha_n}^{-1/2} H_{\overline\alpha_n} \Lambda_{\overline\alpha_n}^{-1/2}\right)^2\right\}\right]^{1/2}\nonumber \\
&\leq \left[\tr\left\{\left(\Lambda_{\overline\alpha_n}^{-1/2} H_{\overline\alpha_n} \Lambda_{\overline\alpha_n}^{-1/2}\right)^2\right\}\right]^{1/2}  \leq \sqrt{p}(\overline\alpha_n/\alpha_0)^{2\nu+d}.
\end{align}
Therefore, for $z=16\log^2 n$ and $\Sigma = \Lambda_{\overline\alpha_n}^{-1/2} H_{\overline\alpha_n} \Lambda_{\overline\alpha_n}^{-1/2}$, given the choice of $\tau$ in \eqref{eq:2kappa.re}, $\tau< 1/2 - (2\nu+d)\overkappa$, so we have that for all sufficiently large $n$,
\begin{align}\label{eq:o2hw}
& \tr(\Sigma) + 2\sqrt{\tr(\Sigma^2)z} + 2\|\Sigma\|_{\op} z \nonumber \\
\leq{} & p(\overline\alpha_n/\alpha_0)^{2\nu+d} + 8\sqrt{p}(\overline\alpha_n/\alpha_0)^{2\nu+d} \log n + 32\sqrt{p}(\overline\alpha_n/\alpha_0)^{2\nu+d} \log^2 n  \nonumber \\
\leq{}&  42p n^{(2\nu+d)\overkappa} \log^2 n < n^{1/2-\tau}/128 .
\end{align}

We now apply Lemma \ref{lem:Hsuetal12} to \eqref{eq:theta.tilde.2.1} with $Z=Z_n(\overline\alpha_n)$, $z=16\log^2 n$, and $\Sigma = \Lambda_{\overline\alpha_n}^{-1/2} H_{\overline\alpha_n} \Lambda_{\overline\alpha_n}^{-1/2}$ to obtain that for all sufficiently large $n$,
\begin{align}\label{eq:theta.tilde.2.3}
& \pr \left( \sqrt{n}\widetilde\theta_{\overline\alpha_n}^{(2)} > \frac{\theta_0}{16} n^{-\tau} \right)  \nonumber \\
\leq{}& \pr \left(  Z_n(\overline\alpha_n)^\top \Lambda_{\overline\alpha_n}^{-1/2} H_{\overline\alpha_n} \Lambda_{\overline\alpha_n}^{-1/2} Z_n(\overline\alpha_n) > \frac{1}{32} n^{1/2 -\tau} \right) \nonumber \\
\leq {}& \pr \left( Z_n(\overline\alpha_n)^\top \Lambda_{\overline\alpha_n}^{-1/2} H_{\overline\alpha_n} \Lambda_{\overline\alpha_n}^{-1/2} Z_n(\overline\alpha_n) > \tr(\Sigma) + 2\sqrt{\tr(\Sigma^2)z} + 2\|\Sigma\|_{\op} z\right) \nonumber \\
\leq{}& \exp(-z) = \exp(-16\log^2 n).
\end{align}
This proves \eqref{eq:theta.alpha.o2}.
\vspace{2mm}

The proof of \eqref{eq:theta.alpha0.2} is similar to the proof of \eqref{eq:theta.alpha.o2} above. \eqref{eq:theta.tilde.2.1} still holds by replacing all $\overline\alpha_n$ with $\alpha_0$. We notice that $\Lambda_{\alpha_0}=I_n$, $\lambda_{\max}(\Lambda_{\alpha_0}^{-1})=1$, so the three upper bounds in \eqref{eq:hw-3sigma} become $p, p , \sqrt{p}$, respectively. With $\overline\alpha_n$ replaced by $\alpha_0$, the left-hand side \eqref{eq:o2hw} is upper bounded by $p+8\sqrt{p}\log n +32\sqrt{p}\log^2 n$, which is smaller than $n^{1/2-\tau}/32$ for all sufficiently large $n$. Hence \eqref{eq:theta.tilde.2.3} holds with $\overline\alpha_n$ replaced by $\alpha_0$. This proves \eqref{eq:theta.alpha0.2}.

The proof of \eqref{eq:theta.alpha.u2} is also similar to the proof of \eqref{eq:theta.alpha.o2} above. \eqref{eq:theta.tilde.2.1} still holds by replacing all $\overline\alpha_n$ with $\underline\alpha_n$. We notice that from \ref{lem:specden_lambda}, $\lambda_{i,n}(\underline \alpha_n)\geq 1$ for all $i=1,\ldots,n$, so $\lambda_{\max}\left(\Lambda_{\underline\alpha_n}^{-1}\right) \leq 1$. As a result, the three upper bounds in \eqref{eq:hw-3sigma} become $p, p , \sqrt{p}$, respectively. With $\overline\alpha_n$ replaced by $\underline\alpha_n$, the left-hand side \eqref{eq:o2hw} is upper bounded by $p+8\sqrt{p}\log n +32\sqrt{p}\log^2 n$, which is smaller than $n^{1/2-\tau}/32$ for all sufficiently large $n$. Hence \eqref{eq:theta.tilde.2.3} holds with $\overline\alpha_n$ replaced by $\underline\alpha_n$. This proves \eqref{eq:theta.alpha.u2}.
\end{proof}

\vspace{8mm}

\begin{lemma}\label{lem:theta3.bound.2ends}
For $\widetilde\theta_{\alpha}^{(3)}$ defined in \eqref{tildetheta2.1}, for $d\in \ZZ^+$ and $\nu\in \RR^+$, there exists a large integer $N_2'$ that only depends on $\nu,d,T,\beta_0,\theta_0,\alpha_0$ and the $\Wcal_2^{\nu+d/2}(\Scal)$ norms of $\bbm_1(\cdot),\ldots,\bbm_p(\cdot)$, such that for all $n>N_2'$ and $\underline\alpha_n,\overline\alpha_n,\tau$ defined in \eqref{eq:2kappa.re},
\begin{align}
&\pr \left(\sqrt{n}\widetilde\theta_{\alpha_0}^{(3)} >   \theta_0 n^{-\tau}/16 \right) \leq \exp(-16\log^2 n), \label{eq:theta.alpha0.3} \\
&\pr \left(\sqrt{n}\widetilde\theta_{\underline\alpha_n}^{(3)} >  \theta_0 n^{-\tau}/16 \right) \leq \exp(-16\log^2 n), \label{eq:theta.alpha.u3} \\
&\pr \left(\sqrt{n}\widetilde\theta_{\overline\alpha_n}^{(3)} >  \theta_0 n^{-\tau}/16 \right) \leq \exp(-16\log^2 n). \label{eq:theta.alpha.o3}
\end{align}
\end{lemma}

\begin{proof}[Proof of Lemma \ref{lem:theta3.bound.2ends}]

We first prove \eqref{eq:theta.alpha.o3}. The proofs of \eqref{eq:theta.alpha0.3} and \eqref{eq:theta.alpha.u3} follow similarly.

In the definition of $\widetilde\theta_{\overline\alpha_n}^{(3)}$ in \eqref{tildetheta2.1}, we directly drop the positive semidefinite matrix $\big(M_n^\top R_{\overline\alpha_n}^{-1} M_n + \Omega_{\beta}\big)^{-1}$ in the middle bracket, and obtain that
\begin{align} \label{eq:theta3.under1}
\widetilde\theta_{\overline\alpha_n}^{(3)}
={}& \frac{\overline\alpha_n^{2\nu} Y_n^\top  R_{\overline\alpha_n}^{-1}M_n\left[ \big(M_n^\top R_{\overline\alpha_n}^{-1} M_n \big)^{-1} - \big(M_n^\top R_{\overline\alpha_n}^{-1} M_n + \Omega_{\beta}\big)^{-1} \right] M_n^\top R_{\overline\alpha_n}^{-1} Y_n}{n-p}  \nonumber \\
\leq {}& \frac{\overline\alpha_n^{2\nu} Y_n^\top  R_{\overline\alpha_n}^{-1}M_n \big(M_n^\top R_{\overline\alpha_n}^{-1} M_n \big)^{-1} M_n^\top R_{\overline\alpha_n}^{-1} Y_n}{n-p}  \nonumber \\
={}& \frac{\overline\alpha_n^{2\nu} (M_n\beta_0 + X_n)^\top  R_{\overline\alpha_n}^{-1}M_n \big(M_n^\top R_{\overline\alpha_n}^{-1} M_n \big)^{-1} M_n^\top R_{\overline\alpha_n}^{-1} (M_n\beta_0 + X_n)}{n-p} .
\end{align}
For two vectors $u,v\in \RR^n$ and an $n\times n$ symmetric positive definite matrix $\Sigma$, we have the following inequality:
\begin{align} \label{eq:uvSigma}
& (u+v)^\top \Sigma (u+v) = u^\top \Sigma u + v^\top \Sigma v + 2u^\top \Sigma v \nonumber\\
&= u^\top \Sigma u + v^\top \Sigma v + 2(\Sigma^{1/2}u)^\top (\Sigma^{1/2}v) \leq u^\top \Sigma u + v^\top \Sigma v + u^\top \Sigma u + v^\top \Sigma v \nonumber \\
&= 2\left(u^\top \Sigma u + v^\top \Sigma v\right) .
\end{align}
We apply \eqref{eq:uvSigma} to the right-hand side of \eqref{eq:theta3.under1}, with $u=M_n\beta_0$, $v=X_n$, and\\ $\Sigma=R_{\overline\alpha_n}^{-1}M_n \big(M_n^\top R_{\overline\alpha_n}^{-1} M_n \big)^{-1} M_n^\top R_{\overline\alpha_n}^{-1}$ to obtain that
\begin{align} \label{eq:theta3.under2}
\widetilde\theta_{\overline\alpha_n}^{(3)} &\leq \frac{\overline\alpha_n^{2\nu} (M_n\beta_0 + X_n)^\top  R_{\overline\alpha_n}^{-1}M_n \big(M_n^\top R_{\overline\alpha_n}^{-1} M_n \big)^{-1} M_n^\top R_{\overline\alpha_n}^{-1} (M_n\beta_0 + X_n)}{n-p} \nonumber \\
&\leq \frac{2\overline\alpha_n^{2\nu}}{n-p} \left[\beta_0^\top M_n^\top R_{\overline\alpha_n}^{-1} M_n \beta_0 + X_n^\top R_{\overline\alpha_n}^{-1}M_n \big(M_n^\top R_{\overline\alpha_n}^{-1} M_n \big)^{-1} M_n^\top R_{\overline\alpha_n}^{-1} X_n \right] \nonumber \\
&\leq \frac{2\theta_0}{n-p} \beta_0^\top M_n^\top \left[(\theta_0/\overline\alpha_n^{2\nu})R_{\overline\alpha_n}\right]^{-1} M_n \beta_0 + \frac{2\overline\alpha_n^{2\nu}}{n-p} X_n^\top R_{\overline\alpha_n}^{-1}M_n \big(M_n^\top R_{\overline\alpha_n}^{-1} M_n \big)^{-1} M_n^\top R_{\overline\alpha_n}^{-1} X_n .
\end{align}
We bound the two terms in \eqref{eq:theta3.under2}. Because $\bbm_1,\ldots,\bbm_p\in \Wcal_2^{\nu+d/2}(\Scal)$ by Assumption \ref{assump.m.func}, Lemma \ref{lem:matern.sobolev} implies that $\bbm_1,\ldots,\bbm_p\in\Hcal_{\sigma_0^2 K_{\alpha_0,\nu}}$, the RKHS of Mat\'ern kernel $\sigma^2 K_{\alpha,\nu}$ for any $(\sigma^2,\alpha)\in \RR^+ \times \RR^+$. Let $\bbm_{j,n}=(\bbm_j(s_1),\ldots,\bbm_j(s_n))^\top\in \RR^n$ for $j=1,\ldots,p$. Then we can apply Lemma \ref{lem:rkhs.quadratic}, Lemma \ref{lem:rkhs.ordering}, and Lemma \ref{lem:matern.sobolev} to the first term in \eqref{eq:theta3.under2} and obtain that
\begin{align} \label{eq:theta3.under3}
& \frac{2\theta_0}{n-p} \beta_0^\top M_n^\top \left[(\theta_0/\overline\alpha_n^{2\nu})R_{\overline\alpha_n}\right]^{-1} M_n \beta_0 \nonumber \\
\leq{}&  \frac{2\theta_0}{n-p} \beta_0^\top \beta_0 \cdot \lambda_{\max} \left(M_n^\top \left[(\theta_0/\overline\alpha_n^{2\nu})R_{\overline\alpha_n}\right]^{-1} M_n \right)  \nonumber \\
\leq{}&  \frac{2\theta_0}{n-p} \|\beta_0\|^2 \cdot \tr \left(M_n^\top \left[(\theta_0/\overline\alpha_n^{2\nu})R_{\overline\alpha_n}\right]^{-1} M_n \right)  \nonumber \\
={}& \frac{2\theta_0}{n-p} \|\beta_0\|^2 \cdot  \sum_{j=1}^p \bbm_{j,n}^\top \left[(\theta_0/\overline\alpha_n^{2\nu})R_{\overline\alpha_n}\right]^{-1} \bbm_{j,n}  \nonumber \\
\stackrel{(i)}{\leq} {}&  \frac{2\theta_0}{n-p} \|\beta_0\|^2  \cdot  \sum_{j=1}^p \|\bbm_j\|_{\Hcal_{(\theta_0/\overline\alpha_n^{2\nu})K_{\overline\alpha_n,\nu}}}^2  \nonumber \\
\stackrel{(ii)}{\leq} {}&  \frac{2\theta_0}{n-p} \|\beta_0\|^2  \cdot  \sum_{j=1}^p (\overline\alpha_n/\alpha_0)^{2\nu+d} \|\bbm_j\|_{\Hcal_{\sigma_0^2K_{\alpha_0,\nu}}}^2  \nonumber \\
\stackrel{(iii)}{\leq} {}&  \frac{2\theta_0}{n-p} \|\beta_0\|^2 (\overline\alpha_n/\alpha_0)^{2\nu+d} \cdot c_2(\sigma_0,\alpha_0)^2 \sum_{j=1}^p \|\bbm_j\|_{\Wcal_2^{\nu+d/2}(\Scal)}^2   \nonumber \\
\leq{}& \frac{4\|\beta_0\|^2}{\alpha_0^{2\nu+d}} c_2(\sigma_0,\alpha_0)^2 \left(\sum_{j=1}^p \|\bbm_j\|_{\Wcal_2^{\nu+d/2}(\Scal)}^2\right)\cdot \theta_0 n^{(2\nu+d)\overkappa -1} \stackrel{(iv)}{\leq} \theta_0 n^{-\tau-1/2}/32,
\end{align}
for all sufficiently large $n$, where (i) follows by applying Lemma \ref{lem:rkhs.quadratic} to each $\bbm_1(\cdot),\ldots,\bbm_p(\cdot)$ with the covariance kernel $(\theta_0/\overline\alpha_n^{2\nu})K_{\overline\alpha_n,\nu}$, (ii) follows from Lemma \ref{lem:rkhs.ordering}, (iii) follows from Lemma \ref{lem:matern.sobolev} with the constant $c_2(\sigma_0,\alpha_0)$ defined in Lemma \ref{lem:matern.sobolev}, and (iv) follows from the definition of $\tau$ in \eqref{eq:2kappa.re} and $\tau < 1/2-(2\nu+d)\overkappa$.

For the second term in \eqref{eq:theta3.under2}, we notice that the exact term
$$\frac{\overline\alpha_n^{2\nu}}{n-p} X_n^\top R_{\overline\alpha_n}^{-1}M_n \big(M_n^\top R_{\overline\alpha_n}^{-1} M_n \big)^{-1} M_n^\top R_{\overline\alpha_n}^{-1} X_n $$
shows up as an upper bound for $\widetilde\theta_{\overline\alpha_n}^{(2)}$ in \eqref{eq:theta.tilde.2.1} in the proof of Lemma \ref{lem:theta2.bound.2ends}. Therefore, we can directly make use of the inequalities in \eqref{eq:theta.tilde.2.1}, \eqref{eq:o2hw}, and \eqref{eq:theta.tilde.2.3} to conclude that for all sufficiently large $n$,
\begin{align} \label{eq:theta3.under4}
& \pr\left(\frac{2\overline\alpha_n^{2\nu}}{n-p} X_n^\top R_{\overline\alpha_n}^{-1}M_n \big(M_n^\top R_{\overline\alpha_n}^{-1} M_n \big)^{-1} M_n^\top R_{\overline\alpha_n}^{-1} X_n > \theta_0 n^{-\tau-1/2}/32 \right) \nonumber \\
\leq{}& \pr \left( \frac{4\theta_0}{n} Z_n(\overline\alpha_n)^\top \Lambda_{\overline\alpha_n}^{-1/2} H_{\overline\alpha_n} \Lambda_{\overline\alpha_n}^{-1/2} Z_n(\overline\alpha_n) > \theta_0 n^{-\tau-1/2}/32 \right) \nonumber \\
={}& \pr \left( Z_n(\overline\alpha_n)^\top \Lambda_{\overline\alpha_n}^{-1/2} H_{\overline\alpha_n} \Lambda_{\overline\alpha_n}^{-1/2} Z_n(\overline\alpha_n) > n^{1/2-\tau}/128 \right) \nonumber \\
\leq{}& \exp(-16\log^2 n).
\end{align}

Therefore, we can combine \eqref{eq:theta3.under2}, \eqref{eq:theta3.under3}, and \eqref{eq:theta3.under4} together to conclude that for all sufficiently large $n$,
\begin{align}\label{eq:theta3.under5}
&\quad~ \pr \left(\sqrt{n}\widetilde\theta_{\overline\alpha_n}^{(3)} > \theta_0 n^{-\tau}/16 \right) \nonumber \\
& \leq
\pr \left(\sqrt{n}\cdot \frac{2\theta_0}{n-p} \beta_0^\top M_n^\top \left[(\theta_0/\overline\alpha_n^{2\nu})R_{\overline\alpha_n}\right]^{-1} M_n \beta_0 > \theta_0 n^{-\tau}/32 \right) \nonumber\\
&\quad +
\pr \left(\sqrt{n}\cdot \frac{2\overline\alpha_n^{2\nu}}{n-p} X_n^\top R_{\overline\alpha_n}^{-1}M_n \big(M_n^\top R_{\overline\alpha_n}^{-1} M_n \big)^{-1} M_n^\top R_{\overline\alpha_n}^{-1} X_n > \theta_0 n^{-\tau}/32 \right) \nonumber \\
&\leq 0 + \exp(-16\log^2 n) = \exp(-16\log^2 n).
\end{align}
This proves \eqref{eq:theta.alpha.o3}.

For the proofs of \eqref{eq:theta.alpha0.3} and \eqref{eq:theta.alpha.u3}, we only need to modify the proof above for \eqref{eq:theta.alpha.o3} for a looser upper bound. In particular, the relation \eqref{eq:theta3.under2} still holds with $\overline\alpha_n$ replaced by both $\alpha_0$ and $\underline\alpha_n$; in the inequality \eqref{eq:theta3.under3}, $(\overline\alpha_n/\alpha_0)^{2\nu+d}$ in step (ii) will be replaced by 1 if $\overline\alpha_n$ is replaced by both  $\alpha_0$ and $\underline\alpha_n$, such that $n^{(2\nu+d)\overkappa-1}$ before the last step of \eqref{eq:theta3.under3} is replaced by the smaller $n^{-1}$, which means that \eqref{eq:theta3.under3} remains true if $\overline\alpha_n$ is replaced by both $\alpha_0$ and $\underline\alpha_n$. Given Lemma \ref{lem:theta2.bound.2ends}, \eqref{eq:theta3.under4} still holds true if $\overline\alpha_n$ is replaced by both $\alpha_0$ and $\underline\alpha_n$. Therefore, \eqref{eq:theta3.under5} holds for both $\alpha_0$ and $\underline\alpha_n$. This completes the proof.
\end{proof}

\vspace{8mm}

\begin{lemma}\label{lem:theta1.bound.2ends}
For $\widetilde\theta_{\alpha}$ defined in \eqref{tildetheta2}, for $d\in\{1,2,3\}$ and $\nu\in \RR^+$, there exists a large integer $N_3'$ that only depends on $\nu,d,T,\beta_0,\theta_0,\alpha_0$ and the $\Wcal_2^{\nu+d/2}(\Scal)$ norms of $\bbm_1(\cdot),\ldots,\bbm_p(\cdot)$, such that for all $n>N_3'$,
\begin{align}
&\pr \left(0\leq \sqrt{n} \left(\widetilde \theta_{\overline {\alpha}_n} - \widetilde \theta_{\alpha_0} \right) \leq \frac{\theta_0}{2} n^{-\tau} \right) \geq 1-2\exp(-4\log^2 n), \label{eq:theta.diff.right} \\
&\pr \left(0\leq \sqrt{n} \left(\widetilde \theta_{\overline {\alpha}_n}^{(1)} - \widetilde \theta_{\alpha_0}^{(1)} \right) \leq \frac{\theta_0}{4} n^{-\tau} \right) \geq 1- \exp(-4\log^2 n), \label{eq:theta.diff.right.1} \\
&\pr \left(0\leq \sqrt{n} \left(\widetilde \theta_{\alpha_0} - \widetilde \theta_{\underline {\alpha}_n} \right) \leq \frac{\theta_0}{2} n^{-\tau} \right) \geq 1-2\exp(-4\log^2 n), \label{eq:theta.diff.left} \\
&\pr \left(0\leq \sqrt{n} \left(\widetilde \theta_{\alpha_0}^{(1)} - \widetilde \theta_{\underline {\alpha}_n}^{(1)} \right) \leq \frac{\theta_0}{4} n^{-\tau} \right) \geq 1- \exp(-4\log^2 n), \label{eq:theta.diff.left.1}
\end{align}
where $\tau, \underline\alpha_n ,\overline\alpha_n$ are as defined the same as in \eqref{eq:2kappa.re}.
\end{lemma}

\begin{proof}[Proof of Lemma \ref{lem:theta1.bound.2ends}]
\vspace{2mm}

\noindent \underline{Proof of \eqref{eq:theta.diff.right} and \eqref{eq:theta.diff.right.1} (for the case of $\overline \alpha_n=n^{\overkappa}$).}
\vspace{2mm}

Since $\overkappa>0$, $\overline \alpha_n = n^{\overline \kappa} > \alpha_0$ for all sufficiently large $n$. By Lemma \ref{lem:theta1.monotone}, we have $\widetilde \theta_{\overline \alpha_n} \geq \widetilde \theta_{\alpha_0}$ and $\widetilde \theta_{\overline \alpha_n}^{(1)} \geq \widetilde \theta_{\alpha_0}^{(1)}$. By the decomposition of $\widetilde \theta_{\alpha}$ in \eqref{tildetheta2.1} of Lemma \ref{lem:REML.decomp} and the fact that $0\leq \widetilde \theta_{\alpha}^{(2)}\leq \widetilde \theta_{\alpha}^{(1)}, \widetilde \theta_{\alpha}^{(3)}\geq 0$, we can rewrite the difference inside the probability in \eqref{eq:theta.diff.right} as
\begin{align}\label{eq:theta.diff1.1}
0&\leq \widetilde \theta_{\overline \alpha_n} - \widetilde \theta_{\alpha_0} \nonumber \\
&= \widetilde \theta_{\overline \alpha_n}^{(1)} - \widetilde \theta_{\overline \alpha_n}^{(2)} + \widetilde \theta_{\overline \alpha_n}^{(3)} - \widetilde \theta_{\alpha_0}^{(1)} + \widetilde \theta_{\alpha_0}^{(2)} - \widetilde \theta_{\alpha_0}^{(3)} \nonumber\\
&\leq \widetilde \theta_{\overline \alpha_n}^{(1)} - \widetilde \theta_{\alpha_0}^{(1)} + \widetilde \theta_{\overline \alpha_n}^{(3)} +\widetilde \theta_{\alpha_0}^{(2)}.
\end{align}
According to the definition of $\lambda_{k,n}(\alpha)$ ($k=1,\ldots,n$) in \eqref{diagonalize}, we have that for any $\alpha>0$,
\begin{align}\label{eq:RUU}
\alpha^{2\nu}R_{\alpha}^{-1} = \theta_0 \sigma^{-2}R_{\alpha}^{-1} = \theta_0  U_{\alpha} \Lambda_{\alpha}^{-1} U_{\alpha}^\top ,\quad \alpha_0^{2\nu}R_{\alpha_0}^{-1} = \theta_0 \sigma_0^{-2}R_{\alpha_0}^{-1} = \theta_0  U_{\alpha}  U_{\alpha}^\top ,
\end{align}
where $\Lambda_{\alpha}=\diag\{\lambda_{k,n}(\alpha):k=1,\ldots,n\}$. Similar to the proof of Lemma \ref{lem:theta2.bound.2ends}, for any $\alpha>0$, we define $Z_n(\alpha)=(Z_{1,n}(\alpha),\ldots,Z_{n,n}(\alpha))^\top = U_{\alpha}^\top X_n$. Since $X_n\sim \Ncal(0_n,\sigma_0^2 R_{\alpha_0})$, we have $Z_n(\alpha)\sim \Ncal(0_n,I_n)$ for any $\alpha>0$.

Then it follows that for $n$ sufficiently large,
\begin{align}\label{eq:theta.diff1.2}
\sqrt{n}\left(\widetilde \theta_{\overline \alpha_n}^{(1)} - \widetilde \theta_{\alpha_0}^{(1)}\right) &= \frac{\sqrt{n}}{n-p} X_n^\top \left(\overline\alpha_n^{2\nu}R_{\overline\alpha_n}^{-1} - \alpha_0^{2\nu}R_{\alpha_0}^{-1}\right) X_n \nonumber \\
&= \frac{\sqrt{n}\theta_0}{n-p} X_n^\top U_{\overline\alpha_n} \left(\Lambda_{\overline\alpha_n}^{-1}-I_n\right) U_{\overline\alpha_n}^\top X_n \nonumber \\
&= \frac{\sqrt{n}\theta_0}{n-p} \sum_{i=1}^n \left|\lambda_{i,n}(\overline\alpha_n)^{-1}-1\right| Z_{i,n}(\overline\alpha_n)^2.
\end{align}
$\sqrt{n}/(n-p)\leq 2n^{-1/2}$ for all large $n$. Now we apply Lemma \ref{lem:LauMas00} and Lemma \ref{lem:weight.bound} to \eqref{eq:theta.diff1.2},  with $z=4\log^2 n$, $Z_i=Z_{i,n}(\overline\alpha_n)$, $w_i=w_i(\overline\alpha_n)=\left|\lambda_{i,n}(\overline\alpha_n)^{-1}-1\right|/\sqrt{n}$, to obtain that
\begin{align} \label{eq:theta.diff1.3}
& \pr \left(\sqrt{n} \left|\widetilde \theta_{\overline {\alpha}_n}^{(1)} - \widetilde \theta_{\alpha_0}^{(1)} \right| > \frac{\theta_0}{4} n^{-\tau} \right) \nonumber \\
\leq{}& \pr \Big( \frac{1}{\sqrt{n}} \sum_{i=1}^n \left|\lambda_{i,n}(\overline\alpha_n)^{-1}-1\right| Z_{i,n}(\overline\alpha_n)^2 \nonumber\\
& \quad > \|w(\overline\alpha_n)\|_1 + 4\|w(\overline\alpha_n)\|\log n + 8\|w(\overline\alpha_n)\|_{\infty}\log^2 n  \Big) \nonumber \\
\leq{}& \exp(-4\log^2 n).
\end{align}
This proves \eqref{eq:theta.diff.right.1}.

We combine \eqref{eq:theta.diff1.1}, \eqref{eq:theta.diff1.3} with \eqref{eq:theta.alpha.o3} from Lemma \ref{lem:theta3.bound.2ends} and \eqref{eq:theta.alpha0.2} from Lemma \ref{lem:theta2.bound.2ends} to obtain that for all sufficiently large $n$,
\begin{align}\label{eq:theta.diff1.5}
&\pr \left(\sqrt{n} \left(\widetilde \theta_{\overline {\alpha}_n} - \widetilde \theta_{\alpha_0} \right) > \frac{\theta_0}{2} n^{-\tau} \right) \nonumber\\
\leq{}& \pr \left(\sqrt{n} \left(\widetilde \theta_{\overline {\alpha}_n}^{(1)} - \widetilde \theta_{\alpha_0}^{(1)} + \widetilde \theta_{\overline\alpha_n}^{(3)} + \widetilde \theta_{\alpha_0}^{(2)} \right) > \frac{\theta_0}{2} n^{-\tau} \right) \nonumber\\
\leq{}& \pr \left(\sqrt{n} \left(\widetilde \theta_{\overline {\alpha}_n}^{(1)} - \widetilde \theta_{\alpha_0}^{(1)} \right) > \frac{\theta_0}{4} n^{-\tau} \right)  + \pr \left(\sqrt{n}  \widetilde \theta_{\overline\alpha_n}^{(3)} > \frac{\theta_0}{16} n^{-\tau} \right) +  \pr \left(\sqrt{n}  \widetilde \theta_{\alpha_0}^{(2)} > \frac{\theta_0}{16} n^{-\tau} \right) \nonumber\\
\leq{}& \exp(-4\log^2 n) + 2\exp(-16\log^2 n) < 2\exp(-4\log^2 n),
\end{align}
which proves \eqref{eq:theta.diff.right}.

\vspace{5mm}

\noindent \underline{Proof of \eqref{eq:theta.diff.left} and \eqref{eq:theta.diff.left.1} (for the case of $\underline \alpha_n=n^{\underkappa}$).}
\vspace{2mm}

The proof for the case of $\underline \alpha_n=n^{\underkappa}$ is similar to the previous case of $\overline \alpha_n=n^{\overkappa}$. First by Lemma \ref{lem:theta1.monotone}, we have $\widetilde \theta_{\underline \alpha_n} \leq \widetilde \theta_{\alpha_0}$ and $\widetilde \theta_{\underline \alpha_n}^{(1)} \leq \widetilde \theta_{\alpha_0}^{(1)}$ for large $n$. By the decomposition of $\widetilde \theta_{\alpha}$ in \eqref{tildetheta2.1} of Lemma \ref{lem:REML.decomp} and the fact that $0\leq \widetilde \theta_{\alpha}^{(2)}\leq \widetilde \theta_{\alpha}^{(1)}, \widetilde \theta_{\alpha}^{(3)}\geq 0$, we can rewrite the difference inside the probability in \eqref{eq:theta.diff.left} as
\begin{align}\label{eq:theta.diff2.1}
0 &\leq \widetilde \theta_{\alpha_0} - \widetilde \theta_{\underline \alpha_n} \nonumber \\
&= \widetilde \theta_{\alpha_0}^{(1)}  - \widetilde \theta_{\alpha_0}^{(2)} + \widetilde \theta_{\alpha_0}^{(3)} -  \widetilde \theta_{\underline \alpha_n}^{(1)} + \widetilde \theta_{\underline \alpha_n}^{(2)} - \widetilde \theta_{\underline \alpha_n}^{(3)} \nonumber\\
&\leq \widetilde \theta_{\alpha_0}^{(1)} -  \widetilde \theta_{\underline \alpha_n}^{(1)} + \widetilde \theta_{\alpha_0}^{(3)} + \widetilde \theta_{\underline \alpha_n}^{(2)}.
\end{align}
Using Lemma \ref{lem:LauMas00} and Lemma \ref{lem:weight.bound} with $z=4\log^2 n$, $Z_i=Z_{i,n}(\underline\alpha_n)$, $w_i=w_i(\underline\alpha_n)=\left|\lambda_{i,n}(\underline\alpha_n)^{-1}-1\right|/\sqrt{n}$, we have that
\begin{align} \label{eq:theta.diff2.2}
& \pr \left(\sqrt{n} \left|\widetilde \theta_{\underline {\alpha}_n}^{(1)} - \widetilde \theta_{\alpha_0}^{(1)} \right| > \frac{\theta_0}{2} n^{-\tau} \right) \nonumber \\
\leq{}& \pr \Big( \frac{1}{\sqrt{n}} \sum_{i=1}^n \left|\lambda_{i,n}(\underline\alpha_n)^{-1}-1\right| Z_{i,n}(\overline\alpha_n)^2 \nonumber\\
& \quad > \|w(\underline\alpha_n)\|_1 + 4\|w(\underline\alpha_n)\|\log n + 8\|w(\underline\alpha_n)\|_{\infty}\log^2 n  \Big) \nonumber \\
\leq{}& \exp(-4\log^2 n).
\end{align}
This proves \eqref{eq:theta.diff.left.1}.

We then combine \eqref{eq:theta.diff2.1}, \eqref{eq:theta.diff2.2} with \eqref{eq:theta.alpha0.3} from Lemma \ref{lem:theta3.bound.2ends} and \eqref{eq:theta.alpha.u2} from Lemma \ref{lem:theta2.bound.2ends}  to obtain that for all sufficiently large $n$,
\begin{align}\label{eq:theta.diff2.4}
& \pr \left(\sqrt{n} \left(\widetilde \theta_{\alpha_0} - \widetilde \theta_{\underline\alpha_n} \right) > \frac{\theta_0}{2} n^{-\tau} \right) \nonumber\\
\leq{}& \pr \left(\sqrt{n} \left( \widetilde \theta_{\alpha_0}^{(1)} - \widetilde \theta_{\underline\alpha_n}^{(1)} + \widetilde \theta_{\alpha_0}^{(3)}  +  \widetilde \theta_{\underline\alpha_n}^{(2)} \right) > \frac{\theta_0}{2} n^{-\tau} \right) \nonumber\\
\leq{}& \pr \left(\sqrt{n} \left(\widetilde \theta_{\alpha_0}^{(1)} - \widetilde \theta_{\underline\alpha_n}^{(1)} \right) > \frac{\theta_0}{4} n^{-\tau} \right) + \pr \left(\sqrt{n}  \widetilde \theta_{\alpha_0}^{(3)} > \frac{\theta_0}{16} n^{-\tau} \right) + \pr \left(\sqrt{n}  \widetilde \theta_{\underline\alpha_n}^{(2)} > \frac{\theta_0}{16} n^{-\tau} \right) \nonumber\\
\leq{}& \exp(-4\log^2 n) + 2\exp(-16\log^2 n) < 2\exp(-4\log^2 n),
\end{align}
which proves \eqref{eq:theta.diff.left}.
\end{proof}

We restate and strengthen the uniform convergence in Part (ii) of Lemma \ref{lem:dimension reduction} in the main text as the following lemma. The inequality in Part (ii) of Lemma \ref{lem:dimension reduction} is implied by \eqref{theta.uniformbound1} below.
\begin{lemma}[Uniform Convergence of $\widetilde \theta_{\alpha}$ in Lemma \ref{lem:dimension reduction} in the Main Text] \label{lem:sup.theta1}
Suppose that $d\in\{1,2,3\}$. For $\widetilde\theta_{\alpha}$ defined in \eqref{tildetheta2} and $N_3'$ defined in Lemma \ref{lem:theta1.bound.2ends}, for all $n>N_3'$,
\begin{align}
&\pr \left(\sup_{\alpha\in [\underline\alpha_n,\overline\alpha_n]} \sqrt{n} \left|\widetilde \theta_{\alpha} - \widetilde \theta_{\alpha_0} \right| \leq \frac{\theta_0}{2} n^{-\tau} \right) \geq 1- 4\exp(-4\log^2 n), \label{theta.uniformbound1} \\
&\pr \left(\sup_{\alpha\in [\underline\alpha_n,\overline\alpha_n]} \sqrt{n} \left|\widetilde \theta_{\alpha}^{(1)} - \widetilde \theta_{\alpha_0}^{(1)} \right| \leq \frac{\theta_0}{4} n^{-\tau} \right) \geq 1- 2\exp(-4\log^2 n), \label{theta.uniformbound1.1}
\end{align}
where $\tau, \underline\alpha_n ,\overline\alpha_n$ are as defined the same as in \eqref{eq:2kappa.re}.
\end{lemma}

\begin{proof}[Proof of Lemma \ref{lem:sup.theta1}]
From Lemma \ref{lem:theta1.monotone}, we have that $\widetilde\theta_{\alpha}$ and $\widetilde\theta_{\alpha}^{(1)}$ are both non-decreasing in $\alpha$. Therefore,
\begin{align*}
& \sup_{\alpha \in[\underline\alpha_n, \alpha_0]} \left|\widetilde \theta_{\alpha} - \widetilde \theta_{\alpha_0}\right| = \widetilde \theta_{\alpha_0} - \widetilde \theta_{\underline\alpha_n} , \\
& \sup_{\alpha \in[\alpha_0, \overline \alpha_n]} \left|\widetilde \theta_{\alpha} - \widetilde \theta_{\alpha_0}\right| = \widetilde \theta_{\overline \alpha_n} - \widetilde \theta_{\alpha_0} , \\
& \sup_{\alpha \in[\underline\alpha_n, \overline \alpha_n]} \left|\widetilde \theta_{\alpha} - \widetilde \theta_{\alpha_0}\right| = \max\left( \widetilde \theta_{\alpha_0}- \widetilde \theta_{\underline \alpha_n}, \widetilde \theta_{\overline \alpha_n} - \widetilde \theta_{\alpha_0} \right), \\
& \sup_{\alpha \in[\underline\alpha_n, \overline \alpha_n]} \left|\widetilde \theta_{\alpha}^{(1)} - \widetilde \theta_{\alpha_0}^{(1)} \right| = \max\left( \widetilde \theta_{\alpha_0}^{(1)}- \widetilde \theta_{\underline \alpha_n}^{(1)}, \widetilde \theta_{\overline \alpha_n}^{(1)} - \widetilde \theta_{\alpha_0}^{(1)} \right).
\end{align*}
We can then combine \eqref{eq:theta.diff.right} and \eqref{eq:theta.diff.left} from Lemma \ref{lem:theta1.bound.2ends} to obtain that for all $n>N_3'$,
\begin{align}
& \pr \left(\sup_{\alpha \in[\underline\alpha_n, \overline \alpha_n]}\sqrt{n} \left|\widetilde \theta_{\alpha}- \widetilde \theta_{\alpha_0} \right| > \frac{\theta_0}{2} n^{-\tau} \right) \nonumber \\
={} & \pr \Bigg(\sqrt{n} \left(\widetilde \theta_{\alpha_0}- \widetilde \theta_{\underline \alpha_n}\right) > \frac{\theta_0}{2} n^{-\tau}  \text{ or }  \sqrt{n} \left(\widetilde \theta_{\underline \alpha_n} - \widetilde \theta_{\alpha_0}\right) > \frac{\theta_0}{2} n^{-\tau} \Bigg) \nonumber \\
\leq{}&  \pr \left(\sqrt{n} \left(\widetilde \theta_{\alpha_0} - \widetilde \theta_{\underline \alpha_n} \right) > \frac{\theta_0}{2} n^{-\tau} \right) + \pr \left(\sqrt{n} \left(\widetilde \theta_{\overline \alpha_n} - \widetilde \theta_{\alpha_0}\right) > \frac{\theta_0}{2} n^{-\tau} \right) \nonumber \\
\leq{}& 4 \ee^{-4\log^2 n} . \nonumber
\end{align}
The inequality \eqref{theta.uniformbound1.1} follows similarly using a union bound from \eqref{eq:theta.diff.left.1} and \eqref{eq:theta.diff.right.1} in Lemma \ref{lem:theta1.bound.2ends}.
\end{proof}

\vspace{5mm}

In the next lemma, we prove the asymptotic normality of $\widetilde\theta_{\alpha}$ for a fixed $\alpha>0$ in Theorem \ref{thm:bvm1:theta} in the main text. We also bound the tail probability of $\big|\widetilde \theta_{\alpha_0}-\theta_0\big|$.
\begin{lemma}\label{lem:theta.alpha0}
For $d\in\ZZ^+$ and $\nu\in \RR^+$, there exists a large integer $N_4'$ that only depends on $\nu,d,T,\beta_0,\theta_0,\alpha_0$ and the $\Wcal_2^{\nu+d/2}(\Scal)$ norms of $\bbm_1(\cdot),\ldots,\bbm_p(\cdot)$, such that for all $n>N_4'$,
\begin{align} \label{theta.alpha0}
&\pr \left(\sqrt{n} \left| \widetilde \theta_{\alpha_0} - \theta_0 \right| \leq 5 \theta_0 \log n \right) \geq 1 - 3\exp(-4\log^2 n).
\end{align}
Furthermore, for $d\in\{1,2,3\}$ and $\nu\in \RR^+$, for any fixed $\alpha>0$, as $n\to\infty$,
\begin{align}\label{eq:REML.normal}
& \sqrt{n}\left(\widetilde\theta_{\alpha}-\theta_0\right) \overset{\Dcal}{\rightarrow} \mathcal{N}\left(0,2\theta_0^2\right).
\end{align}
\end{lemma}

\begin{proof}[Proof of Lemma \ref{lem:theta.alpha0}]
Let $W_n = (W_{1,n},\ldots,W_{n,n})^\top = \sigma_0^{-1} R_{\alpha_0}^{-1/2} X_n \sim \Ncal(0_n,I_n)$. Using the decomposition in \eqref{tildetheta2.1}, we have
\begin{align*}
\sqrt{n} \left( \widetilde \theta_{\alpha_0} - \theta_0 \right) = \sqrt{n} \left( \widetilde \theta_{\alpha_0}^{(1)} - \theta_0 \right) - \sqrt{n} \widetilde\theta_{\alpha_0}^{(2)} + \sqrt{n} \widetilde\theta_{\alpha_0}^{(3)}.
\end{align*}
Since $\widetilde\theta_{\alpha_0}^{(1)}=\alpha_0^{2\nu} X_n^\top R_{\alpha_0}^{-1} X_n/(n-p) = \theta_0 W_n^\top W_n/(n-p)$, by the central limit theorem for $\chi^2_1$ random variables, we have that as $n\to\infty$,
\begin{align}\label{eq:REML0.normal1}
\sqrt{n}\left(\widetilde\theta_{\alpha_0}^{(1)}-\theta_0\right)=\sqrt{n}\theta_0\left(\frac{W_n^\top W_n}{n-p} - 1 \right) \overset{\Dcal}{\rightarrow} \mathcal{N}(0, 2\theta_0^2).
\end{align}
The first inequality in Lemma \ref{lem:LauMas00} with $Z_i=W_{i,n}$, $w_i=1$ for $i=1,\ldots,n$ and $z=4\log^2 n$ implies that for all sufficiently large $n$,
\begin{align} \label{eq:theta1.1.alpha0}
& \pr \left(\sqrt{n}\left(\widetilde\theta_{\alpha_0}^{(1)}-\theta_0\right) > 4.5 \theta_0\log n \right) = \pr \left( W_n^\top W_n > n - p + \frac{4.5(n-p)\log n}{\sqrt{n}} \right) \nonumber \\
\leq{}&  \pr \left( W_n^\top W_n > n + 4\sqrt{n}\log n + 8\log^2 n \right) \leq \exp(-4\log^2 n).
\end{align}
The second inequality in Lemma \ref{lem:LauMas00} with $Z_i=W_{i,n}$, $w_i=1$ for $i=1,\ldots,n$ and $z=4\log^2 n$ implies that for all sufficiently large $n$,
\begin{align} \label{eq:theta1.2.alpha0}
& \pr \left(\sqrt{n}\left(\widetilde\theta_{\alpha_0}^{(1)}-\theta_0\right) < -  4.5\theta_0\log n \right) = \pr \left( W_n^\top W_n < n - p - \frac{4.5(n-p)\log n}{\sqrt{n}} \right) \nonumber \\
\leq{}&  \pr \left( W_n^\top W_n < n-  4\sqrt{n}\log n  \right) \leq \exp(-4\log^2 n).
\end{align}
We combine \eqref{eq:theta1.1.alpha0}, \eqref{eq:theta1.2.alpha0}, \eqref{eq:theta.alpha0.2} from Lemma \ref{lem:theta2.bound.2ends} and \eqref{eq:theta.alpha0.3} from Lemma \ref{lem:theta3.bound.2ends} to obtain that for all sufficiently large $n$,
\begin{align*}
& \pr \left(  \sqrt{n} \left| \widetilde \theta_{\alpha_0} - \theta_0 \right| > 5 \theta_0 \log n \right) \\
\leq{}& \pr \left( \sqrt{n} \left| \widetilde \theta_{\alpha_0}^{(1)} - \theta_0 \right| > 4.5 \theta_0 \log n \right) + \pr \left(\sqrt{n} \widetilde\theta_{\alpha_0}^{(2)} > \frac{\theta_0}{4} \log n \right) + \pr \left(\sqrt{n} \widetilde\theta_{\alpha_0}^{(3)} > \frac{\theta_0}{4} \log n \right) \\
\leq{}& \pr \left(\sqrt{n}\left(\widetilde\theta_{\alpha_0}^{(1)}-\theta_0\right) > 4.5 \theta_0\log n \right) + \pr \left(\sqrt{n}\left(\widetilde\theta_{\alpha_0}^{(1)}-\theta_0\right) < -4.5 \theta_0 \log n\right) \\
& ~~ +  \pr \left(\sqrt{n} \widetilde\theta_{\alpha_0}^{(2)} > \theta_0 n^{-\tau}/16 \right) +  \pr \left(\sqrt{n} \widetilde\theta_{\alpha_0}^{(3)} > \theta_0 n^{-\tau}/16 \right)  \\
\leq{}& 2\exp(-4\log^2 n)  + 2\exp(-16\log^2 n) < 3\exp(-4\log^2 n),
\end{align*}
which has proved \eqref{theta.alpha0}.

Now for \eqref{eq:REML.normal}, we notice that \eqref{eq:theta.alpha0.2} from Lemma \ref{lem:theta2.bound.2ends} and \eqref{eq:theta.alpha0.3} from Lemma \ref{lem:theta3.bound.2ends} imply that both $\sqrt{n}\widetilde\theta_{\alpha_0}^{(2)}$ and $\sqrt{n}\widetilde\theta_{\alpha_0}^{(3)}$ converge to zero in $P_{(\beta_0,\sigma_0^2,\alpha_0)}$-probability as $n\to\infty$. Therefore, we combine this with \eqref{eq:REML0.normal1} and apply the Slutsky's theorem to obtain that as $n\to\infty$,
\begin{align}\label{eq:REML0.normal2}
& \sqrt{n}\left(\widetilde\theta_{\alpha_0}-\theta_0\right)= \sqrt{n}\left(\widetilde\theta_{\alpha_0}^{(1)}-\theta_0\right) - \sqrt{n}\widetilde\theta_{\alpha_0}^{(2)} + \sqrt{n}\widetilde\theta_{\alpha_0}^{(3)} \overset{\Dcal}{\rightarrow} \mathcal{N}(0, 2\theta_0^2).
\end{align}
Since $\alpha>0$ is fixed, it will be eventually covered by the interval $[\underline\alpha_n,\overline\alpha_n]$ as $n\to\infty$. Therefore, by Lemma \ref{lem:sup.theta1}, for any fixed $\alpha>0$, $\sqrt{n}\big|\widetilde \theta_{\alpha}-\widetilde \theta_{\alpha_0}\big|\to 0$ in $P_{(\beta_0,\sigma_0^2,\alpha_0)}$-probability as $n\to\infty$. We combine this with \eqref{eq:REML0.normal2} and apply the Slutsky's theorem again to conclude that
as $n\to\infty$,
\begin{align}\label{eq:REML.normal2}
& \sqrt{n}\left(\widetilde\theta_{\alpha}-\theta_0\right)= \sqrt{n}\left(\widetilde\theta_{\alpha_0}-\theta_0\right) + \sqrt{n}\left(\widetilde \theta_{\alpha}-\widetilde \theta_{\alpha_0}\right)  \overset{\Dcal}{\rightarrow} \mathcal{N}(0, 2\theta_0^2).
\end{align}
This completes the proof.
\end{proof}

\vspace{5mm}

\subsection{Auxiliary RKHS Theory} \label{supsec:rkhs}
In this subsection, we present some auxiliary technical results on the reproducing kernel Hilbert space (RKHS) of Mat\'ern kernels that are used to handle the regression functions $\bbm_1,\ldots,\bbm_p$. We define some concepts for a generic positive definite covariance function $K(\cdot,\cdot)$ on a fixed domain $\Scal=[0,T]^d$. Let $L_2(\Scal)$ be the space of square integrable functions on $\Scal$, and $\Ccal(\Scal)$ be the space of continuous functions on $\Scal$. We assume that $K(\cdot,\cdot)$ is symmetric with $K(s,t)=K(t,s)$ for any $s,t\in \Scal$. The reproducing kernel Hilbert space (RKHS) associated with $K$, denoted by $\Hcal_K$ (suppressing its dependence on the domain $\Scal$), can be defined to be the space endowed with an inner product $\langle\cdot,\cdot \rangle_{\Hcal_K}$ such that: (i) $K(s,\cdot)\in \Hcal_K$ for each $s\in \Scal$; (ii) reproducing property: for any $f\in \Hcal_K$, $\langle f, K(\cdot,s)\rangle_{\Hcal_K} = f(s)$ for all $s\in \Scal$ (see Definition 6.1 of \citet{RasWil06}).

For shift-invariant kernels (including the isotropic Mat\'ern in this paper), an alternative and equivalent definition of the RKHS norm is based on the spectral density of the kernel. Details can be found in \citet{Wen05}. Let $\imath=\sqrt{-1}$ and $\Fcal[f](\omega) = (2\pi)^{-d/2} \int_{\Scal} f(x) \ee^{- \imath x^\top \omega} \ud \omega$ for any $\omega\in \RR^d$. If $K(\cdot,\cdot)$ is a shift-invariant kernel on $\Scal$, with $\Phi(s-s')\equiv K(s,s')$ for any $s,s'\in \Scal$, then Theorem 10.12 of \citet{Wen05} has shown that the RKHS associated with $K$ can be written as
\begin{align}\label{eq:rkhs.specdef}
\Hcal_K &= \Big\{f\in L_2(\Scal)\cap \Ccal(\Scal) : \exists g\in L_2(\RR^d)\cap \Ccal(\RR^d) \text{ such that } g\big|_{\Scal}=f, \nonumber \\
&\|f\|_{\Hcal_K}^2 = \|g\|_{\Hcal_K}^2  = (2\pi)^{-d/2} \int_{\RR^d} \frac{|\Fcal[g](\omega)|^2}{\Fcal[\Phi](\omega)} \ud \omega <\infty\Big \},
\end{align}
where $g\big|_{\Scal}$ is the restriction of $g$ to the domain $\Scal$. For ease of notation, we suppress the dependence on $\Scal$ in the notation $\Hcal_K$.

In particular, for the isotropic Mat\'ern covariance function $\sigma^2 K_{\alpha,\nu}$ as defined in \eqref{eq:MaternCov} of the main text, we know that $\Fcal[\sigma^2 K_{\alpha,\nu}](\omega) = \frac{2^{d/2}\Gamma(\nu+d/2)}{\Gamma(\nu)} \frac{\sigma^2\alpha^{2\nu}}{(\alpha^2+\|\omega\|^2)^{\nu+d/2}}$. So the RKHS associated with $\sigma^2 K_{\alpha,\nu}$ can be written as
\begin{align}\label{eq:rkhs.matern}
& \Hcal_{\sigma^2 K_{\alpha,\nu}} = \Big\{f\in L_2(\Scal)\cap \Ccal(\Scal) : \exists g\in L_2(\RR^d)\cap \Ccal(\RR^d) \text{ such that } g\big|_{\Scal}=f, \nonumber\\
&\|f\|_{\Hcal_{\sigma^2 K_{\alpha,\nu}}}^2 =  \frac{\Gamma(\nu)}{2^{d}\pi^{d/2}\Gamma(\nu+d/2)\sigma^2\alpha^{2\nu}} \int_{\RR^d} (\alpha^2+\|\omega\|^2)^{\nu+d/2}|\Fcal[g](\omega)|^2 \ud \omega <\infty\Big \},
\end{align}

\vspace{2mm}

\begin{lemma}\label{lem:matern.sobolev}
(\citet{Wen05} Corollary 10.48) For any fixed $(\sigma^2,\alpha)\in \RR^+ \times \RR^+$, $d\in\ZZ^+$, $\nu\in\RR^+$, $\Hcal_{\sigma^2 K_{\alpha,\nu}}$ is norm equivalent to the Sobolev space $\Wcal_{2}^{\nu+d/2}(\Scal)$. In other words, there exist constants $0<c_1(\sigma,\alpha)\leq c_2(\sigma,\alpha)<\infty$, such that for any $f\in \Hcal_{\sigma^2 K_{\alpha,\nu}}$,
$$c_1(\sigma,\alpha) \|f\|_{\Wcal_{2}^{\nu+d/2}(\Scal)} \leq \|f\|_{\Hcal_{\sigma^2 K_{\alpha,\nu}}} \leq c_2(\sigma,\alpha) \|f\|_{\Wcal_{2}^{\nu+d/2}(\Scal)}.$$
\end{lemma}

\vspace{2mm}

\begin{lemma} \label{lem:rkhs.quadratic}
Suppose that $f\in \Hcal_K$ for a covariance function $K(\cdot,\cdot)$ defined on the fixed domain $\Scal=[0,T]^d$. Let $\Scal_n = \{s_1,\ldots,s_n\}$ be a set of distinct points in $\Scal$, $f_n = (f(s_1),\ldots,f(s_n))^\top$, and $K(\Scal_n,\Scal_n)$ be the matrix with $(i,j)$-entry equal to $K(s_i,s_j)$, for $i,j=1,\ldots,n$. Then
$f_n^\top K(\Scal_n,\Scal_n)^{-1} f_n \leq \|f\|_{\Hcal_K}^2$.
\end{lemma}

\begin{proof}[Proof of Lemma \ref{lem:rkhs.quadratic}]
We denote the $(i,j)$-entry of the matrix $K(\Scal_n,\Scal_n)^{-1}$ by $\big\{K(\Scal_n,\Scal_n)^{-1}\big\}_{ij}$, for $i,j=1,\ldots,n$. Let $K(\Scal_n,s)=(K(s_1,s),\ldots,K(s_n,s))^\top$ for any $s\in \Scal$. Because the function $K(s,\cdot)\in \Hcal_K$ for any $s\in \Scal$ by the definition of RKHS, we have that the function $f_n^\top K(\Scal_n,\Scal_n)^{-1} K(\Scal_n,\cdot)\in \Hcal_K$. For any $a=(a_1,\ldots,a_n)^\top \in \RR^n$, the RKHS norm of the function
$a^\top K(\Scal_n,\cdot)$ is
$$\left\|a^\top K(\Scal_n,\cdot)\right\|_{\Hcal_K}^2=\sum_{i=1}^n\sum_{j=1}^n a_ia_jK(s_i,s_j).$$
Therefore, the RKHS norm of $f_n^\top K(\Scal_n,\Scal_n)^{-1} K(\Scal_n,\cdot)$ is
\begin{align}\label{eq:fKK1}
&\quad~ \left\|f_n^\top K(\Scal_n,\Scal_n)^{-1} K(\Scal_n,\cdot)\right\|_{\Hcal_K}^2 \nonumber\\
&= \sum_{i=1}^n \sum_{j=1}^n \sum_{k=1}^n \sum_{l=1}^n \big\{K(\Scal_n,\Scal_n)^{-1}\big\}_{ij}  \big\{K(\Scal_n,\Scal_n)^{-1}\big\}_{kl}\cdot  f(s_i)f(s_k) K(s_j,s_l) \nonumber\\
&= \sum_{i=1}^n \sum_{k=1}^n f(s_i)f(s_k)  \left\{\sum_{j=1}^n \sum_{l=1}^n \big\{K(\Scal_n,\Scal_n)^{-1}\big\}_{ij}  \big\{K(\Scal_n,\Scal_n)^{-1}\big\}_{kl} K(s_j,s_l)\right\} \nonumber\\
&\stackrel{(i)}{=} \sum_{i=1}^n \sum_{k=1}^n f(s_i)f(s_k) \big\{K(\Scal_n,\Scal_n)^{-1}\big\}_{ik} = f_n^\top K(\Scal_n,\Scal_n)^{-1} f_n,
\end{align}
where the equality (i) follows from the expression of $(i,k)$-entry in the matrix multiplication $K(\Scal_n,\Scal_n)^{-1} K(\Scal_n,\Scal_n) K(\Scal_n,\Scal_n)^{-1}$.

On the other hand, using the RKHS inner product and the fact that $f\in \Hcal_K$, we have
\begin{align} \label{eq:fKK2}
&\quad ~ f_n^\top K(\Scal_n,\Scal_n)^{-1} f_n = \sum_{i=1}^n \sum_{k=1}^n f(s_i)f(s_k) \big\{K(\Scal_n,\Scal_n)^{-1}\big\}_{ik}  \nonumber \\
&= \left\langle \sum_{i=1}^n\sum_{k=1}^n f(s_k)\big\{K(\Scal_n,\Scal_n)^{-1}\big\}_{ik} K(s_i,\cdot), ~~ f(\cdot) \right \rangle_{\Hcal_K} \nonumber \\
&\stackrel{(i)}{\leq} \left\|\sum_{i=1}^n\sum_{k=1}^n f(s_k)\big\{K(\Scal_n,\Scal_n)^{-1}\big\}_{ik} K(s_i,\cdot) \right\|_{\Hcal_k} \cdot \|f\|_{\Hcal_K} \nonumber \\
&= \left\|f_n^\top K(\Scal_n,\Scal_n)^{-1} K(\Scal_n,\cdot)\right\|_{\Hcal_k} \cdot \|f\|_{\Hcal_K} \nonumber \\
&\stackrel{(ii)}{=} \sqrt{f_n^\top K(\Scal_n,\Scal_n)^{-1} f_n} \cdot \|f\|_{\Hcal_K} ,
\end{align}
where the inequality (i) follows from $\langle f_1,f_2 \rangle_{\Hcal_K} \leq \|f_1\|_{\Hcal_K} \|f_2\|_{\Hcal_K}$ for any $f_1,f_2\in \Hcal_K$, and the equality (ii) follows from \eqref{eq:fKK1}. Therefore, we conclude from the left-hand side and the right-hand side of \eqref{eq:fKK2} that $\sqrt{f_n^\top K(\Scal_n,\Scal_n)^{-1} f_n}  \leq \|f\|_{\Hcal_K}$, or
$f_n^\top K(\Scal_n,\Scal_n)^{-1} f_n  \leq \|f\|_{\Hcal_K}^2$.
\end{proof}
\vspace{5mm}

\begin{lemma}\label{lem:rkhs.ordering}
For any $f\in \Wcal_2^{\nu+d/2}(\Scal)$, any $d\in\ZZ^+$, $\nu\in\RR^+$, $\alpha\in \RR^+$,
\begin{align}\label{eq:rkhs12}
\|f\|_{\Hcal_{(\theta_0/\alpha^{2\nu}) K_{\alpha,\nu}}} \leq \max\left\{\left(\frac{\alpha}{\alpha_0}\right)^{\nu+d/2},1 \right\} \cdot \|f\|_{\Hcal_{\sigma_0^2 K_{\alpha_0,\nu}}} .
\end{align}
\end{lemma}

\begin{proof}[Proof of Lemma \ref{lem:rkhs.ordering}]
From \eqref{eq:rkhs.matern}, one can see that for any function $f\in \Wcal_2^{\nu+d/2}(\Scal)$, for any $\alpha>0$,
\begin{align}
&\quad~ \|f\|_{\Hcal_{(\theta_0/\alpha^{2\nu})K_{\alpha,\nu}}}^2 \nonumber \\
&=  \frac{\Gamma(\nu)}{2^{d}\pi^{d/2}\Gamma(\nu+d/2)\theta_0} \int_{\RR^d} (\alpha^2+\|\omega\|^2)^{\nu+d/2}|\Fcal[f](\omega)|^2 \ud \omega  \nonumber \\
& \leq \sup_{\omega\in \RR^d} \left(\frac{\alpha^2 + \|\omega\|^2}{\alpha_0^2 + \|\omega\|^2} \right)^{\nu+d/2} \cdot  \frac{\Gamma(\nu)}{2^{d}\pi^{d/2}\Gamma(\nu+d/2)\theta_0} \int_{\RR^d} (\alpha_0^2+\|\omega\|^2)^{\nu+d/2}|\Fcal[f](\omega)|^2 \ud \omega  \nonumber \\
& \leq \sup_{\omega\in \RR^d} \left(\frac{\alpha^2 + \|\omega\|^2}{\alpha_0^2 + \|\omega\|^2} \right)^{\nu+d/2} \cdot \|f\|_{\Hcal_{(\theta_0/\alpha_0^{2\nu})K_{\alpha_0,\nu}}}^2  \nonumber \\
& \leq \max\left\{\left(\frac{\alpha}{\alpha_0}\right)^{2(\nu+d/2)},1 \right\} \cdot \|f\|_{\Hcal_{\sigma_0^2 K_{\alpha_0,\nu}}}^2 . \nonumber
\end{align}
Hence the conclusion follows.
\end{proof}

\vspace{5mm}

\subsection{Auxiliary Results on Spectral Analysis of Mat\'ern Covariance Functions} \label{supsec:spectral}
In this subsection, we present a series of technical lemmas on the spectral analysis of Mat\'ern covariance functions. For a detailed background theory on the equivalence of Gaussian measures on Hilbert spaces, we refer the interested readers to Chapter III of \citet{IbrRoz78} and Chapter 4 of \citet{Stein99a}. Our Lemmas \ref{lem:fxi}, \ref{lem:qngn}, and \ref{lem:zetabound} below will use similar techniques in Section 4 of \citet{WangLoh11}. The key difference is that the theory of \citet{WangLoh11} only works for a \textit{fixed and known} value of range parameter $\alpha$. As a result, all those probabilistic error bounds in \citet{WangLoh11} do not depend on $\alpha$ and cannot be directly applied to varying values of $\alpha$ drawn from a posterior distribution. In contrast, our lemmas below will make all error bounds explicitly dependent on the value of $\alpha$. This is made possible by using our new results on Mat\'ern spectral densities in Lemma \ref{lem:specden_lambda}, which is not shown in \citet{WangLoh11}. These lemmas will be used for showing the \textit{uniform convergence} of $\big|\widetilde\theta_{\alpha}-\widetilde\theta_{\alpha_0}\big|$ over a large range of values of $\alpha$ as proved in Lemma \ref{lem:sup.theta1}, which is fundamental for deriving the limiting joint posterior distribution of $(\theta,\alpha)$.

We first consider the case when $\sigma^2\alpha^{2\nu}=\theta_0=\sigma_0^2\alpha_0^{2\nu}$. If $d\in\{1,2,3\}$, then the two Gaussian measures $\gp(0,\sigma^2K_{\alpha,\nu})$ and $\gp(0,\sigma_0^2K_{\alpha_0,\nu})$ are equivalent (\citet{Zhang04}). For a generic $\alpha>0$, we consider the two Mat\'ern covariance matrices $\sigma_0^2R_{\alpha_0}$ and $\sigma^2 R_{\alpha}$. We have the following lemma.

\begin{lemma}\label{lem:URU}
For any pair $(\sigma,\alpha)\in \RR^+ \times \RR^+$ that satisfies $\sigma^2\alpha^{2\nu}=\theta_0=\sigma_0^2\alpha_0^{2\nu}$, for all $d\in \ZZ^+$ and $\nu\in \RR^+$, there exists an $n\times n$ invertible matrix $U_{\alpha}$ that depends on $\alpha,\alpha_0,\sigma_0^2,\nu$, such that
\begin{align}\label{diagonalize}
& \sigma_0^2 U_{\alpha}^\top R_{\alpha_0} U_{\alpha}= I_n,\qquad \sigma^2 U_{\alpha}^\top R_{\alpha} U_{\alpha}= \diag\{\lambda_{k,n}(\alpha):k=1,\ldots,n\} \equiv \Lambda_{\alpha},
\end{align}
where $I_n$ is the $n\times n$ identity matrix, and $\{\lambda_{k,n}(\alpha),k=1,\ldots,n\}$ are the positive diagonal entries of the diagonal matrix $\Lambda_{\alpha}$.
\end{lemma}

\begin{proof}[Proof of Lemma \ref{lem:URU}]
The existence of such an invertible $U_{\alpha}$ is guaranteed by Theorem 7.6.4 and Corollary 7.6.5 on page 465--466 of \citet{HorJoh85}. For completeness, we directly prove the existence of such an invertible matrix in the following general claim.
\vspace{2mm}

\noindent \underline{\bf Claim:} Suppose that $A$ and $B$ are two generic $n\times n$ symmetric positive definite matrices. Then there always exists an invertible matrix $U$, such that
\begin{align} \label{eq:UAU}
& U^\top A U = I_n,\qquad U^\top B U = \Lambda,
\end{align}
where $I_n$ is the $n\times n$ identity matrix and $\Lambda$ is an $n\times n$ diagonal matrix whose diagonal entries are all positive.
\vspace{2mm}

\noindent \underline{Proof of the Claim:}
Since $B$ is symmetric positive definite, let $B=LL^\top$ be the Cholesky decomposition of $B$, where $L$ is an $n\times n$ lower triangular matrix with all positive diagonal entries and $L$ is invertible. Let $G=L^{-1} A L^{-\top}$. Then obviously $G$ is also a symmetric positive definite matrix with $G^\top=G$. Suppose that $G$ has the spectral decomposition $G=PDP^{-1}$ where $P$ is an $n\times n$ orthogonal matrix ($P^{-1}=P^\top$) and $D$ is a $n\times n$ diagonal matrix whose diagonal entries are all eigenvalues of $G$ and they are all positive. Then $P^\top G P = D$. We let $U=L^{-\top} P D^{-1/2}$. It follows that
\begin{align*}
U^\top A U &= D^{-1/2} P^\top L^{-1} A L^{-\top} P D^{-1/2} \\
&= D^{-1/2} P^\top G P D^{-1/2} = D^{-1/2} D D^{-1/2} = I_n,\\
U^\top B U &= D^{-1/2} P^\top L^{-1} B L^{-\top} P D^{-1/2} \\
&= D^{-1/2} P^\top L^{-1} LL^\top L^{-\top} P D^{-1/2} = D^{-1/2} P^\top P  D^{-1/2} = D^{-1}.
\end{align*}
We set $\Lambda=D^{-1}$ which is an $n\times n$ diagonal matrix whose diagonal entries are all positive. This proves the claim.
\vspace{3mm}

Based on the claim, if we set $A=\sigma_0^2 R_{\alpha_0}$ and $B=\sigma^2 R_{\alpha}$, then we can find an invertible matrix $U$ such that \eqref{eq:UAU} holds. Because $\sigma^2\alpha^{2\nu}=\theta_0=\sigma_0^2\alpha_0^{2\nu}$, and $\sigma_0^2,\alpha_0,\nu$ are assumed to be fixed numbers, we can see that $U$ only changes with $\alpha$ and we can write it as $U_{\alpha}$. Similarly, we write $\Lambda_{\alpha}$ to highlight its dependence on $\alpha$. Correspondingly, we have $\sigma_0^2 U_{\alpha}^\top R_{\alpha_0} U_{\alpha}= I_n$ and $\sigma^2 U_{\alpha}^\top R_{\alpha} U_{\alpha}= \diag\{\lambda_{k,n}(\alpha):k=1,\ldots,n\} \equiv \Lambda_{\alpha}$. This proves Lemma \ref{lem:URU}.
\end{proof}

\vspace{5mm}

Let $\imath=\sqrt{-1}$. For $\omega \in \RR^d$, let
\begin{align}\label{f.specden}
f_{\sigma,\alpha}(\omega) &= \frac{1}{(2\pi)^d}\int_{\RR^d} \ee^{-\imath \omega^\top x} \sigma^2 K_{\alpha,\nu}(x) \ud x \nonumber \\
&= \frac{\Gamma(\nu+d/2)}{\Gamma(\nu)}\cdot \frac{\sigma^2\alpha^{2\nu}}{\pi^{d/2}\left(\alpha^2+\|\omega\|^2\right)^{\nu+d/2}},
\end{align}
be the isotropic spectral density of the Gaussian process with isotropic Mat\'ern covariance function defined in \eqref{eq:MaternCov} of the main text. For any given pair $(\sigma,\alpha)$, let $\|\psi\|_{f_{\sigma,\alpha}}^2 = \langle \psi,\psi\rangle_{f_{\sigma,\alpha}} =\int_{\RR^d} |\psi(\omega)|^2 f_{\sigma,\alpha}(\omega)\ud \omega$ be the norm of a generic function $\psi$ in the Hilbert space $L_2(f_{\sigma,\alpha})$, with inner product $\langle \psi_1,\psi_2\rangle_{f_{\sigma,\alpha}} =\int_{\RR^d} \psi_1(\omega) \overline{\psi_2(\omega)} f_{\sigma,\alpha}(\omega)\ud \omega$ for any $\psi_1,\psi_2\in L_2(f_{\sigma,\alpha})$.

According to the spectral analysis in Section 4 of \citet{WangLoh11}, using the same notation as theirs, for any given pair $(\sigma,\alpha)$ that satisfies $\sigma^2\alpha^{2\nu}=\theta_0=\sigma_0^2\alpha_0^{2\nu}$, there exist orthonormal basis functions $\psi_1,\ldots,\psi_n \in L_2(f_{\sigma_0,\alpha_0})$ such that for any $j,k\in \{1,\ldots,n\}$,
\begin{align}\label{wangloh.eq13}
& \langle \psi_j,\psi_k \rangle_{f_{\sigma_0,\alpha_0}} =  \Ical(j=k), \qquad \langle \psi_j,\psi_k \rangle_{f_{\sigma,\alpha}} =  \lambda_{j,n}(\alpha) \Ical(j=k),
\end{align}
where $\Ical(\cdot)$ is the indicator function.

We prove the following lemma for the spectral density $f_{\sigma,\alpha}$ and the sequence $\{\lambda_{k,n}(\alpha),k=1,\ldots,n\}$.

\begin{lemma}\label{lem:specden_lambda}
Suppose that $d\in \ZZ^+$ and $\nu\in \RR^+$. For any pair $(\sigma,\alpha)\in \RR^+ \times \RR^+$ that satisfies $\sigma^2\alpha^{2\nu}=\theta_0=\sigma_0^2\alpha_0^{2\nu}$, and for all $\omega \in \RR^d$, the following relations hold:
\begin{align}
& \min\left\{\left(\frac{\alpha_0}{\alpha}\right)^{2\nu+d},1\right\} \leq \frac{f_{\sigma,\alpha}(\omega)}{f_{\sigma_0,\alpha_0}(\omega)} \leq \max \left\{\left(\frac{\alpha_0}{\alpha}\right)^{2\nu+d},1\right\}, \label{ratio.ff1} \\
& \left|\frac{f_{\sigma,\alpha}(\omega)}{f_{\sigma_0,\alpha_0}(\omega)} -1 \right| \leq \frac{(2\nu+d)\max(\alpha_0^2,\alpha^2) \max\left(\alpha_0^{2\nu+d-2},\alpha^{2\nu+d-2}\right)}{\alpha^{2\nu+d-2}(\alpha^2+\|\omega\|^2)}, \label{ratio.ff2} \\
& \lambda_{k,n}(\alpha) \leq \max \left\{\left(\frac{\alpha_0}{\alpha}\right)^{2\nu+d},1\right\}, \label{lambda.upper1} \\
& \lambda_{k,n}(\alpha) \geq \min\left\{\left(\frac{\alpha_0}{\alpha}\right)^{2\nu+d},1\right\}, \label{lambda.lower1}
\end{align}
for all $k=1,\ldots,n$.
\end{lemma}

\begin{proof}[Proof of Lemma \ref{lem:specden_lambda}]
For \eqref{ratio.ff1}, when $\sigma^2\alpha^{2\nu}=\theta_0$, we have that
\begin{align*}
\frac{f_{\sigma,\alpha}(\omega)}{f_{\sigma_0,\alpha_0}(\omega)}  & = \left(\frac{\alpha_0^2+\|\omega\|^2}{\alpha^2+\|\omega\|^2}\right)^{\nu+d/2}.
\end{align*}
If $\alpha\geq \alpha_0$, then this ratio is an increasing function in $\|\omega\|$, which implies that $f_{\sigma,\alpha}(\omega)/f_{\sigma_0,\alpha_0}(\omega) \leq 1$ (attained when $\|\omega\|\to +\infty$), and $f_{\sigma,\alpha}(\omega)/f_{\sigma_0,\alpha_0}(\omega) \geq (\alpha_0/\alpha)^{2\nu+d}$ (attained when $\|\omega\|\to 0$). The case of $\alpha < \alpha_0 $ follows similarly. \eqref{ratio.ff1} summarizes the two cases.

For \eqref{ratio.ff2}, if $\nu+d/2\geq 1$, then using a first order Taylor expansion, we have that
\begin{align}\label{ratio.ff21}
& \left|\frac{f_{\sigma,\alpha}(\omega)}{f_{\sigma_0,\alpha_0}(\omega)} -1 \right| = \left|\frac{\left(\alpha_0^2+\|\omega\|^2\right)^{\nu+d/2}}{\left(\alpha^2+\|\omega\|^2\right)^{\nu+d/2}} -1\right| \nonumber \\
&\leq \frac{(\nu+d/2)(\alpha_1^{2} + \|\omega\|^2 )^{\nu+d/2-1}\cdot 2\alpha_1\cdot |\alpha-\alpha_0|}{\left(\alpha^2+\|\omega\|^2\right)^{\nu+d/2}} \nonumber \\
&\leq (2\nu+d)\max(\alpha_0^2,\alpha^2) \left(\frac{\max(\alpha_0,\alpha)^2+\|\omega\|^2}{\alpha^2+\|\omega\|^2}\right)^{\nu+d/2-1} \cdot \frac{1}{\alpha^2 + \|\omega\|^2} \nonumber \\
&\leq \frac{(2\nu+d)\max(\alpha_0^2,\alpha^2) \max\left(\alpha_0^{2\nu+d-2},\alpha^{2\nu+d-2}\right)}{\alpha^{2\nu+d-2}(\alpha^2+\|\omega\|^2)} ,
\end{align}
where $\alpha_1$ is a value between $\alpha_0$ and $\alpha$.

If $\nu+d/2 < 1$, then we have that
\begin{align}\label{ratio.ff22}
& \left|\frac{f_{\sigma,\alpha}(\omega)}{f_{\sigma_0,\alpha_0}(\omega)} -1 \right| = \left|\frac{\left(\alpha_0^2+\|\omega\|^2\right)^{\nu+d/2}}{\left(\alpha^2+\|\omega\|^2\right)^{\nu+d/2}} -1\right| \nonumber \\
&\leq \frac{(\nu+d/2)(\alpha_1^{2} + \|\omega\|^2 )^{\nu+d/2-1}\cdot 2\alpha_1\cdot |\alpha-\alpha_0|}{\left(\alpha^2+\|\omega\|^2\right)^{\nu+d/2}} \nonumber \\
&\leq (2\nu+d)\max(\alpha_0^2,\alpha^2) \left(\frac{\alpha^2+\|\omega\|^2}{\alpha_1^2+\|\omega\|^2}\right)^{1-(\nu+d/2)} \cdot \frac{1}{\alpha^2 + \|\omega\|^2}.
\end{align}
In \eqref{ratio.ff22}, if $\alpha\geq \alpha_1 \geq\alpha_0$, then the function $\left(\frac{\alpha^2+\|\omega\|^2}{\alpha_1^2+\|\omega\|^2}\right)^{1-(\nu+d/2)}$ is decreasing in $\|\omega\|^2$, so
\begin{align*}
&\left(\frac{\alpha^2+\|\omega\|^2}{\alpha_1^2+\|\omega\|^2}\right)^{1-(\nu+d/2)}\leq \left(\frac{\alpha}{\alpha_1}\right)^{2-(2\nu+d)} = \left(\frac{\alpha_1}{\alpha}\right)^{2\nu+d-2} \leq \left(\frac{\alpha_0}{\alpha}\right)^{2\nu+d-2}.
\end{align*}
If $\alpha\leq \alpha_1 \leq\alpha_0$, then the function $\left(\frac{\alpha^2+\|\omega\|^2}{\alpha_1^2+\|\omega\|^2}\right)^{1-(\nu+d/2)}$ is increasing in $\|\omega\|^2$, so
\begin{align*}
&\left(\frac{\alpha^2+\|\omega\|^2}{\alpha_1^2+\|\omega\|^2}\right)^{1-(\nu+d/2)}\leq 1.
\end{align*}
Considering both cases, then from \eqref{ratio.ff2}, we can derive that
\begin{align}\label{ratio.ff23}
\left|\frac{f_{\sigma,\alpha}(\omega)}{f_{\sigma_0,\alpha_0}(\omega)} -1 \right|
&\leq (2\nu+d)\max(\alpha_0^2,\alpha^2) \left(\frac{\alpha^2+\|\omega\|^2}{\alpha_1^2+\|\omega\|^2}\right)^{1-(\nu+d/2)} \cdot \frac{1}{\alpha^2 + \|\omega\|^2} \nonumber \\
&\leq \frac{(2\nu+d)\max(\alpha_0^2,\alpha^2)}{\alpha^2 + \|\omega\|^2} \max\left\{\left(\frac{\alpha_0}{\alpha}\right)^{2\nu+d-2},1\right\} \nonumber \\
&\leq \frac{(2\nu+d)\max(\alpha_0^2,\alpha^2) \max\left(\alpha_0^{2\nu+d-2},\alpha^{2\nu+d-2}\right)}{\alpha^{2\nu+d-2}(\alpha^2+\|\omega\|^2)}.
\end{align}
\eqref{ratio.ff21} for $\nu+d/2\geq 1$ and \eqref{ratio.ff23} for $\nu+d/2<1$ lead to \eqref{ratio.ff2}.

For \eqref{lambda.upper1} and \eqref{lambda.lower1}, we use the relation $\lambda_{k,n}(\alpha)= \int_{\RR^d} |\psi_k(\omega)|^2 f_{\sigma_0,\alpha_0}(\omega) \cdot \frac{f_{\sigma,\alpha}(\omega)}{f_{\sigma_0,\alpha_0}(\omega)} \ud \omega$ for $k=1,\ldots,n$ and the bounds in \eqref{ratio.ff1} to obtain that
\begin{align}\label{lambda30}
\lambda_{k,n}(\alpha)&\leq \sup_{\omega \in \RR^d} \frac{f_{\sigma,\alpha}(\omega)}{f_{\sigma_0,\alpha_0}(\omega)}  \cdot \int_{\RR^d} |\psi_k(\omega)|^2 f_{\sigma_0,\alpha_0}(\omega) \ud \omega \leq \max \left\{\left(\frac{\alpha_0}{\alpha}\right)^{2\nu+d},1\right\}, \nonumber \\
\lambda_{k,n}(\alpha)&\geq \inf_{\omega \in \RR^d} \frac{f_{\sigma,\alpha}(\omega)}{f_{\sigma_0,\alpha_0}(\omega)}  \cdot \int_{\RR^d} |\psi_k(\omega)|^2 f_{\sigma_0,\alpha_0}(\omega) \ud \omega \geq \min\left\{\left(\frac{\alpha_0}{\alpha}\right)^{2\nu+d},1\right\}.
\end{align}	
\end{proof}

In the rest of this subsection, we focus exclusively on the case of $d\in\{1,2,3\}$. For any $a>0$, define $m_a= \lfloor a+d/2\rfloor +1$. For $\omega \in \RR^d$, let
\begin{align}
c_0(x) &= \|x\|^{\frac{\nu+d/2}{2m_{\nu}}-d} \Ical(\|x\|\leq 1), \label{c0func} \\
\xi_0(\omega) &= \int_{\RR^d} \ee^{-\imath x^\top w} c_0(x) \ud x, \label{xi0func}
\end{align}
and $\xi_1(\omega)=\xi_0(\omega)^{2 m_{\nu}}$ for all $\omega \in \RR^d$. If $c_1=c_0\ast \ldots \ast c_0$ is the $2m_{\nu}$-fold convolution of the function $c_0$ with itself, then $\xi_1(\omega)$ is the Fourier transform of $c_1(x)$. Then Lemma 6 in \citet{WangLoh11} has proved that for $d=1,2,3$, $\xi_0(\omega)\asymp \|\omega\|^{-\frac{\nu+d/2}{2m_{\nu}}}$ as $\|\omega\|\to \infty$, which means that $\xi_1(\omega)\asymp \|\omega\|^{-(\nu+d/2)}$. This implies that if $\sigma^2\alpha^{2\nu}=\theta_0$, then $f_{\sigma,\alpha}(\omega)/\xi_1(\omega) \asymp 1$ as $\|\omega\|\to \infty$. In fact, using Lemma 6 in \citet{WangLoh11}, we can prove the following lower and upper bound for his ratio.

\begin{lemma}\label{lem:fxi}
Suppose that $d\in \{1,2,3\}$ and $\nu\in \RR^+$. For any pair $(\sigma,\alpha)\in \RR^+ \times \RR^+$, the following holds for all $\omega \in \RR^d$:
\begin{align}\label{proxy.bound1}
& \underline c_{\xi} \sigma^2\alpha^{2\nu} \min\left\{\left(\frac{\alpha_0}{\alpha}\right)^{2\nu+d},1\right\} \leq \frac{f_{\sigma,\alpha}(\omega)}{\xi_1(\omega)^2} \leq \overline c_{\xi} \sigma^2 \alpha^{2\nu} \max\left\{\left(\frac{\alpha_0}{\alpha}\right)^{2\nu+d},1\right\},
\end{align}
where $\underline c_{\xi}$ and $\overline c_{\xi}$ are two positive constants that only depend on $d$, $\nu$ and $\alpha_0$.
\end{lemma}

\begin{proof}[Proof of Lemma \ref{lem:fxi}]
Lemma 6 in \citet{WangLoh11} has proved that for $d=1,2,3$, $\xi_0(\omega)\asymp \|\omega\|^{-\frac{\nu+d/2}{2m_{\nu}}}$ as $\|\omega\|\to \infty$. This implies that there exists two positive absolute constants $\underline c_{\xi 0}$ and $\overline c_{\xi 0}$ that only depend on $d$, $\nu$ and $\alpha_0$, such that
$$\underline c_{\xi 0} \leq (\alpha_0^2+\|\omega\|^2)^{\frac{\nu+d/2}{4m_{\nu}}}\xi_0(\omega) \leq \overline c_{\xi 0},$$
for all $\omega \in \RR^d$. According to the definition of $\xi_1(\omega)$, this implies that
\begin{align}\label{eq:xi1.12}
& \underline c_{\xi 0}^{2m_{\nu}} \leq (\alpha_0^2+\|\omega\|^2)^{\frac{\nu+d/2}{2}}\xi_1(\omega) \leq \overline c_{\xi 0}^{2m_{\nu}},
\end{align}
for all $\omega \in \RR^d$. Now, from the definition of $f_{\sigma,\alpha}$ in \eqref{f.specden}, we have that
\begin{align}\label{fxi1}
\frac{f_{\sigma,\alpha}(\omega)}{\xi_1(\omega)^2} & = \frac{\sigma^2\alpha^{2\nu}(\alpha_0^2+\|\omega\|^2)^{\nu+d/2}}{\pi^{d/2}\left(\alpha^2+\|\omega\|^2\right)^{\nu+d/2}} \cdot \frac{1}{(\alpha_0^2+\|\omega\|^2)^{\nu+d/2}\xi_1(\omega)^2}.
\end{align}
Since
\begin{align*}
& \min\left\{\left(\frac{\alpha_0}{\alpha}\right)^{2\nu+d},1\right\} \leq \left(\frac{\alpha_0^2+\|\omega\|^2}{\alpha^2+\|\omega\|^2}\right)^{\nu+d/2} \leq \max\left\{\left(\frac{\alpha_0}{\alpha}\right)^{2\nu+d},1\right\},
\end{align*}
we have from \eqref{eq:xi1.12} and \eqref{fxi1} that
\begin{align*}
\frac{f_{\sigma,\alpha}(\omega)}{\xi_1(\omega)^2} &\geq \frac{\sigma^2\alpha^{2\nu}}{\pi^{d/2}\overline c_{\xi 0}^{4m_{\nu}}} \min\left\{\left(\frac{\alpha_0}{\alpha}\right)^{2\nu+d},1\right\}, \\
\frac{f_{\sigma,\alpha}(\omega)}{\xi_1(\omega)^2} &\leq \frac{\sigma^2\alpha^{2\nu}}{\pi^{d/2}\underline c_{\xi 0}^{4m_{\nu}}} \max\left\{\left(\frac{\alpha_0}{\alpha}\right)^{2\nu+d},1\right\}.
\end{align*}
Finally, we let $\underline c_{\xi}= 1/(\pi^{d/2}\overline c_{\xi 0}^{4m_{\nu}})$ and $\overline c_{\xi}= 1/(\pi^{d/2}\underline c_{\xi 0}^{4m_{\nu}})$ and the conclusion follows.
\end{proof}

\vspace{5mm}

Now to proceed, we define the function
\begin{align}\label{etafunc}
\eta(\omega) &= \frac{f_{\sigma,\alpha}(\omega)-f_{\sigma_0,\alpha_0}(\omega)}{\xi_1(\omega)^2}, \quad \forall \omega \in \RR^d.
\end{align}
Note that $\eta$ depends on $(\sigma,\alpha)$, but we suppress the dependence for the ease of notation.

For any given pair $(\sigma,\alpha)\in \RR^+ \times \RR^+$, from \eqref{ratio.ff2} in Lemma \ref{lem:specden_lambda} and \eqref{proxy.bound1} in Lemma \ref{lem:fxi}, we have that
\begin{align}\label{eq:eta.bound1}
& \int_{\RR^d} \eta_n(\omega)^2\ud \omega = \int_{\RR^d} \left\{\frac{f_{\sigma,\alpha}(\omega)-f_{\sigma_0,\alpha_0}(\omega)}{\xi_1(\omega)^2}\right\}^2\ud \omega  \nonumber \\
&= \int_{\RR^d} \left\{\frac{f_{\sigma,\alpha}(\omega)-f_{\sigma_0,\alpha_0}(\omega)}{f_{\sigma_0,\alpha_0}(\omega)}\right\}^2 \cdot \left(\frac{f_{\sigma_0,\alpha_0}(\omega)}{\xi_1(\omega)^2}\right)^2 \ud \omega  \nonumber \\
&\leq \sup_{\omega \in \RR^d} \left(\frac{f_{\sigma_0,\alpha_0}(\omega)}{\xi_1(\omega)^2}\right)^2 \cdot \int_{\RR^d} \left|\frac{f_{\sigma,\alpha}(\omega)}{f_{\sigma_0,\alpha_0}(\omega)} -1 \right|^2 \ud \omega  \nonumber \\
&\leq \overline c_{\xi}^2 \theta_0^2 \cdot \int_{\RR^d} \left\{\frac{(2\nu+d)\max(\alpha_0^2,\alpha^2) \max\left(\alpha_0^{2\nu+d-2},\alpha^{2\nu+d-2}\right)}{\alpha^{2\nu+d-2}(\alpha^2+\|\omega\|^2)}\right\}^2 \ud \omega   \nonumber \\
&= \frac{\overline c_{\xi}^2 \theta_0^2 (2\nu+d)^2 \max(\alpha_0^4,\alpha^4) \max\left\{\alpha_0^{2(2\nu+d-2)},\alpha^{2(2\nu+d-2)}\right\}}{\alpha^{2(2\nu+d-2)}}  \nonumber \\
&\quad \times \int_0^{\infty} \frac{r^{d-1}}{(\alpha^2+r^2)^2} \ud r <\infty,
\end{align}
where the last integral is finite because $\alpha>0$ and $4-(d-1)\geq 2$ for $d=1,2,3$. Therefore, we have shown that $\eta(\omega)$ is a square-integrable function of $w$. From the theory of Fourier transforms of $L_2(\RR^d)$, there exists a square-integrable function $g:\RR^d\to \RR$ such that
$$\int_{\RR^d} \left\{\eta(\omega)-\hat g_k(\omega)\right\}^2 \ud \omega \rightarrow 0, \text{ as } k\to\infty,$$
where
\begin{align}\label{gkfun}
& \hat g_k(\omega) = \int_{\RR^d} \ee^{-\imath \omega^\top x} g(x) \Ical (\|x\|_{\infty}\leq k) \ud x .
\end{align}
Furthermore, for any fixed number $a>0$ and $0< \cbeta < \min(4-d,2)$, we define the sequence $\varepsilon_n = n^{-1/(4a+2d+\cbeta)}$, such that $ \varepsilon_n \to 0$ as $n\to\infty$. We define the following functions similar to Equations (35) and (36) in \citet{WangLoh11}. Let
\begin{align*}
\tilde c_0(x) &= \|x\|^{\frac{a+d/2}{2m_a}-d} \Ical(\|x\|\leq 1), \quad \forall x \in \RR^d,
\end{align*}
and $\tilde c_1(x) = c_0 \ast\ldots \ast c_0(x)$ be the $2m_a$-fold convolution of $c_0$ with itself. Let $C_q = \int_{\RR^d} \tilde c_1(x) \ud x$. Define the following functions
\begin{align}\label{4func}
\tilde \xi_0(\omega) &= \int_{\RR^d} \ee^{-\imath x^\top w} \tilde c_0(x)  \ud x, \quad \forall \omega\in \RR^d, \nonumber \\
\tilde \xi_1(\omega) &= \int_{\RR^d} \ee^{-\imath x^\top w} \tilde c_1(x)  \ud x = \tilde \xi_0(\omega)^{2m_a} , \quad \forall \omega\in \RR^d, \nonumber \\
q_n(x) &= \frac{1}{C_q \varepsilon_n^d} \tilde c_1\left(\frac{x}{\varepsilon_n}\right), \quad \forall x \in \RR^d, \nonumber \\
\hat q_n(\omega) &= \int_{\RR^d} \ee^{-\imath \omega^\top x} q_n(x) \ud x = \frac{1}{C_q}  \int_{\RR^d} \ee^{-\imath \varepsilon_n \omega^\top x} \tilde c_1(x) \ud x = \frac{\tilde \xi_1(\varepsilon_n w)}{C_q},
\quad \forall \omega\in \RR^d.
\end{align}
Then using Lemma 6 of \citet{WangLoh11}, there exists a finite positive constant $C_{\hat q}$ that only depends on $d,\nu,a,\cbeta$, such that
\begin{align}\label{qn.bound1}
&\left|\hat q_n(\omega)\right| \leq \frac{C_{\hat q}}{(1+\varepsilon_n\|\omega\|)^{a+d/2}},\quad \forall \omega \in \RR^d.
\end{align}

\begin{lemma}\label{lem:qngn}
Suppose that $d\in \{1,2,3\}$ and $\nu\in \RR^+$. Let $a>0$ and $0< \cbeta < \min(4-d,2)$ be fixed constants. Let $\varepsilon_n = n^{-1/(4a+2d+\cbeta)}$. For the $g$ function in \eqref{gkfun} and the $q_n$ function in \eqref{4func}, there exists a positive constant $C_{g,q}$ that depends only on $d,\nu,\alpha_0,a,\cbeta$, such that
\begin{align*}
& \left\{\int_{\RR^d} \left|q_n \ast g(x) - g(x) \right|^2 \ud x \right\}^{1/2} \leq C_{g,q} \frac{\max(\alpha_0^4,\alpha^4) \max\left\{\alpha_0^{2(2\nu+d-2)},\alpha^{2(2\nu+d-2)}\right\}}{\alpha^{4\nu+3d/2-\cbeta/2}} \varepsilon_n^{\cbeta/2},
\end{align*}
where $q_n\ast g(x) = \int_{\RR^d} q_n(y) g(x-y) \ud y$ for any $x\in \RR^d$.
\end{lemma}

\begin{proof}[Proof of Lemma \ref{lem:qngn}]
We have the following derivation:
\begin{align}\label{lemma5wang}
& \left\{\int_{\RR^d} \left|q_n \ast g(x) - g(x) \right|^2 \ud x \right\}^{1/2} \nonumber \\
={}& \left[ \int_{\RR^d} \left| \int_{\|y\|\leq 2m_a \varepsilon_n} \{g(x-y)-g(x)\}q_n(y)\ud y\right|^2\right]^{1/2} \nonumber \\
\stackrel{(i)}{\leq} {}& \int_{\|y\|\leq 2m_a \varepsilon_n} \left[\int_{\RR^d} | g(x-y) - g(x) |^2 \ud x \right]^{1/2} q_n(y) \ud y \nonumber \\
\stackrel{(ii)}{=} {}& \int_{\|y\|\leq 2m_a \varepsilon_n} \left[\frac{1}{(2\pi)^d}\int_{\RR^d} |( \ee^{-\imath \omega^\top y}-1)\eta(\omega) |^2 \ud \omega\right]^{1/2} q_n(y) \ud y \nonumber \\
\stackrel{(iii)}{=} {}& \int_{\|y\|\leq 2m_a \varepsilon_n} \left[\frac{1}{(2\pi)^d}\int_{\RR^d} \left|( \ee^{-\imath \omega^\top y}-1) \cdot \frac{f_{\sigma,\alpha}(\omega)-f_{\sigma_0,\alpha_0}(\omega)}{f_{\sigma_0,\alpha_0}(\omega)}\cdot \frac{f_{\sigma_0,\alpha_0}(\omega)}{\xi_1(\omega)^2} \right|^2 \ud \omega\right]^{1/2} \nonumber \\
{}&\quad \cdot q_n(y) \ud y \nonumber \\
\leq {}& \frac{1}{(2\pi)^{d/2}}\sup_{\omega \in \RR^d} \frac{f_{\sigma_0,\alpha_0}(\omega)}{\xi_1(\omega)^2} \cdot \nonumber \\
{}& \int_{\|y\|\leq 2m_a \varepsilon_n} \left[\int_{\RR^d} \left|( \ee^{-\imath \omega^\top y}-1) \cdot
\left\{\frac{f_{\sigma,\alpha}(\omega)}{f_{\sigma_0,\alpha_0}(\omega)}-1\right\} \right|^2 \ud \omega\right]^{1/2} q_n(y) \ud y \nonumber \\
\stackrel{(iv)}{\leq} {}& \frac{1}{(2\pi)^{d/2}}\sup_{\omega \in \RR^d} \frac{f_{\sigma_0,\alpha_0}(\omega)}{\xi_1(\omega)^2} \cdot \nonumber \\
{}& 2^{1-\cbeta/2}\int_{\|y\|\leq 2m_a \varepsilon_n} \left[\int_{\RR^d} \|\omega\|^{\cbeta} \left|
\left\{\frac{f_{\sigma,\alpha}(\omega)}{f_{\sigma_0,\alpha_0}(\omega)}-1\right\} \right|^2 \ud \omega\right]^{1/2} \|y\|^{\cbeta/2} q_n(y) \ud y \nonumber \\
\stackrel{(v)}{\leq} {} & \frac{2^{1-\cbeta/2}}{(2\pi)^{d/2}}\sup_{\omega \in \RR^d} \frac{f_{\sigma_0,\alpha_0}(\omega)}{\xi_1(\omega)^2} \nonumber\\
{}& \cdot \left[\int_{\RR^d} \left\{\frac{(2\nu+d)\max(\alpha_0^2,\alpha^2) \max\left(\alpha_0^{2\nu+d-2},\alpha^{2\nu+d-2}\right)}{\alpha^{2\nu+d-2}}\right\}^2 \frac{\|\omega\|^{\cbeta}}{(\alpha^2+\|\omega\|^2)^2}\ud \omega\right]^{1/2} \nonumber \\
{}& \cdot \int_{\|y\|\leq 2m_a \varepsilon_n} \|y\|^{\cbeta/2} q_n(y) \ud y \nonumber \\
\stackrel{(vi)}{\leq} {}& \frac{2^{1-\cbeta/2}\theta_0}{(2\pi)^{d/2}} \cdot  \overline c_{\xi} \sigma^2 \alpha^{2\nu} \max\left\{\left(\frac{\alpha_0}{\alpha}\right)^{2\nu+d},1\right\} \nonumber \\
&\cdot \frac{(2\nu+d)\max(\alpha_0^2,\alpha^2) \max\left(\alpha_0^{2\nu+d-2},\alpha^{2\nu+d-2}\right)}{\alpha^{2\nu+d-2}} \nonumber \\
{}& \cdot \alpha^{\cbeta/2+d/2-2} \cdot \left[\int_{0}^{\infty} \frac{r^{\cbeta+d-1}}{(1+r^2)^2} \ud r\right]^{1/2} \cdot (2m_a \varepsilon_n)^{\cbeta/2} \nonumber \\
\leq {}& \left[\int_{0}^{\infty} \frac{r^{\cbeta+d-1}}{(1+r^2)^2} \ud r\right]^{1/2} \nonumber \\
{}& \quad \cdot \frac{2\overline c_{\xi} \theta_0(2\nu+d)m_a^{\cbeta/2}\max(\alpha_0^4,\alpha^4) \max\left(\alpha_0^{2(2\nu+d-2)},\alpha^{2(2\nu+d-2)}\right)}{(2\pi)^{d/2} \alpha^{4\nu+3d/2-\cbeta/2}}\cdot \varepsilon_n^{\cbeta/2}.
\end{align}
In the derivations above: (i) follows from the Minkowski's integral inequality; (ii) follows from the Plancherel's theorem; (iii) is based on the definition of $\eta(\omega)$ in \eqref{etafunc}; (iv) uses the fact that $| \ee^{\imath a}-1|^2=4\sin^2(a/2) \leq 2^{2-\cbeta}|a|^{\cbeta}$ for any $a\in \RR$ and all $0<\cbeta< 2$; (v) follows from \eqref{ratio.ff2} in Lemma \ref{lem:specden_lambda}. (vi) follows from \eqref{proxy.bound1} in Lemma \ref{lem:fxi}. Since $\cbeta<4-d$, the integral in the last display exists and hence the conclusion follows.
\end{proof}

\vspace{3mm}

\begin{lemma}\label{lem:zetabound}
Suppose that $d\in \{1,2,3\}$ and $\nu\in \RR^+$. Let $(\sigma,\alpha)\in \RR^+ \times \RR^+$ satisfy $\sigma^2\alpha^{2\nu}=\theta_0=\sigma_0^2\alpha_0^{2\nu}$. Let $a>0$ and $0<\cbeta< \min(4-d,2)$ be fixed constants. Let $\varepsilon_n = n^{-1/(4a+2d+\cbeta)}$.  For the $\lambda_{k,n}(\alpha)$ in \eqref{wangloh.eq13}, for any $\alpha>0$, there exist positive constants $C_1^{\dagger}, C_1^{\ddagger},C_2^{\ddagger}$ that depend only on $d,\nu,T,\alpha_0,a,\cbeta$, such that
\begin{align}\label{lambda.bound1}
& \sum_{k=1}^n \left|\lambda_{k,n}(\alpha)-1\right|  \nonumber \\
&\leq C_1^{\dagger} \frac{ \max(\alpha_0^{6},\alpha^{6}) \max\left\{\alpha_0^{3(2\nu+d-2)},\alpha^{3(2\nu+d-2)}\right\} \sqrt{n}\varepsilon_n^{\cbeta/2}}{\alpha^{4\nu+3d/2-\cbeta/2}} \nonumber \\
& ~~+ C^{\ddagger}_1 \frac{[\max(\alpha_0,\alpha)]^{2\nu+d}}{\varepsilon_n^{2a+d}} + C^{\ddagger}_2 \frac{\max(\alpha_0^{6},\alpha^{6}) \max\left\{\alpha_0^{3(2\nu+d-2)},\alpha^{3(2\nu+d-2)}\right\}}{\alpha^{2(3\nu+d)}}.
\end{align}
\end{lemma}

\begin{proof}[Proof of Lemma \ref{lem:zetabound}]
For any $x,y\in \Scal$, let $b(x,y)={\EE}_{(\sigma,\alpha)}\{X(x)X(y)\} - {\EE}_{(\sigma_0,\alpha_0)}\{X(x)X(y)\}$. Then using the definition of $c_0(x)$ in \eqref{c0func} and $c_1(x)$ with the support of $c_1$ in $[-2m_{\nu},2m_{\nu}]^d$, the derivation after Equation (39) of \citet{WangLoh11} has shown that for $s,t\in \Scal$,
\begin{align}\label{b2eq}
b(x,y) &= (2\pi)^d \int_{\RR^d} \int_{\RR^d} g(s-t)c_1(x-s)c_1(y-t) \ud s \ud t \nonumber \\
&= \frac{1}{(2\pi)^d} \int_{\RR^{2d}} \ee^{\imath (\omega^\top x - v^\top y)} \eta_n^*\left(\frac{w+v}{2}\right) \vartheta\left(\frac{w-v}{2}\right) \xi_1(\omega)\xi_1(v) \ud \omega \ud v \nonumber \\
&~~~+ \frac{1}{(2\pi)^d} \int_{\RR^{2d}} \ee^{\imath(\omega^\top x - v^\top y)} \xi_1(\omega) \xi_1(v) \Bigg\{ \int_{\|u\|_{\infty}\leq 2m_{\nu}+2m_a+T} \ee^{-\imath (\omega^\top u-v^\top u)} \nonumber \\
&\quad \times \hat q_n(\omega)\eta(v) \ud u \Bigg\} \ud v \ud \omega,
\end{align}
where $\eta_n^*:\RR^d\to \CC$ is the Fourier transform of $g-q_n\ast g$ for $g$ defined in \eqref{gkfun} and $q_n$ in defined in \eqref{4func}, such that $\int_{\RR^d}\left|\eta^*_n(\omega)\right|^2\ud \omega = \int_{\RR^d} \left|q_n \ast g(x) - g(x) \right|^2 \ud x$ which can be upper bounded by Lemma \ref{lem:qngn}; $\vartheta(\omega)$ in \eqref{b2eq} is defined in the same way as Equation (23) of \citet{WangLoh11}:
\begin{align}\label{vartheta}
\vartheta(\omega) &= \frac{1}{2^d} \int_{\RR^d} e^{-\imath t^\top w} \Ical\left(\|t\|_{\infty}\leq 4m_{\nu}+2T\right) \ud t, \quad \text{ for all } \omega \in \RR^d.
\end{align}
Lemma 3 of \citet{WangLoh11} has proved that $\int_{\RR^d}\vartheta(\omega)^2\ud \omega<\infty$ and its value only depends on $d,\nu,T$.

Note that by the definition of covariance function,
\begin{align}\label{b3eq}
b(x,y) &= {\EE}_{(\sigma,\alpha)}\{X(x)X(y)\} - {\EE}_{(\sigma_0,\alpha_0)}\{X(x)X(y)\} \nonumber \\
&= \int_{\RR^d} \ee^{\imath(x-y)^\top w} \left\{f_{\sigma,\alpha}(\omega) - f_{\sigma_0,\alpha_0}(\omega)\right\}\ud \omega.
\end{align}
Hence, for any pair $(\sigma,\alpha)$ that satisfies $\sigma^2\alpha^{2\nu} =\theta_0=\sigma_0^2\alpha_0^{2\nu}$, for the $\{\psi_k:k=1\ldots,n\}$ functions in \eqref{wangloh.eq13}, we have that for $k=1,\ldots,n$,
\begin{align}\label{lambda2parts}
\lambda_{k,n}(\alpha)-1 & = \langle \psi_k,\psi_k \rangle_{f_{\sigma,\alpha}}  - \langle \psi_k,\psi_k \rangle_{f_{\sigma_0,\alpha_0}} := \zeta^{\dagger}_{k,n}+\zeta^{\ddagger}_{k,n},
\end{align}
where
\begin{align}\label{2zetafunc}
\zeta^{\dagger}_{k,n} &= \frac{1}{(2\pi)^d} \int_{\RR^{2d}} \psi_k(\omega)\overline{\psi_k(v)} \eta^*_n\left(\frac{w+v}{2}\right) \vartheta\left(\frac{w-v}{2}\right) \xi_1(\omega)\xi_1(v) \ud \omega \ud v, \nonumber \\
\zeta^{\ddagger}_{k,n} &= \frac{1}{(2\pi)^d} \int_{\RR^{2d}} \psi_k(\omega)\overline{\psi_k(v)} \xi_1(\omega)\xi_1(v) \hat q_n(\omega)\eta(v)\nonumber \\
&~~~ \times \left\{\int_{\|u\|_{\infty}\leq 2m_{\nu}+2m_a+T} \ee^{-\imath (\omega^\top u -v^\top u)}\ud u\right\}\ud \omega \ud v.
\end{align}
We follow the derivations on page 258-259 of \citet{WangLoh11}. By the Bessel's inequality, we have that
\begin{align}\label{zeta1bound1}
\sum_{k=1}^n \left|\zeta^{\dagger}_{k,n} \right|^2 &= \sum_{k=1}^n \left\{\frac{1}{(2\pi)^d} \int_{\RR^{2d}} \psi_k(\omega)\overline{\psi_k(\omega)} \eta^*_n\left(\frac{w+v}{2}\right) \vartheta\left(\frac{w-v}{2}\right) \xi_1(\omega)\xi_1(v) \ud \omega \ud v \right\}^2\nonumber \\
&\leq \frac{1}{(2\pi)^{2d}} \int_{\RR^{2d}}  \left|\eta^*_n\left(\frac{w+v}{2}\right) \vartheta\left(\frac{w-v}{2}\right) \right|^2 \frac{\xi_1(\omega)^2}{f_{\sigma,\alpha}(\omega)}\frac{\xi_1(v)^2}{f_{\sigma,\alpha}(v)} \ud \omega \ud v\nonumber \\
&\stackrel{(i)}{\leq} \frac{1}{2^d \pi^{2d}} \left\{\sup_{\omega \in \RR^d} \frac{\xi_1(\omega)^2}{f_{\sigma,\alpha}(\omega)}\right\}^2 \int_{\RR^d} \left |\vartheta(v)\right|^2\ud v \int_{\RR^d}\left|\eta^*_n(\omega)\right|^2\ud \omega \nonumber \\
&\stackrel{(ii)}{\leq} \frac{1}{2^d \pi^{2d}} \cdot \left\{\frac{\max\left\{(\alpha/\alpha_0)^{2\nu+d},1\right\}}{\underline c_{\xi} \theta_0}\right\}^2 \cdot \int_{\RR^d} \left |\vartheta(v)\right|^2\ud v \nonumber \\
&~~~ \times  C_{g,q}^2 \left[\frac{\max(\alpha_0^4,\alpha^4) \max\left\{\alpha_0^{2(2\nu+d-2)},\alpha^{2(2\nu+d-2)}\right\}}{ \alpha^{4\nu+3d/2-\cbeta/2}}\right]^2 \cdot \varepsilon_n^{\cbeta} \nonumber \\
&\leq (C_1^{\dagger})^2 \frac{\max(\alpha_0^{12},\alpha^{12}) \max\left\{\alpha_0^{6(2\nu+d-2)},\alpha^{6(2\nu+d-2)}\right\}}{\alpha^{2(4\nu+3d/2-\cbeta/2)}} \varepsilon_n^{\cbeta},
\end{align}
where (i) follows from the Cauchy-Schwarz inequality; (ii) follows from Lemma \ref{lem:fxi} and Lemma \ref{lem:qngn}, and $C_1^{\dagger}$ is a positive constant that depends only on $d,\nu,T,\alpha_0,a,\cbeta$.

For $\zeta^{\ddagger}_{k,n}$, we apply the Bessel's inequality to obtain that
\begin{align} \label{zeta2bound1}
&\quad \sum_{k=1}^n \left|\zeta^{\ddagger}_{k,n} \right|  \nonumber \\
& \leq \frac{1}{(2\pi)^{d}} \sum_{k=1}^n \int_{\|u\|_{\infty}\leq 2m_{\nu}+2m_a+T} \left|\int_{\RR^{d}} \ee^{-\imath \omega^\top u}\psi_k(\omega) \xi_1(\omega) \hat q_n(\omega) \ud \omega\right|  \nonumber \\
&~~~ \times \left|\int_{\RR^{d}} \ee^{\imath v^\top u}\overline \psi_k(v)\xi_1(v)\eta(v) \ud v\right|\ud u \nonumber \\
& \leq \frac{1}{2(2\pi)^d}    \int_{\|u\|_{\infty}\leq 2m_{\nu}+2m_a+T}  \sum_{k=1}^n \Bigg\{\left|\int_{\RR^d} \ee^{-\imath \omega^\top u} \psi_k(\omega) \frac{\xi_1(\omega)}{f_{\sigma,\alpha}(\omega)}\hat q_n(\omega)f_{\sigma,\alpha}(\omega) \ud \omega\right|^2 \nonumber \\
&~~~ +  \left|\int_{\RR^d} \ee^{-\imath v^\top u} \overline \psi_k(v) \frac{\xi_1(v)}{f_{\sigma,\alpha}(v)}\eta(v)f_{\sigma,\alpha}(v) \ud v\right|^2 \Bigg\}\ud u \nonumber \\
& \leq \frac{1}{2(2\pi)^d}    \int_{\|u\|_{\infty}\leq 2m_{\nu}+2m_a+T}  \Bigg\{\sup_{\omega \in \RR^d} \frac{\xi_1(\omega)^2}{f_{\sigma,\alpha}(\omega)}\int_{\RR^d} \left|\hat q_n(\omega)\right|^2 \ud \omega \nonumber \\
&~~~ +  \sup_{\omega \in \RR^d} \frac{f_{\sigma_0,\alpha_0}(\omega)}{\xi_1(\omega)^2} \int_{\RR^d} \left|\frac{f_{\sigma,\alpha}(v)}{f_{\sigma_0,\alpha_0}(v)} -1 \right|^2 \ud v\Bigg\} \ud u  \nonumber \\
&\stackrel{(i)}{\leq} \frac{1}{2(2\pi)^d} \cdot (4m_{\nu}+4m_a+2T)^d \cdot \left\{\frac{\max\left\{(\alpha/\alpha_0)^{2\nu+d},1\right\}}{\underline c_{\xi} \theta_0}\right\}   \nonumber \\
&~~~\times \int_{\RR^d}\frac{C_{\hat q}^2}{(1+\varepsilon_n\|\omega\|)^{2a+d}} \ud \omega \nonumber \\
&~~~ + \frac{1}{2(2\pi)^d} \cdot (4m_{\nu}+4m_a+2T)^d \cdot \overline c_{\xi} \theta_0
\max\left\{(\alpha_0/\alpha)^{2\nu+d},1 \right\}  \nonumber \\
&~~~\times \int_{\RR^d} \left\{\frac{(2\nu+d)\max(\alpha_0^2,\alpha^2) \max\left(\alpha_0^{2\nu+d-2},\alpha^{2\nu+d-2}\right)}{\alpha^{2\nu+d-2}}\right\}^2 \frac{1}{(\alpha^2+\|v\|^2)^2} \ud v \nonumber \\
&\leq \frac{(4m_{\nu}+4m_a+2T)^d}{2(2\pi)^d}\cdot \frac{C_{\hat q}^2 [\max(\alpha_0,\alpha)]^{2\nu+d} }{\underline c_{\xi}\theta_0 \alpha_0^{2\nu+d}\varepsilon_n^{2a+d}}\left\{\int_0^{\infty}\frac{r^{d-1}}{(1+r)^{2a+d}} \ud r \right\} \nonumber \\
&~~~ + \frac{(4m_{\nu}+4m_a+2T)^d \overline c_{\xi}\theta_0}{2(2\pi)^d} \cdot \frac{(2\nu+d)^2\max(\alpha_0^6,\alpha^6) \max\left\{\alpha_0^{3(2\nu+d-2)},\alpha^{3(2\nu+d-2)}\right\}}{\alpha^{3(2\nu+d)-4}} \nonumber \\
&~~~\times \alpha^{d-4} \left\{\int_0^{\infty}\frac{r^{d-1}}{(1+r^2)^2} \ud r \right\} \nonumber \\
&\leq C^{\ddagger}_1 \frac{[\max(\alpha_0,\alpha)]^{2\nu+d}}{\varepsilon_n^{2a+d}} + C^{\ddagger}_2 \frac{\max(\alpha_0^6,\alpha^6) \max\left\{\alpha_0^{3(2\nu+d-2)},\alpha^{3(2\nu+d-2)}\right\}}{\alpha^{2(3\nu+d)}},
\end{align}
where (i) follows from Lemma \ref{lem:specden_lambda}, Lemma \ref{lem:fxi}, and the inequality \eqref{qn.bound1}, and $C_1^{\ddagger},C_2^{\ddagger}$ are positive constants that depend only on $d,\nu,T,\alpha_0,a,\cbeta$.

Finally, we combine \eqref{zeta1bound1} and \eqref{zeta2bound1} to conclude that for any pair $(\sigma,\alpha)$ that satisfies $\sigma^2\alpha^{2\nu} =\theta_0=\sigma_0^2\alpha_0^{2\nu}$,
\begin{align}
& \sum_{k=1}^n \left|\lambda_{k,n}(\alpha)-1\right|  \leq \sum_{k=1}^n \left(\left|\zeta^{\dagger}_{k,n} \right|+\left|\zeta^{\ddagger}_{k,n} \right|\right) \leq \left(n\sum_{k=1}^n \left|\zeta^{\dagger}_{k,n} \right|^2\right)^{1/2} + \sum_{k=1}^n \left|\zeta^{\ddagger}_{k,n} \right| \nonumber \\
&\leq C_1^{\dagger} \frac{ \max(\alpha_0^{6},\alpha^{6}) \max\left\{\alpha_0^{3(2\nu+d-2)},\alpha^{3(2\nu+d-2)}\right\} \sqrt{n}\varepsilon_n^{\cbeta/2}}{\alpha^{4\nu+3d/2-\cbeta/2}} \nonumber \\
& + C^{\ddagger}_1 \frac{[\max(\alpha_0,\alpha)]^{2\nu+d}}{\varepsilon_n^{2a+d}} + C^{\ddagger}_2 \frac{\max(\alpha_0^{6},\alpha^{6}) \max\left\{\alpha_0^{3(2\nu+d-2)},\alpha^{3(2\nu+d-2)}\right\}}{\alpha^{2(3\nu+d)}}. \nonumber
\end{align}
\end{proof}

\vspace{3mm}

\begin{lemma}\label{lem:LauMas00}
(\citet{LauMas00} Lemma 1) Let $Z_1,\ldots,Z_n$ be i.i.d. $\mathcal{N}(0,1)$ random variables. Let $\{w_i:i=1\ldots,n\}$ be nonnegative constants. Let $\|w\|_{\infty}=\max_{1\leq i\leq n} w_i$, $\|w\|_1=\sum_{i=1}^n w_i$, and $\|w\|^2=\sum_{i=1}^n w_i^2$. Then for any positive $z>0$,
\begin{align*}
&\pr \left\{\sum_{i=1}^n w_iZ_i^2 \geq \|w\|_1 + 2\|w\|\sqrt{z} + 2\|w\|_{\infty}z \right\} \leq \ee^{-z}, \\
&\pr \left\{\sum_{i=1}^n w_iZ_i^2 \leq \|w\|_1 - 2\|w\|\sqrt{z}\right\} \leq \ee^{-z}.
\end{align*}
\end{lemma}

\vspace{3mm}

\begin{lemma}\label{lem:weight.bound}
Suppose that $d\in \{1,2,3\}$ and $\nu\in \RR^+$. For any $\alpha>0$, we define $w_i(\alpha)=\left|\lambda_{i,n}(\alpha)^{-1} - 1 \right|/\sqrt{n}$ for $i=1,\ldots,n$ and $w(\alpha)=(w_1(\alpha),\ldots,w_n(\alpha))^\top$, where $\lambda_{i,n}(\alpha)$'s are as defined in \eqref{diagonalize} and \eqref{wangloh.eq13}. Then there exists a large integer $N_5'$ that only depends on $\nu,d,T,\alpha_0$, such that for all $n>N_5'$, for $\tau,\underline\alpha_n,\overline\alpha_n$ defined in \eqref{eq:2kappa.re},
\begin{align*}
\sup_{\alpha\in [\underline\alpha_n, \overline \alpha_n]} \left\{\|w(\alpha)\|_1 + 4\|w(\alpha)\|\log n + 8\|w(\alpha)\|_{\infty}\log^2 n \right\} \leq n^{-\tau} /8 ,
\end{align*}
where $\|w(\alpha)\|_1=\sum_{i=1}^n |w_i(\alpha)|$, $\|w(\alpha)\|=\left(\sum_{i=1}^n w_i(\alpha)^2\right)^{1/2}$, and $\|w(\alpha)\|_{\infty}=\max_{1\leq i\leq n} |w_i(\alpha)|$.
\end{lemma}

\begin{proof}[Proof of Lemma \ref{lem:weight.bound}]
For abbreviation, we use $\Gamma$ to denote the right-hand side of Equation \ref{lambda.bound1} in Lemma \ref{lem:zetabound}. From Lemma \ref{lem:specden_lambda} and Lemma \ref{lem:zetabound}, we can obtain that
\begin{align}\label{w.bound1}
&\|w(\alpha)\|_1=\sum_{i=1}^n w_i(\alpha) = \frac{1}{\sqrt{n}} \sum_{i=1}^n \left|\lambda_{i,n}(\alpha)^{-1} - 1 \right|\nonumber \\
&\leq \frac{1}{\sqrt{n}\min_{1\leq i\leq n} \lambda_{i,n}(\alpha)} \sum_{i=1}^n \left|\lambda_{i,n}(\alpha) - 1 \right| \leq \frac{\{\max(\alpha_0,\alpha)\}^{2\nu+d}}{\sqrt{n}\alpha_0^{2\nu+d} } \times \Gamma,
\end{align}

\begin{align}\label{w.bound2}
&\|w(\alpha)\|^2 = \sum_{i=1}^n w_i^2 = \frac{1}{n} \sum_{i=1}^n \left|\lambda_{i,n}(\alpha)^{-1} - 1 \right|^2\nonumber \\
&\leq \frac{1}{n\left\{\min_{1\leq i\leq n} \lambda_{i,n}(\alpha)\right\}^2} \sum_{i=1}^n \left|\lambda_{i,n}(\alpha) - 1 \right|^2 \nonumber\\
&\leq \frac{1}{n\left\{\min_{1\leq i\leq n} \lambda_{i,n}(\alpha)\right\}^2} \left(\sum_{i=1}^n \left|\lambda_{i,n}(\alpha) - 1 \right| \right)^2  \leq \frac{\{\max(\alpha_0,\alpha)\}^{2(2\nu+d)}}{n\alpha_0^{2(2\nu+d)} } \times \Gamma^2.
\end{align}
We can see the upper bound in \eqref{w.bound2} is exactly the square of the upper bound in \eqref{w.bound1}.
\begin{align}\label{w.bound3}
&\|w(\alpha)\|_{\infty} = \max_{1\leq i\leq n} w_i = \frac{1}{\sqrt{n}} \max_{1\leq i\leq n} \left|\lambda_{i,n}(\alpha)^{-1} - 1 \right|\nonumber \\
&\leq \frac{1}{\sqrt{n}\min_{1\leq i\leq n} \lambda_{i,n}(\alpha)} \max_{1\leq i\leq n} \left|\lambda_{i,n}(\alpha) - 1 \right| \leq \frac{\max_{1\leq i\leq n} \lambda_{i,n}(\alpha) + 1}{\sqrt{n}\min_{1\leq i\leq n} \lambda_{i,n}(\alpha)}  \nonumber \\
&\leq \frac{\max\left\{(\alpha_0/\alpha)^{2\nu+d},1\right\}+1}{\sqrt{n} \min\left\{(\alpha_0/\alpha)^{2\nu+d},1\right\}}
\leq \frac{\left[\{\max(\alpha_0,\alpha)\}^{2\nu+d}+\alpha^{2\nu+d}\right]\{\max(\alpha_0,\alpha)\}^{2\nu+d}}
{\sqrt{n}\alpha_0^{2\nu+d} \alpha^{2\nu+d}} \nonumber \\
&\leq \frac{2\{\max(\alpha_0,\alpha)\}^{2(2\nu+d)}} {\sqrt{n}\alpha_0^{2\nu+d} \alpha^{2\nu+d}} .
\end{align}

Since $\varepsilon_n=n^{-1/(4a+2d+\cbeta)}$ in Lemma \ref{lem:qngn} and Lemma \ref{lem:zetabound}, we have $\sqrt{n}\varepsilon_n^{\cbeta/2}=1/\varepsilon_n^{2a+d}=n^{(2a+d)/(4a+2d+\cbeta)}$. Let $z=4\log^2 n$ in Lemma \ref{lem:LauMas00}. In the following, we analyze the necessary condition for $\underline \alpha_n$ and $\overline \alpha_n $ such that $\|w(\alpha)\|_1 + 4\|w(\alpha)\|\sqrt{z} + 8\|w(\alpha)\|_{\infty}z=o(1)$ for any $\alpha\in [\underline \alpha_n,\overline \alpha_n]$ as $n\to\infty$. We consider two situations according to the value of $\alpha$, each of which has two further sub-cases according to the sign of $2\nu+d-2$.

\vspace{4mm}

\noindent \textbf{(1)} When $\alpha\in [\alpha_0,\overline \alpha_n]$ and possibly $\alpha\to+\infty$ as $n\to\infty$:
\vspace{2mm}

In this case, in the upper bounds of \eqref{w.bound1} and \eqref{w.bound2}, since $\alpha\geq \alpha_0$, we have that \\ $\max\{\alpha_0^{3(2\nu+d-2)},\alpha^{3(2\nu+d-2)}\} \preceq \alpha^{3(2\nu+d-2)}$ if $2\nu+d-2\geq 0$, and that \\ $\max\{\alpha_0^{3(2\nu+d-2)},\alpha^{3(2\nu+d-2)}\} \preceq 1$ if $-1<2\nu+d-2<0$. We discuss the two sub-cases respectively:
\vspace{2mm}

\noindent \textbf{(1)-(i)} When $2\nu+d-2\geq 0$, we have $\max\{\alpha_0^{3(2\nu+d-2)},\alpha^{3(2\nu+d-2)}\} \preceq \alpha^{3(2\nu+d-2)}$.
Using \eqref{w.bound1}, \eqref{w.bound2}, and \eqref{w.bound3}, we can see that (neglecting all multiplicative constants by using the order relation $\preceq$):
\begin{align}\label{3wbound1}
& ~~~~ \|w(\alpha)\|_1 + 2\|w(\alpha)\|\sqrt{z} + 2\|w(\alpha)\|_{\infty}z \nonumber \\
&\preceq \frac{\alpha^{2\nu+d}\log n}{\sqrt{n}} \left(\alpha^{2\nu+3d/2+\cbeta/2} \sqrt{n}\varepsilon_n^{\cbeta/2} + \frac{\alpha^{2\nu+d}}{\varepsilon_n^{2a+d}} + \alpha^d \right) + \frac{\alpha^{2\nu+d}\log^2 n}{\sqrt{n}} \nonumber \\
&\preceq \frac{\overline \alpha_n^{2\nu+d}\log n}{\sqrt{n}} \cdot \overline \alpha_n ^{2\nu+3d/2+\cbeta/2} n^{(2a+d)/(4a+2d+\cbeta)} +  \frac{\overline \alpha_n ^{2\nu+d}\log^2 n}{\sqrt{n}} \nonumber \\
&= \frac{\overline \alpha_n^{4\nu + 5d/2 + \cbeta/2} \log n}{n^{\cbeta/(8a+4d+2\cbeta)}} +  \frac{\overline \alpha_n ^{2\nu+d}\log^2 n}{\sqrt{n}}.
\end{align}
In order to make the last upper bound $o(1)$, given that $\overline \alpha_n\succ 1$, we further need
\begin{align}\label{alphaover1}
& \overline \alpha_n \prec n^{\frac{\cbeta}{(4a+2d+\cbeta)(8\nu+5d+\cbeta)}} (\log n)^{-\frac{2}{8\nu+5d+\cbeta}},\quad
\overline \alpha_n \prec n^{\frac{1}{2(2\nu+d)}} (\log n)^{-\frac{2}{2\nu+d}},
\end{align}
which holds as long as
\begin{align}\label{const1}
& \overkappa  < \frac{\cbeta}{(4a+2d+\cbeta)(8\nu+5d+\cbeta)},\quad \overkappa  < \frac{1}{2(2\nu+d)}.
\end{align}
\vspace{3mm}

\noindent \textbf{(1)-(ii)} When $-1< 2\nu+d-2 < 0$, we have $\max\{\alpha_0^{3(2\nu+d-2)},\alpha^{3(2\nu+d-2)}\} \preceq 1$. Note that this special case can only happen when $d=1$ and $\nu\in (0,1/2)$. Using \eqref{w.bound1}, \eqref{w.bound2}, and \eqref{w.bound3}, we can see that:
\begin{align}\label{3wbound1.2}
& ~~~~ \|w(\alpha)\|_1 + 2\|w(\alpha)\|\sqrt{z} + 2\|w(\alpha)\|_{\infty}z \nonumber \\
&\preceq \frac{\alpha^{2\nu+d}\log n}{\sqrt{n}} \Big\{\alpha^{6-4\nu-3d/2+\cbeta/2} n^{(2a+d)/(4a+2d+\cbeta)} + \alpha^{2\nu+d}n^{(2a+d)/(4a+2d+\cbeta)} \nonumber \\
& \quad + \alpha^{6-6\nu-2d} \Big\} + \frac{\alpha^{2\nu+d}\log^2 n}{\sqrt{n}} .
\end{align}
Therefore,
\begin{align}\label{3wbound1.3}
& ~~~~ \|w(\alpha)\|_1 + 2\|w(\alpha)\|\sqrt{z} + 2\|w(\alpha)\|_{\infty}z \nonumber \\
&\preceq \frac{\overline \alpha_n^{2\nu+d}\log n}{\sqrt{n}} \Big\{\overline \alpha_n^{6-4\nu-3d/2+\cbeta/2} n^{(2a+d)/(4a+2d+\cbeta)}  + \overline \alpha_n^{6-6\nu-2d} \Big\} + \frac{\overline \alpha_n^{2\nu+d}\log^2 n}{\sqrt{n}} \nonumber \\
&\preceq \frac{\overline \alpha_n^{6-2\nu-d/2+\cbeta/2} \log n}{n^{\cbeta/(8a+4d+2\cbeta)}} + \frac{\overline \alpha_n^{6-4\nu-d}\log n}{\sqrt{n}} +  \frac{\overline \alpha_n ^{2\nu+d}\log^2 n}{\sqrt{n}}.
\end{align}
In order to make the last upper bound $o(1)$, given that $\overline \alpha_n\succ 1$ and $d=1$, we need
\begin{align}\label{alphaover1.2}
& \overline \alpha_n \prec n^{\frac{\cbeta}{(4a+2+\cbeta)(11-4\nu+\cbeta)}} (\log n)^{-\frac{2}{11-4\nu+\cbeta}},\nonumber \\
& \overline \alpha_n \prec n^{\frac{1}{2(5-4\nu)}} (\log n)^{-\frac{1}{5-4\nu}},\quad  \overline \alpha_n \prec n^{\frac{1}{2(2\nu+1)}} (\log n)^{-\frac{2}{2\nu+1}}.
\end{align}
Since $d=1$ and $\nu\in (0,1/2)$ in this case, $11-4\nu+\cbeta>0$ and $10> 2(5-4\nu) > 2(2\nu+1)$. We need that
\begin{align} \label{const2}
\overkappa & < \frac{\cbeta}{(4a+2+\cbeta)(11-4\nu+\cbeta)},\quad \overline \kappa< \frac{1}{10}.
\end{align}

\vspace{4mm}

\noindent \textbf{(2)} When $\alpha\in [\underline \alpha_n, \alpha_0]$ and possibly $\alpha\to 0+$ as $n\to\infty$:
\vspace{2mm}

In this case, in the upper bounds of \eqref{w.bound1} and \eqref{w.bound2}, since $\alpha\leq \alpha_0$, we have that \\ $\max\{\alpha_0^{3(2\nu+d-2)},\alpha^{3(2\nu+d-2)}\} \preceq 1$ if $2\nu+d-2\geq 0$, and that \\ $\max\{\alpha_0^{3(2\nu+d-2)},\alpha^{3(2\nu+d-2)}\} \preceq \alpha^{3(2\nu+d-2)}$ if $-1<2\nu+d-2<0$. We discuss the two sub-cases respectively:
\vspace{2mm}

\noindent \textbf{(2)-(i)} When $2\nu+d-2\geq 0$, we have $\max\{\alpha_0^{3(2\nu+d-2)},\alpha^{3(2\nu+d-2)}\} \preceq 1$ and $\max(\alpha_0,\alpha)\preceq 1$. Using \eqref{w.bound1}, \eqref{w.bound2}, and \eqref{w.bound3}, we can see that in this case:
\begin{align}\label{3wbound2.1}
& ~~~~ \|w(\alpha)\|_1 + 2\|w(\alpha)\|\sqrt{z} + 2\|w(\alpha)\|_{\infty}z \nonumber \\
&\preceq \frac{\log n}{\sqrt{n}} \left( \frac{\sqrt{n}\varepsilon_n^{\cbeta/2}}{\alpha^{4\nu+3d/2-\cbeta/2}} + \frac{1}{\varepsilon_n^{2a+d}} + \frac{1}{\alpha^{2(3\nu+d)}} \right) + \frac{\log^2 n}{\sqrt{n}\alpha^{2\nu+d}} \nonumber \\
&\preceq  \frac{\log n}{\underline\alpha_n^{4\nu+3d/2-\cbeta/2}n^{\frac{\cbeta}{2(4a+2d+\cbeta)}}} + \frac{\log n}{\sqrt{n}\underline\alpha_n^{2(3\nu+d)}} +   \frac{\log^2 n}{\sqrt{n}\underline\alpha_n^{2\nu+d}}.
\end{align}
In order to make the last upper bound $o(1)$, given that $\underline \alpha_n \prec 1$, we need that
\begin{align}\label{alphaunder1.1}
& \underline \alpha_n \succ n^{-\frac{\cbeta}{(4a+2d+\cbeta)(8\nu+3d-\cbeta)}} (\log n)^{\frac{2}{8\nu+3d-\cbeta}},\nonumber \\
&
\underline \alpha_n \succ n^{-\frac{1}{4(3\nu+d)}} (\log n)^{\frac{1}{2(3\nu+d)}}, \quad \underline \alpha_n \succ n^{-\frac{1}{2(2\nu+d)}} (\log n)^{\frac{2}{2\nu+d}}.
\end{align}
Since $4(3\nu+d)>2(2\nu+d)$, we only need
\begin{align}\label{const3}
\underkappa & < \frac{\cbeta}{(4a+2d+\cbeta)(8\nu+3d-\cbeta)}, \quad \underkappa < \frac{1}{4(3\nu+d)}.
\end{align}

\vspace{2mm}
\noindent \textbf{(2)-(ii)} When $-1<2\nu+d-2<0$, we have $\max\{\alpha_0^{3(2\nu+d-2)},\alpha^{3(2\nu+d-2)}\} \preceq \alpha^{3(2\nu+d-2)}$ and $\max(\alpha_0,\alpha)\preceq \alpha^{3(2\nu+d-2)}$. Note that this special case can only happen when $d=1$ and $\nu\in (0,1/2)$. Using \eqref{w.bound1}, \eqref{w.bound2}, and \eqref{w.bound3}, we can see that in this case:
\begin{align}\label{3wbound2}
& ~~~~ \|w(\alpha)\|_1 + 2\|w(\alpha)\|\sqrt{z} + 2\|w(\alpha)\|_{\infty}z \nonumber \\
&\preceq \frac{\log n}{\sqrt{n}} \left( \frac{\sqrt{n}\varepsilon_n^{\cbeta/2}}{\alpha^{6-2\nu-3d/2-\cbeta/2}} + \frac{1}{\varepsilon_n^{2a+d}} + \frac{1}{\alpha^{6-d}} \right) + \frac{\log^2 n}{\sqrt{n}\alpha^{2\nu+d}} \nonumber \\
&\preceq \frac{\log n}{\underline\alpha_n^{6-2\nu-3d/2-\cbeta/2}n^{\frac{\cbeta}{2(4a+2d+\cbeta)}}} + \frac{\log n}{\sqrt{n}\underline\alpha_n^{6-d}} +   \frac{\log^2 n}{\sqrt{n}\underline\alpha_n^{2\nu+d}}.
\end{align}
In order to make the last upper bound $o(1)$, given that $\underline \alpha_n \prec 1$ and $d=1$, we only need that
\begin{align}\label{alphaunder1}
& \underline \alpha_n \succ n^{- \frac{\cbeta}{(4a+2+\cbeta)(9-4\nu-\cbeta)}} (\log n)^{\frac{2}{9-4\nu-\cbeta}},\nonumber \\
& \underline \alpha_n \succ n^{-\frac{1}{10}} (\log n)^{\frac{1}{5}}, \quad \underline \alpha_n \succ n^{-\frac{1}{2(2\nu+1)}} (\log n)^{\frac{2}{2\nu+1}}.
\end{align}
Note that since $d=1$ and $\nu\in (0,1/2)$ in this case, $9-4\nu-\cbeta>0$ and $2(2\nu+d)<10$. Therefore, we only need
\begin{align} \label{const4}
\underkappa & < \frac{\cbeta}{(4a+2+\cbeta)(9-4\nu-\cbeta)}, \quad
\underkappa  < \frac{1}{10}.
\end{align}
\vspace{2mm}

Since all the right-hand sides of \eqref{const1}, \eqref{const2}, \eqref{const3}, and \eqref{const4} are positive, we choose $a=0.01$ and $\cbeta=0.9$ such that $a>0$ and $0<\cbeta<\min(4-d,2)$ with $d\in \{1,2,3\}$ are both satisfied. Then the choice of $\overkappa$ and $\underkappa$ in \eqref{eq:2kappa.re} satisfy \eqref{const1}, \eqref{const2}, \eqref{const3}, and \eqref{const4}. Furthermore, for $\tau$ defined in \eqref{eq:2kappa.re}, $n^{-\tau}/8$ is strictly larger in order than the maximum of the right-hand sides of \eqref{3wbound1}, \eqref{3wbound1.3}, \eqref{3wbound2.1}, and \eqref{3wbound2}.

With this $\tau$ and $z=4\log^2 n$, we have shown that uniformly for all $\alpha\in [\underline\alpha_n, \overline \alpha_n]$, there exists a large integer $N_5'$ that depends only on $\nu,d,T,\alpha_0$, such that for all $n>N_5'$,
\begin{align*}
\|w(\alpha)\|_1 + 4\|w(\alpha)\|\log n + 8\|w(\alpha)\|_{\infty}\log^2 n \leq n^{-\tau} /8 .
\end{align*}
\end{proof}

\vspace{3mm}

\begin{lemma}\label{lem:Hsuetal12}
(\citealt{Hsuetal12} Proposition 1.1) Let $Z_1,\ldots,Z_n$ be i.i.d. $\mathcal{N}(0,1)$ random variables and $Z=(Z_1,\ldots,Z_n)^\top$. Let $\Sigma$ be an $n\times n$ symmetric positive semidefinite matrix. Then for any positive $z>0$,
\begin{align*}
&\pr \left\{Z^\top \Sigma Z \geq \tr(\Sigma) + 2\sqrt{\tr(\Sigma^2)z} + 2\|\Sigma\|_{\op} z \right\} \leq e^{-z}.
\end{align*}
\end{lemma}

\vspace{5mm}

\section{Technical Lemmas for Profile Restricted Log-Likelihood} \label{supsec:proflik}

In this section, we derive some useful results for the profile restricted log-likelihood $\widetilde \Lcal_n(\alpha)$ defined in \eqref{def:prologlik} of the main text. In particular, we show Lemma \ref{lem:prolik_loosebound}, Lemma \ref{lem:profilelk.rightlower}, Lemma \ref{lem:profilelk.leftupper}, and Lemma \ref{lem:profilelk.rightupper}. These four lemmas play key roles in controlling the tail part of the posterior of $\alpha$, and will be used in the proof of Theorem \ref{thm:bvm2:joint}. Finally, Lemma \ref{lem:alpha.exist} proves the existence of the profile posterior $\widetilde \pi(\alpha|Y_n)$ as stated in Theorem \ref{thm:bvm2:joint}.
\vspace{5mm}

We recall from the main text that the profile restricted log-likelihood $\widetilde \Lcal_n(\alpha)$ defined in \eqref{def:prologlik} of the main text is
\begin{align} \label{def:prologlik2}
\widetilde \Lcal_n(\alpha) & \equiv \Lcal_n(\alpha^{-2\nu}\widetilde \theta_{\alpha}, \alpha) \nonumber \\
&= -\frac{n-p}{2}\log  \frac{Y_n^\top  \left[R_{\alpha}^{-1} - R_{\alpha}^{-1} M_n \big(M_n^\top R_{\alpha}^{-1} M_n + \Omega_{\beta}\big)^{-1} M_n^\top R_{\alpha}^{-1}  \right] Y_n }{n-p} \nonumber \\
&\quad - \frac{1}{2}\log \left|R_{\alpha}\right| - \frac{1}{2}\log \big|M_n^\top R_{\alpha}^{-1} M_n + \Omega_{\beta}\big| -\frac{n-p}{2}.
\end{align}

\begin{lemma}\label{lem:prolik_loosebound}
Suppose that $d\in \ZZ^+$ and $\nu\in \RR^+$. The profile restricted log-likelihood function defined in \eqref{def:prologlik2} satisfies that for any $0<\alpha_1<\alpha_2<\infty$, for all possible value of $Y_n\in \RR^n$,
\begin{align*}
& \left(\frac{\alpha_1}{\alpha_2}\right)^{n(\nu+d/2)} < \exp\left\{\widetilde \Lcal_n(\alpha_2) - \widetilde \Lcal_n(\alpha_1)\right\} < \left(\frac{\alpha_2}{\alpha_1}\right)^{n(\nu+d/2)}.
\end{align*}
\end{lemma}

\begin{proof}[Proof of Lemma \ref{lem:prolik_loosebound}]
From the expression \eqref{def:prologlik2}, we have that for any $0<\alpha_1<\alpha_2<\infty$,
\begin{align} \label{eq:prolik:diff}
&\quad \widetilde \Lcal_n(\alpha_2) - \widetilde \Lcal_n(\alpha_1) \nonumber \\
&= -\frac{n-p}{2}\log  \frac{Y_n^\top  \left[R_{\alpha_2}^{-1} - R_{\alpha_2}^{-1} M_n \big(M_n^\top R_{\alpha_2}^{-1} M_n + \Omega_{\beta}\big)^{-1} M_n^\top R_{\alpha_2}^{-1}  \right] Y_n }{Y_n^\top  \left[R_{\alpha_1}^{-1} - R_{\alpha_1}^{-1} M_n \big(M_n^\top R_{\alpha_1}^{-1} M_n + \Omega_{\beta}\big)^{-1} M_n^\top R_{\alpha_1}^{-1}  \right] Y_n}  \nonumber \\
&\quad  -\frac{1}{2}\log \frac{|R_{\alpha_2}|}{|R_{\alpha_1}|} - \frac{1}{2}\log \frac{\big|M_n^\top R_{\alpha_2}^{-1}M_n + \Omega_{\beta}\big|}{\big|M_n^\top R_{\alpha_1}^{-1}M_n + \Omega_{\beta}\big|} .
\end{align}
From \eqref{eq:theta1.diff.1} in the proof of Lemma \ref{lem:theta1.monotone}, we have that for any value of $Y_n\in \RR^n$,
\begin{align}\label{eq:alpha12.term1.1}
& \frac{Y_n^\top  \left[R_{\alpha_2}^{-1} - R_{\alpha_2}^{-1} M_n \big(M_n^\top R_{\alpha_2}^{-1} M_n + \Omega_{\beta}\big)^{-1} M_n^\top R_{\alpha_2}^{-1}  \right] Y_n }{Y_n^\top  \left[R_{\alpha_1}^{-1} - R_{\alpha_1}^{-1} M_n \big(M_n^\top R_{\alpha_1}^{-1} M_n + \Omega_{\beta}\big)^{-1} M_n^\top R_{\alpha_1}^{-1}  \right] Y_n} \geq \left(\frac{\alpha_1}{\alpha_2}\right)^{2\nu}.
\end{align}
Similar to the proof of \eqref{eq:theta1.diff.1}, now we notice that the second relation in Lemma \ref{lem:alpha.monotone.matrix} implies that $\alpha_1^{-d}R_{\alpha_1}^{-1} > \alpha_2^{-d}R_{\alpha_2}^{-1}$ for any $0<\alpha_1<\alpha_2<\infty$. Therefore, we apply Lemma \ref{lem:2posdef} with $A_1=\alpha_2^{-d}R_{\alpha_2}^{-1}$, $A_2=\alpha_1^{-d}R_{\alpha_1}^{-1}$, $G=M_n$, and $\Omega=\alpha_2^{-d}\Omega_{\beta}$ to obtain that
\begin{align}\label{eq:Rd.diff}
0_{n\times n} &\overset{(i)}{\leq} \left[\alpha_1^{-d} R_{\alpha_1}^{-1} - \alpha_1^{-d} R_{\alpha_1}^{-1} M_n\big(\alpha_1^{-d} M_n^\top R_{\alpha_1}^{-1}M_n + \alpha_2^{-d} \Omega_{\beta}\big)^{-1} M_n^\top \big(\alpha_1^{-d} R_{\alpha_1}^{-1}\big) \right] \nonumber \\
& \quad - \left[\alpha_2^{-d} R_{\alpha_2}^{-1} - \alpha_2^{-d} R_{\alpha_2}^{-1} M_n\big(\alpha_2^{-d} M_n^\top R_{\alpha_2}^{-1}M_n + \alpha_2^{-d} \Omega_{\beta}\big)^{-1} M_n^\top \big(\alpha_2^{-d} R_{\alpha_2}^{-1}\big) \right] \nonumber \\
&\overset{(ii)}{\leq} \left[\alpha_1^{-d} R_{\alpha_1}^{-1} - \alpha_1^{-d} R_{\alpha_1}^{-1} M_n\big(\alpha_1^{-d} M_n^\top R_{\alpha_1}^{-1}M_n + \alpha_1^{-d} \Omega_{\beta}\big)^{-1} M_n^\top \big(\alpha_1^{-d} R_{\alpha_1}^{-1}\big) \right] \nonumber \\
& \quad - \left[\alpha_2^{-d} R_{\alpha_2}^{-1} - \alpha_2^{-d} R_{\alpha_2}^{-1} M_n\big(\alpha_2^{-d} M_n^\top R_{\alpha_2}^{-1}M_n + \alpha_2^{-d} \Omega_{\beta}\big)^{-1} M_n^\top \big(\alpha_2^{-d} R_{\alpha_2}^{-1}\big) \right] \nonumber \\
&= \alpha_1^{-d}\left[ R_{\alpha_1}^{-1} - R_{\alpha_1}^{-1} M_n\big( M_n^\top R_{\alpha_1}^{-1}M_n +  \Omega_{\beta}\big)^{-1} M_n^\top  R_{\alpha_1}^{-1} \right] \nonumber \\
& \quad - \alpha_2^{-d}\left[ R_{\alpha_2}^{-1} -  R_{\alpha_2}^{-1} M_n\big( M_n^\top R_{\alpha_2}^{-1}M_n + \Omega_{\beta}\big)^{-1} M_n^\top  R_{\alpha_2}^{-1}\right] ,
\end{align}
where (i) follows from the conclusion of Lemma \ref{lem:2posdef} and (ii) follows from replacing $\alpha_2^{-d}\Omega_{\beta}$ inside the first inverse by $\alpha_1^{-d}\Omega_{\beta}$. This implies that the right-hand side of \eqref{eq:Rd.diff} is positive semidefinite. Therefore, we have that if $\alpha_1<\alpha_2$, then for any value of $Y_n\in \RR^n$,
\begin{align} \label{eq:alpha12.term1.2}
& \frac{Y_n^\top  \left[R_{\alpha_2}^{-1} - R_{\alpha_2}^{-1} M_n \big(M_n^\top R_{\alpha_2}^{-1} M_n + \Omega_{\beta}\big)^{-1} M_n^\top R_{\alpha_2}^{-1}  \right] Y_n }{Y_n^\top  \left[R_{\alpha_1}^{-1} - R_{\alpha_1}^{-1} M_n \big(M_n^\top R_{\alpha_1}^{-1} M_n + \Omega_{\beta}\big)^{-1} M_n^\top R_{\alpha_1}^{-1}  \right] Y_n} \leq \left(\frac{\alpha_2}{\alpha_1}\right)^{d}.
\end{align}
Using Lemma \ref{lem:alpha.monotone.matrix} again, we can see that $\alpha_2^{2\nu}R_{\alpha_2}^{-1} > \alpha_1^{2\nu}R_{\alpha_1}^{-1}$ and $\alpha_2^{d}R_{\alpha_2} > \alpha_1^{d}R_{\alpha_1}$ imply
\begin{align} \label{eq:alpha12.term2}
& \left(\frac{\alpha_1}{\alpha_2}\right)^{nd} \leq \frac{|R_{\alpha_2}|}{|R_{\alpha_1}|} \leq \left(\frac{\alpha_2}{\alpha_1}\right)^{2n\nu} .
\end{align}
Next we find upper and lower bounds for the last term in \eqref{eq:prolik:diff} involving $\big|M_n^\top R_{\alpha}^{-1}M_n + \Omega_{\beta}\big|$. We first notice that
\begin{align} \label{eq:det.ratio1}
& \frac{\big|M_n^\top R_{\alpha_2}^{-1}M_n + \Omega_{\beta}\big|}{\big|M_n^\top R_{\alpha_1}^{-1}M_n + \Omega_{\beta}\big|} =
\left| \big(M_n^\top R_{\alpha_1}^{-1}M_n + \Omega_{\beta}\big)^{-1} \big(M_n^\top R_{\alpha_2}^{-1}M_n + \Omega_{\beta}\big)\right| .
\end{align}
For a lower bound of this ratio, we use the result of Lemma \ref{lem:alpha.monotone.matrix} that $\alpha_2^{2\nu} R_{\alpha_2}^{-1} > \alpha_1^{2\nu} R_{\alpha_1}^{-1}$ if $\alpha_1<\alpha_2$ and derive that
\begin{align} \label{eq:det.ratio2}
&\quad ~ \frac{\big|M_n^\top R_{\alpha_2}^{-1}M_n + \Omega_{\beta}\big|}{\big|M_n^\top R_{\alpha_1}^{-1}M_n + \Omega_{\beta}\big|} =
\left| \big(M_n^\top R_{\alpha_1}^{-1}M_n + \Omega_{\beta}\big)^{-1} \big(M_n^\top R_{\alpha_2}^{-1}M_n + \Omega_{\beta}\big)\right|  \nonumber \\
&= \left|\big(M_n^\top R_{\alpha_1}^{-1}M_n + \Omega_{\beta}\big)^{-1} \left[\alpha_2^{-2\nu}\big(\alpha_2^{2\nu}M_n^\top R_{\alpha_2}^{-1}M_n - \alpha_1^{2\nu} M_n^\top R_{\alpha_1}^{-1}M_n +\alpha_1^{2\nu} M_n^\top R_{\alpha_1}^{-1}M_n \big) + \Omega_{\beta} \right] \right| \nonumber \\
&\overset{(i)}{\geq} \left|\big(M_n^\top R_{\alpha_1}^{-1}M_n + \Omega_{\beta}\big)^{-1} \left[\left(\frac{\alpha_1}{\alpha_2}\right)^{2\nu} M_n^\top R_{\alpha_1}^{-1}M_n  + \Omega_{\beta} \right] \right| \nonumber \\
&\overset{(ii)}{\geq} \left|\big(M_n^\top R_{\alpha_1}^{-1}M_n + \Omega_{\beta}\big)^{-1}\left(\frac{\alpha_1}{\alpha_2}\right)^{2\nu} \left[ M_n^\top R_{\alpha_1}^{-1}M_n  + \Omega_{\beta} \right] \right| = \left(\frac{\alpha_1}{\alpha_2}\right)^{2p\nu} ,
\end{align}
where (i) follows from that $\alpha_2^{2\nu}M_n^\top R_{\alpha_2}^{-1}M_n - \alpha_1^{2\nu} M_n^\top R_{\alpha_1}^{-1}M_n$ is positive semidefinite and that the determinant $|A+B|\geq |B|$ if both $A$ and $B$ are positive semidefinite matrices, and (ii) follows from $(\alpha_1/\alpha_2)^{2\nu}<1$ and that the matrix inside the determinant is $p\times p$.

Similarly, we have the upper bound from Lemma \ref{lem:alpha.monotone.matrix} that $\alpha_2^{-d} R_{\alpha_2}^{-1} < \alpha_1^{-d} R_{\alpha_1}^{-1}$ if $\alpha_1<\alpha_2$:
\begin{align} \label{eq:det.ratio3}
&\quad ~ \frac{\big|M_n^\top R_{\alpha_2}^{-1}M_n + \Omega_{\beta}\big|}{\big|M_n^\top R_{\alpha_1}^{-1}M_n + \Omega_{\beta}\big|} =
\left| \big(M_n^\top R_{\alpha_1}^{-1}M_n + \Omega_{\beta}\big)^{-1} \big(M_n^\top R_{\alpha_2}^{-1}M_n + \Omega_{\beta}\big)\right|  \nonumber \\
&= \left|\big(M_n^\top R_{\alpha_1}^{-1}M_n + \Omega_{\beta}\big)^{-1} \left[\alpha_2^{d}\big(\alpha_2^{-d}M_n^\top R_{\alpha_2}^{-1}M_n - \alpha_1^{-d} M_n^\top R_{\alpha_1}^{-1}M_n +\alpha_1^{-d} M_n^\top R_{\alpha_1}^{-1}M_n \big) + \Omega_{\beta} \right] \right| \nonumber \\
&\leq \left|\big(M_n^\top R_{\alpha_1}^{-1}M_n + \Omega_{\beta}\big)^{-1} \left[\left(\frac{\alpha_2}{\alpha_1}\right)^{d} M_n^\top R_{\alpha_1}^{-1}M_n  + \Omega_{\beta} \right] \right| \nonumber \\
&\leq \left|\big(M_n^\top R_{\alpha_1}^{-1}M_n + \Omega_{\beta}\big)^{-1}\left(\frac{\alpha_2}{\alpha_1}\right)^{d} \left[ M_n^\top R_{\alpha_1}^{-1}M_n  + \Omega_{\beta} \right] \right| = \left(\frac{\alpha_2}{\alpha_1}\right)^{pd}.
\end{align}
Therefore, we can combine the inequalities in \eqref{eq:alpha12.term1.1}, \eqref{eq:alpha12.term1.2}, \eqref{eq:alpha12.term2}, \eqref{eq:det.ratio2}, and \eqref{eq:det.ratio3} with \eqref{eq:prolik:diff} to conclude that for any $0<\alpha_1<\alpha_2<\infty$,
\begin{align*}
& \widetilde \Lcal_n(\alpha_2) - \widetilde \Lcal_n(\alpha_1) \nonumber \\
\geq{}& -\frac{n-p}{2} \log\left(\frac{\alpha_2}{\alpha_1}\right)^{d}  - \frac{1}{2}\log \left(\frac{\alpha_2}{\alpha_1}\right)^{2n\nu} - \frac{1}{2}\log \left(\frac{\alpha_2}{\alpha_1}\right)^{pd} = n(\nu+d/2)\log \left(\frac{\alpha_1}{\alpha_2}\right), \nonumber \\
& \widetilde \Lcal_n(\alpha_2) - \widetilde \Lcal_n(\alpha_1) \nonumber \\
\leq{}& -\frac{n-p}{2} \log\left(\frac{\alpha_1}{\alpha_2}\right)^{2\nu}  - \frac{1}{2}\log \left(\frac{\alpha_1}{\alpha_2}\right)^{nd} - \frac{1}{2}\log \left(\frac{\alpha_1}{\alpha_2}\right)^{2p\nu} = n(\nu+d/2)\log \left(\frac{\alpha_2}{\alpha_1}\right) .
\end{align*}
Exponentiating both sides leads to the conclusion.
\end{proof}
\vspace{4mm}

The following lemma is a consequence of Lemmas \ref{lem:REML.decomp}, \ref{lem:theta2.bound.2ends}, \ref{lem:theta3.bound.2ends}, \ref{lem:theta1.bound.2ends}, \ref{lem:sup.theta1} in Section \ref{supsec:lem.dimred.main}. It will be used in proving Lemma \ref{lem:profilelk.rightlower}, Lemma \ref{lem:profilelk.leftupper} and Lemma \ref{lem:profilelk.rightupper} below.
\begin{lemma}\label{lem:MRM.high.order}
For $\tau,\underline\alpha_n,\overline\alpha_n$ defined in \eqref{eq:2kappa.re} and $\widetilde\theta_{\alpha},\widetilde\theta_{\alpha}^{(1)}$ defined in \eqref{tildetheta2.1}, for $d\in \{1,2,3\}$ and $\nu\in \RR^+$, there exists a large integer $N_{6,1}'$ that only depends on $\nu,d,T,\beta_0,\theta_0,\alpha_0$ and the $\Wcal_2^{\nu+d/2}(\Scal)$ norms of $\bbm_1(\cdot),\ldots,\bbm_p(\cdot)$, such that for all $n>N_{6,1}'$,
\begin{align} \label{eq:theta01.ho1}
& \pr \left(\sup_{\alpha\in [\underline\alpha_n,\overline\alpha_n]} \frac{\big|\widetilde\theta_{\alpha}-\widetilde\theta_{\alpha}^{(1)}\big|}{\widetilde\theta_{\alpha}^{(1)}} \leq 2n^{-1/2-\tau} \right) \geq 1 - 10\exp(-4\log^2 n).
\end{align}
Furthermore, for any given $c\geq 1/(2\nu+d)$, for all $d\in \ZZ^+$ and $\nu\in \RR^+$, there exists a large integer $N_{6,2}'$ that only depends on $c,\nu,d,T,\beta_0,\theta_0,\alpha_0$ and the $\Wcal_2^{\nu+d/2}(\Scal)$ norms of $\bbm_1(\cdot),\ldots,\bbm_p(\cdot)$, such that for all $n>N_{6,2}'$,
\begin{align} \label{eq:theta01.ho2}
& \pr \left(\sup_{\alpha\in \big[(1-n^{-c})\alpha_0,(1+n^{-c})\alpha_0\big]} \frac{\big|\widetilde\theta_{\alpha}-\widetilde\theta_{\alpha}^{(1)}\big|}{\widetilde\theta_{\alpha}^{(1)}} \leq n^{-1} \log^4 n \right) \geq 1 - 8\exp(-4\log^2 n) , \nonumber \\
& \pr \left(\sup_{\alpha\in \big[(1-n^{-c})\alpha_0,(1+n^{-c})\alpha_0\big]} \big|\widetilde\theta_{\alpha}-\widetilde\theta_{\alpha_0}\big| \leq 10\theta_0n^{-(2\nu+d)c} \right) \geq 1 - 8\exp(-4\log^2 n).
\end{align}
\end{lemma}

\begin{proof}[Proof of Lemma \ref{lem:MRM.high.order}]
\noindent \underline{Proof of \eqref{eq:theta01.ho1}:}
\vspace{2mm}

We consider the case of $d\in\{1,2,3\}$. From the inequalities  \eqref{eq:theta.alpha0.2}, \eqref{eq:theta.alpha0.3}, \eqref{theta.uniformbound1}, \eqref{theta.uniformbound1.1} and \eqref{theta.alpha0}, a simple union bound shows that for all sufficiently large $n$,
\begin{align} \label{eq:3lem.ineq}
& \pr \Big(\sqrt{n}\widetilde\theta_{\alpha_0}^{(2)} \leq \frac{\theta_0}{16} n^{-\tau},~~ \sqrt{n}\widetilde\theta_{\alpha_0}^{(3)} \leq \frac{\theta_0}{16} n^{-\tau}, ~~ \sup_{\alpha\in [\underline\alpha_n,\overline\alpha_n]} \sqrt{n} \left|\widetilde \theta_{\alpha}^{(1)} - \widetilde \theta_{\alpha_0}^{(1)} \right| \leq \frac{\theta_0}{4} n^{-\tau}, \nonumber \\
& \quad ~~ \sup_{\alpha\in [\underline\alpha_n,\overline\alpha_n]} \sqrt{n} \left|\widetilde \theta_{\alpha} - \widetilde \theta_{\alpha_0} \right| \leq \frac{\theta_0}{2} n^{-\tau}, ~~ \sqrt{n} \left| \widetilde \theta_{\alpha_0} - \theta_0 \right| \leq 5 \theta_0 \log n \Big) \nonumber \\
&\geq 1 - \exp(-16\log^2 n) - \exp(-16\log^2 n) - 2\exp(-4\log^2 n) \nonumber \\
&\quad - 4\exp(-4\log^2 n) - 3\exp(-4\log^2 n ) > 1 - 10\exp(- 4\log^2 n).
\end{align}
From Lemma \ref{lem:REML.decomp}, we have $\widetilde\theta_{\alpha}=\widetilde\theta_{\alpha}^{(1)}-\widetilde\theta_{\alpha}^{(2)}+\widetilde\theta_{\alpha}^{(3)}$ , $\widetilde\theta_{\alpha}^{(1)}\geq \widetilde\theta_{\alpha}^{(2)} \geq 0$, and $\widetilde\theta_{\alpha}^{(3)}\geq 0$ for all $\alpha\in \RR^+$. Therefore, with probability at least $1- 10\exp(-4\log^2 n)$, uniformly over all $\alpha\in [\underline\alpha_n,\overline\alpha_n]$,
\begin{align}
\frac{\big|\widetilde\theta_{\alpha}- \widetilde\theta_{\alpha}^{(1)}\big|}{\widetilde\theta_{\alpha}^{(1)}} & = \frac{\left|\left(\widetilde\theta_{\alpha} - \widetilde\theta_{\alpha_0} \right) - \left(\widetilde\theta_{\alpha}^{(1)} - \widetilde\theta_{\alpha_0}^{(1)} \right)  - \widetilde\theta_{\alpha_0}^{(2)} + \widetilde\theta_{\alpha_0}^{(3)} \right|}
{\left(\widetilde\theta_{\alpha}^{(1)} - \widetilde\theta_{\alpha_0}^{(1)} \right) + \widetilde\theta_{\alpha_0}^{(2)} - \widetilde\theta_{\alpha_0}^{(3)}  + \left(\widetilde\theta_{\alpha_0} - \theta_0\right) + \theta_0 }  \nonumber \\
& \leq \frac{\left|\widetilde\theta_{\alpha} - \widetilde\theta_{\alpha_0} \right| + \left|\widetilde\theta_{\alpha}^{(1)} - \widetilde\theta_{\alpha_0}^{(1)} \right| + \widetilde\theta_{\alpha_0}^{(2)} + \widetilde\theta_{\alpha_0}^{(3)}}
{\theta_0 - \left|\widetilde\theta_{\alpha}^{(1)} - \widetilde\theta_{\alpha_0}^{(1)} \right| - \left|\widetilde\theta_{\alpha_0} - \theta_0\right| - \widetilde\theta_{\alpha_0}^{(2)} - \widetilde\theta_{\alpha_0}^{(3)}  } \nonumber  \\
&\leq \frac{(\theta_0/2)n^{-\frac{1}{2}-\tau} +(\theta_0/4)n^{-\frac{1}{2}-\tau} + (\theta_0/16)n^{-\frac{1}{2}-\tau} + (\theta_0/16)n^{-\frac{1}{2}-\tau} }
{\theta_0 - (\theta_0/4)n^{-\frac{1}{2}-\tau} - 5\theta_0 n^{-1/2}\log n - (\theta_0/16)n^{-\frac{1}{2}-\tau} - (\theta_0/16)n^{-\frac{1}{2}-\tau}} \nonumber \\
&\leq 2n^{-1/2-\tau}. \nonumber
\end{align}

\noindent \underline{Proof of \eqref{eq:theta01.ho2}:}
\vspace{2mm}

Now we consider the case of $d\in \ZZ^+$ and change the interval of supremum to $[(1-n^{-c})\alpha_0,(1+n^{-c})\alpha_0]$. According to \eqref{lambda.upper1} and \eqref{lambda.lower1} in Lemma \ref{lem:specden_lambda}, if $\alpha\in[\alpha_0,(1+n^{-c})\alpha_0]$, then for all $k=1,\ldots,n$ and all sufficiently large $n$,
\begin{align}\label{lam.11}
& 1\geq \lambda_{k,n}(\alpha) \geq \left(\frac{\alpha_0}{\alpha}\right)^{2\nu+d} \geq (1+n^{-c})^{-(2\nu+d)} \geq 1 - 2n^{-(2\nu+d)c} > \frac{1}{2}.
\end{align}
If $\alpha\in[(1-n^{-c})\alpha_0,\alpha_0]$, then for all $k=1,\ldots,n$ and all sufficiently large $n$,
\begin{align}\label{lam.12}
& 1\leq \lambda_{k,n}(\alpha) \leq \left(\frac{\alpha_0}{\alpha}\right)^{2\nu+d} \leq (1-n^{-c})^{-(2\nu+d)} \leq 1 + 2n^{-(2\nu+d)c} < 2.
\end{align}

For short, we let $\alpha_{1n}=(1-n^{-c})\alpha_0$ and $\alpha_{2n}=(1+n^{-c})\alpha_0$. Following a similar argument to the proof of Lemmas \ref{lem:theta2.bound.2ends} and \ref{lem:theta3.bound.2ends}, we can show that for all sufficiently large $n$, with probability $1-6\exp(-16\log^2 n)$,
\begin{align} \label{eq:theta23.alpha1n}
& \widetilde\theta_{\alpha_0}^{(2)} \leq (\theta_0/16) n^{-1} \log^3 n, \qquad  \widetilde\theta_{\alpha_0}^{(3)} \leq (\theta_0/16) n^{-1} \log^3 n, \nonumber \\
& \widetilde\theta_{\alpha_{1n}}^{(2)} \leq (\theta_0/16) n^{-1} \log^3 n, \quad \text{and } \widetilde\theta_{\alpha_{1n}}^{(3)} \leq (\theta_0/16) n^{-1} \log^3 n \nonumber \\
& \widetilde\theta_{\alpha_{2n}}^{(2)} \leq (\theta_0/16) n^{-1} \log^3 n, \quad \text{and } \widetilde\theta_{\alpha_{2n}}^{(3)} \leq (\theta_0/16) n^{-1} \log^3 n .
\end{align}
For $\widetilde\theta_{\alpha_{2n}}^{(1)}$, we first notice that by Lemma \ref{lem:theta1.monotone}, $\widetilde\theta_{\alpha_0}^{(1)}\leq \widetilde\theta_{\alpha}^{(1)} \leq \widetilde\theta_{\alpha_{2n}}^{(1)}$ for all $\alpha \in [\alpha_0, (1+n^{-c})\alpha_0]$. Similar to \eqref{eq:theta.diff1.2} in the proof of Lemma \ref{lem:theta1.bound.2ends}, we have that
\begin{align} \label{eq:alpha.2n.1}
\widetilde \theta_{\alpha_{2n}}^{(1)} - \widetilde \theta_{\alpha_0}^{(1)} &=  \frac{\theta_0}{n-p} \sum_{i=1}^n \left\{\lambda_{i,n}(\alpha_{2n})^{-1}-1\right\} Z_{i,n}(\alpha_{2n})^2,
\end{align}
where $Z_n(\alpha)=(Z_{1,n}(\alpha),\ldots,Z_{n,n}(\alpha))^\top = U_{\alpha}^\top X_n$ with $U_{\alpha}$ given in Lemma \ref{lem:URU}. We let $w=(w_1,\ldots,w_n)^\top$ with $w_i= \frac{\theta_0}{n-p}\left|\lambda_{i,n}(\alpha_{2n})^{-1}-1\right| $ for $i=1,\ldots,n$. Then by \eqref{lam.12}, we have
\begin{align*}
\|w\|_1 &\leq \frac{2\theta_0}{n} \frac{\sum_{i=1}^n [1-\lambda_{i,n}(\alpha_{2n})]}{\min_{1\leq i\leq n}\lambda_{i,n}(\alpha_{2n})} \leq \frac{8\theta_0 n \cdot n^{-(2\nu+d)c}}{n} = 8\theta_0 n^{- (2\nu+d)c} , \\
\|w\| &\leq \frac{2\theta_0}{n} \frac{\left\{\sum_{i=1}^n [1-\lambda_{i,n}(\alpha_{2n})]^2\right\}^{1/2}}{\min_{1\leq i\leq n}\lambda_{i,n}(\alpha_{2n})} \leq \frac{4\theta_0 \left(n \cdot 4n^{-2(2\nu+d)c}\right)^{1/2}}{n} = 8\theta_0 n^{-1/2-(2\nu+d)c} , \\
\|w\|_{\infty} &\leq \frac{2\theta_0}{n} \frac{\max_{1\leq i\leq n}[1-\lambda_{i,n}(\alpha_{2n})]}{\min_{1\leq i\leq n}\lambda_{i,n}(\alpha_{2n})} \leq 8\theta_0 n^{-1-(2\nu+d)c}.
\end{align*}
Therefore, if we apply the first inequality in Lemma \ref{lem:LauMas00} with $z=16\log^2 n$ and $w_i$'s given as above, we obtain that for all sufficiently large $n$,
\begin{align}\label{eq:alpha.2n.2}
&\quad~ \pr \left( \sup_{\alpha\in[\alpha_0, (1+n^{-c})\alpha_0]} \left(\widetilde \theta_{\alpha}^{(1)} - \widetilde \theta_{\alpha_0}^{(1)}\right) > 9\theta_0 n^{-(2\nu+d)c} \right) \nonumber \\
&=\pr \left( \widetilde \theta_{\alpha_{2n}}^{(1)} - \widetilde \theta_{\alpha_0}^{(1)} > 9\theta_0 n^{-(2\nu+d)c} \right) \nonumber \\
&\leq \pr \left( \widetilde \theta_{\alpha_{2n}}^{(1)} - \widetilde \theta_{\alpha_0}^{(1)} > \|w\|_1 + 8\|w\|\log n + 32\|w\|_{\infty} \log^2 n \right) \nonumber \\
&\leq \exp(-16\log^2 n).
\end{align}
Similarly we can show from \eqref{lam.11} that
\begin{align}\label{eq:alpha.1n.2}
&\quad~ \pr \left( \sup_{\alpha\in[(1-n^{-c})\alpha_0, \alpha_0]} \left(\widetilde \theta_{\alpha}^{(1)} - \widetilde \theta_{\alpha_0}^{(1)}\right) > 9\theta_0 n^{-(2\nu+d)c} \right) \leq \exp(-16\log^2 n).
\end{align}
\eqref{eq:alpha.2n.2} and \eqref{eq:alpha.1n.2} together imply that for all sufficiently large $n$,
\begin{align}\label{eq:alpha.12n}
&\quad~ \pr \left( \sup_{\alpha\in \big[(1-n^{-c})\alpha_0, (1+n^{-c})\alpha_0\big]} \left|\widetilde \theta_{\alpha}^{(1)} - \widetilde \theta_{\alpha_0}^{(1)}\right| > 9\theta_0 n^{-(2\nu+d)c} \right) \leq 2\exp(-16\log^2 n).
\end{align}
Finally, from Lemma \ref{lem:REML.decomp}, Lemma \ref{lem:theta.alpha0}, \eqref{eq:theta23.alpha1n}, \eqref{eq:alpha.2n.2} and \eqref{eq:alpha.1n.2}, we obtain that for all sufficiently large $n$, with probability at least $1-8\exp(-4\log^2 n)$, uniformly over all $\alpha\in [\alpha_0,(1+n^{-c})\alpha_0]$,
\begin{align*}
\frac{\big|\widetilde\theta_{\alpha}- \widetilde\theta_{\alpha}^{(1)}\big|}{\widetilde\theta_{\alpha}^{(1)}}
&\leq \frac{\left(\widetilde\theta_{\alpha_{2n}} - \widetilde\theta_{\alpha_0} \right) + \left(\widetilde\theta_{\alpha_{2n}}^{(1)} - \widetilde\theta_{\alpha_0}^{(1)} \right) + \widetilde\theta_{\alpha_0}^{(2)} + \widetilde\theta_{\alpha_0}^{(3)}}
{\theta_0 - \left(\widetilde\theta_{\alpha_{2n}}^{(1)} - \widetilde\theta_{\alpha_0}^{(1)} \right) - \left|\widetilde\theta_{\alpha_0} - \theta_0\right| - \widetilde\theta_{\alpha_0}^{(2)} - \widetilde\theta_{\alpha_0}^{(3)}  } \\
&\leq \frac{2\left(\widetilde\theta_{\alpha_{2n}}^{(1)} - \widetilde\theta_{\alpha_0}^{(1)} \right) + 2\widetilde\theta_{\alpha_0}^{(2)} + 2\widetilde\theta_{\alpha_0}^{(3)} + \widetilde\theta_{\alpha_{2n}}^{(2)} + \widetilde\theta_{\alpha_{2n}}^{(3)}}
{\theta_0 - \left(\widetilde\theta_{\alpha_{2n}}^{(1)} - \widetilde\theta_{\alpha_0}^{(1)} \right) - \left|\widetilde\theta_{\alpha_0} - \theta_0\right| - \widetilde\theta_{\alpha_0}^{(2)} - \widetilde\theta_{\alpha_0}^{(3)}  } \\
&\leq \frac{18\theta_0n^{-(2\nu+d)c} + (\theta_0/4) n^{-1}\log^3 n + (\theta_0/8)n^{-1}\log^3 n}
{\theta_0 - 9\theta_0n^{-(2\nu+d)c} - 5\theta_0n^{-1/2}\log n - (\theta_0/8)n^{-1}\log^3 n} \\
&\leq n^{-\min\{(2\nu+d)c, 1\}} \log^4 n \overset{(i)}{=} n^{-1} \log^4 n,
\end{align*}
and similarly for all $\alpha\in [(1-n^{-c})\alpha_0,\alpha_0]$, $\big|\widetilde \theta_{\alpha} - \widetilde \theta_{\alpha}^{(1)}\big|/\widetilde\theta_{\alpha}^{(1)} \leq n^{-1} \log^4 n$. The step (i) follows from our condition $c\geq 1/(2\nu+d)$. This proves the first inequality in \eqref{eq:theta01.ho2}. The second inequality in \eqref{eq:theta01.ho2} follows from combining the first inequality with \eqref{eq:alpha.12n}.
\end{proof}

\vspace{5mm}

\begin{lemma}\label{lem:profilelk.rightlower}
For $\tau,\underline\alpha_n,\overline\alpha_n$ defined in \eqref{eq:2kappa.re}, for all $d\in \ZZ^+,\nu\in \RR^+$, for any $c>1/(2\nu+d)$, there exists a large integer $N_7'$ that only depends on $c,\nu,d,T,\beta_0,\theta_0,\alpha_0$ and the $\Wcal_2^{\nu+d/2}(\Scal)$ norms of $\bbm_1(\cdot),\ldots,\bbm_p(\cdot)$, such that with probability at least $1-9\exp(-4\log^2 n)$, for all $n>N_7'$,
\begin{align}\label{diffpro1}
\inf_{\alpha\in \left[\alpha_0,(1+n^{-c})\alpha_0 \right]} \exp\left\{\widetilde \Lcal_n(\alpha) - \widetilde \Lcal_n(\alpha_0)\right\} \geq \exp \left( - 3\log^4 n \right).
\end{align}
\end{lemma}

\begin{proof}[Proof of Lemma \ref{lem:profilelk.rightlower}]
Let $\overline \lambda_n(\alpha) = \left\{\prod_{k=1}^n \lambda_{k,n}(\alpha)\right\}^{1/n}$. \eqref{lam.11} implies that for all $\alpha \in [\alpha_0,(1+n^{-c})\alpha_0]$, $\overline \lambda_n(\alpha) \leq 1$. Let $Z_n(\alpha) = U_{\alpha}^\top X_n = (Z_{1,n}(\alpha),\ldots,Z_{n,n}(\alpha))^\top \sim \Ncal(0_n,I_n)$ for any given $\alpha>0$, where $U_{\alpha}$ is given in \eqref{diagonalize} of Lemma \ref{lem:URU}. Then using \eqref{diagonalize} in Lemma \ref{lem:URU} and the definition $\widetilde \theta_{\alpha}^{(1)}$ in \eqref{tildetheta2.1} in Lemma \ref{lem:REML.decomp}, we have that
\begin{align} \label{eq:prolik.simplify}
& -\frac{n-p}{2} \log \frac{\alpha^{-2\nu} \widetilde \theta_{\alpha}^{(1)}}{ \alpha_0^{-2\nu} \widetilde \theta_{\alpha_0}^{(1)}} -\frac{1}{2}\log \frac{|R_{\alpha}|}{|R_{\alpha_0}|} \nonumber \\
={}& -\frac{n-p}{2}\log \frac{ \alpha^{-2\nu} X_n^\top U_{\alpha} \Lambda_{\alpha}^{-1} U_{\alpha}^\top X_n} {\alpha_0^{-2\nu} X_n^\top U_{\alpha} U_{\alpha}^\top X_n} - \frac{1}{2}\log \frac{\alpha^{2\nu n}\prod_{k=1}^n \lambda_{k,n}(\alpha)}{\left|U_{\alpha}\right|^2}  +  \frac{1}{2}\log \frac{\alpha_0^{2\nu n}}{\left|U_{\alpha}\right|^2} \nonumber \\
={}& -\frac{n-p}{2}\log \frac{\sum_{k=1}^n \lambda_{k,n}(\alpha)^{-1} Z_{k,n}(\alpha)^2}{\sum_{k=1}^n Z_{k,n}(\alpha)^2} - \frac{1}{2}\sum_{k=1}^n \log \lambda_{k,n}(\alpha) - p\nu \log\frac{\alpha}{\alpha_0} .
\end{align}
Denote the event on the left-hand side of the first inequality in \eqref{eq:theta01.ho2} in Lemma \ref{lem:MRM.high.order} as $\Acal_{1n}$ such that $\pr(\Acal_{1n})\geq 1- 8\exp(-4\log^2 n)$ given the condition $c> 1/(2\nu+d)$. Then from the expression \eqref{def:prologlik2} and the relation \eqref{diagonalize}, we have that on the event $\Acal_{1n}$, uniformly over all $\alpha\in [\alpha_0, (1+n^{-c})\alpha_0]$,
\begin{align} \label{eq:prolik:diff2}
&\quad~ \widetilde \Lcal_n(\alpha) - \widetilde \Lcal_n(\alpha_0) \nonumber \\
&= -\frac{n-p}{2}\log  \frac{Y_n^\top  \left[R_{\alpha}^{-1} - R_{\alpha}^{-1} M_n \big(M_n^\top R_{\alpha}^{-1} M_n + \Omega_{\beta}\big)^{-1} M_n^\top R_{\alpha}^{-1}  \right] Y_n }{Y_n^\top  \left[R_{\alpha_0}^{-1} - R_{\alpha_0}^{-1} M_n \big(M_n^\top R_{\alpha_0}^{-1} M_n + \Omega_{\beta}\big)^{-1} M_n^\top R_{\alpha_0}^{-1}  \right] Y_n}  \nonumber \\
&\quad  -\frac{1}{2}\log \frac{|R_{\alpha}|}{|R_{\alpha_0}|} - \frac{1}{2}\log \frac{\big|M_n^\top R_{\alpha_0}^{-1}M_n + \Omega_{\beta}\big|}{\big|M_n^\top R_{\alpha_0}^{-1}M_n + \Omega_{\beta}\big|} \nonumber \\
&= -\frac{n-p}{2}\log\frac{\alpha^{-2\nu} \widetilde \theta_{\alpha}}{\alpha_0^{-2\nu} \widetilde \theta_{\alpha_0}} -\frac{1}{2}\log \frac{|R_{\alpha}|}{|R_{\alpha_0}|} - \frac{1}{2}\log \frac{\big|M_n^\top R_{\alpha_0}^{-1}M_n + \Omega_{\beta}\big|}{\big|M_n^\top R_{\alpha_0}^{-1}M_n + \Omega_{\beta}\big|} \nonumber \\
&\overset{(i)}{\geq} -\frac{n-p}{2} \log \frac{\alpha^{-2\nu} \widetilde \theta_{\alpha}^{(1)}\left(1 + n^{-1} \log^4 n \right)}{ \alpha_0^{-2\nu} \widetilde \theta_{\alpha_0}^{(1)}\left(1 - n^{-1} \log^4 n \right)} -\frac{1}{2}\log \frac{|R_{\alpha}|}{|R_{\alpha_0}|} - \frac{1}{2}\log \frac{\big|M_n^\top R_{\alpha}^{-1}M_n + \Omega_{\beta}\big|}{\big|M_n^\top R_{\alpha_0}^{-1}M_n + \Omega_{\beta}\big|} \nonumber \\
&\overset{(ii)}{=} -\frac{n-p}{2}\log \frac{\sum_{k=1}^n \lambda_{k,n}(\alpha)^{-1} Z_{k,n}(\alpha)^2}{\sum_{k=1}^n Z_{k,n}(\alpha)^2} - \frac{1}{2}\sum_{k=1}^n \log \lambda_{k,n}(\alpha)  - p\nu \log\frac{\alpha}{\alpha_0} \nonumber \\
&\quad + \frac{n-p}{2}\log \frac{1 - n^{-1} \log^4 n}{1 + n^{-1} \log^4 n} - \frac{1}{2}\log \frac{\big|M_n^\top R_{\alpha}^{-1}M_n + \Omega_{\beta}\big|}{\big|M_n^\top R_{\alpha_0}^{-1}M_n + \Omega_{\beta}\big|}\nonumber \\
&\overset{(iii)}{\geq} -\frac{n-p}{2}\log \frac{\sum_{k=1}^n \lambda_{k,n}(\alpha)^{-1} Z_{k,n}(\alpha)^2}{\sum_{k=1}^n Z_{k,n}(\alpha)^2} - \frac{1}{2}\sum_{k=1}^n \log \lambda_{k,n}(\alpha) - \frac{p\nu}{2\nu+d} \log 2 \nonumber \\
&\quad - 2 \log^4 n - \frac{pd\log 2}{2(2\nu+d)} \nonumber \\
&= -\frac{n-p}{2}\log \frac{\sum_{k=1}^n \lambda_{k,n}(\alpha)^{-1} Z_{k,n}(\alpha)^2}{\sum_{k=1}^n Z_{k,n}(\alpha)^2} - \frac{1}{2}\sum_{k=1}^n \log \lambda_{k,n}(\alpha) - 2 \log^4 n - \frac{p}{2}\log 2 ,
\end{align}
where (i) follows from \eqref{eq:theta01.ho2} in Lemma \ref{lem:MRM.high.order}; (ii) follows from \eqref{eq:prolik.simplify}; to derive (iii), we first apply
\begin{align}\label{eq:log.ineq1}
& \quad~ \frac{n-p}{2}\log \frac{1 - n^{-1} \log^4 n}{1 + n^{-1} \log^4 n}  \geq \frac{n}{2} \cdot \left(-3n^{-1} \log^4 n\right) = - 2 \log^4 n ,
\end{align}
for all sufficiently large $n$, then notice that $p\nu\log(\alpha/\alpha_0) \leq \frac{p\nu}{2\nu+d} \log 2$ for all $\alpha \in [\alpha_0,(1+n^{-c})\alpha_0]\subseteq [\alpha_0,2^{1/(2\nu+d)}\alpha_0]$, and finally apply \eqref{eq:det.ratio3} to obtain that
$$ - \frac{1}{2}\log \frac{\big|M_n^\top R_{\alpha}^{-1}M_n + \Omega_{\beta}\big|}{\big|M_n^\top R_{\alpha_0}^{-1}M_n + \Omega_{\beta}\big|} \geq -\frac{1}{2}\log \left(\frac{\alpha}{\alpha_0}\right)^{pd} \geq -\frac{1}{2} \log 2^{pd/(2\nu+d)} = \frac{pd\log 2}{2(2\nu+d)} , $$
for all $\alpha \in [\alpha_0,(1+n^{-c})\alpha_0]\subseteq [\alpha_0,2^{1/(2\nu+d)}\alpha_0]$.

Now we further control the first two terms on the right-hand side of \eqref{eq:prolik:diff2}. Since $\overline\lambda_n(\alpha)\leq 1$ for all $\alpha\in [\alpha_0,(1+n^{-c})\alpha_0]$, we have that
\begin{align}\label{eq:prolik:11}
&\quad  \exp\left\{-\frac{n-p}{2}\log \frac{\sum_{k=1}^n \lambda_{k,n}(\alpha)^{-1} Z_{k,n}(\alpha)^2}{\sum_{k=1}^n Z_{k,n}(\alpha)^2} - \frac{1}{2}\sum_{k=1}^n \log \lambda_{k,n}(\alpha)\right\} \nonumber \\
&= \left[\frac{\sum_{k=1}^n \lambda_{k,n}(\alpha)^{-1} Z_{k,n}(\alpha)^2}{\sum_{k=1}^n Z_{k,n}(\alpha)^2}\cdot \left\{\prod_{k=1}^n \lambda_{k,n}(\alpha)\right\}^{1/(n-p)}\right]^{-(n-p)/2} \nonumber \\
&\geq \left[\frac{\sum_{k=1}^n \lambda_{k,n}(\alpha)^{-1} Z_{k,n}(\alpha)^2}{\sum_{k=1}^n Z_{k,n}(\alpha)^2} \right]^{-(n-p)/2} \nonumber \\
&= \left[1 + \frac{\sum_{k=1}^n \left\{\lambda_{k,n}(\alpha)^{-1}-1\right\} Z_{k,n}(\alpha)^2}{\sum_{k=1}^n Z_{k,n}(\alpha)^2} \right]^{-(n-p)/2}.
\end{align}
By \eqref{eq:alpha.2n.1} and \eqref{eq:alpha.2n.2} in the proof of Lemma \ref{lem:MRM.high.order}, we have that on the event $\Acal_{1n}$,
\begin{align}\label{eq:short1up1}
&\quad~ \sup_{\alpha \in [\alpha_0, (1+n^{-c})\alpha_0]} \sum_{k=1}^n \left\{\lambda_{k,n}(\alpha)^{-1}-1\right\} Z_{k,n}(\alpha)^2 \nonumber\\
& \leq  \sup_{\alpha \in [\alpha_0, (1+n^{-c})\alpha_0] } (n-p) \left(\widetilde \theta_{\alpha}^{(1)} - \widetilde \theta_{\alpha_0}^{(1)}\right) / \theta_0  \leq 9n^{1-(2\nu+d)c} .
\end{align}
On the other hand, for any $\alpha>0$,
\begin{align}\label{eq:Z2xi2}
& \sum_{k=1}^n Z_{k,n}(\alpha)^2 = Z_n(\alpha)^\top Z_n(\alpha) = X_n^\top U_{\alpha} U_{\alpha}^\top X_n = X_n^\top (\sigma_0^2 R_{\alpha_0})^{-1} X_n \nonumber \\
& \qquad = W_n^\top W_n = \sum_{k=1}^n W_{k,n}^2,
\end{align}
where $W_n = (W_{1,n},\ldots,W_{n,n})^\top =\sigma_0^{-1}R_{\alpha_0}^{-1/2} X_n \sim \mathcal{N}(0_n,I_n)$. Therefore, we apply the second inequality in Lemma \ref{lem:LauMas00} directly to the $\chi^2_1$ random variables of $\{W_{k,n}^2:k=1,\ldots,n\}$ with $z=4\log^2 n$ and obtain that for all sufficiently large $n$,
\begin{align} \label{eq:Y2up1}
& \pr \left(\inf_{\alpha\in [\underline\alpha_n, \overline\alpha_n]}\sum_{k=1}^n Z_{k,n}(\alpha)^2 \leq n- 4\sqrt{n} \log n \right) \nonumber \\
&= \pr \left( \sum_{k=1}^n W_{k,n}^2 \leq n- 4\sqrt{n} \log n \right) \leq \exp(-4\log^2 n).
\end{align}
We combine \eqref{eq:prolik:diff2}, \eqref{eq:prolik:11}, \eqref{eq:short1up1} and \eqref{eq:Y2up1} to obtain that with probability at least $1 - 9\exp(-4\log^2 n)$, uniformly for all $\alpha \in [\alpha_0,(1+n^{-c})\alpha_0]$ and for all sufficiently large $n$,
\begin{align}\label{eq:prolik:12}
&\quad \inf_{\alpha\in [\alpha_0, (1+n^{-c})\alpha_0]} \exp\left\{\widetilde \Lcal_n(\alpha) - \widetilde \Lcal_n(\alpha_0)\right\}  \nonumber \\
&\geq \left[1 + \frac{\sup_{\alpha\in [\alpha_0,(1+n^{-c})\alpha_0]}\sum_{k=1}^n \left\{\lambda_{k,n}(\alpha)^{-1}-1\right\} Z_{k,n}(\alpha)^2}{\inf_{\alpha\in [\alpha_0,(1+n^{-c})\alpha_0]}\sum_{k=1}^n Z_{k,n}(\alpha)^2} \right]^{-(n-p)/2} \nonumber \\
&\quad \times \exp\left\{- 2 \log^4 n - \frac{p}{2}\log 2 \right\}  \nonumber \\
&\geq \left(1 + \frac{9n^{1-(2\nu+d)c}}{n-4\sqrt{n}\log n} \right)^{-(n-p)/2} \cdot \exp\left\{- 2 \log^4 n - \frac{p}{2}\log 2 \right\}  \nonumber \\
&\geq \left( 1 + \frac{10}{n^{(2\nu+d)c}} \right)^{-(n-p)/2} \cdot  \exp\left\{- 2 \log^4 n - \frac{p}{2}\log 2 \right\}  \nonumber \\
&\overset{(i)}{\geq} \exp\left\{- 10n^{1-(2\nu+d)c} - 2 \log^4 n - \frac{p}{2}\log 2 \right\} \nonumber \\
& \geq \exp\left(-3\log^4 n \right),
\end{align}
where in (i), we apply the relation $(1+x^{-1})^x \leq \exp(1)$ for all $x>0$ and the condition $c>1/(2\nu+d)$.
\end{proof}

\vspace{5mm}

\begin{lemma}\label{lem:profilelk.leftupper}
For $\tau,\underline\alpha_n,\overline\alpha_n$ defined in \eqref{eq:2kappa.re}, for $d\in\{1,2,3\}$ and $\nu\in \RR^+$, there exists a large integer $N_8'$ that only depends on $\nu,d,T,\beta_0,\theta_0,\alpha_0$ and the $\Wcal_2^{\nu+d/2}(\Scal)$ norms of $\bbm_1(\cdot),\ldots,\bbm_p(\cdot)$, such that with probability at least $1-10\exp(-4\log^2 n)$, for all $n>N_8'$,
\begin{align}\label{diffpro2}
\sup_{\alpha\in \left[\underline \alpha_n,\alpha_0 \right]} \exp\left\{\widetilde \Lcal_n(\alpha) - \widetilde \Lcal_n(\alpha_0)\right\} < \exp\left(3n^{1/2-\tau}\right).
\end{align}
\end{lemma}

\begin{proof}[Proof of Lemma \ref{lem:profilelk.leftupper}]
According to \eqref{lambda.upper1} and \eqref{lambda.lower1} in Lemma \ref{lem:specden_lambda}, we have that for all $k=1,\ldots,n$ and all $\alpha \in [\underline \alpha_n, \alpha_0]$,
\begin{align}\label{lam.21}
& 1\leq \lambda_{k,n}(\alpha) \leq \left(\frac{\alpha_0}{\alpha}\right)^{2\nu+d} \leq \left(\frac{\alpha_0}{\underline \alpha_n}\right)^{2\nu+d}.
\end{align}
Let $\overline \lambda_n(\alpha) = \left\{\prod_{k=1}^n \lambda_{k,n}(\alpha)\right\}^{1/n}$. \eqref{lam.21} implies that $\overline \lambda_n(\alpha) \geq 1$. For any $\alpha>0$, let $Z_n(\alpha) = U_{\alpha}^\top X_n = (Z_{1,n}(\alpha),\ldots,Z_{n,n}(\alpha))^\top$ with $U_{\alpha}$ given in \eqref{diagonalize}.

Denote the event on the left-hand side of \eqref{eq:theta01.ho1} in Lemma \ref{lem:MRM.high.order} as $\Acal_{2n}$ such that $\pr(\Acal_{2n})\geq 1- 10\exp(-4\log^2 n)$. Then using the relation \eqref{eq:prolik.simplify}, we have that on the event $\Acal_{2n}$,
\begin{align}\label{eq:prolik:21}
&\quad  \exp\left\{\Lcal_n(\alpha) - \widetilde \Lcal_n(\alpha_0)\right\} \nonumber \\
&\leq \exp\Bigg\{-\frac{n-p}{2} \log \frac{\alpha^{-2\nu} \widetilde \theta_{\alpha}^{(1)}\left(1 - 2n^{-1/2-\tau} \right)}{ \alpha_0^{-2\nu} \widetilde \theta_{\alpha_0}^{(1)}\left(1 + 2n^{-1/2-\tau} \right)} -\frac{1}{2}\log \frac{|R_{\alpha}|}{|R_{\alpha_0}|} - \frac{1}{2}\log \frac{\big|M_n^\top R_{\alpha}^{-1}M_n + \Omega_{\beta}\big|}{\big|M_n^\top R_{\alpha_0}^{-1}M_n + \Omega_{\beta}\big|} \Bigg\} \nonumber \\
&\overset{(i)}{\leq} \left[\frac{\sum_{k=1}^n \lambda_{k,n}(\alpha)^{-1} Z_{k,n}(\alpha)^2}{\sum_{k=1}^n Z_{k,n}(\alpha)^2}\cdot \left\{\prod_{k=1}^n \lambda_{k,n}(\alpha)\right\}^{1/(n-p)}\right]^{-(n-p)/2} \nonumber \\
&\quad \times  \exp\left\{ 2n^{1/2-\tau} - p\nu\log \frac{\alpha}{\alpha_0} + \frac{1}{2}\log \left(\frac{\alpha_0}{\alpha}\right)^{pd} \right\}   \nonumber \\
&\overset{(ii)}{\leq} \left[\frac{\sum_{k=1}^n \lambda_{k,n}(\alpha)^{-1} Z_{k,n}(\alpha)^2}{\sum_{k=1}^n Z_{k,n}(\alpha)^2} \right]^{-(n-p)/2} \cdot \exp\left\{ 2n^{1/2-\tau} + \frac{p(2\nu+d)}{2}\log \left(\frac{\alpha_0}{\alpha}\right) \right\}   \nonumber \\
&= \left[1 + \frac{\sum_{k=1}^n \left\{\lambda_{k,n}(\alpha)^{-1}-1\right\} Z_{k,n}(\alpha)^2}{\sum_{k=1}^n Z_{k,n}(\alpha)^2} \right]^{-(n-p)/2} \cdot \exp\left\{ 2n^{1/2-\tau} + \frac{p(2\nu+d)}{2}\log \left(\frac{\alpha_0}{\alpha}\right) \right\}  ,
\end{align}
where in (i), we use the inequality
\begin{align}\label{eq:log.ineq2}
& \frac{n-p}{2}\log \frac{1 + 2n^{-1/2-\tau}}{1 - 2n^{-1/2-\tau}} \leq \frac{n}{2} \cdot \left(4n^{-1/2-\tau}\right) =  2n^{1/2-\tau} ,
\end{align}
for all sufficiently large $n$ and \eqref{eq:det.ratio3} similar to the derivation of \eqref{eq:prolik:diff2}; in (ii) we use the fact that $\overline \lambda_n(\alpha) \geq 1$.

Notice that $\lambda_{k,n}^{-1}(\alpha)-1\leq 0$ for all $k=1,\ldots,n$ for all $\alpha\in [\underline\alpha_n,\alpha_0]$. Then using the relation \eqref{eq:theta.diff1.2} in the proof of Lemma \ref{lem:theta1.bound.2ends}, on the event $\Acal_{2n}$, uniformly for all $\alpha \in [\underline\alpha_n,\alpha_0]$ and for all sufficiently large $n$,
\begin{align} \label{eq:neglamb}
&\quad~ \inf_{\alpha \in [\underline\alpha_n,\alpha_0]}\sum_{k=1}^n \left\{\lambda_{k,n}(\alpha)^{-1}-1\right\} Z_{k,n}(\alpha)^2 \nonumber\\
&= \inf_{\alpha \in [\underline\alpha_n,\alpha_0]} \frac{(n-p)\left(\widetilde\theta_{\alpha}^{(1)}-\widetilde\theta_{\alpha_0}^{(1)}\right)}{\theta_0}\geq -n^{1/2-\tau} /4.
\end{align}
We combine \eqref{eq:prolik:21}, \eqref{eq:neglamb}, and \eqref{eq:Y2up1} together to derive that uniformly for all all $\alpha \in [\underline\alpha_n,\alpha_0]$, for all sufficiently large $n$, with probability at least $1- 10\exp(-4\log^2 n)$,
\begin{align}\label{eq:prolik:22}
&\quad  \exp\left\{\widetilde \Lcal_n(\alpha) - \widetilde \Lcal_n(\alpha_0)\right\} \nonumber \\
& \leq \left[1 - \frac{n^{1/2-\tau}/4}{\inf_{\alpha \in [\underline\alpha_n, \alpha_0]} \sum_{k=1}^n Z_{k,n}(\alpha)^2} \right]^{-(n-p)/2} \cdot \exp\left\{ 2n^{1/2-\tau} + \inf_{\alpha \in [\underline\alpha_n, \alpha_0]} \frac{p(2\nu+d)}{2}\log \left(\frac{\alpha_0}{\alpha}\right) \right\} \nonumber \\
&\stackrel{(i)}{\leq} \left(1 - \frac{n^{1/2-\tau}/4}{n- 4\sqrt{n} \log n} \right)^{-(n-p)/2} \cdot \exp\left\{ 2n^{1/2-\tau} + \frac{p(2\nu+d)}{2}\left(\log \alpha_0 + \underkappa \log n\right) \right\} \nonumber \\
&\leq \left(1 - \frac{1}{2n^{1/2+\tau}} \right)^{-n/2} \cdot \exp\left\{ 2n^{1/2-\tau} + \frac{p(2\nu+d)}{2}\left(\log \alpha_0 + \underkappa \log n\right) \right\}
\nonumber \\
&= \left\{\left(1-\frac{1}{2n^{1/2+\tau}}\right)^{2n^{1/2+\tau}}\right\}^{-n^{1/2-\tau}/4} \cdot \exp\left\{ 2n^{1/2-\tau} + \frac{p(2\nu+d)}{2}\left(\log \alpha_0 + \underkappa \log n\right) \right\}
\nonumber \\
& \stackrel{(ii)}{<} \exp\left( n^{1/2-\tau}/2\right) \cdot \exp\left\{ 2n^{1/2-\tau} + \frac{p(2\nu+d)}{2}\left(\log \alpha_0 + \underkappa \log n\right) \right\} < \exp \left(3n^{1/2-\tau}\right) ,
\end{align}
where (i) follows from \eqref{eq:Y2up1}, and for (ii), we use the fact that the function $(1-x^{-1})^x$ is continuous and monotonically increasing to $1/\ee$ for $x>1$, so $(1-x^{-1})^x > 1/\ee^2$ for $x=n^{1/2+\tau}$ given that $n$ is sufficiently large.
\end{proof}

\vspace{5mm}

\begin{lemma} \label{lem:geomean}
Suppose that the sequence $\{w_i:i=1,\ldots,n\}$ satisfies $\sum_{i=1}^n w_i \geq n-c_1n^{b_1}$, $\max_{1\leq i\leq n} w_i\leq 1$ and $\min_{1\leq i\leq n} w_i\geq c_2n^{-b_2}$, where $0<b_2<b_1<1$, $c_1>0$, and $c_2>0$ are all constants. Then $\prod_{i=1}^n w_i \geq \exp\left(-4b_2c_1n^{b_1}\log n\right)$ for all $n>\max\left\{c_2^{-1/b_2},(2c_2)^{1/b_2}\right\}$.
\end{lemma}

\begin{proof}[Proof of Lemma \ref{lem:geomean}]
Given the constraints in the lemma, minimizing $\prod_{i=1}^n w_i$ is equivalent to choosing as many $w_i$'s to reach the lower bound of $c_2n^{-b_2}$ as possible. On the other hand, the constraints $\sum_{i=1}^n w_i \geq n-c_1n^{b_1}$ and $\max_{1\leq i\leq n} w_i\leq 1$ imply that the number of $w_i$'s that attain the lower bound cannot be too large. Suppose that out of $n$ terms of $w_i$'s, $w_1=\ldots=w_k=c_2n^{-b_2}$, where $k$ is an integer between $1$ and $n$. Then $k$ must satisfy the relation (since all $w_i$'s satisfy $w_i\leq 1$):
$$kc_2n^{-b_2} + (n-k)\cdot 1 \geq n - c_1n^{b_1},$$
which implies that $k\leq c_1n^{b_1}/(1-c_2n^{-b_2})$. Therefore,
\begin{align*}
\prod_{i=1}^n w_i &\geq (c_2n^{-b_2})^k \cdot 1^{n-k} \geq (c_2n^{-b_2})^{\frac{ c_1n^{b_1}}{1-c_2n^{-b_2}}}.
\end{align*}
Finally, for all $n>\max\left\{c_2^{-1/b_2},(2c_2)^{1/b_2}\right\}$, we have that $c_2>n^{-b_2}$ and $1-c_2n^{-b_2}<1/2$. Hence the conclusion follows.
\end{proof}

\vspace{5mm}

\begin{lemma}\label{lem:profilelk.rightupper}
For $\tau,\underline\alpha_n,\overline\alpha_n$ defined in \eqref{eq:2kappa.re}, for $d\in\{1,2,3\}$ and $\nu\in \RR^+$, there exist constants $\kappa_1\in (1/2-\tau,1)$, $C_{p,1}>0$, and a large integer $N_9'$ that only depend on $\nu,d,T,\beta_0,\theta_0,\alpha_0$ and the $\Wcal_2^{\nu+d/2}(\Scal)$ norms of $\bbm_1(\cdot),\ldots,\bbm_p(\cdot)$, such that with probability at least $1-10\exp(-4\log^2 n)$, for all $n>N_9'$,
\begin{align}\label{diffpro3}
\sup_{\alpha\in \left[\alpha_0, \overline\alpha_n \right]} \exp\left\{\widetilde \Lcal_n(\alpha) - \widetilde \Lcal_n(\alpha_0)\right\} < \exp\left(C_{p,1} n^{\kappa_1} \log n\right).
\end{align}
\end{lemma}

\begin{proof}[Proof of Lemma \ref{lem:profilelk.rightupper}]
According to \eqref{lambda.upper1} and \eqref{lambda.lower1} in Lemma \ref{lem:specden_lambda}, we have that for all $k=1,\ldots,n$ and all $\alpha \in \left[\alpha_0, \overline\alpha_n \right]$,
\begin{align}\label{lam.31}
& 1\geq \lambda_{k,n}(\alpha) \geq \left(\frac{\alpha_0}{\alpha}\right)^{2\nu+d} \geq \left(\frac{\alpha_0}{\overline \alpha_n}\right)^{2\nu+d} = \frac{\alpha_0^{2\nu+d}}{n^{(2\nu+d)\overkappa}}.
\end{align}
Let $\overline \lambda_n(\alpha) = \left\{\prod_{k=1}^n \lambda_{k,n}(\alpha)\right\}^{1/n}$.

If $2\nu+d-2\geq 0$, then by \eqref{lambda.bound1} of Lemma \ref{lem:zetabound}, for all $\alpha \in [\alpha_0,\overline\alpha_n]$, and for all sufficiently large $n$,
\begin{align}\label{lam.32.1}
& \sum_{k=1}^n \left\{1-\lambda_{k,n}(\alpha)\right\}  \nonumber \\
\preceq {}& n^{(2\nu+3d/2+\cbeta/2)\overkappa} \cdot n^{(2a+d)/(4a+2d+\cbeta)}
+ n^{(2\nu+d)\overkappa} \cdot n^{(2a+d)/(4a+2d+\cbeta)} + n^{d\overkappa}  .
\end{align}
Given the definition of $\overkappa$ in \eqref{eq:2kappa.re} and $d\geq 1$, with the choice $a=0.01$ and $\cbeta=0.9$,
\begin{align*}
& (2\nu+3d/2+\cbeta/2)\overkappa + \frac{2a+d}{4a+2d+\cbeta} < 1 ,\\
& (2\nu+d)\overkappa + \frac{2a+d}{4a+2d+\cbeta} <1, \quad d\overkappa <1.
\end{align*}
Therefore, \eqref{lam.32.1} implies that there exist constants $\kappa_1\in (0,1)$ ($\kappa_1$ can be chosen close to 1) and $C_1>0$, such that $\sum_{k=1}^n \left\{1-\lambda_{k,n}(\alpha)\right\} < C_1 n^{\kappa_1}$.

If $-1< 2\nu+d-2 <0$ ($d=1$ and $\nu\in (0,1/2)$), then for all $\alpha \in [\alpha_0,\overline\alpha_n]$, and for all sufficiently large $n$, \eqref{lambda.bound1} of Lemma \ref{lem:zetabound} implies that
\begin{align}\label{lam.32.2}
& \sum_{k=1}^n \left\{1-\lambda_{k,n}(\alpha)\right\}  \nonumber \\
\preceq {}&  n^{(6-4\nu-3d/2+\cbeta/2)\overkappa} \cdot n^{(2a+d)/(4a+2d+\cbeta)}
+  n^{(2\nu+d)\overkappa} \cdot n^{(2a+d)/(4a+2d+\cbeta)} +  n^{d\overkappa}.
\end{align}
Again given $\overkappa$ in \eqref{eq:2kappa.re} and the choice $a=0.01$, $\cbeta=0.9$, we have that
\begin{align*}
& (6-4\nu-3d/2+\cbeta/2)\overkappa + \frac{2a+d}{4a+2d+\cbeta} < 1 ,\\
& (2\nu+d)\overkappa + \frac{2a+d}{4a+2d+\cbeta} <1, \quad d\overkappa <1.
\end{align*}
Therefore, \eqref{lam.32.2} also implies that there exist constants $\kappa_1\in (0,1)$  ($\kappa_1$ can be chosen close to 1) and $C_1>0$, such that $\sum_{k=1}^n \left\{1-\lambda_{k,n}(\alpha)\right\} < C_1 n^{\kappa_1}$. Combining \eqref{lam.32.1} and \eqref{lam.32.2}, we have that for all sufficiently large $n$,
\begin{align} \label{lam.32}
& \sum_{k=1}^n \left\{ 1- \lambda_{k,n}(\alpha) \right\} \leq C_1 n^{\kappa_1},\quad\text{or } \sum_{k=1}^n \lambda_{k,n}(\alpha) \geq n-C_1 n^{\kappa_1}.
\end{align}
Now in Lemma \ref{lem:geomean}, we set $w_i=\lambda_{i,n}$, $c_1=C_1$, $b_1=\kappa_1$, $c_2=\alpha_0^{2\nu+d}$, $b_2=(2\nu+d)\overkappa$, and use \eqref{lam.31} and \eqref{lam.32} to obtain that for all sufficiently large $n$,
\begin{align}\label{lam.33}
\inf_{\alpha \in [\alpha_0,\overline\alpha_n]} \overline \lambda_n(\alpha) & = \left(\inf_{\alpha \in [\alpha_0,\overline\alpha_n]} \prod_{k=1}^n \lambda_{k,n}(\alpha)\right)^{1/n} \geq \exp\left\{-4C_1(2\nu+d)\overkappa n^{\kappa_1-1} \log n\right\}.
\end{align}
On the other hand, \eqref{lam.31} implies that
\begin{align}\label{eq:lambda.Y.positive}
& \sum_{k=1}^n \left\{\lambda_{k,n}(\alpha)^{-1}-1\right\} Y_{k,n}(\alpha)^2\geq 0.
\end{align}
Therefore, on the event $\Acal_{2n}$ (the event on the left-hand side of \eqref{eq:theta01.ho1} in Lemma \ref{lem:MRM.high.order}, where for any $\alpha \in [\underline\alpha_n,\overline\alpha_n]$, $\big|\widetilde\theta_{\alpha}-\widetilde\theta_{\alpha}^{(1)}\big|/\widetilde\theta_{\alpha}^{(1)} \leq 2n^{-1/2-\tau}$), we have that for all $\alpha\in [\alpha_0, \overline\alpha_n]$, for all sufficiently large $n$,
\begin{align}\label{eq:prolik:31}
&\quad \exp\left\{\widetilde \Lcal_n(\alpha) - \widetilde \Lcal_n(\alpha_0)\right\} \nonumber \\
&\leq \exp\Bigg\{-\frac{n-p}{2} \log \frac{\alpha^{-2\nu} \widetilde \theta_{\alpha}^{(1)}\left(1 - 2n^{-1/2-\tau} \right)}{ \alpha_0^{-2\nu} \widetilde \theta_{\alpha_0}^{(1)}\left(1 + 2n^{-1/2-\tau} \right)} -\frac{1}{2}\log \frac{|R_{\alpha}|}{|R_{\alpha_0}|} - \frac{1}{2}\log \frac{\big|M_n^\top R_{\alpha}^{-1}M_n + \Omega_{\beta}\big|}{\big|M_n^\top R_{\alpha_0}^{-1}M_n + \Omega_{\beta}\big|} \Bigg\} \nonumber \\
&\stackrel{(i)}{\leq} \overline \lambda_n (\alpha)^{-(n-p)/2} \left[1 +  \frac{\sum_{k=1}^n \left\{\lambda_{k,n}(\alpha)^{-1}-1\right\} Z_{k,n}(\alpha)^2}{\sum_{k=1}^n Z_{k,n}(\alpha)^2} \right]^{-(n-p)/2} \nonumber \\
&\quad \times \exp\left\{2n^{1/2-\tau} - p\nu \log \frac{\alpha}{\alpha_0} - \frac{1}{2}\log \frac{\big|M_n^\top R_{\alpha}^{-1}M_n + \Omega_{\beta}\big|}{\big|M_n^\top R_{\alpha_0}^{-1}M_n + \Omega_{\beta}\big|} \right\} \nonumber \\
&\stackrel{(ii)}{\leq} \overline \lambda_n (\alpha)^{-(n-p)/2}\cdot 1^{-(n-p)/2} \cdot \exp\left\{ 2n^{1/2-\tau} - p\nu \log \frac{\alpha}{\alpha_0} - \frac{1}{2} \log \left(\frac{\alpha_0}{\alpha}\right)^{2p\nu } \right\} \nonumber \\
&\overset{(iii)}{\leq} \exp\left\{ 2C_1(2\nu+d)\overkappa n^{\kappa_1} \log n\right\} \cdot \exp\left( 2n^{1/2-\tau}\right)  \nonumber \\
&\overset{(iv)}{\leq} \exp\left\{ 3C_1(2\nu+d)\overkappa n^{\kappa_1} \log n\right\},
\end{align}
where (i) follows from  \eqref{eq:prolik.simplify} and \eqref{eq:log.ineq2}; (ii) follows from \eqref{eq:det.ratio2} and \eqref{eq:lambda.Y.positive}; (iii) follows from \eqref{lam.33}; (iv) follows since we can choose $\kappa_1\in (1/2-\tau,1)$. The conclusion follows by taking $C_{p,1}=3C_1(2\nu+d)\overkappa$.
\end{proof}

\vspace{5mm}

\begin{lemma}\label{lem:alpha.exist}
Suppose that Assumptions \ref{assump.m.func}, \ref{prior.1} and \ref{prior.3} hold. Then for all $d\in \ZZ^+$ and $\nu\in \RR^+$, the profile posterior distribution of $\alpha$ given by $\widetilde \pi(\alpha|Y_n)$ in \eqref{profile:post1} is a proper posterior almost surely $P_{(\beta_0,\sigma_0^2,\alpha_0)}$ for any given $n\geq p$.
\end{lemma}

\begin{proof}[Proof of Lemma \ref{lem:alpha.exist}]
We consider a fixed $n\geq p$. Since the Mat\'ern covariance function is continuous in $\alpha\in \RR^+$, $R_{\alpha}$ is also continuous in $\alpha\in \RR^+$, and so is the profile restricted likelihood $\exp\{\widetilde\Lcal_n(\alpha)\}$. Furthermore, both $\pi(\theta_0|\alpha)$ and $\pi(\alpha)$ are continuous functions in $\alpha\in \RR^+$ by Assumptions \ref{prior.1} and \ref{prior.3}. As a result, the profile posterior in \eqref{profile:post1} is well defined as long as the function $\exp\{\widetilde\Lcal_n(\alpha)\} \pi(\theta_0|\alpha) \pi(\alpha)$ is integrable as $\alpha\to +\infty$ and $\alpha \to 0+$.

As $\alpha\to +\infty$, $R_{\alpha} \to I_n$ elementwise. Since $M_n$ is rank-$p$ for all $n\geq p$ by Assumption \ref{assump.m.func}, $M_n^\top M_n$ is invertible for each fixed $n$ and $\Scal_n$. Therefore, as $\alpha\to +\infty$, the profile restricted likelihood $\exp\{\widetilde\Lcal_n(\alpha)\}$ becomes proportional to
\begin{align*}
&\exp\left\{-\frac{n-p}{2}\log \frac{Y_n^\top\big[I_n-M_n(M_n^\top M_n+\Omega_{\beta})^{-1} M_n^\top\big] Y_n}{n} - \frac{1}{2}\log \big|M_n^\top M_n + \Omega_{\beta} \big|\right\} \\
&= \left(\frac{Y_n^\top \big[I_n-M_n(M_n^\top M_n+ \Omega_{\beta})^{-1} M_n^\top\big] Y_n}{n}\right)^{-(n-p)/2}\cdot \big|M_n^\top M_n + \Omega_{\beta} \big|^{-1/2},
\end{align*}
which is a finite positive number almost surely $P_{(\beta_0,\sigma_0^2,\alpha_0)}$ for any given $n\geq p$. Since Assumption \ref{prior.3} says that $\int_0^{\infty} \pi(\theta_0|\alpha) \pi(\alpha) \ud \alpha <\infty$, and $\exp\{\widetilde\Lcal_n(\alpha)\}$ is a continuous function in $\alpha$, it follows that the integral of $ \exp\{\widetilde\Lcal_n(\alpha)\} \pi(\theta_0|\alpha) \pi(\alpha)$ on $\alpha\in [1,+\infty)$ is finite.

Then we consider the case when $\alpha \to 0+$. The property of the Mat\'ern covariance function as $\alpha \to 0+$ has been analyzed in detail in \citet{Beretal01} and \citet{Guetal18}. Lemma 3.3 of \citet{Guetal18} has shown that for given $n$, $M_n$ and $Y_n$, the profile restricted likelihood function converges to zero as $\alpha \to 0+$ with the following rates:
\begin{align*}
& \exp\{\widetilde\Lcal_n(\alpha)\} \leq \left\{
\begin{array}{ll}
C(n,M_n,Y_n)\alpha^{\nu}, & \text{if } \nu\in (0,1),\\
C(n,M_n,Y_n)\alpha \left\{\log(1/\alpha)\right\}^{1/2}, & \text{if } \nu=1,\\
C(n,M_n,Y_n)\alpha,  & \text{if } \nu>1,
\end{array}
\right.
\end{align*}
where $C(n,M_n,Y_n)$ is a finite positive number that depends on $d$, $n$, $M_n$ and $Y_n$ but not $\alpha$. In all three cases, $\exp\{\widetilde\Lcal_n(\alpha)\} \to 0$ as $\alpha \to 0+$. Together with $\int_0^{\infty} \pi(\theta_0|\alpha) \pi(\alpha) \ud \alpha <\infty$ from Assumption \ref{prior.3}, we conclude that the integral of $ \exp\{\widetilde\Lcal_n(\alpha)\} \pi(\theta_0|\alpha) \pi(\alpha)$ on $\alpha\in (0,1)$ is also finite. Therefore, $\int_0^{\infty} \exp\{\widetilde\Lcal_n(\alpha)\} \pi(\theta_0|\alpha) \pi(\alpha) \ud \alpha <\infty$, and the profile posterior defined in \eqref{profile:post1} is a proper posterior almost surely $P_{(\beta_0,\sigma_0^2,\alpha_0)}$ for any given $n\geq p$.
\end{proof}

\vspace{5mm}

\section{Proof of Theorems \ref{thm:bvm1:theta} and \ref{thm:bvm2:joint}} \label{sec:auxiliary}
In this section, we provide the proof of Theorems \ref{thm:bvm1:theta} and \ref{thm:bvm2:joint} in the main text. We first prove a useful Lemma \ref{lemma:gndiff1} that establishes the local asymptotic normality (LAN) condition for the microergodic parameter $\theta$ for a given $\alpha$. This lemma is essential for showing the limiting normal posterior for $\theta$. In Section \ref{subsec:d5}, we present the theory on the limiting posterior distribution of $(\theta,\alpha)$ for the case of $d\geq 5$.

\subsection{Proof of Lemma \ref{lemma:gndiff1}} \label{supsec:gndiff}

For a given $\alpha>0$, let $t=\sqrt{n-p}(\theta-\widetilde \theta_{\alpha})$ be the local parameter. We define the following function:
\begin{align}\label{func:gn}
\varrho_n(t;\alpha) &= \exp\left\{\Lcal_n(\alpha^{-2\nu}(\widetilde \theta_{\alpha} +\tfrac{t}{\sqrt{n-p}}),\alpha)-\Lcal_n(\alpha^{-2\nu}\widetilde \theta_{\alpha},\alpha)\right\}\cdot \frac{\pi\left(\widetilde \theta_{\alpha} +\frac{t}{\sqrt{n-p}} ~\Big |~ \alpha\right)}{\pi(\theta_0|\alpha)} \nonumber \\
& \quad - \exp\left(-\frac{t^2}{4\theta_0^2}\right).
\end{align}
\begin{lemma}\label{lemma:gndiff1}
Suppose that Assumption \ref{assump.m.func} and \ref{prior.1} hold. Then for all $d\in \ZZ^+, \nu\in\RR^+$, for any fixed $\alpha>0$, for any positive sequences $\epsilon_{1n} \to 0$ as $n\to\infty$ and $1\preceq s_n \prec \min\left(n^{1/6},\epsilon_{1n}^{-1/2}\right)$ that do not depend on $\alpha$, for all sufficiently large $n$, the $\varrho_n$ function in \eqref{func:gn} satisfies the following upper bound on the event $\Ecal_1(\epsilon_{1n},\alpha)=\{|\widetilde\theta_{\alpha}-\theta_0|<\epsilon_{1n}\}$:
\begin{align}\label{gnt1}
\int_{\RR} |\varrho_n(t;\alpha)| \ud t &\leq B_n(\alpha),
\end{align}
where
\begin{align}\label{en:rho}
B_n(\alpha) \equiv &~ 4\theta_0 \exp\left(-\frac{n-p}{64}\right) + \frac{\sqrt{n-p}}{\pi(\theta_0|\alpha)} \exp\{-0.007(n-p)\} \nonumber \\
&+ 10\theta_0 \exp\left(-\frac{4s_n^2}{125\theta_0^2}\right)\cdot  \sup_{\theta\in\left(\frac{1}{2}\theta_0,\frac{3}{2}\theta_0\right)} \frac{\pi(\theta|\alpha)}{\pi(\theta_0|\alpha)} + 4\theta_0 \exp\left(-\frac{s_n^2}{4\theta_0^2}\right) \nonumber \\
&+ \frac{8}{\theta_0^2} \left(s_n^2 \epsilon_{1n} + \frac{2s_n^3}{\sqrt{n-p}}\right)\cdot \sup_{\theta\in\left(\frac{3}{4}\theta_0,\frac{3}{2}\theta_0\right)} \frac{\pi(\theta|\alpha)}{\pi(\theta_0|\alpha)} \nonumber \\
&+ 4\theta_0 \sup_{\theta\in \left(\frac{3}{4}\theta_0,\frac{3}{2}\theta_0\right)} \left|\frac{\partial \log \pi(\theta|\alpha)}{\partial \theta} \right| \sup_{\theta\in\left(\frac{3}{4}\theta_0,\frac{3}{2}\theta_0\right)} \frac{\pi(\theta|\alpha)}{\pi(\theta_0|\alpha)}\cdot \left(\epsilon_{1n}+\frac{s_n}{\sqrt{n-p}}\right).
\end{align}
\end{lemma}

\begin{proof}[Proof of Lemma \ref{lemma:gndiff1}]
we first take the difference of the log-likelihood in \eqref{eq:loglik2} and the profile restricted log-likelihood in \eqref{def:prologlik} of the main text, and use the definition of $\widetilde\theta_{\alpha}$ in \eqref{tildetheta1} of the main text to obtain that
\begin{align}
\Lcal_n(\alpha^{-2\nu}\theta,\alpha)-\Lcal_n(\alpha^{-2\nu}\widetilde \theta_{\alpha},\alpha)&=
-\frac{n-p}{2}\log \frac{\theta}{\widetilde \theta_{\alpha}} + \frac{(n-p)(\theta-\widetilde \theta_{\alpha})}{2\theta} \label{llh.diff1} \\
&= -\frac{n-p}{2}\log \left(1+\frac{t}{\sqrt{n-p} \cdot \widetilde \theta_{\alpha}}\right) + \frac{\sqrt{n-p}\cdot t}{2\left(\widetilde \theta_{\alpha}+\frac{t}{\sqrt{n-p}}\right)} \label{llh.diff2}
\end{align}

We decompose the integral in \eqref{gnt1} into three parts:
\begin{align}\label{gnt2}
\int_{\RR} |\varrho_n(t;\alpha)| \ud t = \int_{A_1} |\varrho_n(t;\alpha)| \ud t + \int_{A_2} |\varrho_n(t;\alpha)| \ud t  +\int_{A_3} |\varrho_n(t;\alpha)| \ud t,
\end{align}
where $A_1=\{t\in \RR: ~ |t|\geq (\theta_0/4)\sqrt{n-p}\}$, $A_2=\{t\in \RR:~ s_n\leq |t|< (\theta_0/4)\sqrt{n-p}\}$, and $A_3=\{t\in \RR:~|t|< s_n\}$, with the sequence $s_n$ as specified in the lemma.
\vspace{2mm}

\noindent \underline{Bound the first term in \eqref{gnt2}:} We have
\begin{align}\label{ga1.1}
\int_{A_1} |\varrho_n(t;\alpha)| \ud t &\leq \int_{A_1} \exp\left\{\Lcal_n(\alpha^{-2\nu}\theta,\alpha)-\Lcal_n(\alpha^{-2\nu}\widetilde \theta_{\alpha},\alpha)\right\} \frac{\pi\left(\widetilde \theta_{\alpha} +\frac{t}{\sqrt{n-p}} ~\Big |~ \alpha\right)}{\pi(\theta_0|\alpha)} \ud t \nonumber \\
&~~~~ + \int_{A_1}  \ee^{-\frac{t^2}{4\theta_0^2}}  \ud t .
\end{align}
The second term in \eqref{ga1.1} can be bounded by
\begin{align}\label{ga1.bound2}
\int_{A_1}  \ee^{-\frac{t^2}{4\theta_0^2}} \ud t & \leq 2\sqrt{\pi}\theta_0\cdot \int_{|t|\geq (\theta_0/4)\sqrt{n-p}}  \frac{1}{\sqrt{2\pi\cdot 2\theta_0^2}} \ee^{-\frac{t^2}{4\theta_0^2}} \ud t \nonumber \\
&\leq 2\sqrt{\pi}\theta_0 \exp\left\{-\frac{(n-p) (\theta_0/4)^2}{4\theta_0^2}\right\} = 2\sqrt{\pi}\theta_0 \exp\left(-\frac{n-p}{64}\right),
\end{align}
where the last inequality follows from the tail bounds for a normal random variable: if $Z\sim \Ncal(0,1)$, then for any $z>0$,
\begin{align}\label{normal.tail}
\pr (|Z|>z) \leq \ee^{-z^2/2}.
\end{align}
For the first term in \eqref{ga1.1}, we note that $\theta$ is a linear transformation of $t$. We use the relation \eqref{llh.diff1} and obtain that
\begin{align}\label{ga1.bound1.1}
& \int_{A_1} \exp\left\{\Lcal_n(\alpha^{-2\nu}\theta,\alpha)-\Lcal_n(\alpha^{-2\nu}\widetilde \theta_{\alpha},\alpha)\right\} \frac{\pi\left(\widetilde \theta_{\alpha} +\frac{t}{\sqrt{n-p}} ~\Big |~ \alpha\right)}{\pi(\theta_0|\alpha)} \ud t \nonumber \\
={}& \int_{|t|\geq (\theta_0/4)\sqrt{n-p}} \exp\left\{-\frac{n-p}{2}\log \frac{\theta}{\widetilde \theta_{\alpha}} + \frac{(n-p)(\theta-\widetilde \theta_{\alpha})}{2\theta}\right\} \frac{\pi\left(\widetilde \theta_{\alpha} +\frac{t}{\sqrt{n-p}} ~\Big |~ \alpha\right)}{\pi(\theta_0|\alpha)} \ud t \nonumber \\
\leq{}& \sqrt{n-p}\int_{|\theta-\widetilde \theta_{\alpha}|\geq \theta_0/4} \frac{\pi\left(\widetilde \theta_{\alpha} +\frac{t}{\sqrt{n-p}} ~\Big |~ \alpha\right)}{\pi(\theta_0|\alpha)} \cdot\exp\left\{-\frac{n-p}{2}\varphi\left(\frac{\widetilde \theta_{\alpha}}{\theta}\right)\right\} \ud \theta.
\end{align}
For any constant $\epsilon>0$, define the event $\Ecal_1(\epsilon,\alpha)=\{|\widetilde \theta_{\alpha}-\theta_0|<\epsilon\}$. Let $0<\epsilon_{1n}<\theta_0/4$, where $\epsilon_{1n}\to 0$ as $n\to\infty$ and its order will be determined later. Then, on the event $\Ecal_1(\epsilon_{1n},\alpha)$ and $\{|\theta-\widetilde \theta_{\alpha}|\geq \theta_0/4\}$, we consider two cases: If $\theta>\widetilde\theta_{\alpha}+\theta_0/4$, then
\begin{align*}
1-\frac{\widetilde \theta_{\alpha}}{\theta} &= 1-\frac{\widetilde \theta_{\alpha}}{\theta-\widetilde \theta_{\alpha}+\widetilde \theta_{\alpha}}\geq 1- \frac{\widetilde \theta_{\alpha}}{\theta_0/4+\widetilde \theta_{\alpha}} = \frac{\theta_0/4}{\theta_0/4+\widetilde \theta_{\alpha}} > \frac{\theta_0/4}{\theta_0/4+\theta_0+\epsilon_{1n}} > \frac{1}{6}.
\end{align*}
If $\theta<\widetilde\theta_{\alpha}-\theta_0/4$, then
\begin{align*}
\frac{\widetilde \theta_{\alpha}}{\theta}-1 &= \frac{\widetilde \theta_{\alpha}}{\theta-\widetilde \theta_{\alpha}+\widetilde \theta_{\alpha}} -1 \geq \frac{\widetilde \theta_{\alpha}}{-\theta_0/4+\widetilde \theta_{\alpha}}-1 = \frac{\theta_0/4}{\widetilde \theta_{\alpha}-\theta_0/4} > \frac{\theta_0/4}{\theta_0+\epsilon_{1n}-\theta_0/4} > \frac{1}{4}.
\end{align*}
This implies that on the event $\Ecal_1(\epsilon_{1n},\alpha)$ and $\{|\theta-\widetilde \theta_{\alpha}|\geq \theta_0/4\}$, we must have either $\widetilde \theta_{\alpha}/\theta<\frac{5}{6}$ or $\widetilde \theta_{\alpha}/\theta>\frac{5}{4}$.
Since the function $\varphi(u)=u-\log u-1$ is monotonically decreasing on $(0,1)$ and monotonically increasing on $[1,+\infty)$, we have that on the event $\Ecal_1(\epsilon_{1n},\alpha)$ and $\{|\theta-\widetilde \theta_{\alpha}|\geq \theta_0/4\}$, either $\varphi(\widetilde\theta_{\alpha}/\theta)>\min\{\varphi(5/6),\varphi(5/4)\} >0.015$. Therefore, from \eqref{ga1.bound1.1}, we obtain that on the event $\Ecal_1(\epsilon_{1n},\alpha)$,
\begin{align}\label{ga1.bound1.2}
& \int_{A_1} \exp\left\{\Lcal_n(\alpha^{-2\nu}\theta,\alpha)-\Lcal_n(\alpha^{-2\nu}\widetilde \theta_{\alpha},\alpha)\right\} \frac{\pi\left(\widetilde \theta_{\alpha} +\frac{t}{\sqrt{n-p}} ~\Big |~ \alpha\right)}{\pi(\theta_0|\alpha)} \ud t \nonumber \\
\leq{}& \sqrt{n-p} \int_{|\theta-\widetilde \theta_{\alpha}|\geq \theta_0/4} \frac{\pi\left(\theta  | \alpha\right)}{\pi(\theta_0|\alpha)}\cdot\exp\left\{-\frac{0.015(n-p)}{2}\right\} \ud\theta \nonumber \\
<{}& \frac{\sqrt{n-p}}{\pi(\theta_0|\alpha)} \exp\{-0.007(n-p)\},
\end{align}
where in the last inequality, we use the fact that $\pi(\theta|\alpha)$ is a proper prior density. Thus, combining \eqref{ga1.1}, \eqref{ga1.bound2} and \eqref{ga1.bound1.2} yields that on the event $\Ecal_1(\epsilon_{1n},\alpha)$,
\begin{align}\label{ga1.1.bound}
\int_{A_1} |\varrho_n(t;\alpha)| \ud t &\leq 2\sqrt{\pi}\theta_0 \exp\left(-\frac{n-p}{64}\right) + \frac{\sqrt{n-p}}{\pi(\theta_0|\alpha) }  \exp\{-0.007(n-p)\}.
\end{align}

\vspace{2mm}

\noindent \underline{Bound the second term in \eqref{gnt2}:} On the event $\Ecal_1(\epsilon_{1n},\alpha)$ and $\{|\theta-\widetilde \theta_{\alpha}|< \theta_0/4\}$ with $0<\epsilon_{1n}<\theta_0/4$, if $\theta\geq \widetilde \theta_{\alpha}$, then
\begin{align*}
1-\frac{\widetilde \theta_{\alpha}}{\theta} &= 1-\frac{\widetilde \theta_{\alpha}}{\theta-\widetilde \theta_{\alpha}+\widetilde \theta_{\alpha}}< 1- \frac{\widetilde \theta_{\alpha}}{\theta_0/4+\widetilde \theta_{\alpha}} = \frac{\theta_0/4}{\theta_0/4+\widetilde \theta_{\alpha}} \leq \frac{\theta_0/4}{\theta_0/4+\theta_0-\epsilon_{1n}} < \frac{1}{4}.
\end{align*}
If $\theta<\widetilde\theta_{\alpha}$, then
\begin{align*}
\frac{\widetilde \theta_{\alpha}}{\theta}-1 &= \frac{\widetilde \theta_{\alpha}}{\theta-\widetilde \theta_{\alpha}+\widetilde \theta_{\alpha}} -1 < \frac{\widetilde \theta_{\alpha}}{-\theta_0/4+\widetilde \theta_{\alpha}}-1 = \frac{\theta_0/4}{\widetilde \theta_{\alpha}-\theta_0/4} < \frac{\theta_0/4}{\theta_0-\epsilon_{1n}-\theta_0/4} < \frac{1}{2}.
\end{align*}
Hence on the event $\Ecal_1(\epsilon_{1n},\alpha)$ and $\{|\theta-\widetilde \theta_{\alpha}|< \theta_0/4 \}$, $\widetilde \theta_{\alpha}/\theta\in (\frac{3}{4},\frac{3}{2})$. For any $u\in (\frac{3}{4},\frac{3}{2})$, by simple calculus, we have
\begin{align}\label{varphi.bound1}
\left|\varphi(u) - \frac{1}{2}\left(\frac{1}{u}-1\right)^2 \right| \leq \frac{6}{5}\left|\frac{1}{u}-1\right|^3.
\end{align}
Let
\begin{align*}
g_n(t) &= \frac{1}{n-p}\left[\Lcal_n(\alpha^{-2\nu}(\widetilde \theta_{\alpha} +\tfrac{t}{\sqrt{n-p}}),\alpha)-\Lcal_n(\alpha^{-2\nu}\widetilde \theta_{\alpha},\alpha)\right] - \frac{t^2}{2(n-p)\widetilde \theta_{\alpha}^2}\\
&= \varphi\left(\left[1+\frac{t}{\sqrt{n-p}\cdot \widetilde \theta_{\alpha}}\right]^{-1}\right)- \frac{t^2}{2(n-p) \widetilde \theta_{\alpha}^2}.
\end{align*}
In \eqref{varphi.bound1}, if we set $u=\widetilde \theta_{\alpha}/\theta$, then $\frac{1}{2}\left(\frac{1}{u}-1\right)^2=t^2/[2(n-p)\widetilde \theta_{\alpha}^2]$. Thus, we can obtain that on the event $\Ecal_1(\epsilon_{1n},\alpha)$ and $t\in A_2$ (so that $|\theta-\widetilde\theta_{\alpha}|<\theta_0/4$),
\begin{align}\label{rn.bound1}
|g_n(t)|&= \left|\varphi\left(\left[1+\frac{t}{\sqrt{n-p}\cdot \widetilde \theta_{\alpha}}\right]^{-1}\right)- \frac{t^2}{2(n-p) \widetilde \theta_{\alpha}^2}\right| \leq \frac{6|t|^3}{5(n-p)^{3/2}\widetilde\theta_{\alpha}^3} = \frac{6|\theta-\widetilde\theta_{\alpha}|^3}{5\widetilde\theta_{\alpha}^3}\nonumber \\
&\leq \frac{12|\theta-\widetilde\theta_{\alpha}|}{5\widetilde\theta_{\alpha}} \cdot \frac{|\theta-\widetilde\theta_{\alpha}|^2}{2\widetilde\theta_{\alpha}^2}
\leq
\frac{4}{5} \cdot \frac{|\theta-\widetilde\theta_{\alpha}|^2}{2\widetilde\theta_{\alpha}^2} = \frac{2t^2}{5(n-p)\widetilde\theta_{\alpha}^2}.
\end{align}
Therefore, on the event $\Ecal_1(\epsilon_{1n},\alpha)$ with $0<\epsilon_{1n}<\theta_0/4$,
\begin{align}\label{ga1.2.bound}
& ~\quad \int_{A_2} |\varrho_n(t;\alpha)| \ud t \leq \int_{A_2} \exp\left\{-\frac{n-p}{2}\varphi(\widetilde \theta_{\alpha}/\theta)\right\} \frac{\pi\left(\widetilde \theta_{\alpha} +\frac{t}{\sqrt{n-p}} ~\Big |~ \alpha\right)}{\pi(\theta_0|\alpha)} \ud t  + \int_{A_2}  \ee^{-\frac{t^2}{4\theta_0^2}} \ud t  \nonumber \\
&\leq  \int_{A_2} \exp\left\{- \frac{t^2}{4\widetilde \theta_{\alpha}^2} + \frac{n-p}{2}|g_n(t)|\right\} \frac{\pi\left(\widetilde \theta_{\alpha} +\frac{t}{\sqrt{n-p}} ~\Big |~ \alpha\right)}{\pi(\theta_0|\alpha)} + \int_{A_2}  \ee^{-\frac{t^2}{4\theta_0^2}} \ud t  \nonumber \\
&\stackrel{(i)}{\leq} \int_{A_2} \exp\left\{- \frac{t^2}{20\widetilde \theta_{\alpha}^2}\right\} \frac{\pi\left(\widetilde \theta_{\alpha} +\frac{t}{\sqrt{n-p}} ~\Big |~ \alpha\right)}{\pi(\theta_0|\alpha)} \ud t + \int_{A_2}  \ee^{-\frac{t^2}{4\theta_0^2}}  \ud t  \nonumber \\
&\leq \sup_{|\theta-\widetilde\theta_{\alpha}|<\theta_0/4} \frac{\pi\left(\theta |\alpha\right) }{\pi\left(\theta_0|\alpha\right)}\cdot \int_{A_2} \exp\left\{- \frac{t^2}{20\widetilde \theta_{\alpha}^2}\right\} \ud t + \int_{A_2}  \ee^{-\frac{t^2}{4\theta_0^2}}  \ud t  \nonumber \\
&\stackrel{(ii)}{\leq} \sup_{\theta\in\left(\frac{1}{2}\theta_0,\frac{3}{2}\theta_0\right)}
\frac{\pi\left(\theta |\alpha\right) }{\pi\left(\theta_0|\alpha\right)} \cdot \int_{|t|>s_n} \exp\left\{- \frac{t^2}{20\widetilde \theta_{\alpha}^2}\right\} \ud t + \int_{|t|>s_n}  \ee^{-\frac{t^2}{4\theta_0^2}}  \ud t  \nonumber \\
&\stackrel{(iii)}{\leq} \sup_{\theta\in\left(\frac{1}{2}\theta_0,\frac{3}{2}\theta_0\right)} \frac{\pi\left(\theta |\alpha\right) }{\pi\left(\theta_0|\alpha\right)} \cdot 2\sqrt{5\pi}\widetilde \theta_{\alpha} \exp\left(-\frac{s_n^2}{20\widetilde \theta_{\alpha}^2}\right) + 2\sqrt{\pi}\theta_0  \exp\left(-\frac{s_n^2}{4\theta_0^2}\right) \nonumber \\
&\stackrel{(iv)}{\leq} \sup_{\theta\in\left(\frac{1}{2}\theta_0,\frac{3}{2}\theta_0\right)} \frac{\pi\left(\theta |\alpha\right) }{\pi\left(\theta_0|\alpha\right)} \cdot \frac{5}{2}\sqrt{5\pi}\theta_0 \exp\left(-\frac{4s_n^2}{125\theta_0^2}\right) + 2\sqrt{\pi}\theta_0 \exp\left(-\frac{s_n^2}{4\theta_0^2}\right),
\end{align}
where (i) is from the upper bound of $g_n(t)$ in \eqref{rn.bound1}; (ii) is based on the relation $|\theta-\theta_0|\leq |\theta-\widetilde \theta_{\alpha}| + |\widetilde \theta_{\alpha}-\theta_0|< \theta_0/4 +\epsilon_{1n}< \theta_0/2$; (iii) follows from the normal tail inequality \eqref{normal.tail}; (iv) is based on the relation $\widetilde \theta_{\alpha}\leq \theta_0+\epsilon_{1n}< \theta_0+\theta_0/4< 5\theta_0/4$.
\vspace{3mm}

\noindent \underline{Bound the third term in \eqref{gnt2}:} We continue to use the bound in \eqref{varphi.bound1} and \eqref{rn.bound1} for $t\in A_3$ on the event $\Ecal_1(\epsilon_{1n},\alpha)$ and obtain that
\begin{align}\label{rn.bound2}
|g_n(t)|& \leq \frac{6|t|^3}{5(n-p)^{3/2}\widetilde\theta_{\alpha}^3} \leq \frac{6s_n^3}{5(n-p)^{3/2}\widetilde\theta_{\alpha}^3}.
\end{align}
Therefore,
\begin{align}\label{ga1.3.0}
&~~~~\int_{A_3} |\varrho_n(t;\alpha)| \ud t \nonumber \\
&= \int_{A_3} \left|\exp\left\{-\frac{n-p}{2}\varphi(\widetilde\theta_{\alpha}/\theta)\right\} \frac{\pi\left(\widetilde \theta_{\alpha} +\frac{t}{\sqrt{n-p}} ~\Big |~ \alpha\right)}{\pi(\theta_0|\alpha)} - \ee^{-\frac{t^2}{4\theta_0^2}} \right| \ud t \nonumber \\
&= \int_{A_3} \Bigg|\exp\Big\{-\frac{t^2}{4\widetilde \theta_{\alpha}^2} - \frac{n-p}{2} g_n(t)\Big\} \frac{\pi\left(\widetilde \theta_{\alpha} +\frac{t}{\sqrt{n-p}} ~\Big |~ \alpha\right)}{\pi(\theta_0|\alpha)} - \ee^{-\frac{t^2}{4\theta_0^2}} \Bigg| \ud t \nonumber \\
&\leq \int_{A_3} \Bigg|\exp\Big\{-\frac{t^2}{4\widetilde \theta_{\alpha}^2} - \frac{n-p}{2} g_n(t)\Big\} - \exp\Big(-\frac{t^2}{4\theta_0^2} \Big)\Bigg| \cdot \frac{\pi\left(\widetilde \theta_{\alpha} +\frac{t}{\sqrt{n-p}} ~\Big |~ \alpha\right)}{\pi(\theta_0|\alpha)} \ud t  \nonumber \\
&\quad + \int_{A_3} \ee^{-\frac{t^2}{4\theta_0^2}}\Bigg|  \frac{\pi\left(\widetilde \theta_{\alpha} +\frac{t}{\sqrt{n-p}} ~\Big |~ \alpha\right)}{\pi(\theta_0|\alpha)} - 1\Bigg| \ud t \nonumber \\
&\leq  \sup_{|t|<s_n} \left|\exp\left\{\frac{t^2}{4}\left(\theta_0^{-2}-\widetilde \theta_{\alpha}^{-2}\right) - \frac{n-p}{2} g_n(t) \right\} -1 \right| \cdot \sup_{|t|<s_n}\frac{\pi\left(\widetilde \theta_{\alpha} +\frac{t}{\sqrt{n-p}} ~\Big |~ \alpha\right)}{\pi(\theta_0|\alpha)} \cdot \int_{|t|<s_n} \ee^{-\frac{t^2}{4\theta_0^2}}\ud t   \nonumber \\
&\quad +\sup_{|t|<s_n} \Bigg|  \frac{\pi\left(\widetilde \theta_{\alpha} +\frac{t}{\sqrt{n-p}} ~\Big |~ \alpha\right)}{\pi(\theta_0|\alpha)} - 1 \Bigg|  \times \int_{|t|<s_n} \ee^{-\frac{t^2}{4\theta_0^2}}\ud t \nonumber \\
&\leq 2\sqrt{\pi}\theta_0\cdot \sup_{|t|<s_n} \left|\exp\left\{\frac{t^2}{4}\left(\theta_0^{-2}-\widetilde \theta_{\alpha}^{-2}\right)- \frac{n-p}{2} g_n(t)\right\} -1 \right| \cdot \sup_{|t|<s_n} \frac{\pi\left(\widetilde \theta_{\alpha} +\frac{t}{\sqrt{n-p}} ~\Big |~ \alpha\right)}{\pi(\theta_0|\alpha)}   \nonumber \\
&\quad + 2\sqrt{\pi}\theta_0\cdot \sup_{|t|<s_n} \left|  \frac{\pi\left(\widetilde \theta_{\alpha} +\frac{t}{\sqrt{n-p}} ~\Big |~ \alpha\right)}{\pi(\theta_0|\alpha)} - 1 \right|.
\end{align}
For the first term in \eqref{ga1.3.0}, we can choose $\epsilon_{1n} \to 0$ as $n\to \infty$ and $\epsilon_{1n}<\theta_0/4$, such that on the event $\Ecal_1(\epsilon_{1n},\alpha)$, for all $|t|<s_n$, using \eqref{rn.bound2}, we have
\begin{align}\label{ga1.3.1}
& \left|\frac{t^2}{4}\left(\theta_0^{-2}-\widetilde \theta_{\alpha}^{-2}\right) - \frac{n-p}{2} g_n(t)\right| \leq  \frac{s_n^2}{4}\frac{\left|\widetilde \theta_{\alpha}^2 - \theta_0^2\right|}{\widetilde \theta_{\alpha}^2\theta_0^2} + \left|\frac{n-p}{2} g_n(t)\right| \nonumber \\
&\leq \frac{s_n^2 \epsilon_{1n}}{4}\frac{\left|\widetilde \theta_{\alpha}+ \theta_0\right|}{\widetilde \theta_{\alpha}^2\theta_0^2} + \left|\frac{n-p}{2} g_n(t)\right| \leq \frac{s_n^2 \epsilon_{1n}}{4}\frac{(2\theta_0+\epsilon_{1n})}{(\theta_0-\epsilon_{1n})^2\theta_0^2} + \frac{3s_n^3}{5\sqrt{n-p}\widetilde\theta_{\alpha}^3} \nonumber \\
&<\frac{s_n^2 \epsilon_{1n}}{\theta_0^3} + \frac{2s_n^3}{\sqrt{n-p}\theta_0^3}.
\end{align}
We choose sufficiently large $n$ that satisfies $\epsilon_{1n}\leq \frac{\theta_0^3}{2s_n^2}$ and $n\geq \frac{16s_n^6}{\theta_0^6}+p$, such that the upper bound in \eqref{ga1.3.1} is smaller than 1. Then we can apply the inequality $| \ee^u-1|\leq 2|u|$ for all $|u|\leq 1$ and obtain that
\begin{align}\label{ga1.3.2}
&\quad \sup_{|t|<s_n} \left|\exp\left\{\frac{t^2}{4}\left(\theta_0^{-2}-\widetilde \theta_{\alpha}^{-2}\right)- \frac{n-p}{2} g_n(t)\right\} -1 \right| \nonumber \\
&\leq 2\left|\frac{t^2}{4}\left(\theta_0^{-2}-\widetilde \theta_{\alpha}^{-2}\right) - \frac{n-p}{2} g_n(t)\right|  <\frac{2s_n^2 \epsilon_{1n}}{\theta_0^3} + \frac{4s_n^3}{\sqrt{n-p}\theta_0^3}.
\end{align}
Furthermore, we can choose $n\geq \frac{16s_n^2}{\theta_0^2}+p$ such that for all $|t|<s_n$, on the event $\Ecal_1(\epsilon_{1n},\alpha)$, $\widetilde\theta_{\alpha}+t/\sqrt{n-p} \leq \theta_0+\epsilon_{1n} + s_n/\sqrt{n-p}<\frac{3}{2}\theta_0$ and $\widetilde\theta_{\alpha}+t/\sqrt{n-p} >\theta_0-\epsilon_{1n}>\frac{3}{4}\theta_0$. Then from Assumption \ref{prior.1} (ii), we have that on the interval $(\frac{3}{4}\theta_0,\frac{3}{2}\theta_0)$,
\begin{align}\label{ga1.3.3}
\sup_{|t|<s_n} \frac{\pi\left(\widetilde \theta_{\alpha} +\frac{t}{\sqrt{n-p}} ~\Big |~ \alpha\right)}{\pi(\theta_0|\alpha)} &\leq \sup_{\theta\in\left(\frac{3}{4}\theta_0,\frac{3}{2}\theta_0\right)} \frac{\pi(\theta|\alpha)}{\pi(\theta_0|\alpha)}.
\end{align}
For the second term in \eqref{ga1.3.1}, by Assumption \ref{prior.1} and the fact that $\epsilon_{1n}\to 0, s_n/\sqrt{n-p}\to 0$, we have that on the event $\Ecal_1(\epsilon_{1n},\alpha)$, for all sufficiently large $n$,
\begin{align}\label{ga1.3.4}
& \quad \sup_{|t|<s_n} \left|  \frac{\pi\left(\widetilde \theta_{\alpha} +\frac{t}{\sqrt{n-p}} ~\Big |~ \alpha\right)}{\pi(\theta_0|\alpha)} - 1 \right| \leq \sup_{\theta\in (3\theta_0/4,3\theta_0/2)} \left|\frac{\pi(\theta|\alpha) - \pi(\theta_0|\alpha)}{\pi(\theta_0|\alpha)}\right| \nonumber \\
& \leq \sup_{\theta\in (3\theta_0/4,3\theta_0/2)} \left|\frac{\partial \log \pi(\theta|\alpha)}{\partial \theta} \right| \cdot \sup_{\theta\in (3\theta_0/4,3\theta_0/2)} \frac{\pi(\theta|\alpha)}{\pi(\theta_0|\alpha)}
\cdot \sup_{|t|<s_n} \left|\widetilde \theta_{\alpha} +\frac{t}{\sqrt{n-p}} - \theta_0 \right| \nonumber\\
&\leq  \sup_{\theta\in \left(\frac{3}{4}\theta_0,\frac{3}{2}\theta_0\right)} \left|\frac{\partial \log \pi(\theta|\alpha)}{\partial \theta} \right| \sup_{\theta\in \left(\frac{3}{4}\theta_0,\frac{3}{2}\theta_0\right)} \frac{\pi(\theta|\alpha)}{\pi(\theta_0|\alpha)} \cdot \left(\epsilon_{1n}+\frac{s_n}{\sqrt{n-p}}\right).
\end{align}
Therefore, \eqref{ga1.3.0}, \eqref{ga1.3.2}, \eqref{ga1.3.3}, and \eqref{ga1.3.4} together yield that on the event $\Ecal_1(\epsilon_{1n},\alpha)$, with $\epsilon_{1n}\leq \min\left(\frac{\theta_0^3}{2s_n^2},\frac{\theta_0}{4}\right)$ and $n\geq \max\left(\frac{16s_n^6}{\theta_0^6},\frac{16s_n^2}{\theta_0^2}\right)+p$,
\begin{align}\label{ga1.3.bound}
&\int_{A_3} |\varrho_n(t;\alpha)| \ud t \nonumber\\
\leq{} & ~\frac{4\sqrt{\pi}}{\theta_0^2} \left(s_n^2 \epsilon_{1n} + \frac{2s_n^3}{\sqrt{n-p}}\right)\cdot \sup_{\theta\in\left(\frac{3}{4}\theta_0,\frac{3}{2}\theta_0\right)} \frac{\pi(\theta|\alpha)}{\pi(\theta_0|\alpha)} \nonumber \\
& + 2\sqrt{\pi}\theta_0  \sup_{\theta\in \left(\frac{3}{4}\theta_0,\frac{3}{2}\theta_0\right)} \left|\frac{\partial \log \pi(\theta|\alpha)}{\partial \theta} \right| \sup_{\theta\in \left(\frac{3}{4}\theta_0,\frac{3}{2}\theta_0\right)} \frac{\pi(\theta|\alpha)}{\pi(\theta_0|\alpha)} \cdot\left(\epsilon_{1n}+\frac{s_n}{\sqrt{n-p}}\right).
\end{align}

Finally, we combine \eqref{ga1.1.bound}, \eqref{ga1.2.bound}, and \eqref{ga1.3.bound} to conclude that on the event $\Ecal_1(\epsilon_{1n},\alpha)$ with $\epsilon_{1n}\leq \min\left(\frac{\theta_0^3}{2s_n^2},\frac{\theta_0}{4}\right)$ and $n\geq \max\left(\frac{16s_n^6}{\theta_0^6},\frac{16s_n^2}{\theta_0^2}\right)+p$,
\begin{align}\label{ga1.bound.full1}
& \int_{\RR} |\varrho_n(t;\alpha)| \ud t \leq 2\sqrt{\pi}\theta_0 \exp\{-(n-p)/64\} + \frac{\sqrt{n-p}}{\pi(\theta_0|\alpha)} \exp\{-0.007(n-p)\} \nonumber \\
&~~+ \sup_{\theta\in\left(\frac{1}{2}\theta_0,\frac{3}{2}\theta_0\right)} \frac{\pi(\theta|\alpha)}{\pi(\theta_0|\alpha)}\cdot \frac{5}{2}\sqrt{5\pi}\theta_0 \exp\left(-\frac{4s_n^2}{125\theta_0^2}\right) + 2\sqrt{\pi}\theta_0 \exp\left(-\frac{s_n^2}{4\theta_0^2}\right) \nonumber \\
&~~+ \frac{4\sqrt{\pi}}{\theta_0^2} \left(s_n^2 \epsilon_{1n} + \frac{2s_n^3}{\sqrt{n-p}}\right)\cdot \sup_{\theta\in\left(\frac{3}{4}\theta_0,\frac{3}{2}\theta_0\right)} \frac{\pi(\theta|\alpha)}{\pi(\theta_0|\alpha)} \nonumber \\
&~~ + 2\sqrt{\pi}\theta_0  \sup_{\theta\in \left(\frac{3}{4}\theta_0,\frac{3}{2}\theta_0\right)} \left|\frac{\partial \log \pi(\theta|\alpha)}{\partial \theta} \right| \sup_{\theta\in \left(\frac{3}{4}\theta_0,\frac{3}{2}\theta_0\right)} \frac{\pi(\theta|\alpha)}{\pi(\theta_0|\alpha)} \cdot \left(\epsilon_{1n}+\frac{s_n}{\sqrt{n-p}}\right).
\end{align}
By adjusting the constants to be slightly larger, we obtain the bound in \eqref{en:rho}.
\end{proof}

\vspace{4mm}

The proof of Theorem \ref{thm:bvm1:theta} has used on the following lemmas.
\begin{lemma}\label{lemma:intdiff}
For two nonnegative functions $f$ and $g$, if their integrals are $F=\int f$ and $G=\int g$, then
\begin{align*}
\int \left|\frac{f}{\int f} - \frac{g}{\int g}\right| \leq \frac{2\int |f-g|}{G}.
\end{align*}
\end{lemma}

\begin{proof}[Proof of Lemma \ref{lemma:intdiff}]
\begin{align*}
& \int \left|\frac{f}{\int f} - \frac{g}{\int g}\right| = \int \frac{|fG-gF|}{FG} \leq \int \frac{f|G-F| + F|g-f|}{FG} \\
&= \frac{|G-F|\int f + F\int |f-g|}{FG} \leq \frac{F\int|f-g| + F\int |f-g|}{FG} = \frac{2\int |f-g|}{G}.
\end{align*}
\end{proof}

\begin{lemma}\label{lemma:normaltv}
For two univariate normal distributions $\mathcal{N}(\mu_1,\sigma^2)$ and $\mathcal{N}(\mu_2,\sigma^2)$ on $\RR$, their total variation distance is given by
\begin{align*}
& \left\|\mathcal{N}(\mu_1,\sigma^2) - \mathcal{N}(\mu_2,\sigma^2)\right\|_{\tv} = 2\Phi\left(\frac{|\mu_1-\mu_2|}{2\sigma}\right)-1,
\end{align*}
where $\Phi(x)=\int_{-\infty}\frac{1}{\sqrt{2\pi}} \ee^{-z^2/2}\ud z$ is the standard normal cdf.
\end{lemma}

\begin{proof}[Proof of Lemma \ref{lemma:normaltv}]
Let $f_i(x)$ be the normal density of $\mathcal{N}(\mu_i,\sigma^2)$, $i=1,2$. Suppose that $\mu_1<\mu_2$ without loss of generality. Then it is clear that $f_1(x)>f_2(x)$ if $x<(\mu_1+\mu_2)/2$ and $f_1(x)<f_2(x)$ if $x>(\mu_1+\mu_2)/2$. Therefore,
\begin{align*}
&\quad ~ \left\|\mathcal{N}(\mu_1,\sigma^2) - \mathcal{N}(\mu_2,\sigma^2)\right\|_{\tv} \\
&= \frac{1}{2} \int_{-\infty}^{+\infty} |f_1(x)-f_2(x)|\ud x \\
&= \frac{1}{2}\int_{-\infty}^{(\mu_1+\mu_2)/2} \{f_1(x)-f_2(x)\}\ud x + \frac{1}{2}\int_{(\mu_1+\mu_2)/2}^{+\infty} \{f_2(x)-f_1(x)\}\ud x \\
&= \frac{1}{2} \left[\Phi\left(\frac{\mu_2-\mu_1}{2\sigma}\right) - \Phi\left(\frac{\mu_1-\mu_2}{2\sigma}\right) + 1 - \Phi\left(\frac{\mu_1-\mu_2}{2\sigma}\right) - \left\{1-\Phi\left(\frac{\mu_2-\mu_1}{2\sigma}\right)\right\} \right]\\
&= 2\Phi\left(\frac{\mu_2-\mu_1}{2\sigma}\right)-1 .
\end{align*}
\end{proof}

\vspace{3mm}

\subsection{Proof of Theorem \ref{thm:bvm1:theta}} \label{supsubsec:thm1.proof}
\begin{proof}[Proof of Theorem \ref{thm:bvm1:theta}]
The asymptotic normality of $\widetilde\theta_{\alpha}$, i.e., $\sqrt{n}\big(\widetilde \theta_{\alpha}-\theta_0\big)\overset{\Dcal}{\rightarrow} \Ncal(0,2\theta_0^2)$ as $n\to\infty$, has already been proved in Lemma \ref{lem:theta.alpha0}. In the following, we focus on proving the normal limit for the conditional posterior of $\theta$.

From \eqref{post:theta:rho}, the posterior density of $\theta$ can be written as
\begin{align}\label{post:t:rho1}
\pi(\theta|Y_n,\alpha)&= \frac{ \ee^{\Lcal_n\left(\alpha^{-2\nu}\theta,\alpha\right)}\pi\left(\theta|\alpha\right)}
{\int_0^{\infty} \ee^{\Lcal_n\left(\alpha^{-2\nu}\theta,\alpha\right)}\pi\left(\theta|\alpha\right)\ud \theta} = \frac{ \ee^{\Lcal_n\left(\alpha^{-2\nu}\theta,\alpha\right)
-\Lcal_n\left(\alpha^{-2\nu}\widetilde\theta_{\alpha},\alpha\right)}
\frac{\pi\left(\theta|\alpha\right)}{\pi\left(\theta_0|\alpha\right)}}
{\int_0^{\infty} \ee^{\Lcal_n\left(\alpha^{-2\nu}\theta,\alpha\right)
-\Lcal_n\left(\alpha^{-2\nu}\widetilde\theta_{\alpha},\alpha\right)}
\frac{\pi\left(\theta|\alpha\right)}{\pi\left(\theta_0|\alpha\right)}\ud \theta}.
\end{align}
We can rewrite \eqref{gnt1} in Lemma \ref{lemma:gndiff1} in terms of $\theta=\widetilde \theta_{\alpha}+(n-p)^{-1/2}t$:
\begin{align}\label{cond:denom1}
&\int_{\RR} \left| \ee^{\Lcal_n\left(\alpha^{-2\nu}\theta,\alpha\right)
-\Lcal_n\left(\alpha^{-2\nu}\widetilde\theta_{\alpha},\alpha\right)}\frac{\pi\left(\theta|\alpha\right)}{\pi(\theta_0|\alpha)}
- \ee^{-\frac{(n-p)(\theta-\widetilde\theta_{\alpha})^2}{4\theta_0^2}}\right| \ud \theta  \leq  \frac{B_n(\alpha)}{\sqrt{n-p}}.
\end{align}
For the fixed $\alpha>0$, define the events $\Ecal_1'(\epsilon,\alpha)=\{|\widetilde\theta_{\alpha}-\widetilde\theta_{\alpha_0}|<\epsilon\}$ and   $\Ecal_1''(\epsilon)=\{|\widetilde\theta_{\alpha_0}-\theta_0|<\epsilon\}$ for any $\epsilon>0$. From Lemma \ref{lem:sup.theta1}, $\pr\left\{\Ecal_1'(\theta_0 n^{-1/2-\tau}/2,\alpha)\right\} \geq 1-4\exp(-4\log^2 n)$ for all sufficiently large $n$. From Lemma \ref{lem:theta.alpha0}, $\pr\left\{\Ecal_1''(5\theta_0 n^{-1/2}\log n )\right\} \geq 1-3\exp(-4\log^2 n)$ for all sufficiently large $n$. Since when $n$ is sufficiently large,
\begin{align*}
\Ecal_1'(\theta_0 n^{-1/2-\tau}/2,\alpha) \cap \Ecal_1''(5\theta_0 n^{-1/2}\log n,\alpha) \subseteq \Ecal_1(6\theta_0 n^{-1/2}
\log n, \alpha),
\end{align*}
we have that $\pr\left\{\Ecal_1(6\theta_0 n^{-1/2}\log n,\alpha)\right\} \geq 1-7\exp(-4\log^2 n)$. In the expression of $B_n(\alpha)$ in \eqref{en:rho}, we set $\epsilon_{1n}=6\theta_0 n^{-1/2}\log n$ and $s_n=\log n$ which satisfies the conditions in Lemma \ref{lemma:gndiff1}. By Assumption \ref{prior.1}, for a fixed $\alpha>0$, there exists some finite constant $C_1>0$ that depends on $\alpha$, such that
\begin{align} \label{eq:C1}
\sup_{\theta\in\left(\frac{1}{2}\theta_0,\frac{3}{2}\theta_0\right)}\frac{\pi(\theta|\alpha)}{\pi(\theta_0|\alpha)} \leq C_1,\quad \sup_{\theta\in\left(\frac{3}{4}\theta_0,\frac{3}{2}\theta_0\right)} \left|\frac{\partial \log \pi(\theta|\alpha)}{\partial\theta}\right| \leq C_1.
\end{align}
Hence, on the event $\Ecal_1(6\theta_0 n^{-1/2}\log n,\alpha)$, the order of $B_n(\alpha)$ can be quantified from \eqref{en:rho} in Lemma \ref{lemma:gndiff1}:
\begin{align}\label{eq:bn.bound1}
B_n(\alpha) & \leq 4\theta_0 \exp\left(-\frac{n-p}{64}\right) + \frac{\sqrt{n}}{\pi(\theta_0|\alpha)} \exp\left\{-0.007(n-p)\right\} \nonumber\\
&\quad + 10C_1\theta_0 \exp\left(-\frac{4\log^2 n}{125\theta_0^2}\right) + 4\theta_0 \exp\left(-\frac{\log^2 n}{4\theta_0^2}\right)  \nonumber \\
&\quad + \frac{8C_1}{\theta_0^2} \left(6\theta_0 n^{-1/2}\log^3 n + 2(n-p)^{-1/2}\log^3 n\right) + 4C_1^2 \theta_0  \left(6\theta_0+1\right) (n-p)^{-1/2}\log n \nonumber \\
&\leq C_2 n^{-1/2}\log^3 n \rightarrow 0, \text{ as } n\to\infty,
\end{align}
for some constant $C_2>0$ that depends on $\theta_0,p,\pi(\theta_0|\alpha)$ and $C_1$ in \eqref{eq:C1}.
This together with \eqref{cond:denom1} implies that on the event $\Ecal_1(6\theta_0 n^{-1/2}\log n,\alpha)$, the denominator of \eqref{post:t:rho1} converges to
$$\int_{\RR} \exp\left\{-\frac{(n-p)(\theta-\widetilde\theta_{\alpha})^2}{4\theta_0^2}\right\}\ud \theta = 2\theta_0\sqrt{\pi/(n-p)} .$$

Now in Lemma \ref{lemma:intdiff}, we set $f$ to be the numerator of \eqref{post:t:rho1} and $g$ to be $\exp\left\{-\frac{(n-p)(\theta-\widetilde\theta_{\alpha})^2}{4\theta_0^2}\right\}$. Using \eqref{eq:bn.bound1}, we obtain that on the event $\Ecal_1(6\theta_0 n^{-1/2}\log n,\alpha)$, as $n\to\infty$,
\begin{align}\label{eq:bn.bound2}
&\int_{\RR} \left|\pi(\theta|Y_n,\alpha)-\frac{1}{2\sqrt{\pi/(n-p)}\theta_0}\exp\left\{-\frac{(n-p)(\theta-\widetilde\theta_{\alpha})^2}{4\theta_0^2}\right\}\right|\ud \theta \nonumber \\
&\leq
\frac{2\int_{\RR} \left| \ee^{\Lcal_n\left(\alpha^{-2\nu}\theta,\alpha\right)
-\Lcal_n\left(\alpha^{-2\nu}\widetilde\theta_{\alpha},\alpha\right)}\frac{\pi\left(\theta|\alpha\right)}{\pi\left(\theta_0|\alpha\right)}
- \exp\left\{-\frac{(n-p)(\theta-\widetilde\theta_{\alpha})^2}{4\theta_0^2}\right\}\right| \ud \theta}
{2\theta_0\sqrt{\pi/(n-p)}} \nonumber \\
&\leq \frac{B_n(\alpha)/\sqrt{n-p}}{\theta_0\sqrt{\pi/(n-p)}} = \frac{B_n(\alpha)}{\theta_0\sqrt{\pi}} \leq C_3 n^{-1/2}\log^3 n \rightarrow 0,
\end{align}
for some constant $C_3>0$ that depends on $\theta_0,p,\pi(\theta_0|\alpha)$ and $C_1$ in \eqref{eq:C1}.

Since $\pr\left(\left\{\Ecal_1(6\theta_0 n^{-1/2}\log n,\alpha)\right\}^c\right) \leq 7\exp(-4\log^2 n)$ and $\sum_{n=1}^{\infty} 7\exp(-4\log^2 n)<\infty$, by the Borel-Cantelli lemma, we have shown that as $n\to\infty$ almost surely $P_{(\beta_0,\sigma_0^2,\alpha_0)}$,
\begin{align}\label{eq:cond.pi.np}
\left\| \Pi(\ud\theta|Y_n,\alpha) - \Ncal \left(\ud\theta \big| \widetilde\theta_{\alpha}, 2\theta_0^2/(n-p) \right)\right\|_{\tv} \leq \frac{B_n(\alpha)}{2\theta_0\sqrt{\pi}} \to 0 .
\end{align}
On the other hand, Theorem 1.3 of \citet{Devetal18} implies that
\begin{align}\label{eq:cond.np.n}
& \left\| \Ncal \left(\ud\theta \big| \widetilde\theta_{\alpha}, 2\theta_0^2/n \right) - \Ncal \left(\ud\theta \big| \widetilde\theta_{\alpha}, 2\theta_0^2/(n-p) \right)\right\|_{\tv} \nonumber \\
\leq{}& \frac{3}{2} \cdot \frac{2\theta_0^2/(n-p) - 2\theta_0^2/n}{2\theta_0^2/n} = \frac{3p}{2(n-p)} \to 0, ~~ \text{ as } n\to\infty.
\end{align}
Therefore, by \eqref{eq:cond.pi.np}, \eqref{eq:cond.np.n}, and the triangle inequality, we have
\begin{align}
\left\| \Pi(\ud\theta|Y_n,\alpha) - \Ncal \left(\ud\theta \big| \widetilde\theta_{\alpha}, 2\theta_0^2/n \right)\right\|_{\tv} \leq C_3 n^{-1/2}\log^3 n + \frac{3p}{2(n-p)}  \leq C_4 n^{-1/2}\log^3 n \to 0 ,  \nonumber
\end{align}
as $n\to\infty$ almost surely $P_{(\beta_0,\sigma_0^2,\alpha_0)}$, for some constant $C_4>0$ that depends on $\theta_0,p,\pi(\theta_0|\alpha)$ and $C_1$ in \eqref{eq:C1}. This completes the proof of Theorem \ref{thm:bvm1:theta}.
\end{proof}

\vspace{5mm}

\subsection{Proof of Theorem \ref{thm:bvm2:joint}} \label{supsubsec:thm2.proof}
\begin{proof}[Proof of Theorem \ref{thm:bvm2:joint}]
It has been proved in Lemma \ref{lem:alpha.exist} that the profile posterior density \eqref{profile:post1} is well defined almost surely for every $n\geq p$. The convergence in total variation norm for the marginal posterior distributions of $\theta$ and $\alpha$ will follow trivially once the convergence for the joint posterior is proved. The convergence in total variation norm for the joint posterior \eqref{joint:theta1} is implied by adding the following relations using a triangle inequality:
\begin{align}
& \int_{0}^{\infty}\int_{\RR} \left|\pi(\theta,\alpha|Y_n) - \frac{\sqrt{n-p}}{2\sqrt{\pi}\theta_0} \ee^{-\frac{(n-p)(\theta-\widetilde\theta_{\alpha})^2}{4\theta_0^2}} \cdot \widetilde \pi(\alpha|Y_n)\right| \ud \theta \ud \alpha \rightarrow 0, \label{joint:theta2} \\
& \int_{0}^{\infty}  \int_{\RR} \left|\frac{\sqrt{n-p}}{2\sqrt{\pi}\theta_0} \ee^{-\frac{(n-p)(\theta-\widetilde\theta_{\alpha})^2}{4\theta_0^2}} - \frac{\sqrt{n}}{2\sqrt{\pi}\theta_0} \ee^{-\frac{n(\theta-\widetilde\theta_{\alpha_0})^2}{4\theta_0^2}} \right| \cdot \widetilde \pi(\alpha|Y_n) \ud \theta \ud \alpha \rightarrow 0, \label{joint:theta3}
\end{align}
as $n\to\infty$ almost surely $P_{(\beta_0,\sigma_0^2,\alpha_0)}$. We prove \eqref{joint:theta2} and \eqref{joint:theta3} respectively.

\vspace{6mm}

\noindent \underline{Proof of \eqref{joint:theta2}:}
\vspace{3mm}

In Lemma \ref{lemma:intdiff}, we take
\begin{align*}
& f = \ee^{\Lcal_n(\alpha^{-2\nu}\theta,\alpha)-\Lcal_n(\alpha^{-2\nu}\widetilde\theta_{\alpha},\alpha)} \pi(\theta|\alpha) \cdot \ee^{\Lcal_n(\alpha^{-2\nu}\widetilde\theta_{\alpha},\alpha)} \pi(\alpha), \\
& g = \ee^{-\frac{(n-p)(\theta-\widetilde\theta_{\alpha})^2}{4\theta_0^2}} \pi(\theta_0|\alpha) \cdot \ee^{\Lcal_n(\alpha^{-2\nu}\widetilde\theta_{\alpha},\alpha)} \pi(\alpha),
\end{align*}
such that by applying Lemma \ref{lemma:intdiff}, we can obtain that
\begin{align}\label{joint:theta23}
&\int_0^{\infty} \int_{\RR} \left|\pi(\theta,\alpha|Y_n) - \frac{\sqrt{n-p}}{2\sqrt{\pi}\theta_0} \ee^{-\frac{(n-p)(\theta-\widetilde\theta_{\alpha})^2}{4\theta_0^2}} \cdot \widetilde \pi(\alpha|Y_n)\right| \ud \theta \ud \alpha \nonumber \\
=& \int_0^{\infty} \int_{\RR} \Bigg|\frac{ \ee^{\Lcal_n(\alpha^{-2\nu}\theta,\alpha)-\Lcal_n(\alpha^{-2\nu}\widetilde\theta_{\alpha},\alpha)}\cdot \ee^{\Lcal_n(\alpha^{-2\nu}\widetilde\theta_{\alpha},\alpha)}
\pi(\theta|\alpha)\pi(\alpha)}
{\int_0^{\infty} \int_0^{\infty} \ee^{\Lcal_n(\alpha^{-2\nu}\theta,\alpha)-\Lcal_n(\alpha^{-2\nu}\widetilde\theta_{\alpha},\alpha)}
\cdot \ee^{\Lcal_n(\alpha^{-2\nu}\widetilde\theta_{\alpha},\alpha)}\pi(\theta|\alpha)\pi(\alpha)\ud \theta \ud \alpha}
\nonumber \\
& - \frac{ \ee^{-\frac{(n-p)(\theta-\widetilde\theta_{\alpha})^2}{4\theta_0^2}}\cdot \ee^{\Lcal_n(\alpha^{-2\nu}\widetilde\theta_{\alpha},\alpha)}
\pi(\theta_0|\alpha)\pi(\alpha)}
{\int_0^{\infty} \int_{\RR} \ee^{-\frac{(n-p)(\theta-\widetilde\theta_{\alpha})^2}{4\theta_0^2}}\cdot \ee^{\Lcal_n(\alpha^{-2\nu}\widetilde\theta_{\alpha},\alpha)} \pi(\theta_0|\alpha)\pi(\alpha) \ud \theta \ud \alpha}
\Bigg| \ud \theta \ud \alpha \leq \frac{\numer}{\denom},
\end{align}
where (with $\varrho_n(t;\alpha)$ defined in \eqref{func:gn})
\begin{align}
\numer & = 2\int_0^{\infty} \int_{\RR} \left| \ee^{\Lcal_n(\alpha^{-2\nu}\theta,\alpha)-\Lcal_n(\alpha^{-2\nu}\widetilde\theta_{\alpha},\alpha)} \frac{\pi(\theta|\alpha)}{\pi(\theta_0|\alpha)}
- \ee^{-\frac{(n-p)(\theta-\widetilde\theta_{\alpha})^2}{4\theta_0^2}}
\right| \nonumber \\
&\quad \times \ee^{\Lcal_n(\alpha^{-2\nu}\widetilde\theta_{\alpha},\alpha)} \pi(\theta_0|\alpha) \pi(\alpha) \ud \theta \ud \alpha  \nonumber \\
& = 2\int_0^{\infty} \int_{\RR} \left|\varrho_n(\sqrt{n-p}(\theta-\widetilde\theta_{\alpha});\alpha)\right| \ee^{\Lcal_n(\alpha^{-2\nu}\widetilde\theta_{\alpha},\alpha)} \pi(\theta_0|\alpha) \pi(\alpha) \ud \theta \ud \alpha , \label{numerator1} \\
\denom & = \int_0^{\infty} \int_{\RR} \ee^{-\frac{(n-p)(\theta-\widetilde\theta_{\alpha})^2}{4\theta_0^2}}\cdot \ee^{\Lcal_n(\alpha^{-2\nu}\widetilde\theta_{\alpha},\alpha)} \pi(\theta_0|\alpha)\pi(\alpha) \ud \theta \ud \alpha \nonumber \\
& = \frac{2\theta_0\sqrt{\pi}}{\sqrt{n-p}}\int_0^{\infty} \ee^{\Lcal_n(\alpha^{-2\nu}\widetilde\theta_{\alpha},\alpha)} \pi(\theta_0|\alpha)\pi(\alpha) \ud \alpha, \label{denominator1}
\end{align}
We decompose the numerator in \eqref{numerator1} into three terms:
\begin{align}\label{numerator2}
\numer &= {\numer}_1 + {\numer}_2 + {\numer}_3 , \nonumber \\
{\numer}_1 &= 2\int_{\underline\alpha_n}^{\overline\alpha_n} \int_{\RR} \left|\varrho_n(\sqrt{n-p}(\theta-\widetilde\theta_{\alpha});\alpha)\right| \ee^{\Lcal_n(\alpha^{-2\nu}\widetilde\theta_{\alpha},\alpha)} \pi(\theta_0|\alpha) \pi(\alpha) \ud \theta \ud \alpha, \nonumber  \\
{\numer}_2 &= 2\int_0^{\underline\alpha_n} \int_{\RR}
\left|\varrho_n(\sqrt{n-p}(\theta-\widetilde\theta_{\alpha});\alpha)\right| \ee^{\Lcal_n(\alpha^{-2\nu}\widetilde\theta_{\alpha},\alpha)} \pi(\theta_0|\alpha) \pi(\alpha) \ud \theta \ud \alpha, \nonumber \\
{\numer}_3 &= 2\int_{\overline\alpha_n}^{\infty} \int_{\RR} \left|\varrho_n(\sqrt{n-p}(\theta-\widetilde\theta_{\alpha});\alpha)\right| \ee^{\Lcal_n(\alpha^{-2\nu}\widetilde\theta_{\alpha},\alpha)} \pi(\theta_0|\alpha) \pi(\alpha) \ud \theta \ud \alpha,
\end{align}
To show \eqref{joint:theta2}, from \eqref{joint:theta23} and \eqref{numerator2}, it suffices to show that ${\numer}_j/\denom\to 0$ for $j=1,2,3$ as $n\to\infty$ almost surely $P_{(\beta_0,\sigma_0^2,\alpha_0)}$.
\vspace{5mm}

\noindent \underline{Proof of ${\numer}_1/\denom \to 0$:}
\vspace{3mm}

We consider all $\alpha \in [\underline\alpha_n,\overline\alpha_n]$. For any $\epsilon>0$, define three events
\begin{align} \label{eq:Ecal.234}
& \Ecal_2(\epsilon) = \Big\{\sup_{\alpha\in [\underline \alpha_n, \overline \alpha_n]} |\widetilde\theta_{\alpha}-\theta_0| < \epsilon \Big\}, \quad \Ecal_3(\epsilon) = \Big\{\sup_{\alpha\in [\underline \alpha_n, \overline \alpha_n]} |\widetilde\theta_{\alpha}-\widetilde \theta_{\alpha_0}| < \epsilon \Big\}, \nonumber \\
& \Ecal_4(\epsilon) = \Big\{|\widetilde\theta_{\alpha_0}-\theta_0| < \epsilon \Big\}.
\end{align}

For sufficiently large $n$, Lemma \ref{lem:sup.theta1} shows that  $\pr\{\Ecal_3(\theta_0n^{-1/2-\tau}/2)\}\geq 1-4\exp(-4\log^2 n)$ for some constant $\tau\in (0,1/2)$. Lemma \ref{lem:theta.alpha0} shows that $\pr\{\Ecal_4(5\theta_0n^{-1/2}\log n)\}\geq 1-3\exp(-4\log^2 n)$. By the triangle inequality, for sufficiently large $n$,
$$\Ecal_2(6\theta_0n^{-1/2}\log n)\supseteq \Ecal_3(\theta_0n^{-1/2-\tau}/2) \cap \Ecal_4(5\theta_0n^{-1/2}\log n),$$
it follows that $\pr\{\Ecal_2(6\theta_0n^{-1/2}\log n)\}\geq 1 - 7\exp(-4\log^2 n)$.

We again use the inequality \eqref{cond:denom1} from Lemma \ref{lemma:gndiff1},
with $B_n(\alpha)$ defined in \eqref{en:rho} with $\epsilon_{1n}=6\theta_0n^{-1/2}\log n$ and $s_n=\log n$. Since $\Ecal_1(6\theta_0n^{-1/2}\log n, \alpha) \supseteq \Ecal_2(6\theta_0n^{-1/2}\log n)$ for every $\alpha\in [\underline \alpha_n, \overline \alpha_n]$, Lemma \ref{lemma:gndiff1} can be applied to all $\alpha\in [\underline \alpha_n, \overline \alpha_n]$ with $\epsilon_{1n}=6\theta_0n^{-1/2}\log n$ and $s_n=\log n$. Therefore, \eqref{cond:denom1} holds uniformly for all $\alpha\in [\underline \alpha_n, \overline \alpha_n]$ on the event $\Ecal_2(6\theta_0n^{-1/2}\log n)$, such that $\pr\{\Ecal_2(6\theta_0n^{-1/2}\log n)\}\geq 1 - 7\exp(-4\log^2 n)$.

Integrating \eqref{cond:denom1} over the interval $[\underline \alpha_n, \overline \alpha_n]$ gives that
\begin{align}\label{joint:theta22}
& \int_{\underline\alpha_n}^{\overline\alpha_n} \int_{\RR} \left| \ee^{\Lcal_n(\alpha^{-2\nu}\theta,\alpha)-\Lcal_n(\alpha^{-2\nu}\widetilde\theta_{\alpha},\alpha)} \frac{\pi(\theta|\alpha)}{\pi(\theta_0|\alpha)}
- \ee^{-\frac{(n-p)(\theta-\widetilde\theta_{\alpha})^2}{4\theta_0^2}}
\right| \nonumber \\
&~~ \times \ee^{\Lcal_n(\alpha^{-2\nu}\widetilde\theta_{\alpha},\alpha)} \pi(\theta_0|\alpha)\pi(\alpha) \ud \theta \ud \alpha \nonumber \\
\leq& \int_{\underline\alpha_n}^{\overline\alpha_n} \frac{B_n(\alpha)}{\sqrt{n-p}} \ee^{\Lcal_n(\alpha^{-2\nu}\widetilde\theta_{\alpha},\alpha)} \pi(\theta_0|\alpha) \pi(\alpha) \ud \alpha \nonumber \\
\leq& \frac{\sup_{\alpha\in [\underline\alpha_n,\overline\alpha_n]} B_n(\alpha)}{\sqrt{n-p}} \int_{\underline\alpha_n}^{\overline\alpha_n} \ee^{\Lcal_n(\alpha^{-2\nu}\widetilde\theta_{\alpha},\alpha)} \pi(\theta_0|\alpha) \pi(\alpha) \ud \alpha.
\end{align}
According to Assumption \ref{prior.2}, with $\epsilon_{1n}=6\theta_0n^{-1/2}\log n$ and $s_n=\log n$, $B_n(\alpha)$ as defined in \eqref{en:rho} satisfies that for all sufficiently large $n$,
\begin{align}\label{en:Bnrate}
& \sup_{\alpha\in [\underline\alpha_n,\overline\alpha_n]} B_n(\alpha) \nonumber \\
&\leq 4\theta_0 \exp\left(-\frac{n-p}{64}\right) +  \frac{\sqrt{n-p}}{\inf_{\alpha\in [\underline\alpha_n,\overline\alpha_n]} \pi(\theta_0|\alpha)} \exp\{-0.007(n-p)\} \nonumber \\
&+ \sup_{\alpha\in [\underline\alpha_n,\overline\alpha_n]} \sup_{\theta\in\left(\frac{1}{2}\theta_0,\frac{3}{2}\theta_0\right)} \frac{\pi\left(\theta |\alpha\right)}{ \pi\left(\theta_0|\alpha\right)}  \cdot 10\theta_0 \exp\left(-\frac{4\log^2 n}{125\theta_0^2}\right) + 4\theta_0 \exp\left(-\frac{\log^2 n}{4\theta_0^2}\right) \nonumber \\
&+ \frac{8}{\theta_0^2} \left(\frac{6\theta_0\log^3 n}{\sqrt{n}} + \frac{2\log^3 n}{\sqrt{n-p}}\right)\cdot \sup_{\alpha\in [\underline\alpha_n,\overline\alpha_n]} \sup_{\theta\in\left(\frac{3}{4}\theta_0,\frac{3}{2}\theta_0\right)} \frac{\pi\left(\theta |\alpha\right)}{ \pi\left(\theta_0|\alpha\right)} \nonumber \\
&+ 4\theta_0 \sup_{\alpha\in [\underline\alpha_n,\overline\alpha_n]} \sup_{\theta\in \left(\frac{3}{4}\theta_0,\frac{3}{2}\theta_0\right)} \left|\frac{\partial \log \pi(\theta|\alpha)}{\partial \theta} \right| \nonumber \\
&\times \sup_{\alpha\in [\underline\alpha_n,\overline\alpha_n]} \sup_{\theta\in\left(\frac{3}{4}\theta_0,\frac{3}{2}\theta_0\right)} \frac{\pi(\theta|\alpha)}{\pi(\theta_0|\alpha)} \left(\frac{6\theta_0\log n}{\sqrt{n}}+\frac{\log n}{\sqrt{n-p}}\right) \nonumber \\
&\leq 4\theta_0 \exp\left(-\frac{n-p}{64}\right)  + \exp\left(n^{C_{\pi,3}}\right) \cdot \sqrt{n} \exp\{-0.007(n-p)\} + n^{C_{\pi,2}} \cdot 10\theta_0 \exp\left(-\frac{4\log^2 n}{125\theta_0^2}\right)  \nonumber \\
&+ 4\theta_0 \exp\left(-\frac{\log^2 n}{4\theta_0^2}\right) + \frac{8(6\theta_0+2)}{\theta_0^2} \frac{\log^3 n}{\sqrt{n-p}} \cdot n^{C_{\pi,2}} + 4(6\theta_0+1)\theta_0 n^{C_{\pi,1}+C_{\pi,2}} \cdot \frac{\log n}{\sqrt{n-p}} \nonumber \\
&\rightarrow 0, \text{ as } n\to\infty,
\end{align}
where in the last step, we have used the fact that $C_{\pi,3}<1$ and $C_{\pi,1}+C_{\pi,2}<1/2$ according to Assumption \ref{prior.2}.

Therefore, \eqref{joint:theta22}, \eqref{en:Bnrate}, \eqref{numerator2}, and \eqref{denominator1} together imply that on the event \\
$\Ecal_2(6\theta_0n^{-1/2}\log n)$,
\begin{align}\label{en:bound1}
\frac{{\numer}_1}{\denom} & = \frac{2\int_{\underline\alpha_n}^{\overline\alpha_n} \int_{\RR} \left|\varrho_n(\sqrt{n-p}(\theta-\widetilde\theta_{\alpha});\alpha)\right| \ee^{\Lcal_n(\alpha^{-2\nu}\widetilde\theta_{\alpha},\alpha)} \pi(\theta_0|\alpha) \pi(\alpha) \ud \theta \ud \alpha} {\frac{2\theta_0\sqrt{\pi}}{\sqrt{n-p}}\int_0^{\infty} \ee^{\Lcal_n(\alpha^{-2\nu}\widetilde\theta_{\alpha},\alpha)} \pi(\theta_0|\alpha)\pi(\alpha) \ud \alpha} \nonumber \\
&\leq \frac{ \frac{\sup_{\alpha\in [\underline\alpha_n,\overline\alpha_n]} B_n(\alpha)}{\sqrt{n-p}} \int_{\underline\alpha_n}^{\overline\alpha_n} \ee^{\Lcal_n(\alpha^{-2\nu}\widetilde\theta_{\alpha},\alpha)} \pi(\theta_0|\alpha) \pi(\alpha) \ud \alpha}
{\frac{\theta_0\sqrt{\pi}}{\sqrt{n-p}}\int_{\underline\alpha_n}^{\overline\alpha_n} \ee^{\Lcal_n(\alpha^{-2\nu}\widetilde\theta_{\alpha},\alpha)} \pi(\theta_0|\alpha)\pi(\alpha) \ud \alpha} \nonumber \\
&\leq \frac{\sup_{\alpha\in [\underline\alpha_n,\overline\alpha_n]} B_n(\alpha)}{\theta_0\sqrt{\pi}} \rightarrow 0,
\end{align}
as $n\to\infty$. Since $\pr\{\Ecal_2(6\theta_0n^{-1/2}\log n)^c\}\leq 7\exp(-4\log^2 n)$ and $\sum_{n=1}^{\infty} 7\exp(-4\log^2 n) <\infty$, by the Borel-Cantelli lemma, we have shown that ${\numer}_1/\denom \to 0$ as $n\to\infty$ almost surely $P_{(\beta_0,\sigma_0^2,\alpha_0)}$.

\vspace{7mm}

\noindent \underline{Proof of ${\numer}_2/\denom \to 0$:}
\vspace{3mm}

We start with an upper bound for ${\numer}_2$:
\begin{align}\label{en:numer21}
{\numer}_2 &= 2\int_0^{\underline\alpha_n} \int_{\RR} \left| \ee^{\Lcal_n(\alpha^{-2\nu}\theta,\alpha)-\Lcal_n(\alpha^{-2\nu}\widetilde \theta_{\alpha},\alpha)} \pi\left(\theta | \alpha\right) - \ee^{-\frac{(n-p)(\theta-\widetilde\theta_{\alpha})^2}{4\theta_0^2}}\pi(\theta_0|\alpha)\right| \nonumber \\
&\qquad \times \ee^{\Lcal_n(\alpha^{-2\nu}\widetilde\theta_{\alpha},\alpha)} \pi(\alpha) \ud \theta \ud \alpha, \nonumber \\
&\leq 2\int_0^{\underline\alpha_n} \int_{\RR} \left( \ee^{\Lcal_n(\alpha^{-2\nu}\theta,\alpha)-\Lcal_n(\alpha^{-2\nu}\widetilde \theta_{\alpha},\alpha)} \pi\left(\theta | \alpha\right) + \ee^{-\frac{(n-p)(\theta-\widetilde\theta_{\alpha})^2}{4\theta_0^2}}\pi(\theta_0|\alpha)\right)\nonumber \\
&\qquad \times \ee^{\Lcal_n(\alpha^{-2\nu}\widetilde\theta_{\alpha},\alpha)} \pi(\alpha) \ud \theta \ud \alpha, \nonumber \\
&\stackrel{(i)}{\leq} 2\int_0^{\underline\alpha_n} \left\{\int_0^\infty \pi(\theta|\alpha) \ud \theta\right\} \ee^{\Lcal_n(\alpha^{-2\nu}\widetilde\theta_{\alpha},\alpha)} \pi(\alpha)  \ud \alpha \nonumber \\
&\qquad + 2\int_0^{\underline\alpha_n} \left\{\int_{\RR}  \ee^{-\frac{(n-p)(\theta-\widetilde\theta_{\alpha})^2}{4\theta_0^2}} \ud \theta\right\} \ee^{\Lcal_n(\alpha^{-2\nu}\widetilde\theta_{\alpha},\alpha)} \pi(\theta_0|\alpha)\pi(\alpha)  \ud \alpha \nonumber \\
&\leq 2\int_0^{\underline\alpha_n}  \ee^{\Lcal_n(\alpha^{-2\nu}\widetilde\theta_{\alpha},\alpha)} \pi(\alpha)  \ud \alpha + \frac{4\theta_0\sqrt{\pi}}{\sqrt{n-p}}\int_0^{\underline\alpha_n} \ee^{\Lcal_n(\alpha^{-2\nu}\widetilde\theta_{\alpha},\alpha)} \pi(\theta_0|\alpha)\pi(\alpha)  \ud \alpha,
\end{align}
where (i) follows from the fact that $\Lcal_n(\alpha^{-2\nu}\theta,\alpha) \leq \Lcal_n(\alpha^{-2\nu}\widetilde \theta_{\alpha},\alpha)$ as $\widetilde \theta_{\alpha}$ is the maximizer of $\Lcal_n(\alpha^{-2\nu}\theta,\alpha)$ given $\alpha$.

On the other hand, since $2\nu+d>1$, we choose $c=1>1/(2\nu+d)$ in Lemma \ref{lem:profilelk.rightlower}, and define $\Ecal_5$ to be the event that \eqref{diffpro1} in Lemma \ref{lem:profilelk.rightlower} happens, such that $\pr(\Ecal_5)\geq 1-9\exp(-4\log^2 n)$. Then on the event $\Ecal_5$, the denominator \eqref{denominator1} can be lower bounded by
\begin{align}\label{denominator2}
\denom & \geq \frac{2\theta_0\sqrt{\pi}}{\sqrt{n-p}} \ee^{\widetilde \Lcal_n(\alpha_0)} \int_{\alpha_0}^{(1+n^{-1})\alpha_0}  \ee^{\widetilde \Lcal_n(\alpha)-\widetilde \Lcal_n(\alpha_0)} \pi(\theta_0|\alpha) \pi(\alpha) \ud \alpha \nonumber \\
&\geq \frac{2\theta_0\sqrt{\pi}}{\sqrt{n-p}} \exp\left\{\widetilde \Lcal_n(\alpha_0)- 3 \log^4 n \right\} \int_{\alpha_0}^{(1+n^{-1})\alpha_0}   \pi(\theta_0|\alpha) \pi(\alpha) \ud \alpha  \nonumber \\
&\overset{(i)}{\geq} \frac{2\theta_0\sqrt{\pi}c_{\pi,0}}{n\sqrt{n-p}} \exp\left\{\widetilde \Lcal_n(\alpha_0)- 3\log^4 n \right\},
\end{align}
where $c_{\pi,0} = \pi(\theta_0|\alpha_0) \pi(\alpha_0)\cdot \alpha_0/4$, and the inequality (i) holds because by Assumptions \ref{prior.1} and \ref{prior.3}, $\pi(\theta_0|\alpha)>\pi(\theta_0|\alpha_0)/2 >0$ and $\pi(\alpha)>\pi(\alpha_0)/2>0$ for all $\alpha\in [\alpha_0,(1+n^{-1})\alpha_0]$ and sufficiently large $n$, such that $\int_{\alpha_0}^{(1+n^{-1})\alpha_0}   \pi(\theta_0|\alpha) \pi(\alpha) \ud \alpha \geq n^{-1}\alpha_0 \cdot \pi(\theta_0|\alpha_0) \pi(\alpha_0)/4 = c_{\pi,0} n^{-1}$.

We combine \eqref{en:numer21} and \eqref{denominator2} to obtain that
\begin{align}\label{en:numer22}
\frac{{\numer}_2}{\denom } &\leq \frac{n^{3/2}}{\theta_0\sqrt{\pi}c_{\pi,0}} \exp\left(3\log^4 n\right) \int_0^{\underline\alpha_n}  \ee^{\widetilde \Lcal_n(\alpha) - \widetilde \Lcal_n(\alpha_0)} \pi(\alpha)  \ud \alpha \nonumber \\
&\quad + \frac{2n}{c_{\pi,0}}\exp\left(3\log^4 n\right)\int_0^{\underline\alpha_n} \ee^{\widetilde \Lcal_n(\alpha) - \widetilde \Lcal_n(\alpha_0)} \pi(\theta_0|\alpha)\pi(\alpha)  \ud \alpha.
\end{align}
To upper bound the two terms in \eqref{en:numer22}, we first derive a simple relation for the part $\exp\{\widetilde \Lcal_n(\alpha) - \widetilde \Lcal_n(\alpha_0)\}$. Let $\Ecal_6$ be the event on which \eqref{diffpro2} in Lemma \ref{lem:profilelk.leftupper} happens, such that $\pr(\Ecal_6)\geq 1-10\exp(-4\log^2 n)$ for sufficiently large $n$. On the event $\Ecal_6$, the monotonicity bound from Lemma \ref{lem:prolik_loosebound} and the upper bound from Lemma \ref{lem:profilelk.leftupper} imply that for any $\alpha\in (0,\underline\alpha_n)$,
\begin{align}\label{en:prolik.left1}
&\exp\left\{\widetilde \Lcal_n(\alpha) - \widetilde \Lcal_n(\alpha_0)\right\} \nonumber \\
={} & \exp\left\{\widetilde \Lcal_n(\alpha) - \widetilde \Lcal_n(\underline\alpha_n)\right\} \cdot \exp\left\{\widetilde \Lcal_n(\underline\alpha_n) - \widetilde \Lcal_n(\alpha_0)\right\} \nonumber \\
<{}& \left(\frac{\underline\alpha_n}{\alpha}\right)^{n(\nu+d/2)} \exp\left(3n^{1/2-\tau}\right) \nonumber \\
={}& \alpha^{-n(\nu+d/2)} \exp\left\{-(\nu+d/2)\underkappa n\log n + 3n^{1/2-\tau}\right\},
\end{align}
where $\tau\in (0,1/2)$ and $\underkappa\in (0,1/2)$ are defined in \eqref{eq:2kappa.re}. Since $3\log^4 n / n^{1/2-\tau} \to 0$ as $n\to\infty$, we now plug \eqref{en:prolik.left1} in \eqref{en:numer22} and use Assumption \ref{prior.3} to obtain that on the event $\Ecal_5\cap \Ecal_6$,
\begin{align}\label{en:numer23}
\frac{{\numer}_2}{\denom } &\leq \frac{n^{3/2}}{\theta_0\sqrt{\pi}c_{\pi,0}} \exp\left\{-(\nu+d/2)\underkappa n\log n + 4n^{1/2-\tau}\right\}\int_0^{\underline\alpha_n}  \alpha^{-n(\nu+d/2)} \pi(\alpha)  \ud \alpha \nonumber \\
&\quad + \frac{2n}{c_{\pi,0}}\exp\left\{-(\nu+d/2)\underkappa n\log n + 4n^{1/2-\tau}\right\} \int_0^{\underline\alpha_n} \alpha^{-n(\nu+d/2)} \pi(\theta_0|\alpha)\pi(\alpha)  \ud \alpha \nonumber \\
&\leq \frac{n^{3/2}}{\theta_0\sqrt{\pi}c_{\pi,0}} \exp\left\{-(\nu+d/2)\underkappa n\log n + 4n^{1/2-\tau} + \underline{c_{\pi}}n\log n\right\}  \nonumber \\
&\quad + \frac{2n}{c_{\pi,0}}\exp\left\{-(\nu+d/2)\underkappa n\log n + 4n^{1/2-\tau} + \underline{c_{\pi}}n\log n\right\}  \nonumber \\
&\rightarrow 0,  \text{ as } n\to\infty,
\end{align}
where the last step follows because $\underline{c_{\pi}}<(\nu+d/2)\underkappa$ by Assumption \ref{prior.3} and $\tau\in (0,1/2)$. Since $\pr\{(\Ecal_5\cap \Ecal_6)^c\}\leq 20\exp(-4\log^2 n)$ and $\sum_{n=1}^{\infty} 20\exp(-4\log^2 n) <\infty$, by the Borel-Cantelli lemma, we have shown that ${\numer}_2/\denom \to 0$ as $n\to\infty$ almost surely $P_{(\beta_0,\sigma_0^2,\alpha_0)}$.
\vspace{5mm}

\noindent \underline{Proof of ${\numer}_3/\denom \to 0$:}
\vspace{3mm}

Similar to the derivation of \eqref{en:numer21}, we have the following upper bound for ${\numer}_3$:
\begin{align}\label{en:numer31}
{\numer}_3
&\leq 2\int_{\overline\alpha_n}^{\infty}  \ee^{\Lcal_n(\alpha^{-2\nu}\widetilde\theta_{\alpha},\alpha)} \pi(\alpha)  \ud \alpha + \frac{4\theta_0\sqrt{\pi}}{\sqrt{n-p}}\int_{\overline\alpha_n}^{\infty} \ee^{\Lcal_n(\alpha^{-2\nu}\widetilde\theta_{\alpha},\alpha)} \pi(\theta_0|\alpha)\pi(\alpha)  \ud \alpha.
\end{align}
\eqref{denominator2} and \eqref{en:numer31} imply that on the event $\Ecal_5$,
\begin{align}\label{en:numer32}
\frac{{\numer}_3}{\denom } &\leq \frac{n^{3/2}}{\theta_0\sqrt{\pi}c_{\pi,0}} \exp\left(3\log^4 n\right)\int_{\overline\alpha_n}^{\infty}  \ee^{\widetilde \Lcal_n(\alpha) - \widetilde \Lcal_n(\alpha_0)} \pi(\alpha)  \ud \alpha \nonumber \\
&\quad + \frac{2n}{c_{\pi,0}}\exp\left(3\log^4 n\right)\int_{\overline\alpha_n}^{\infty} \ee^{\widetilde \Lcal_n(\alpha) - \widetilde \Lcal_n(\alpha_0)} \pi(\theta_0|\alpha)\pi(\alpha)  \ud \alpha.
\end{align}
Let $\Ecal_7$ be the event on which \eqref{diffpro3} in Lemma \ref{lem:profilelk.rightupper} happens, such that $\pr(\Ecal_7)\geq 1-10\exp(-4\log^2 n)$ for sufficiently large $n$. Similar to the proof of ${\numer}_2/\denom  \to 0$, on the event $\Ecal_7$, we use Lemma \ref{lem:prolik_loosebound} and Lemma \ref{lem:profilelk.rightupper} to obtain that for any $\alpha\in (\overline\alpha_n,+\infty)$,
\begin{align}\label{en:prolik.right1}
&\exp\left\{\widetilde \Lcal_n(\alpha) - \widetilde \Lcal_n(\alpha_0)\right\} \nonumber \\
={} & \exp\left\{\widetilde \Lcal_n(\alpha) - \widetilde \Lcal_n(\overline\alpha_n)\right\} \cdot \exp\left\{\widetilde \Lcal_n(\overline\alpha_n) - \widetilde \Lcal_n(\alpha_0)\right\} \nonumber \\
<{}& \left(\frac{\alpha}{\overline\alpha_n}\right)^{n(\nu+d/2)} \exp\left(C_{p,1}n^{\kappa_1}\log n\right) \nonumber \\
={}& \alpha^{n(\nu+d/2)} \exp\left\{-(\nu+d/2)\overline \kappa n\log n + C_{p,1}n^{\kappa_1}\log n\right\},
\end{align}
where $C_{p,1}>0$ and $\kappa_1 \in (1/2-\tau,1)$ are given in Lemma \ref{lem:profilelk.rightupper}, and $\overkappa\in (0,1/2)$ is given in \eqref{eq:2kappa.re}. Since $3\log^4 n/ (C_{p,1}n^{\kappa_1}\log n) \to 0 $ as $n\to\infty$, we now plug \eqref{en:prolik.right1} in \eqref{en:numer32} and use Assumption \ref{prior.3} to obtain that on the event $\Ecal_5\cap \Ecal_7$,
\begin{align}\label{en:numer33}
\frac{{\numer}_3}{\denom }
&\leq \frac{n^{3/2}}{\theta_0\sqrt{\pi}c_{\pi,0}} \exp\left\{-(\nu+d/2)\overline \kappa n\log n + 2C_{p,1}n^{\kappa_1}\log n\right\} \times \int_{\overline\alpha_n}^{\infty}  \alpha^{n(\nu+d/2)} \pi(\alpha)  \ud \alpha \nonumber \\
&\quad + \frac{2n}{c_{\pi,0}}\exp\left\{-(\nu+d/2)\overline \kappa n\log n + 2C_{p,1}n^{\kappa_1}\log n \right\} \times \int_{\overline\alpha_n}^{\infty} \alpha^{n(\nu+d/2)} \pi(\theta_0|\alpha)\pi(\alpha)  \ud \alpha \nonumber \\
&\leq \frac{n^{3/2}}{\theta_0\sqrt{\pi}c_{\pi,0}} \exp\left\{-(\nu+d/2)\overline \kappa n\log n + 2C_{p,1}n^{\kappa_1}\log n + \overline{c_{\pi}}n\log n\right\} \nonumber \\
&\quad + \frac{2n}{c_{\pi,0}}\exp\left\{-(\nu+d/2)\overline \kappa n\log n + 2C_{p,1}n^{\kappa_1}\log n + \overline{c_{\pi}}n\log n \right\}  \nonumber \\
&\rightarrow 0, \text{ as } n\to\infty,
\end{align}
where the last step follows because $\overline{c_{\pi}}<(\nu+d/2)\overline \kappa$ by Assumption \ref{prior.3} and $\kappa_1 \in (1/2-\tau,1)$. Since $\pr\{(\Ecal_5\cap\Ecal_7)^c\}\leq 20\exp(-4\log^2 n)$ and $\sum_{n=1}^{\infty} 20\exp(-4\log^2 n) <\infty$, by the Borel-Cantelli lemma, we have shown that ${\numer}_3/\denom \to 0$ as $n\to\infty$ almost surely $P_{(\beta_0,\sigma_0^2,\alpha_0)}$.

\vspace{5mm}

\noindent \underline{Proof of \eqref{joint:theta3}:}

We use Lemma \ref{lemma:normaltv} and obtain that
\begin{align}\label{joint:theta31}
& \int_0^{\infty} \int_{\RR} \left|\frac{\sqrt{n-p}}{2\sqrt{\pi}\theta_0} \ee^{-\frac{(n-p)(\theta-\widetilde\theta_{\alpha})^2}{4\theta_0^2}} - \frac{\sqrt{n}}{2\sqrt{\pi}\theta_0} \ee^{-\frac{n(\theta-\widetilde\theta_{\alpha_0})^2}{4\theta_0^2}} \right| \cdot \widetilde \pi(\alpha|Y_n) \ud \theta \ud \alpha \nonumber \\
= & \int_0^{\infty} \left\|\mathcal{N}(\widetilde\theta_{\alpha},2\theta_0^2/(n-p)) -   \mathcal{N}(\widetilde\theta_{\alpha_0},2\theta_0^2/n) \right\|_{\tv} \widetilde\pi(\alpha|Y_n) \ud \alpha \nonumber \\
\stackrel{(i)}{\leq} & \int_0^{\infty} \left\|\mathcal{N}(\widetilde\theta_{\alpha},2\theta_0^2/(n-p)) - \mathcal{N}(\widetilde\theta_{\alpha_0},2\theta_0^2/(n-p)) \right\|_{\tv} \widetilde\pi(\alpha|Y_n) \ud \alpha \nonumber \\
&\quad + \int_0^{\infty} \left\|\mathcal{N}(\widetilde\theta_{\alpha_0},2\theta_0^2/(n-p)) - \mathcal{N}(\widetilde\theta_{\alpha_0},2\theta_0^2/n) \right\|_{\tv} \widetilde\pi(\alpha|Y_n) \ud \alpha \nonumber \\
\stackrel{(ii)}{\leq} & \int_0^{\infty} \left\{2\Phi\left(\frac{(n-p)^{1/2}|\widetilde\theta_{\alpha}-\widetilde\theta_{\alpha_0}|}{2\sqrt{2}\theta_0}\right)-1\right\} \widetilde \pi(\alpha|Y_n) \ud \alpha \nonumber \\
&\quad + \int_0^{\infty}  \frac{3}{2}\cdot \frac{2\theta_0^2/(n-p) - 2\theta_0^2/n}{2\theta_0^2/n}\widetilde \pi(\alpha|Y_n) \ud \alpha \nonumber \\
\stackrel{(iii)}{\leq} & \int_{\underline\alpha_n}^{\overline\alpha_n} \frac{(n-p)^{1/2}}{2\sqrt{\pi}\theta_0}|\widetilde\theta_{\alpha}-\widetilde\theta_{\alpha_0}|\widetilde \pi(\alpha|Y_n) \ud \alpha \nonumber \\
& + \int_0^{\underline\alpha_n} \widetilde \pi(\alpha|Y_n) \ud \alpha + \int_{\overline\alpha_n}^{\infty} \widetilde \pi(\alpha|Y_n) \ud \alpha + \frac{3p}{2(n-p)},
\end{align}
where (i) follows from the triangle inequality of total variation distance; (ii) follows from Lemma \ref{lemma:normaltv} and Theorem 1.3 of \citet{Devetal18}; for (iii), we use the relation $\Phi(x)-0.5 = \Phi(x)-\Phi(0) \leq \phi(0)x =x/\sqrt{2\pi}$ for all $x\geq 0$ (where $\phi(x)$ is the standard normal density), and the direct bound $|2\Phi(x)-1|\leq 1$ for all $x\in \RR$.

On the event $\Ecal_3(\theta_0n^{-1/2-\tau}/2)$, we have that $n^{1/2}|\widetilde\theta_{\alpha}-\widetilde\theta_{\alpha_0}| \leq \theta_0 n^{-\tau}/2$ uniformly for all $\alpha\in [\underline \alpha_n, \overline \alpha_n]$. Together with the fact that $\widetilde \pi(\alpha|Y_n)$ is almost surely a proper probability density from Lemma \ref{lem:alpha.exist}, we can derive from \eqref{joint:theta31} that on the event $\Ecal_3(\theta_0n^{-1/2-\tau}/2) $,
\begin{align}\label{joint:theta32}
& \int_{\underline\alpha_n}^{\overline\alpha_n} \frac{(n-p)^{1/2}}{2\sqrt{\pi}\theta_0}|\widetilde\theta_{\alpha}-\widetilde\theta_{\alpha_0}|\widetilde \pi(\alpha|Y_n) \ud \alpha
\leq  \frac{n^{-\tau}}{4\sqrt{\pi}} \int_0^{\infty} \widetilde \pi(\alpha|Y_n) \ud \alpha \leq \frac{n^{-\tau}}{4\sqrt{\pi}} \to 0,
\end{align}
as $n\to \infty$. Since $\pr\left\{\Ecal_3(\theta_0n^{-1/2-\tau}/2)^c\right\}\leq 4\exp(-4\log^2 n)$ and $\sum_{n=1}^{\infty} 4\exp(-4\log^2 n) <\infty$, by the Borel-Cantelli lemma, we have shown that \eqref{joint:theta32} holds as $n\to \infty$ almost surely $P_{(\beta_0,\sigma_0^2,\alpha_0)}$.

For the second term on the right-hand side of \eqref{joint:theta31}, we have that by the definition \eqref{profile:post1},
\begin{align*}
 \int_0^{\underline\alpha_n} \widetilde \pi(\alpha|Y_n) \ud \alpha
&\leq \frac{\int_0^{\underline\alpha_n} \ee^{\widetilde \Lcal_n(\alpha)- \widetilde \Lcal_n(\alpha_0)} \pi(\theta_0|\alpha) \pi(\alpha) \ud \alpha}{\int_{\alpha_0}^{(1+n^{-1})\alpha_0} \ee^{\widetilde \Lcal_n(\alpha)- \widetilde \Lcal_n(\alpha_0)} \pi(\theta_0|\alpha) \pi(\alpha) \ud \alpha}.
\end{align*}
The denominator is lower bounded by $c_{\pi,0}n^{-1}\exp(-3\log^4 n)$ on the event $\Ecal_5$, similar to the proof of \eqref{denominator2}. The numerator can be upper bounded on the event $\Ecal_6$, using the same derivation as in \eqref{en:numer22} and \eqref{en:prolik.left1}. As a result, on the event $\Ecal_5\cap \Ecal_6$, using $\underline{c_{\pi}}<(\nu+d/2)\underkappa$ in Assumption \ref{prior.3}, we have that
\begin{align} \label{joint:theta33}
\int_0^{\underline\alpha_n} \widetilde \pi(\alpha|Y_n) \ud \alpha
&\leq \frac{\exp\left\{-(\nu+d/2)\underkappa n\log n + 3n^{1/2-\tau} \right\}\int_0^{\underline\alpha_n}  \alpha^{-n(\nu+d/2)} \pi(\alpha)  \ud \alpha}{c_{\pi,0}n^{-1}\exp(-3\log^4 n)} \nonumber \\
&\leq \frac{n}{c_{\pi,0}} \exp\left\{-(\nu+d/2)\underkappa n\log n+ 3n^{1/2-\tau} + \underline{c_{\pi}}n\log n + 3\log^4 n \right\}\nonumber \\
& \rightarrow 0,  \text{ as } n\to\infty.
\end{align}
\eqref{joint:theta33} holds as $n\to\infty$ almost surely $P_{(\beta_0,\sigma_0^2,\alpha_0)}$ since $\pr\{(\Ecal_5\cap \Ecal_6)^c\}\leq 20\exp(-4\log^2 n)$ and $\sum_{n=1}^{\infty} 20\exp(-4\log^2 n) <\infty$.

Similarly, for the third term on the right-hand side of \eqref{joint:theta31}, we have that by the definition \eqref{profile:post1},
\begin{align*}
\int_{\overline\alpha_n}^{\infty} \widetilde \pi(\alpha|Y_n) \ud \alpha
&\leq \frac{\int_{\overline\alpha_n}^{\infty} \ee^{\widetilde \Lcal_n(\alpha)- \widetilde \Lcal_n(\alpha_0)} \pi(\theta_0|\alpha) \pi(\alpha) \ud \alpha}{\int_{\alpha_0}^{(1+n^{-1})\alpha_0} \ee^{\widetilde \Lcal_n(\alpha)- \widetilde \Lcal_n(\alpha_0)} \pi(\theta_0|\alpha) \pi(\alpha) \ud \alpha}.
\end{align*}
On the event $\Ecal_5\cap \Ecal_7$, the denominator is lower bounded by $c_{\pi,0}n^{-1}\exp(-3\log^4 n)$, and the numerator can be upper bounded using the same derivation as in \eqref{en:prolik.right1} and \eqref{en:numer33}. As a result, using $\overline{c_{\pi}}<(\nu+d/2)\overkappa$ in Assumption \ref{prior.3}, we have that on $\Ecal_5\cap \Ecal_7$,
\begin{align} \label{joint:theta34}
&~~~~ \int_{\overline\alpha_n}^{\infty} \widetilde \pi(\alpha|Y_n) \ud \alpha \nonumber \\
&\leq \frac{\exp\left\{-(\nu+d/2)\overline \kappa n\log n + C_{p,1}n^{\kappa_1}\log n \right\}\int_{\overline\alpha_n}^{\infty}  \alpha^{n(\nu+d/2)} \pi(\alpha)  \ud \alpha}{c_{\pi,0}n^{-1}\exp(-3\log^4 n)} \nonumber \\
&\leq \frac{n}{c_{\pi,0}} \exp\left\{-(\nu+d/2)\overline \kappa n\log n + C_{p,1}n^{\kappa_1}\log n + \overline{c_{\pi}}n\log n + 3\log^4 n\right\}\nonumber \\
& \rightarrow 0, \text{ as } n\to\infty.
\end{align}
\eqref{joint:theta34} holds as $n\to\infty$ almost surely $P_{(\beta_0,\sigma_0^2,\alpha_0)}$ since $\pr\{(\Ecal_5\cap\Ecal_7)^c\}\leq 20\exp(-4\log^2 n)$ and $\sum_{n=1}^{\infty} 20\exp(-4\log^2 n) <\infty$.

Finally, \eqref{joint:theta32}, \eqref{joint:theta33}, and \eqref{joint:theta34} together imply that the right-hand side of \eqref{joint:theta31} converges to zero as $n\to\infty$ almost surely $P_{(\beta_0,\sigma_0^2,\alpha_0)}$. This has proved \eqref{joint:theta3}, and hence has completed the proof of Theorem \ref{thm:bvm2:joint}.
\end{proof}

\subsection{Limiting Posterior Distribution When $d\geq 5$} \label{subsec:d5}

We present a theorem for the limiting posterior distribution of $(\theta,\alpha)$ when the domain dimension $d\geq 5$ in the universal kriging model \eqref{eq:obs.model} with the isotropic Mat\'ern covariance function \eqref{eq:MaternCov}. The theorem is similar to Theorem \ref{thm:bvm2:joint} for the case of $d\in\{1,2,3\}$ but requires more assumptions and has some important difference in its proof from that of Theorem \ref{thm:bvm2:joint}, mainly because that the range parameter $\alpha$ can be consistently estimated for $d\geq 5$ (\citet{And10}).

For any $\epsilon_1>0, \epsilon_2>0$, we define the set
\begin{align} \label{eq:Bcal0}
\Bcal_0(\epsilon_1,\epsilon_2) &= \Big\{(\beta,\theta,\alpha)\in \RR^p\times \RR^+ \times \RR^+:  |\theta/\theta_0 - 1|<\epsilon_1, |\alpha/\alpha_0 -1|<\epsilon_2 \Big\}.
\end{align}
This set can be viewed as a neighborhood of $(\theta_0,\alpha_0)$. For the case of $d\geq 5$, the following assumptions will replace Assumption \ref{prior.3} in the main text for the case of $d\in \{1,2,3\}$.
\begin{enumerate}[label=(S.\arabic*)]
\item \label{assump.alpha.est}
For the model \eqref{eq:obs.model} with isotropic Mat\'ern covariance function in \eqref{eq:MaternCov} with $d\geq 5$, there exist constants $0<\kappa_1'\leq 1/2$, $1/(2\nu+d)< \kappa_2'\leq 1/2$, $c_{5}>0$ and consistent estimators $\widehat\theta_n$ for $\theta$ and $\widehat \alpha_n$ for $\alpha$ based on $(Y_n,M_n)$, such that for any $\epsilon_1>0$ and $\epsilon_2>0$,
\begin{align} \label{eq:exp.test.alpha}
P_{(\beta_0,\sigma_0^2,\alpha_0)} \left(\left|\widehat\theta_n/\theta_0-1\right| \geq \epsilon_1 /2 \right) & \leq \exp\left\{- c_{5}  \min\big[n^{\kappa_1'}\epsilon_1, (n^{\kappa_1'}\epsilon_1)^2\big] \right\}, \nonumber \\
P_{(\beta_0,\sigma_0^2,\alpha_0)} \left(\left|\widehat\alpha_n/\alpha_0-1\right| \geq \epsilon_2 /2 \right) & \leq \exp\left\{- c_{5} \min\big[n^{\kappa_2'}\epsilon_2, (n^{\kappa_2'}\epsilon_2)^2\big] \right\}, \nonumber \\
\sup_{\Bcal_0(\epsilon_1,\epsilon_2)^c \cap \Fcal_n} P_{(\beta,\theta/\alpha^{2\nu},\alpha)} \left( \left|\widehat\theta_n/\theta_0-1\right| \leq \epsilon_1 /2 \right) & \leq \exp\left\{ - c_{5} \min\big[n^{\kappa_1'}\epsilon_1, (n^{\kappa_1'}\epsilon_1)^2\big] \right\}, \nonumber \\
\sup_{\Bcal_0(\epsilon_1,\epsilon_2)^c\cap \Fcal_n} P_{(\beta,\theta/\alpha^{2\nu},\alpha)} \left( \left|\widehat\alpha_n/\alpha_0-1\right| \leq \epsilon_2 /2 \right) & \leq \exp\left\{ - c_{5} \min\big[n^{\kappa_2'}\epsilon_2, (n^{\kappa_2'}\epsilon_2)^2\big] \right\}  ,
\end{align}
where the sieve $\Fcal_n\subseteq \left\{(\beta,\theta,\alpha)\in \RR^p\times \RR^+ \times \RR^+ \right\}$, such that the prior satisfies $\Pi(\Fcal_n^c) \leq n^{-(3p+6)}$ for all sufficiently large $n$.
\end{enumerate}
Assumption \ref{assump.alpha.est} requires the existence of consistent estimators $\widehat\theta_n$ and $\widehat\alpha_n$. The exponentially small tail bounds in the inequalities in \eqref{eq:exp.test.alpha} imply the convergence rates of $O(n^{-\kappa_1'})$ and $O(n^{-\kappa_2'})$ for $\widehat\theta_n$ and $\widehat\alpha_n$, respectively. The inequalities in \eqref{eq:exp.test.alpha} will be used to construct exponentially consistent tests for $\theta$ and $\alpha$, which are commonly used for showing the posterior consistency and posterior contraction rates in the Bayesian nonparametrics literature; see for example, Sections 6.4 and 8.2 in \citet{GhoVan17}.

Since Assumption \ref{assump.alpha.est} is a high level condition, we explain why such estimators $\widehat\theta_n$ and $\widehat\alpha_n$ exist for the isotropic Mat\'ern covariance function with $d\geq 5$. To the best of our knowledge, \citet{And10} is the only work that has systematically studied the fixed-domain asymptotics for the isotropic Mat\'ern covariance function with domain dimension $d\geq 5$. \citet{And10} has considered a special case of our model \eqref{eq:obs.model}, in which (i) $Y(\cdot)$ is a GP with mean zero and no regression terms $\bbm(\cdot)^\top\beta$, and (ii) the sampling location set $\Scal_n$ consists of equispaced grids in a fixed domain. For this special case, \citet{And10} proposed consistent moment estimators for both $\theta$ and $\alpha$ when $d\geq 5$ if we set their $M$ matrix to be the identity matrix; see their Theorem 1, Theorem 2, and the discussion after the two theorems. The proofs of Theorems 1 and 2 in \citet{And10} have derived tail bound inequalities similar to \eqref{eq:exp.test.alpha}, where both $\kappa_1'$ and $\kappa_2'$ can be taken as $1/2$, which satisfies our condition $0<\kappa_1'\leq 1/2$ and $1/(2\nu+d)< \kappa_2'\leq 1/2$ since $1/(2\nu+d)<1/5$ when $d\geq 5$.

The supremum in the inequalities of \eqref{eq:exp.test.alpha} can often be established using a union bound argument over the set $\Bcal_0(\epsilon_1,\epsilon_2)^c\cap\Fcal_n$. The parameter set $\Fcal_n$ in Assumption \ref{assump.alpha.est} is typically a bounded set whose radius increases slowly with $n$, such that it is a sieve to the whole parameter space of $\left\{(\beta,\theta,\alpha)\in \RR^p\times \RR^+ \times \RR^+ \right\}$. The supremum inequalities and the sieve are also commonly used in Bayesian nonparametrics for showing posterior consistency and contraction rates; see for example, Theorem 6.17, Theorem 8.9 and their proofs in \citet{GhoVan17}. We assume that the prior mass outside the sieve $\Fcal_n$ is polynomially small, which is usually satisfied if $\beta$ is assigned a normal prior and $(\theta,\alpha)$ are assigned the priors described in Section \ref{subsec:prior.assumption}. In Bayesian nonparametrics, it is often assumed that $\Pi(\Fcal_n^c)$ is exponentially small in $n$, so our assumption is weaker in comparison.

Although Assumption \ref{assump.alpha.est} is currently verifiable only for the special case considered in \citet{And10}, we expect that the inequalities in \eqref{eq:exp.test.alpha} continue to hold for more general sampling designs and the model with regression terms in the case of $d \geq 5$, where the two constants $\kappa_1'$ and $\kappa_2'$ can be possibly smaller than $1/2$ depending on the sampling designs. Detailed construction of such consistent estimators $\widehat\theta_n$ and $\widehat\alpha_n$ for $d\geq 5$ in the general universal kriging model \eqref{eq:obs.model} can be based on the recently proposed higher-order quadratic variation techniques in \citet{Loh15} and \citet{Lohetal20} and will be left for future investigation.

Before stating the main theorem for $d\geq 5$, we first prove two technical lemmas. Lemma \ref{lem:d5.denom.lower} can be used to show a theoretical lower bound of the denominator in the posterior distribution for $d\geq 5$. Lemma \ref{lem:exp.test.alpha} proves the posterior contraction for $(\theta,\alpha)$ for $d\geq 5$. This will be used later for truncating the posterior to a shrinking neighborhood of $(\theta_0,\alpha_0)$, which will be important for deriving the limiting posterior distribution for $d\geq 5$.
\begin{lemma} \label{lem:d5.denom.lower}
Suppose that Assumptions \ref{assump.m.func} holds for $d\geq 5$ and $\nu\in\RR^+$. Let
\begin{align*}
\Acal_{n}^{\dagger}&= \big\{(\beta,\theta,\alpha)\in \RR^p\times \RR^+ \times \RR^+:~\|\beta-\beta_0\|\leq n^{-3}, \nonumber \\
&\qquad \theta_0 \leq \theta < \theta_0(1+n^{-2}),~ \alpha_0(1- n^{-2})\leq \alpha \leq \alpha_0 \big\}.
\end{align*}
Then $\inf_{\Acal_{n}^{\dagger}} \left\{ \Lcal_n(\beta,\theta/\alpha^{2\nu},\alpha) - \Lcal_n(\beta_0,\sigma_0^2,\alpha_0) \right\} \geq  c_{5L} n^{-1}$ with probability at least $1-\exp(-16\log^2 n)$ for all sufficiently large $n$, where $c_{5L}>0$ is a constant that depends on $\nu,d,T,\beta_0,\sigma_0^2,\alpha_0$ and the $\Wcal_2^{\nu+d/2}(\Scal)$ norms of $\bbm_1(\cdot),\ldots,\bbm_p(\cdot)$.
\end{lemma}

\begin{proof}[Proof of Lemma \ref{lem:d5.denom.lower}]
By definition of the log-likelihood function $\Lcal_n(\beta,\sigma^2,\alpha)$ in \eqref{eq:loglik} and the true model $Y_n=M_n\beta_0 + X_n$, we have
\begin{align} \label{eq:loglik.decomp1}
&\Lcal_n(\beta,\theta/\alpha^{2\nu},\alpha) - \Lcal_n(\beta_0,\sigma_0^2,\alpha_0) = -\frac{n}{2}\log \frac{\theta}{\theta_0} + \nu n\log \frac{\alpha}{\alpha_0} - \frac{1}{2}\log \frac{|R_{\alpha}|}{|R_{\alpha_0}|} \nonumber \\
&\qquad - (Y_n-M_n\beta)^\top \left(\frac{\alpha^{2\nu} R_{\alpha}^{-1}}{2\theta} - \frac{\alpha_0^{2\nu} R_{\alpha_0}^{-1}}{2\theta_0} \right) (Y_n-M_n\beta)  \nonumber \\
&\qquad + (\beta-\beta_0)^\top M_n^\top \frac{\alpha_0^{2\nu} R_{\alpha_0}^{-1}}{\theta_0} X_n - (\beta-\beta_0)^\top M_n^\top \frac{\alpha_0^{2\nu} R_{\alpha_0}^{-1}}{2\theta_0} M_n (\beta-\beta_0).
\end{align}

On the right-hand side of \eqref{eq:loglik.decomp1}, using Lemma \ref{lem:URU} and Lemma \ref{lem:specden_lambda}, the first line can be lower bounded as follows in the set $\Acal_{n}^{\dagger}$ for all sufficiently large $n$:
\begin{align} \label{eq:loglik.term1.lower}
&\quad~  -\frac{n}{2}\log \frac{\theta}{\theta_0} + \nu n\log \frac{\alpha}{\alpha_0} - \frac{1}{2}\log \frac{|R_{\alpha}|}{|R_{\alpha_0}|} \nonumber \\
&\overset{(i)}{\geq} -\frac{n}{2} \log(1+n^{-2}) + \nu n \log \frac{\alpha}{\alpha_0} - \frac{1}{2}\sum_{k=1}^n \log \lambda_{k,n}(\alpha) - \nu p\log\frac{\alpha}{\alpha_0} \nonumber \\
&\overset{(ii)}{\geq} -\frac{n}{2} \log(1+n^{-2}) +\nu(n-p)\log(1-n^{-2}) + \frac{2\nu+d}{2}\sum_{k=1}^n \log\big(1-n^{-2}\big) \nonumber \\
&\overset{(iii)}{\geq} -\frac{1}{2n} - \frac{2\nu}{n} - \frac{2\nu+d}{n} ,
\end{align}
where (i) follows from Lemma \ref{lem:URU}, (ii) follows from \eqref{lambda.upper1} in Lemma \ref{lem:specden_lambda} given that $\alpha\leq \alpha_0$ on $\Acal_{n}^{\dagger}$, and (iii) follows from $\log(1+x)\leq x$ and $\log (1-x) \geq -2x$ for $x\in (0,1/2)$.

By Lemma \ref{lem:alpha.monotone.matrix}, in the set $\Acal_{n}^{\dagger}$, $\frac{\alpha^{2\nu} R_{\alpha}^{-1}}{2\theta} - \frac{\alpha_0^{2\nu} R_{\alpha_0}^{-1}}{2\theta_0}$ is negative definite. Therefore,
\begin{align} \label{eq:loglik.term2.lower}
& - (Y_n-M_n\beta)^\top \left(\frac{\alpha^{2\nu} R_{\alpha}^{-1}}{2\theta} - \frac{\alpha_0^{2\nu} R_{\alpha_0}^{-1}}{2\theta_0} \right) (Y_n-M_n\beta) \geq 0.
\end{align}

For the third line of \eqref{eq:loglik.decomp1}, since $X_n\sim \Ncal(0,\sigma_0^2 R_{\alpha_0})$, by Lemma \ref{lem:LauMas00}, $\pr(\|\sigma_0^{-1}R_{\alpha_0}^{-1/2} X_n\|^2 \leq n+ 8\log n + 16\log^2 n ) \geq 1-\exp(-16\log^2 n)$. By Assumption \ref{assump.m.func}, using similar derivation to \eqref{eq:theta3.under3} based on Lemmas \ref{lem:matern.sobolev}, \ref{lem:rkhs.quadratic} and \ref{lem:rkhs.ordering}, we have that on the set $\Acal_{n}^{\dagger}$,
\begin{align*}
& \big\|\sigma_0^{-1}R_{\alpha_0}^{-1/2}M_n(\beta-\beta_0)\big\|^2 \leq \sum_{j=1}^p \|\bbm_j\|_{\Hcal_{\sigma_0^2K_{\alpha_0,\nu}}}^2 \cdot \|\beta-\beta_0\|^2 \leq c_2(\sigma_0,\alpha_0)^2 \sum_{j=1}^p \|\bbm_j\|_{\Wcal_2^{\nu+d/2}(\Scal)}^2 \cdot n^{-6}.
\end{align*}
Therefore, by Cauchy-Schwarz inequality, with probability at least $1-\exp(-16\log^2 n)$,
\begin{align} \label{eq:loglik.term3.lower}
&\quad~  (\beta-\beta_0)^\top M_n^\top \frac{\alpha_0^{2\nu} R_{\alpha_0}^{-1}}{\theta_0} X_n - (\beta-\beta_0)^\top M_n^\top \frac{\alpha_0^{2\nu} R_{\alpha_0}^{-1}}{2\theta_0} M_n (\beta-\beta_0) \nonumber \\
&\geq - (n+ 8\log n + 16\log^2 n)^{1/2} c_2(\sigma_0,\alpha_0)\left(\sum_{j=1}^p \|\bbm_j\|_{\Wcal_2^{\nu+d/2}(\Scal)}^2\right)^{1/2} n^{-3} \nonumber\\
&\quad ~ - \frac{c_2(\sigma_0,\alpha_0)^2}{2} \sum_{j=1}^p \|\bbm_j\|_{\Wcal_2^{\nu+d/2}(\Scal)}^2 \cdot n^{-6} \nonumber \\
&\geq - c_1'n^{-2},
\end{align}
where $c_1'>0$ is a constant dependent on the $\Wcal_2^{\nu+d/2}(\Scal)$ norms of $\bbm_1(\cdot),\ldots,\bbm_p(\cdot)$.

Finally, we combine \eqref{eq:loglik.decomp1}, \eqref{eq:loglik.term1.lower}, \eqref{eq:loglik.term2.lower} and \eqref{eq:loglik.term3.lower} to obtain that on the set $\Acal_{n}^{\dagger}$, with probability at least $1-\exp(-16\log^2 n)$,
\begin{align*}
&\Lcal_n(\beta,\theta/\alpha^{2\nu},\alpha) - \Lcal_n(\beta_0,\sigma_0^2,\alpha_0)  \geq -\frac{1}{2n} - \frac{2\nu}{n} - \frac{2\nu+d}{n} - c_1'n^{-2} \geq c_{5L} n^{-1},
\end{align*}
for some constant $c_{5L}>0$. This completes the proof.
\end{proof}

\begin{lemma} \label{lem:exp.test.alpha}
Suppose that Assumptions \ref{assump.m.func}, \ref{prior.1} and \ref{assump.alpha.est} hold for $d\geq 5$ and $\nu\in\RR^+$. Then the profile posterior distribution satisfies
\begin{align}
& \Pi\left(|\theta/\theta_0-1| \leq n^{-\kappa_1'}\log^2 n, ~ |\alpha/\alpha_0-1| \leq n^{-\kappa_2'}\log^2 n ~\big |~ Y_n\right) \rightarrow 1,  \nonumber
\end{align}
as $n\to\infty$ almost surely $P_{(\beta_0,\sigma_0^2,\alpha_0)}$.
\end{lemma}

\begin{proof}[Proof of Lemma \ref{lem:exp.test.alpha}]
The proof proceeds in a similar way to that of the Schwartz's theorem for posterior consistency (\citet{Sch65}); see for example, Theorem 6.17 and its proof in \citet{GhoVan17}. Let $\epsilon_{1n}'= n^{-\kappa_1'} \log^2 n$ and $\epsilon_{2n}'= n^{-\kappa_2'} \log^2 n$. Define the testing function (indicator function):
\begin{align} \label{eq:test}
T_n & = \Ical \left( |\widehat\theta_n/\theta_0-1|\geq \epsilon_{1n}'/2, \text{ or } |\widehat\alpha_n/\alpha_0-1|\geq \epsilon_{2n}'/2 \right),
\end{align}
where $\widehat\theta_n$ and $\widehat\alpha_n$ are the consistent estimators of $\theta$ and $\alpha$ from Assumption \ref{assump.alpha.est}. Recall that the log-likelihood function $\Lcal_n(\beta,\theta/\alpha^{2\nu},\alpha)$ is defined in \eqref{eq:loglik} of the main text. We have the following decomposition:
\begin{align} \label{eq:d5.decomp1}
&~ \quad \Pi \left(\Bcal_0(\epsilon_{1n}',\epsilon_{2n}')^c ~|~ Y_n  \right) \nonumber \\
& = \frac{(T_n+1-T_n)\int_{\Bcal_0(\epsilon_{1n}',\epsilon_{2n}')^c}\exp\left\{ \Lcal_n(\beta,\theta/\alpha^{2\nu},\alpha) \right\} \pi(\beta,\theta,\alpha) \ud \beta \ud \theta \ud \alpha }
{\int_{\RR^p \times \RR^+ \times \RR^+} \exp\left\{\Lcal_n(\beta',\theta'/{\alpha'}^{2\nu},\alpha') \right\} \pi(\beta',\theta',\alpha') \ud \beta' \ud \theta' \ud \alpha' } \nonumber \\
&\leq T_n + \frac{(1-T_n)\int_{\Bcal_0(\epsilon_{1n}',\epsilon_{2n}')^c \cap \Fcal_n}\exp\left\{ \Lcal_n(\beta,\theta/\alpha^{2\nu},\alpha) - \Lcal_n(\beta_0,\sigma_0^2,\alpha_0) \right\} \pi(\beta,\theta,\alpha) \ud \beta \ud \theta \ud \alpha }
{\int_{\RR^p \times \RR^+ \times \RR^+} \exp\left\{\Lcal_n(\beta',\theta'/{\alpha'}^{2\nu},\alpha') - \Lcal_n(\beta_0,\sigma_0^2,\alpha_0) \right\} \pi(\beta',\theta',\alpha') \ud \beta' \ud \theta' \ud \alpha' } \nonumber \\
&\qquad + \frac{\int_{\Fcal_n^c}\exp\left\{ \Lcal_n(\beta,\theta/\alpha^{2\nu},\alpha) - \Lcal_n(\beta_0,\sigma_0^2,\alpha_0) \right\} \pi(\beta,\theta,\alpha) \ud \beta \ud \theta \ud \alpha }
{\int_{\RR^p \times \RR^+ \times \RR^+} \exp\left\{\Lcal_n(\beta',\theta'/{\alpha'}^{2\nu},\alpha') - \Lcal_n(\beta_0,\sigma_0^2,\alpha_0) \right\} \pi(\beta', \theta',\alpha') \ud \beta' \ud \theta' \ud \alpha' }.
\end{align}
By Assumption \ref{assump.alpha.est}, we have that as $n\to\infty$,
\begin{align} \label{eq:d5.post1}
{\EE}_{(\beta_0,\sigma_0^2,\alpha_0)} (T_n) & \leq P_{(\beta_0,\sigma_0^2,\alpha_0)} \big(\big|\widehat\theta_n/\theta_0-1\big| \geq \epsilon_{1n}'/2 \big) +  P_{(\beta_0,\sigma_0^2,\alpha_0)} \left(\left|\widehat\alpha_n/\alpha_0-1\right| \geq \epsilon_{2n}'/2 \right) \nonumber \\
&\leq 2\exp\left(-c_5 \log^2 n \right) .
\end{align}
For the second term in \eqref{eq:d5.decomp1}, we use the same proof technique as the Schwartz's theorem for posterior consistency. By Assumption \ref{assump.alpha.est} and the Fubini's theorem, its numerator has expectation upper bounded by
\begin{align} \label{eq:d5.post2}
&\quad~ {\EE}_{(\beta_0,\sigma_0^2,\alpha_0)} (1-T_n)\int_{\Bcal_0(\epsilon_{1n}',\epsilon_{2n}')^c \cap \Fcal_n}\exp\left\{ \Lcal_n(\beta,\theta/\alpha^{2\nu},\alpha) - \Lcal_n(\beta_0,\sigma_0^2,\alpha_0) \right\} \pi(\beta,\theta,\alpha) \ud \beta \ud \theta \ud \alpha \nonumber \\
&= \int_{\Bcal_0(\epsilon_{1n}',\epsilon_{2n}')^c \cap \Fcal_n}{\EE}_{(\beta_0,\sigma_0^2,\alpha_0)} (1-T_n) \exp\left\{ \Lcal_n(\beta,\theta/\alpha^{2\nu},\alpha) - \Lcal_n(\beta_0,\sigma_0^2,\alpha_0) \right\} \pi(\beta,\theta,\alpha) \ud \beta \ud \theta \ud \alpha \nonumber \\
&= \int_{\Bcal_0(\epsilon_{1n}',\epsilon_{2n}')^c \cap \Fcal_n}{\EE}_{(\beta,\theta/\alpha^{2\nu},\alpha)} (1-T_n) \pi(\beta,\theta,\alpha) \ud \beta \ud \theta \ud \alpha \nonumber \\
&\leq \sup_{\Bcal_0(\epsilon_{1n}',\epsilon_{2n}')^c \cap \Fcal_n} {\EE}_{(\beta,\theta/\alpha^{2\nu},\alpha)} \left\{ (1-T_n) \int_{\Bcal_0(\epsilon_{1n}',\epsilon_{2n}')^c \cap \Fcal_n} \pi(\beta,\theta,\alpha) \ud \beta \ud \theta \ud \alpha \right\} \nonumber \\
&\leq \sup_{\Bcal_0(\epsilon_{1n}',\epsilon_{2n}')^c \cap \Fcal_n} {\EE}_{(\beta,\theta/\alpha^{2\nu},\alpha)} (1-T_n)  \nonumber \\
&\leq \sup_{\Bcal_0(\epsilon_{1n}',\epsilon_{2n}')^c \cap \Fcal_n} P_{(\beta,\theta/\alpha^{2\nu},\alpha)} \big( \big|\widehat\theta_n/\theta_0-1\big| \leq \epsilon_{1n}' /2 \big) \nonumber \\
&\qquad + \sup_{\Bcal_0(\epsilon_{1n}',\epsilon_{2n}')^c\cap \Fcal_n} P_{(\beta,\theta/\alpha^{2\nu},\alpha)} \left( \left|\widehat\alpha_n/\alpha_0-1\right| \leq \epsilon_{2n}' /2 \right)  \nonumber \\
&\leq 2\exp\left(-c_5 \log^2 n \right).
\end{align}
Since $\sum_{n=1}^{\infty} 2\exp\left(-c_5 \log^2 n \right) < \infty$, by applying the Markov's inequality and the Borel-Cantelli lemma, the numerator of the second term in \eqref{eq:d5.decomp1} is upper bounded by $2\exp\left\{-(c_5/2) \log^2 n \right\} $ as $n\to\infty$ almost surely $P_{(\beta_0,\sigma_0^2,\alpha_0)}$. On the other hand, by Lemma \ref{lem:d5.denom.lower}, for $d\geq 5$ and for all sufficiently large $n$, with probability at least $1-\exp(-16\log^2 n)$ the denominator of the second term in \eqref{eq:d5.decomp1} can be lower bounded by
\begin{align} \label{eq:d5.post3}
& \int_{\RR^p \times \RR^+ \times \RR^+} \exp\left\{\Lcal_n(\beta',\theta'/{\alpha'}^{2\nu},\alpha') - \Lcal_n(\beta_0,\sigma_0^2,\alpha_0) \right\} \pi(\beta',\theta',\alpha') \ud \beta' \ud \theta' \ud \alpha'  \nonumber \\
&\geq \int_{\Acal_{n}^{\dagger}} \exp\left\{\Lcal_n(\beta',\theta'/{\alpha'}^{2\nu},\alpha') - \Lcal_n(\beta_0,\sigma_0^2,\alpha_0) \right\} \pi(\beta',\theta',\alpha') \ud \beta' \ud \theta' \ud \alpha' \nonumber \\
&\geq \exp(-c_{5L}n^{-1}) \cdot \Pi(\Acal_{n}^{\dagger}) \overset{(i)}{\geq} c_{6L} n^{-(3p+4)},
\end{align}
for some constant $c_{6L}>0$, where (i) follows because $\exp(-c_{5L}n^{-1})>1/2$ for large $n$, the prior density $\pi(\beta,\theta,\alpha)=\pi(\beta|\theta/\alpha^{2\nu})\pi(\theta|\alpha)\pi(\alpha)$ is lower bounded by constant in the set $\Acal_{n}^{\dagger}$ by Assumptions \ref{assump.m.func} and \ref{prior.1}, and the set $\Acal_{n}^{\dagger}$ defined in Lemma \ref{lem:d5.denom.lower} has a volume of order $n^{-3p}\cdot n^{-2}\cdot n^{-2} = n^{-(3p+4)}$. Therefore, we combine \eqref{eq:d5.post2} and \eqref{eq:d5.post3} to obtain that almost surely $P_{(\beta_0,\sigma_0^2,\alpha_0)}$ as $n\to\infty$, the second term in \eqref{eq:d5.decomp1} is upper bounded by
\begin{align} \label{eq:d5.post4}
& \frac{(1-T_n)\int_{\Bcal_0(\epsilon_{1n}',\epsilon_{2n}')^c \cap \Fcal_n}\exp\left\{ \Lcal_n(\beta,\theta/\alpha^{2\nu},\alpha) - \Lcal_n(\beta_0,\sigma_0^2,\alpha_0) \right\} \pi(\beta,\theta,\alpha) \ud \beta \ud \theta \ud \alpha }
{\int_{\RR^p \times \RR^+ \times \RR^+} \exp\left\{\Lcal_n(\beta',\theta'/{\alpha'}^{2\nu},\alpha') - \Lcal_n(\beta_0,\sigma_0^2,\alpha_0) \right\} \pi(\beta',\theta',\alpha') \ud \beta' \ud \theta' \ud \alpha' } \nonumber \\
&\leq 2c_{6L}^{-1} n^{3p+4} \exp\left\{-(c_5/2) \log^2 n \right\} \to 0, \text{ as } n\to\infty.
\end{align}
For the third term in \eqref{eq:d5.decomp1}, similar to \eqref{eq:d5.post2}, by the Fubini's theorem and Assumption \ref{assump.alpha.est}, we have that
\begin{align} \label{eq:d5.post5}
&\quad~ {\EE}_{(\beta_0,\sigma_0^2,\alpha_0)} \int_{\Fcal_n^c }\exp\left\{ \Lcal_n(\beta,\theta/\alpha^{2\nu},\alpha) - \Lcal_n(\beta_0,\sigma_0^2,\alpha_0) \right\} \pi(\beta,\theta,\alpha) \ud \beta \ud \theta \ud \alpha \nonumber \\
&= \int_{\Fcal_n^c } {\EE}_{(\beta_0,\sigma_0^2,\alpha_0)} \exp\left\{ \Lcal_n(\beta,\theta/\alpha^{2\nu},\alpha) - \Lcal_n(\beta_0,\sigma_0^2,\alpha_0) \right\} \pi(\beta,\theta,\alpha) \ud \beta \ud \theta \ud \alpha \nonumber \\
&= \int_{\Fcal_n^c }{\EE}_{(\beta,\theta/\alpha^{2\nu},\alpha)} \pi(\beta,\theta,\alpha) \ud \beta \ud \theta \ud \alpha= \Pi(\Fcal_n^c) \leq n^{-(3p+6)},
\end{align}
which by the Markov's inequality and the Borel-Cantelli lemma, implies that the numerator of the second term in \eqref{eq:d5.decomp1} is upper bounded by $n^{-(3p+6)}$ as $n\to\infty$ almost surely $P_{(\beta_0,\sigma_0^2,\alpha_0)}$. Therefore, \eqref{eq:d5.post3} and \eqref{eq:d5.post5} imply that the second term in \eqref{eq:d5.decomp1} is upper bounded by
\begin{align} \label{eq:d5.post6}
& \frac{\int_{\Fcal_n^c}\exp\left\{ \Lcal_n(\beta,\theta/\alpha^{2\nu},\alpha) - \Lcal_n(\beta_0,\sigma_0^2,\alpha_0) \right\} \pi(\beta,\theta,\alpha) \ud \beta \ud \theta \ud \alpha }
{\int_{\RR^p \times \RR^+ \times \RR^+} \exp\left\{\Lcal_n(\beta',\theta'/{\alpha'}^{2\nu},\alpha') - \Lcal_n(\beta_0,\sigma_0^2,\alpha_0) \right\} \pi(\beta', \theta',\alpha') \ud \beta' \ud \theta' \ud \alpha' } \nonumber \\
&\leq c_{6L}^{-1} n^{3p+4}\cdot n^{-(3p+6)} = c_{6L}^{-1} n^{-2}  \to 0, \text{ as } n\to\infty.
\end{align}
The conclusion follows by combining \eqref{eq:d5.decomp1}, \eqref{eq:d5.post1}, \eqref{eq:d5.post4}, and \eqref{eq:d5.post6}.
\end{proof}

We state and prove the following Theorem \ref{thm:d5.bvm.joint} for the limiting posterior distribution of the covariance parameters $(\theta,\alpha)$ for the case of $d\geq 5$. Theorem \ref{thm:d5.bvm.joint} for $d\geq 5$ is a parallel to Theorem \ref{thm:bvm2:joint} in the main text for $d\in\{1,2,3\}$. We emphasize that in Theorem \ref{thm:d5.bvm.joint}, we only derive the asymptotic normality for the posterior of $\theta$, since the limiting posterior distribution of the range parameter $\alpha$ will depend on the exact form of sampling design $\Scal_n$. Another difference in Theorem \ref{thm:d5.bvm.joint} from Theorem \ref{thm:bvm2:joint} is that the profile posterior distribution for $\alpha$ will be a truncated distribution to the neighborhood $\big[(1-n^{-\kappa_2'}\log^2 n)\alpha_0, (1+n^{-\kappa_2'}\log^2 n)\alpha_0\big]$, given the posterior contraction result in Lemma \ref{lem:exp.test.alpha}.

\begin{theorem} \label{thm:d5.bvm.joint}
Suppose that Assumptions \ref{assump.m.func}, \ref{prior.1}, \ref{prior.2} and \ref{assump.alpha.est} hold for $d\geq 5$ and $\nu\in\RR^+$. The posterior distributions of $\theta$ and $\alpha$ are asymptotically independent, in the sense that the joint posterior distribution of $(\theta,\alpha)$ satisfies
\begin{align}\label{eq:d5.joint:theta1}
\left\|\Pi(\ud \theta, \ud \alpha |Y_n) - \mathcal{N}\left(\ud \theta \big| \widetilde\theta_{\alpha_0}, 2\theta_0^2/n \right) \times \widetilde \Pi^{\dagger}(\ud \alpha|Y_n)\right\|_{\tv} \rightarrow 0,
\end{align}
as $n\to\infty$ almost surely $P_{(\beta_0,\sigma_0^2,\alpha_0)}$, and $\widetilde \Pi^{\dagger}(\ud \alpha|Y_n)$ is the truncated profile posterior distribution with the density
\begin{align}\label{eq:d5.profile:post1}
\widetilde \pi^{\dagger} (\alpha|Y_n) &= \frac{\exp\big\{\widetilde \Lcal_n(\alpha)\big\}\pi(\alpha|\theta_0)}{\displaystyle \int_{\max\big\{0, \big(1-n^{-\kappa_2'}\log^2 n\big)\alpha_0 \big\}}^{(1+n^{-\kappa_2'}\log^2 n)\alpha_0} \exp\big\{\widetilde \Lcal_n(\alpha')\big\}\pi(\alpha'|\theta_0)\ud \alpha'},
\end{align}
where the profile restricted log-likelihood $\widetilde \Lcal_n(\alpha)$ is given in \eqref{def:prologlik} of the main text and $\pi(\alpha|\theta_0)$ is the conditional prior density of $\alpha$ given $\theta=\theta_0$.
\end{theorem}

\begin{proof}[Proof of Theorem \ref{thm:d5.bvm.joint}]
For short, let $\Bcal_{0n}= \Bcal_0(n^{-\kappa_1'}\log^2 n, n^{-\kappa_2'}\log^2 n)$ as defined in \eqref{eq:Bcal0}. For the joint posterior distribution $\Pi(\ud\theta,\ud\alpha|Y_n)$, we define the truncated posterior distribution $\Pi^{\dagger} (\ud \theta,\ud\alpha | Y_n) = \Pi(\ud\theta,\ud\alpha|Y_n) \cdot \Ical\{(\theta,\alpha)\in \Bcal_{0n}\}/ \Pi(\Bcal_{0n}|Y_n)$ on the truncated support $\Bcal_{0n}$. For all sufficiently large $n$, this support is a subset of $\RR^+ \times \RR^+$. By Lemma \ref{lem:exp.test.alpha}, the posterior probability of the set $\Bcal_{0n}$ converges to 1 as $n\to\infty$ almost surely $P_{(\beta_0,\sigma_0^2,\alpha_0)}$, which immediately implies that
\begin{align} \label{eq:true.truncate}
&\left\|\Pi^{\dagger} (\ud \theta,\ud\alpha | Y_n) - \Pi (\ud \theta,\ud\alpha | Y_n) \right\|_{\tv}  = \sup_{\Acal \in \RR^+ \times \RR^+} \left|\Pi^{\dagger} (\Acal | Y_n) - \Pi (\Acal | Y_n) \right| \nonumber \\
&= \frac{\sup_{\Acal \in \RR^+ \times \RR^+}\Pi\left( \Acal |Y_n\right) \left[1-\Pi\left(\Bcal_{0n}|Y_n\right)\right] }{\Pi\left(\Bcal_{0n}|Y_n\right)}
= \frac{ \Pi\left(\Bcal_{0n}^c |Y_n\right) }{\Pi\left(\Bcal_{0n}|Y_n\right)}
\rightarrow 0,
\end{align}
as $n\to\infty$ almost surely $P_{(\beta_0,\sigma_0^2,\alpha_0)}$. Therefore, to show \eqref{eq:d5.joint:theta1}, it suffices to show that as $n\to\infty$ almost surely $P_{(\beta_0,\sigma_0^2,\alpha_0)}$,
\begin{align}\label{eq:d5.joint:theta2}
\left\|\Pi^{\dagger}(\ud \theta, \ud \alpha |Y_n) - \mathcal{N}\left(\ud \theta \big| \widetilde\theta_{\alpha_0}, 2\theta_0^2/n \right) \times \widetilde \Pi^{\dagger}(\ud \alpha|Y_n)\right\|_{\tv} \rightarrow 0.
\end{align}
The rest of the proof proceeds in a similar way to the proof of Theorem \ref{thm:bvm2:joint}, with a few key differences. Without loss of generality, we only consider those sufficiently large $n$ such that $1-n^{-\kappa_2'}\log^2 n>0$. For short, let $\alpha_{1n}=\big(1-n^{-\kappa_2'}\log^2 n\big)\alpha_0$ and $\alpha_{2n}=\big(1+n^{-\kappa_2'}\log^2 n\big)\alpha_0$. First, \eqref{joint:theta2} and \eqref{joint:theta3} in the proof of Theorem \ref{thm:bvm2:joint} will be replaced by
\begin{align}
& \int_{\alpha_{1n}}^{\alpha_{2n}}  \int_{\RR} \left|\pi^{\dagger}(\theta,\alpha|Y_n) - \frac{\sqrt{n-p}}{2\sqrt{\pi}\theta_0} \ee^{-\frac{(n-p)(\theta-\widetilde\theta_{\alpha})^2}{4\theta_0^2}} \cdot \widetilde \pi^{\dagger}(\alpha|Y_n)\right| \ud \theta \ud \alpha \rightarrow 0, \label{d5.joint:theta2} \\
& \int_{\alpha_{1n}}^{\alpha_{2n}}  \int_{\RR} \left|\frac{\sqrt{n-p}}{2\sqrt{\pi}\theta_0} \ee^{-\frac{(n-p)(\theta-\widetilde\theta_{\alpha})^2}{4\theta_0^2}} - \frac{\sqrt{n}}{2\sqrt{\pi}\theta_0} \ee^{-\frac{n(\theta-\widetilde\theta_{\alpha_0})^2}{4\theta_0^2}} \right| \cdot \widetilde \pi^{\dagger}(\alpha|Y_n) \ud \theta \ud \alpha \rightarrow 0, \label{d5.joint:theta3}
\end{align}
as $n\to\infty$ almost surely $P_{(\beta_0,\sigma_0^2,\alpha_0)}$, where $\pi^{\dagger}(\theta,\alpha|Y_n)$ is the density of $\Pi^{\dagger}(\ud \theta,\ud\alpha | Y_n)$ and $\widetilde \pi^{\dagger}(\alpha|Y_n)$ is as defined in \eqref{eq:d5.profile:post1}. The lower and upper bounds in the integrals of \eqref{d5.joint:theta2} and \eqref{d5.joint:theta3} are because the range parameter $\alpha$ in both $\pi^{\dagger}(\theta,\alpha|Y_n)$ and $\widetilde \pi^{\dagger}(\alpha|Y_n)$ is supported on $[\alpha_{1n},\alpha_{2n}]$. Similar to \eqref{joint:theta23}, using Lemma \ref{lemma:intdiff} and the definition of $\varrho_n(t;\alpha)$ in \eqref{func:gn}, the left-hand side of \eqref{d5.joint:theta2} is smaller than ${\numer}'/{\denom}'$, where
\begin{align}
{\numer}' &= 2\int_{\alpha_{1n}}^{\alpha_{2n}} \int_{\RR} \left|\varrho_n(\sqrt{n-p}(\theta-\widetilde\theta_{\alpha});\alpha)\right| \ee^{\Lcal_n(\alpha^{-2\nu}\widetilde\theta_{\alpha},\alpha)} \pi(\theta_0|\alpha) \pi(\alpha) \ud \theta \ud \alpha , \label{eq:d5.numer} \\
{\denom}' &= \frac{2\theta_0\sqrt{\pi}}{\sqrt{n-p}} \int_{\alpha_{1n}}^{\alpha_{2n}} \ee^{\Lcal_n(\alpha^{-2\nu}\widetilde\theta_{\alpha},\alpha)} \pi(\theta_0|\alpha)\pi(\alpha) \ud \alpha, \label{eq:d5.denom}.
\end{align}
For any $\epsilon>0$, let $\Ecal_2'(\epsilon) = \big\{\sup_{\alpha\in [\alpha_{1n},\alpha_{2n}]} |\widetilde\theta_{\alpha}-\theta_0| < \epsilon \big\}$, $\Ecal_3'(\epsilon) = \big\{\sup_{\alpha\in  [\alpha_{1n},\alpha_{2n}]} |\widetilde\theta_{\alpha}-\widetilde \theta_{\alpha_0}| < \epsilon \big\}$, $\Ecal_4(\epsilon) = \big\{|\widetilde\theta_{\alpha_0}-\theta_0| < \epsilon \big\}$. We can set $c=\kappa_2'$ in Lemma \ref{lem:MRM.high.order}, which satisfies $c=\kappa_2'> 1/(2\nu+d)$ given Assumption \ref{assump.alpha.est} and hence $[\alpha_{1n},\alpha_{2n}]\subseteq \big[\big(1-n^{-1/(2\nu+d)}\big)\alpha_0, \big(1+n^{-1/(2\nu+d)}\big)\alpha_0]$. Thus we can apply \eqref{eq:theta01.ho2} of Lemma \ref{lem:MRM.high.order} to obtain that $\pr\big\{\Ecal_3'(10\theta_0n^{-(2\nu+d)\kappa_2'}\log^4 n)\big\}\geq 1-8\exp(-4\log^2 n)$. Lemma \ref{lem:theta.alpha0} implies that $\pr\big\{\Ecal_4(5\theta_0n^{-1/2}\log n)\big\}\geq 1-3\exp(-4\log^2 n)$. Since $(2\nu+d)\kappa_2'> 1$ from Assumption \ref{assump.alpha.est}, by the triangle inequality, for sufficiently large $n$,
$$\Ecal_2'\big(6\theta_0n^{-1/2}\log n\big)\supseteq \Ecal_3'\big(10\theta_0n^{-(2\nu+d)\kappa_2'}\log^4 n\big) \cap \Ecal_4\big(5\theta_0n^{-1/2}\log n\big),$$
and hence it follows that $\pr\big\{\Ecal_2'(6\theta_0n^{-1/2}\log n)\big\}\geq 1 - 11\exp(-4\log^2 n)$.

Lemma \ref{lemma:gndiff1} still applies when $d\geq 5$ and $\Ecal_1(6\theta_0n^{-1/2}\log n, \alpha)\supseteq \Ecal_2'(6\theta_0n^{-1/2}\log n)$ for every $\alpha\in [\alpha_{1n},\alpha_{2n}]$. Also, under Assumption \ref{prior.2}, the inequality and convergence in \eqref{en:Bnrate} in the proof of Theorem \ref{thm:bvm2:joint} still holds. Since $[\alpha_{1n},\alpha_{2n}] \subseteq [\underline\alpha_n,\overline\alpha_n]$, we apply Lemma \ref{lemma:gndiff1} with $\epsilon_{1n}=6\theta_0n^{-1/2}\log n$ and $s_n=\log n$ and obtain from \eqref{joint:theta22} and \eqref{en:Bnrate} that
\begin{align}
& \int_{\alpha_{1n}}^{\alpha_{2n}} \int_{\RR} \left| \ee^{\Lcal_n(\alpha^{-2\nu}\theta,\alpha)-\Lcal_n(\alpha^{-2\nu}\widetilde\theta_{\alpha},\alpha)} \frac{\pi(\theta|\alpha)}{\pi(\theta_0|\alpha)}
- \ee^{-\frac{(n-p)(\theta-\widetilde\theta_{\alpha})^2}{4\theta_0^2}}
\right| \ee^{\Lcal_n(\alpha^{-2\nu}\widetilde\theta_{\alpha},\alpha)} \pi(\theta_0|\alpha) \pi(\alpha) \ud\theta \ud\alpha \nonumber \\
\leq{}& \frac{\sup_{\alpha\in [\alpha_{1n},\alpha_{2n}]} B_n(\alpha)}{\sqrt{n-p}} \int_{\alpha_{1n}}^{\alpha_{2n}} \ee^{\Lcal_n(\alpha^{-2\nu}\widetilde\theta_{\alpha},\alpha)} \pi(\theta_0|\alpha) \pi(\alpha) \ud \alpha \to 0, \text{ as } n\to\infty .  \nonumber
\end{align}
Therefore, similar to \eqref{en:bound1}, we have that on the event $\Ecal_2'(6\theta_0n^{-1/2}\log n)$,
\begin{align}
\frac{{\numer}'}{{\denom}'} &\leq \frac{2\int_{\alpha_{1n}}^{\alpha_{2n}} \int_{\RR} \left|\varrho_n(\sqrt{n-p}(\theta-\widetilde\theta_{\alpha});\alpha)\right| \ee^{\Lcal_n(\alpha^{-2\nu}\widetilde\theta_{\alpha},\alpha)} \pi(\theta_0|\alpha) \pi(\alpha) \ud \theta \ud \alpha} {\frac{2\theta_0\sqrt{\pi}}{\sqrt{n-p}}\int_{\alpha_{1n}}^{\alpha_{2n}} \ee^{\Lcal_n(\alpha^{-2\nu}\widetilde\theta_{\alpha},\alpha)} \pi(\theta_0|\alpha)\pi(\alpha) \ud \alpha} \nonumber \\
&\leq \frac{ \frac{\sup_{\alpha\in [\alpha_{1n},\alpha_{2n}]} B_n(\alpha)}{\sqrt{n-p}} \int_{\alpha_{1n}}^{\alpha_{2n}} \ee^{\Lcal_n(\alpha^{-2\nu}\widetilde\theta_{\alpha},\alpha)} \pi(\theta_0|\alpha) \pi(\alpha) \ud \alpha}
{\frac{\theta_0\sqrt{\pi}}{\sqrt{n-p}}\int_{\alpha_{1n}}^{\alpha_{2n}} \ee^{\Lcal_n(\alpha^{-2\nu}\widetilde\theta_{\alpha},\alpha)} \pi(\theta_0|\alpha)\pi(\alpha) \ud \alpha} \nonumber \\
&\leq \frac{\sup_{\alpha\in [\alpha_{1n},\alpha_{2n}]} B_n(\alpha)}{\theta_0\sqrt{\pi}} \rightarrow 0,  \text{ as } n\to\infty ,  \nonumber
\end{align}
where the last step follows from \eqref{en:Bnrate}. This has proved \eqref{d5.joint:theta2}.

For \eqref{d5.joint:theta3}, similar to \eqref{joint:theta31} and \eqref{joint:theta32}, on the event $\Ecal_3'(10\theta_0n^{-(2\nu+d)\kappa_2'}\log^4 n)$, we have that for all sufficiently large $n$,
\begin{align}
& \int_{\alpha_{1n}}^{\alpha_{2n}}  \int_{\RR} \left|\frac{\sqrt{n-p}}{2\sqrt{\pi}\theta_0} \ee^{-\frac{(n-p)(\theta-\widetilde\theta_{\alpha})^2}{4\theta_0^2}} - \frac{\sqrt{n}}{2\sqrt{\pi}\theta_0} \ee^{-\frac{n(\theta-\widetilde\theta_{\alpha_0})^2}{4\theta_0^2}} \right| \cdot \widetilde \pi^{\dagger}(\alpha|Y_n) \ud \theta \ud \alpha  \nonumber \\
\leq{}& \int_{\alpha_{1n}}^{\alpha_{2n}} \frac{(n-p)^{1/2}}{2\sqrt{\pi}\theta_0} \big|\widetilde\theta_{\alpha}-\widetilde\theta_{\alpha_0}\big| \widetilde \pi^{\dagger}(\alpha|Y_n) \ud \alpha \nonumber \\
\leq{}& \frac{5}{\sqrt{\pi}} n^{-(2\nu+d)\kappa_2'+1/2}\log^4 n \int_{\alpha_{1n}}^{\alpha_{2n}} \widetilde \pi^{\dagger}(\alpha|Y_n) \ud \alpha \nonumber \\
\overset{(i)}{\leq}{}& \frac{5}{\sqrt{\pi}} n^{-1/2}\log^4 n \rightarrow 0,  \text{ as } n\to\infty ,
\end{align}
where (i) follows from the condition $\kappa_2' > 1/(2\nu+d)$. This has proved \eqref{d5.joint:theta3}. Therefore, \eqref{d5.joint:theta2} and \eqref{d5.joint:theta3} together imply \eqref{eq:d5.joint:theta2}, and \eqref{eq:d5.joint:theta2} together with \eqref{eq:true.truncate} proves Theorem \ref{thm:d5.bvm.joint}.
\end{proof}
\vspace{8mm}

\section{Proof of Propositions \ref{prop:prior2} and \ref{prop:prior3}} \label{supsec:2prop}
In this section, we provide the proof of Propositions \ref{prop:prior2} and \ref{prop:prior3} in the main text, which verify Assumptions \ref{prior.2} and \ref{prior.3} on the prior, respectively.

\begin{proof}[Proof of Proposition \ref{prop:prior2}:]

\noindent (i) Since $\pi(\theta|\alpha)=\pi(\theta)$ and does not depend on $\alpha$, we have that $\frac{\partial \log \pi(\theta|\alpha)}{\partial \theta} = \pi'(\theta)/\pi(\theta)$. Since $\pi(\theta)>0$ and $\pi'(\theta)=\ud \pi(\theta)/\ud \theta$ is continuous on $\RR^+$, \eqref{A2.1} is satisfied for all sufficiently large $n$ since
$$\sup_{\alpha\in [\underline\alpha_n, \overline\alpha_n]} \sup_{\theta\in (\theta_0/2,2\theta_0)}
\left|\frac{\partial \log \pi(\theta|\alpha)}{\partial \theta} \right| \leq \frac{\sup_{\theta\in (\theta_0/2,2\theta_0)}\pi'(\theta)}{\inf_{\theta\in (\theta_0/2,2\theta_0)}\pi(\theta)} < n^{C_{\pi,1}},$$
for arbitrary $C_{\pi,1}>0$.

The prior density $\pi(\theta)$ has finite supremum and positive infimum on $(\theta_0/2,2\theta_0)$. Hence \eqref{A2.2} is satisfied for all sufficiently large $n$ since
$$\sup_{\alpha\in [\underline\alpha_n, \overline\alpha_n]} \sup_{\theta\in (\theta_0/2,2\theta_0)} \frac{\pi(\theta|\alpha)}{\pi(\theta_0|\alpha)} \leq \frac{\sup_{\theta\in (\theta_0/2,2\theta_0)}\pi(\theta)}{\inf_{\theta\in (\theta_0/2,2\theta_0)}\pi(\theta)} < n^{C_{\pi,2}},$$
for arbitrary $C_{\pi,2}>0$. Since $C_{\pi,1}$ and $C_{\pi,2}$ can be arbitrarily small, $C_{\pi,1}+C_{\pi,2}<1/2$ is satisfied. Finally, \eqref{A2.3} is satisfied for all sufficiently large $n$ since $\pi(\theta_0)>0$ and for all sufficiently large $n$,
$$\inf_{\alpha\in [\underline\alpha_n, \overline\alpha_n]} \log \pi(\theta_0|\alpha) = \log \pi(\theta_0) > - n^{C_{\pi,3}},$$
for arbitrarily small $C_{\pi,3}>0$.

\vspace{5mm}

\noindent (ii) If $\pi(\alpha)$ is supported on a compact interval $[\alpha_1,\alpha_2]$, then all $\sup_{\alpha\in [\underline\alpha_n, \overline\alpha_n]}$ can be replaced by $\sup_{\alpha\in [\alpha_1, \alpha_2]}$. Based on the conditions, for all sufficiently large $n$,
$$\sup_{\alpha\in [\alpha_1, \alpha_2]} \sup_{\theta\in (\theta_0/2,2\theta_0)}
\left|\frac{\partial \log \pi(\theta|\alpha)}{\partial \theta} \right| <  n^{C_{\pi,1}},$$
for arbitrary $C_{\pi,1}>0$.

Since $\pi(\theta|\alpha)>0$ for all $(\theta,\alpha)\in \RR^+ \times \RR^+$, for all sufficiently large $n$,
$$\sup_{\alpha\in [\alpha_1, \alpha_2]} \sup_{\theta\in (\theta_0/2,2\theta_0)}
\left|\frac{\partial \log \pi(\theta|\alpha)}{\partial \theta} \right| <  n^{C_{\pi,2}},$$
for arbitrary $C_{\pi,2}>0$. Since $C_{\pi,1}$ and $C_{\pi,2}$ can be arbitrarily small, $C_{\pi,1}+C_{\pi,2}<1/2$ is satisfied.

Since $\pi(\theta|\alpha)>0$ is continuous in $\alpha\in \RR^+$, for all sufficiently large $n$,
$$\inf_{\alpha\in [\alpha_1, \alpha_2]} \log \pi(\theta_0|\alpha) =\inf_{\alpha\in [\alpha_1, \alpha_2]} \log \pi(\theta_0|\alpha) > -n^{C_{\pi,3}},$$
for arbitrarily small $C_{\pi,3}>0$.

\vspace{5mm}

\noindent (iii) If the prior of $\sigma^2$ is independent of $\alpha$, then by the relation $\theta=\sigma^2\alpha^{2\nu}$, the prior of $\theta$ given $\alpha$ is $\pi(\theta|\alpha) = \pi_{\sigma^2}(\theta/\alpha^{2\nu})/\alpha^{2\nu}$, where we use $\pi_{\sigma^2}(\cdot)$ to denote the prior density of $\sigma^2$. Therefore, $\frac{\partial \log \pi(\theta|\alpha)}{\partial \theta} = \frac{\pi_{\sigma^2}'(\theta/\alpha^{2\nu})}{\alpha^{2\nu}\pi_{\sigma^2}(\theta/\alpha^{2\nu})}$. For the transformed beta family density, the derivative is
\begin{align*}
\pi_{\sigma^2}'(\sigma^2) & = \frac{\Gamma(\gamma_1+\gamma_2)}{\Gamma(\gamma_1)\Gamma(\gamma_2)} \frac{\left(\frac{\sigma^2}{b}\right)^{\gamma_2/\gamma-2}\left[\gamma_2-\gamma-\left(\gamma_1+\gamma\right)
\left(\frac{\sigma^2}{b}\right)^{1/\gamma}\right]}{(b\gamma)^2 [1+(\sigma^2/b)^{1/\gamma}]^{\gamma_1+\gamma_2+1}}.
\end{align*}
Therefore, for all sufficiently large $n$,
\begin{align*}
&\sup_{\alpha\in [\underline\alpha_n, \overline\alpha_n]} \sup_{\theta\in (\theta_0/2,2\theta_0)}
\left|\frac{\partial \log \pi(\theta|\alpha)}{\partial \theta} \right|  \leq \sup_{\alpha\in [\underline\alpha_n, \overline\alpha_n]} \sup_{\theta\in (\theta_0/2,2\theta_0)}  \frac{\left|\gamma_2-\gamma-\left(\gamma_1+\gamma\right)
\left(\frac{\theta}{b\alpha^{2\nu}}\right)^{1/\gamma}\right|}{b\gamma\alpha^{2\nu} \left(\frac{\theta}{b\alpha^{2\nu}}\right) [1+\left(\frac{\theta}{b\alpha^{2\nu}}\right)^{1/\gamma}]} \\
&\leq \frac{2|\gamma_2-\gamma|}{\gamma\theta_0} + \frac{2(\gamma_1+\gamma)}{\gamma\theta_0} < n^{C_{\pi,1}},
\end{align*}
for arbitrary $C_{\pi,1}>0$.
\begin{align*}
\sup_{\alpha\in [\underline\alpha_n, \overline\alpha_n]} \sup_{\theta\in (\theta_0/2,2\theta_0)}
\frac{\pi(\theta|\alpha)}{\pi(\theta_0|\alpha)} &\leq \sup_{\alpha\in [\underline\alpha_n, \overline\alpha_n]} \sup_{\theta\in (\theta_0/2,2\theta_0)} \left(\frac{\theta}{\theta_0}\right)^{\gamma_2/\gamma-1} \left[\frac{b^{1/\gamma}\alpha^{2\nu/\gamma}+\theta_0^{1/\gamma}}{b^{1/\gamma}\alpha^{2\nu/\gamma}+\theta^{1/\gamma}}\right]^{\gamma_1+\gamma_2} \\
&\leq \sup_{\theta\in (\theta_0/2,2\theta_0)}  \max\left\{\left(\frac{\theta}{\theta_0}\right)^{\gamma_2/\gamma-1}, \left(\frac{\theta}{\theta_0}\right)^{-\gamma_1/\gamma-1}\right\}< n^{C_{\pi,2}} ,
\end{align*}
for arbitrary $C_{\pi,2}>0$. Since $C_{\pi,1}$ and $C_{\pi,2}$ can be arbitrarily small, $C_{\pi,1}+C_{\pi,2}<1/2$ is satisfied.
\begin{align*}
\inf_{\alpha\in [\underline\alpha_n, \overline\alpha_n]} \log \pi(\theta_0|\alpha)
& \geq \inf_{\alpha\in [\underline\alpha_n, \overline\alpha_n]} \Bigg\{ -2\nu \log \alpha- \log \frac{\Gamma(\gamma_1+\gamma_2)}{b\gamma \Gamma(\gamma_1)\Gamma(\gamma_2)} + \left(\frac{\gamma_2}{\gamma}-1\right)\log \frac{\theta_0}{b\alpha^{2\nu}} \\
&\quad  - \left(\gamma_1+\gamma_2\right) \log \left[1+\left(\frac{\theta_0}{b\alpha^{2\nu}}\right)^{1/\gamma}\right] \Bigg\} \\
& \geq -2\nu \overkappa \log n- \log \frac{\Gamma(\gamma_1+\gamma_2)}{b\gamma \Gamma(\gamma_1)\Gamma(\gamma_2)} + \left(\frac{\gamma_2}{\gamma}-1\right)\log \frac{\theta_0}{b} \\
& \quad - \left|\frac{\gamma_2}{\gamma}-1\right|\cdot 2\nu (\overkappa  + \underkappa) \log n  - \left(\gamma_1+\gamma_2\right) \log \left[1+\left(\frac{\theta_0 n^{2\nu \underkappa}}{b}\right)^{1/\gamma} \right] \\
&\succeq -\log n \succ -n^{C_{\pi,3}},
\end{align*}
for arbitrarily small $C_{\pi,3}>0$.
\end{proof}

\vspace{8mm}

\begin{proof}[Proof of Proposition \ref{prop:prior3}]
We will verify only \eqref{ineq:g.right} with $0<\overline{c_{\pi}}<(\nu+d/2)\overkappa$ for each conditions in the list. The verification of \eqref{ineq:g.left} with $0<\underline{c_{\pi}}<(\nu+d/2)\underkappa$ is similar and omitted.

For $p(\alpha)$ that satisfies $(i)$, we use the change of variable $u=\alpha^{1/\delta_1}$ to obtain that
\begin{align}\label{ineq:g11}
& \int_{\overline\alpha_n}^{\infty} \alpha^{n(\nu+d/2)} p(\alpha) \ud \alpha
\leq \int_{\overline\alpha_n}^{\infty} \alpha^{n(\nu+d/2)} \exp\left(-\alpha^{\delta_1}\right) \ud \alpha \nonumber \\
& \leq \frac{1}{\delta_1} \int_{\overline\alpha_n^{\delta_1}}^{\infty} u^{\{n(\nu+d/2) + 1\}/\delta_1 -1 } \ee^{-u} \ud u
< \frac{1}{\delta_1} \int_0^{\infty} u^{\{n(\nu+d/2) + 1\}/\delta_1 -1 } \ee^{-u} \ud u \nonumber \\
& = \frac{1}{\delta_1} \Gamma\left(\delta_1^{-1}\{n(\nu+d/2) + 1\}\right),
\end{align}
where $\Gamma(x) = \int_0^{\infty} u^{x-1} \ee^{-u} \ud u$ is the gamma function. Using the Stirling's approximation for gamma functions ($\Gamma(x)< 2\sqrt{2\pi x}(x/\ee)^x$ for all large $x>0$), we have that for sufficiently large $n$,
\begin{align}\label{ineq:g12}
& \Gamma\left(\delta_1^{-1}\{n(\nu+d/2) + 1\}\right)  \nonumber \\
&< 2\sqrt{2\pi\delta_1^{-1}\{n(\nu+d/2) + 1\}}\left( \ee^{-1}\delta_1^{-1}\{n(\nu+d/2) + 1\}\right)^{\delta_1^{-1}\{n(\nu+d/2) + 1\}}.
\end{align}
From \eqref{ineq:g11} and \eqref{ineq:g12}, we can see that \eqref{ineq:g.right} will be satisfied if for all sufficiently large $n$,
\begin{align}
&2\delta_1^{-1}\sqrt{2\pi\delta_1^{-1}\{n(\nu+d/2) + 1\}}\left( \ee^{-1}\delta_1^{-1}\{n(\nu+d/2) + 1\}\right)^{\delta_1^{-1}\{n(\nu+d/2) + 1\}} \nonumber \\
&< \exp(\overline{c_{\pi}} n\log n). \nonumber
\end{align}
A comparison of the orders in $n$ on both sides immediately shows that this relation holds for all sufficiently large $n$, as long as $\delta_1^{-1}(\nu+d/2) < \overline{c_{\pi}}$. Since $\overline{c_{\pi}}$ can be chosen as any constant between 0 and $(\nu+d/2)\overkappa$, it suffices to have $\delta_1^{-1}(\nu+d/2) < (\nu+d/2)\overkappa$, or equivalently $\delta_1> 1/\overkappa$.

For $p(\alpha)$ that satisfies $(ii)$, we use the change of variable $u=n^{\delta_2}\alpha$ and the Stirling's approximation to obtain that
\begin{align}
& \int_{\overline\alpha_n}^{\infty} \alpha^{n(\nu+d/2)} p(\alpha) \ud \alpha
\leq \int_{\overline\alpha_n}^{\infty} \alpha^{n(\nu+d/2)} n^{\delta_3} \exp\left(-n^{\delta_2}\alpha\right) \ud \alpha \nonumber \\
& \leq n^{\delta_3-\delta_2\{n(\nu+d/2)+1\}} \int_{n^{\delta_2}\overline\alpha_n}^{\infty} u^{n(\nu+d/2)} \ee^{-u} \ud u
< n^{\delta_3-\delta_2\{n(\nu+d/2)+1\}} \int_0^{\infty} u^{n(\nu+d/2)} \ee^{-u} \ud u \nonumber \\
& = n^{\delta_3-\delta_2\{n(\nu+d/2)+1\}} \Gamma\left(n(\nu+d/2) + 1 \right) \nonumber \\
&\leq n^{\delta_3-\delta_2\{n(\nu+d/2)+1\}} \cdot 2\sqrt{2\pi\{n(\nu+d/2) + 1\}} \times \left( \ee^{-1}\{n(\nu+d/2) + 1\}\right)^{n(\nu+d/2) + 1}. \nonumber
\end{align}
From the last display, \eqref{ineq:g.right} will be satisfied if for all sufficiently large $n$,
\begin{align*}
& n^{\delta_3-\delta_2\{n(\nu+d/2)+1\}} \cdot 2\sqrt{2\pi\{n(\nu+d/2) + 1\}} \nonumber \\
&\quad \times \left( \ee^{-1}\{n(\nu+d/2) + 1\}\right)^{n(\nu+d/2) + 1} < \exp(\overline{c_{\pi}} n\log n).
\end{align*}
A comparison of the orders in $n$ on both sides immediately shows that this relation holds for all sufficiently large $n$, as long as $-\delta_2(\nu+d/2)+ (\nu+d/2) < \overline{c_{\pi}}$. Since $\overline{c_{\pi}}$ can be chosen as any constant between 0 and $(\nu+d/2)\overkappa$, it suffices to have $(1-\delta_2)(\nu+d/2) < (\nu+d/2)\overkappa$, or equivalently $\delta_2> 1-\overkappa$.
\end{proof}

\vspace{3mm}

\section{Proof of Theorem \ref{thm:OU1} and Corollary \ref{cor:OU2}} \label{supsec:OU1}
In this section, we provide the proof of Theorem \ref{thm:OU1} and Corollary \ref{cor:OU2} for the limiting distribution for 1-dimensional Ornstein-Uhlenbeck process. Before that, we first elaborate on the possible choices of prior $\pi(\alpha)$ and its hyperparameters that satisfy the relaxed Assumption \ref{prior.3OU} on the tails of $\pi(\alpha)$.

\begin{itemize}[leftmargin=5mm]
\item If we take $\pi(\alpha)$ to be the gamma density $\pi(\alpha)=\frac{b^a}{\Gamma(a)}\alpha^{a-1} \ee^{-b\alpha}$, then for all sufficiently large $n$,
\begin{align*}
\sqrt{n} \int_0^{\underline\alpha_n} \sqrt{\alpha} \pi(\alpha) \ud \alpha & = \sqrt{n} \int_0^{n^{-\underline\kappa}} \frac{b^a}{\Gamma(a)}\alpha^{a+1/2-1} \ee^{-b\alpha} \ud \alpha \\
&\leq \frac{\sqrt{n} b^a}{\Gamma(a)} \int_0^{n^{-\underline\kappa}} \alpha^{a+1/2-1} \ud \alpha = \frac{b^a}{(a+1/2)\Gamma(a)} n^{-\underline\kappa(a+1/2)+1/2}, \\
\text{and }\quad \sqrt{n} \int_{\overline\alpha_n}^{\infty} \sqrt{\alpha} \pi(\alpha) \ud \alpha &= \sqrt{n} \int_{n^{\overline\kappa}}^{\infty} \frac{b^a}{\Gamma(a)}\alpha^{a+1/2-1} \ee^{-b\alpha} \ud \alpha \\
&\leq \frac{\sqrt{n} b^a}{\Gamma(a)}  \int_{n^{\overline\kappa}}^{\infty} \ee^{-b\alpha/2} \ud \alpha = \frac{2b^{a-1}}{\Gamma(a)}\sqrt{n} \exp(-bn^{\overline\kappa}/2).
\end{align*}
To satisfy \eqref{A3.1.OU} in Assumption \ref{prior.3OU}, we need the condition $-\underline\kappa(a+1/2)+1/2<0$, or $a>(\underline\kappa^{-1}-1)/2$. Therefore, Assumption \ref{prior.3OU} holds for the gamma prior density $\pi(\alpha)$ with hyperparameters $a>(\underline\kappa^{-1}-1)/2$ and all $b>0$.
\item If we take $\pi(\alpha)$ to be the inverse gamma density $\pi(\alpha)=\frac{b^a}{\Gamma(a)}\alpha^{-(a+1)} \ee^{-b/\alpha}$, then similar to the derivation above, we obtain that Assumption \ref{prior.3OU} holds for the inverse gamma prior density $\pi(\alpha)$ with hyperparameters $a>(\overline\kappa^{-1}-1)/2$ and all $b>0$.
\item If we take $\pi(\alpha)$ to be the inverse Gaussian density $\pi(\alpha)= \sqrt{\frac{b}{2\pi\alpha^3}}\exp\left\{-\frac{b(\alpha-a)^2}{2a^2\alpha}\right\}$ for $a>0,b>0$, then
 for all sufficiently large $n$,
\begin{align*}
\sqrt{n} \int_0^{\underline\alpha_n} \sqrt{\alpha} \pi(\alpha) \ud \alpha & = \sqrt{n} \int_0^{n^{-\underline\kappa}} \sqrt{\frac{b}{2\pi\alpha^3}}\exp\left\{-\frac{b(\alpha-a)^2}{2a^2\alpha}\right\} \ud \alpha \\
&\leq \sqrt{n} \sqrt{\frac{b}{2\pi}} \exp(b/a) \int_{n^{\underline\kappa}}^{\infty} t^{-1/2} \exp\left\{-bt/(2a^2)\right\} \ud t    \\
&\leq \sqrt{n} \sqrt{\frac{b}{2\pi}} \exp(b/a) \int_{n^{\underline\kappa}}^{\infty} \exp\left\{-bt/(4a^2)\right\} \ud t    \\
&= \frac{4a^2}{b}\sqrt{\frac{b}{2\pi}} \exp(b/a)\sqrt{n} \exp\left\{ - bn^{\underline\kappa}/(4a^2) \right\} \to 0, \\
\text{and }\quad \sqrt{n} \int_{\overline\alpha_n}^{\infty} \sqrt{\alpha} \pi(\alpha) \ud \alpha &= \sqrt{n} \int_{n^{\overline\kappa}}^{\infty} \sqrt{\frac{b}{2\pi\alpha^3}}\exp\left\{-\frac{b(\alpha-a)^2}{2a^2\alpha}\right\} \ud \alpha \\
&\leq \sqrt{n} \sqrt{\frac{b}{2\pi}} \exp(b/a) \int_{n^{\overline\kappa}}^{\infty} \alpha^{-3/2} \exp\left\{-b\alpha/(2a^2)\right\} \ud \alpha    \\
&\leq \sqrt{n} \sqrt{\frac{b}{2\pi}} \exp(b/a) \int_{n^{\overline\kappa}}^{\infty} \exp\left\{-b\alpha/(2a^2)\right\} \ud \alpha    \\
&= \frac{2a^2}{b}\sqrt{\frac{b}{2\pi}} \exp(b/a)\sqrt{n}  \exp\left\{ - bn^{\overline\kappa}/(2a^2) \right\} \to 0 .
\end{align*}
Therefore, the inverse Gaussian density $\pi(\alpha)$ satisfies \eqref{A3.1.OU} in Assumption \ref{prior.3OU} for all hyperparameter values of $a>0$ and $b>0$.
\item If we take $\pi(\alpha)$ to be the generalized beta density of the second kind: $$\pi(\alpha)=\frac{\Gamma(\gamma_1+\gamma_2)}{\Gamma(\gamma_1)\Gamma(\gamma_2)}\frac{(\alpha/b)^{\gamma_2/\gamma-1}}{b\gamma [1+(\alpha/b)^{1/\gamma}]^{\gamma_1+\gamma_2}}$$
    with parameters $b>0,\gamma>0,\gamma_1>0,\gamma_2>0$, then for all sufficiently large $n$,
\begin{align*}
\sqrt{n} \int_0^{\underline\alpha_n} \sqrt{\alpha} \pi(\alpha) \ud \alpha & = \sqrt{n} \int_0^{n^{-\underline\kappa}} \frac{\Gamma(\gamma_1+\gamma_2)}{\Gamma(\gamma_1)\Gamma(\gamma_2)}\frac{(\alpha/b)^{\gamma_2/\gamma-1}}{b\gamma [1+(\alpha/b)^{1/\gamma}]^{\gamma_1+\gamma_2}} \ud \alpha \\
&\leq \frac{\Gamma(\gamma_1+\gamma_2)}{b^{\gamma_2/\gamma}\gamma\Gamma(\gamma_1)\Gamma(\gamma_2)} \sqrt{n} \int_0^{n^{-\underline\kappa}} \alpha^{\gamma_2/\gamma-1} \ud \alpha \\
&= \frac{\Gamma(\gamma_1+\gamma_2)}{b^{\gamma_2/\gamma}\gamma_2 \Gamma(\gamma_1)\Gamma(\gamma_2)} n^{-\underline\kappa \gamma_2/\gamma+1/2}, \\
\text{and } \quad \sqrt{n} \int_{\overline\alpha_n}^{\infty} \sqrt{\alpha} \pi(\alpha) \ud \alpha &= \sqrt{n} \int_{n^{\overline\kappa}}^{\infty} \frac{\Gamma(\gamma_1+\gamma_2)}{\Gamma(\gamma_1)\Gamma(\gamma_2)}\frac{(\alpha/b)^{\gamma_2/\gamma-1}}{b\gamma [1+(\alpha/b)^{1/\gamma}]^{\gamma_1+\gamma_2}}  \ud \alpha \\
&\leq \frac{\Gamma(\gamma_1+\gamma_2)}{b\gamma \Gamma(\gamma_1)\Gamma(\gamma_2)} \sqrt{n}  \int_{n^{\overline\kappa}}^{\infty}  (\alpha/b)^{-(\gamma_1/\gamma + 1)} \ud \alpha \\
&= \frac{\Gamma(\gamma_1+\gamma_2)}{b^{\gamma_1/\gamma}\gamma_1 \Gamma(\gamma_1)\Gamma(\gamma_2)} n^{-\overline\kappa \gamma_1/\gamma + 1/2}.
\end{align*}
To satisfy \eqref{A3.1.OU} in Assumption \ref{prior.3OU}, we need the conditions $-\underline\kappa \gamma_2/\gamma+1/2<0$ and $-\overline\kappa \gamma_1/\gamma + 1/2<0$, or equivalently, $\gamma_2/\gamma > 1/(2\underline\kappa)$ and $\gamma_1/\gamma> 1/(2\overline\kappa)$. Therefore, if $\pi(\alpha)$ is the generalized beta density of the second kind, then it satisfies Assumption \ref{prior.3OU} if its hyperparameters $(b,\gamma,\gamma_1,\gamma_2)$ satisfy $\gamma_1/\gamma> 1/(2\overline\kappa)$ and $\gamma_2/\gamma > 1/(2\underline\kappa)$.
\end{itemize}

\subsection{Proof of Theorem \ref{thm:OU1}} \label{supsubsec:OU1}

Recall that for Case (i) in Section \ref{subsec:OU1} of the main text, we observe the 1-dimensional Ornstein-Uhlenbeck process without regression term $Y(\cdot)=X(\cdot)\sim \gp(0,\sigma_0^2 K_{\alpha_0,\nu})$ on the grid $s_i=i/n$, for $i=1,\ldots,n$. Since $Y(s_i)=X(s_i)$ for all $i=1,\ldots,n$ in this case, we have
$$A_1=\sum_{i=2}^{n-1} X(s_i)^2, \quad A_2=\sum_{i=1}^{n-1} X(s_i)X(s_{i+1}), \quad A_3=\sum_{i=1}^n X(s_i)^2.$$
In the following, for any random variable $Z_n$, we write $Z_n\asymp 1$ to denote that $Z_n$ is lower bounded away from zero and upper bounded from infinity as $n\to\infty$ in $P_{(\sigma_0^2,\alpha_0)}$-probability. The $O_p(\cdot)$ notation refers to the true probability measure $P_{(\sigma_0^2,\alpha_0)}$.

We introduce two technical Lemmas \ref{lem:OU.orders} and \ref{lem:2limdiff}.
\begin{lemma} \label{lem:OU.orders}
Under the model setup of Theorem \ref{thm:OU1}, we have the following results:
\begin{itemize}[leftmargin=5mm]
\item[(i)] $A_1+A_3-2A_2>0$ a.s. $P_{(\sigma_0^2,\alpha_0)}$;
\item[(ii)] $A_1+A_3-2A_2\asymp 1$ as $n\to\infty$ a.s. $P_{(\sigma_0^2,\alpha_0)}$;
\item[(iii)] $|A_1-A_2|=O_p(1)$ as $n\to\infty$ in $P_{(\sigma_0^2,\alpha_0)}$-probability, and $|A_1-A_2|\preceq \log^2 n$ as $n\to\infty$ a.s. $P_{(\sigma_0^2,\alpha_0)}$;
\item[(iv)] $A_1/n\asymp 1$ and $A_3/n\asymp 1$ as $n\to\infty$ in $P_{(\sigma_0^2,\alpha_0)}$-probability;
\item[(v)] $|u_*|=n|A_1-A_2|/A_1=O_p(1)$ as $n\to\infty$ in $P_{(\sigma_0^2,\alpha_0)}$-probability;
\item[(vi)] $v_*=n(A_1-2A_2+A_3)/A_1\asymp 1$ as $n\to\infty$ in $P_{(\sigma_0^2,\alpha_0)}$-probability;
\item[(vii)] $\frac{A_1A_3-A_2^2}{A_1(A_1-2A_2+A_3)}  = 1 + O_p(n^{-1})$ as $n\to\infty$ in $P_{(\sigma_0^2,\alpha_0)}$-probability;
\item[(viii)] Uniformly over all $\alpha\in [0,n^{1/6}]$,
\begin{align*}
& \frac{\left| \left(A_1 \ee^{-2\alpha/n}-2A_2 \ee^{-\alpha/n} +A_3\right) - \left[ A_1\left(\frac{\alpha}{n}-\frac{A_1-A_2}{A_1}\right)^2 + \frac{A_1A_3-A_2^2}{A_1}\right] \right|}{A_1\left(\frac{\alpha}{n}-\frac{A_1-A_2}{A_1}\right)^2 + \frac{A_1A_3-A_2^2}{A_1}} = O_p\left(n^{-3/2}\right),
\end{align*}
as $n\to\infty$ in $P_{(\sigma_0^2,\alpha_0)}$-probability;;
\item[(ix)] Uniformly over all $\alpha\in [0,n^{1/6}]$,
\begin{align*}
& \sqrt{1- \ee^{-2\alpha/n}} = \sqrt{\frac{2\alpha}{n}}\left[1+O(n^{-5/12})\right],
\end{align*}
as $n\to\infty$.
\end{itemize}
\end{lemma}

\begin{proof}[Proof of Lemma \ref{lem:OU.orders}]
\noindent (i) By definition, $A_1+A_3-2A_2=\sum_{i=1}^{n-1}\left[ X(s_{i+1}) - X(s_i) \right]^2 > 0$ almost surely $P_{(\sigma_0^2,\alpha_0)}$.
\vspace{4mm}

\noindent (ii) Let $W_{i,n}=[X(s_i)- \ee^{-\alpha_0/n}X(s_{i-1})]/\sqrt{\sigma_0^2(1- \ee^{-2\alpha_0/n})}$ for $i=2,\ldots,n$. Then by the Markov property of Ornstein-Uhlenbeck process, $W_{i,n}$'s are i.i.d. $\mathcal{N}(0,1)$ random variables, such that $W_{i,n}$ is independent of $X(s_{i-1})$, and $X(s_i)= \ee^{-\alpha_0/n}X(s_{i-1})+\sqrt{\sigma_0^2(1- \ee^{-2\alpha_0/n})}W_{i,n}$, for $i=2,\ldots,n$. We can derive that
\begin{align} \label{eq:3Aorder1}
A_1+A_3-2A_2 ={}& \sum_{i=1}^{n-1}\left[ X(s_{i+1}) - \ee^{-\alpha_0/n}X(s_i) - (1- \ee^{-\alpha_0/n})X(s_i) \right]^2 \nonumber \\
={}& \sum_{i=1}^{n-1}\left[ X(s_{i+1}) - \ee^{-\alpha_0/n}X(s_i)\right]^2 + \sum_{i=1}^{n-1}(1- \ee^{-\alpha_0/n})^2X(s_i)^2 \nonumber \\
&+2\sum_{i=1}^{n-1}(1- \ee^{-\alpha_0/n})X(s_i)\left[ X(s_{i+1}) - \ee^{-\alpha_0/n}X(s_i)\right].
\end{align}
The first term in \eqref{eq:3Aorder1} is
\begin{align*}
\sum_{i=1}^{n-1}\left[ X(s_{i+1}) - \ee^{-\alpha_0/n}X(s_i)\right]^2 & = \sum_{i=2}^{n}\sigma_0^2(1- \ee^{-2\alpha_0/n})W_{i,n}^2 = \frac{\sigma_0^2\alpha_0[1+o(1)]}{n}\sum_{i=2}^n W_{i,n}^2,
\end{align*}
using a Taylor expansion of $1- \ee^{-x}$ around $x=0$. Since $W_{i,n}$'s are i.i.d. $\mathcal{N}(0,1)$ random variables, we have that $n^{-1}\sum_{i=2}^n W_{i,n}^2\to 1$ as $n\to\infty$ almost surely $P_{(\sigma_0^2,\alpha_0)}$.

The second term in \eqref{eq:3Aorder1} is
\begin{align*}
\sum_{i=1}^{n-1}(1- \ee^{-\alpha_0/n})^2X(s_i)^2 & \leq \frac{\alpha_0^2}{n} \sup_{s\in [0,1]}X(s)^2.
\end{align*}
For the Ornstein-Uhlenbeck process, $\sup_{s\in [0,1]}X(s)^2<\infty$ almost surely $P_{(\sigma_0^2,\alpha_0)}$. Therefore, $\sum_{i=1}^{n-1}(1- \ee^{-\alpha_0/n})^2X(s_i)^2=O(1/n)$ almost surely $P_{(\sigma_0^2,\alpha_0)}$.

The third term in \eqref{eq:3Aorder1} can be upper bounded by
\begin{align}\label{eq:W.cross1}
& 2\sum_{i=1}^{n-1}(1- \ee^{-\alpha_0/n})X(s_i)\left[ X(s_{i+1}) - \ee^{-\alpha_0/n}X(s_i)\right] \nonumber \\
={}& 2\sum_{i=1}^{n-1} \sqrt{\sigma_0^2(1- \ee^{-2\alpha_0/n})} (1- \ee^{-\alpha_0/n}) X(s_i)W_{i+1,n}  \nonumber \\
={}& \frac{2\sqrt{2}\sigma_0\alpha_0^{3/2}[1+o(1)]}{n^{3/2}} \sum_{i=1}^{n-1}X(s_i)W_{i+1,n} \nonumber  \\
\leq{}& \frac{2\sqrt{2}\sigma_0\alpha_0^{3/2}[1+o(1)]}{n} \sqrt{\sum_{i=1}^{n-1}X(s_i)^2W_{i+1,n}^2}  \nonumber \\
\leq{}& \frac{2\sqrt{2}\sigma_0\alpha_0^{3/2}[1+o(1)]}{\sqrt{n}} \sqrt{\sup_{s\in[0,1]}X(s)^2 \frac{1}{n}\sum_{i=2}^{n}W_{i,n}^2},
\end{align}
which shows that the third term is $O(n^{-1/2})$ almost surely $P_{(\sigma_0^2,\alpha_0)}$.

In combination with \eqref{eq:3Aorder1}, we have shown that $A_1+A_3-A_2 \to \sigma_0^2\alpha_0 = \theta_0>0$ as $n\to\infty$ almost surely $P_{(\sigma_0^2,\alpha_0)}$, which means that $A_1+A_3-A_2\asymp 1$.
\vspace{4mm}

\noindent (iii)
\begin{align}
|A_1-A_2| &\leq \frac{1}{2}\left|A_1+A_3-2A_2\right| + \frac{1}{2}[X(s_1)^2+X(s_n)^2].\nonumber
\end{align}
Since $X(s_1)\sim \mathcal{N}(0,\sigma_0^2)$, $X(s_n)\sim \mathcal{N}(0,\sigma_0^2)$, we have $X(s_1)=O_p(1)$ and $X(s_n)=O_p(1)$. Furthermore, by the Borel-Cantelli lemma, $X(s_1)\preceq \log n$ and $X(s_n)\preceq \log n$ as $n\to\infty$ almost surely $P_{(\sigma_0^2,\alpha_0)}$. Then the conclusion follows by combining these relations with Part (ii).
\vspace{4mm}

\noindent (iv) First $A_3/n \leq \sup_{s\in [0,1]}X(s)^2 <\infty$ almost surely $P_{(\sigma_0^2,\alpha_0)}$. The expectation of $A_3/n$ is ${\EE}_{(\sigma_0^2,\alpha_0)}(A_3/n)=\sum_{i=1}^n {\EE}_{(\sigma_0^2,\alpha_0)}(X(s_i)^2)/n = \sigma_0^2$. To calculate the variance of $A_3$, we let $T_{ij}=\left[X(s_i)-\ee^{-\alpha_0|i-j|/n}X(s_j)\right]/\sqrt{\sigma_0^2\left(1-\ee^{-2\alpha_0|i-j|/n}\right)}$ for any $i\neq j$ and $i,j=1,\ldots,n$. By the Markov property of the OU process, $T_{ij}\sim \Ncal(0,1)$. Therefore, given that each $X(s_i)\sim \Ncal(0,\sigma_0^2)$, ${\EE}_{(\sigma_0^2,\alpha_0)}\left[X(s_i)^2\right]=\sigma_0^2$, ${\EE}_{(\sigma_0^2,\alpha_0)}\left[X(s_i)^3\right]=0$, ${\EE}_{(\sigma_0^2,\alpha_0)}\left[X(s_i)^4\right]=3\sigma_0^4$, we have that
\begin{align*}
&\quad~ {\Var}_{(\sigma_0^2,\alpha_0)}(A_3/n) \\
& = \frac{1}{n^2} \left\{ \sum_{i=1}^n {\Var}_{(\sigma_0^2,\alpha_0)}(X(s_i)^2) + \sum_{i\neq j} {\Cov}_{(\sigma_0^2,\alpha_0)}\left(X(s_i)^2, X(s_j)^2\right) \right\} \\
&= \frac{1}{n^2} \left\{2n\sigma_0^4 + \sum_{i\neq j} \left[{\EE}_{(\sigma_0^2,\alpha_0)}\left(X(s_i)^2X(s_j)^2\right)- {\EE}_{(\sigma_0^2,\alpha_0)}\left(X(s_i)^2\right){\EE}_{(\sigma_0^2,\alpha_0)}\left(X(s_j)^2\right) \right] \right\} \\
&= \frac{1}{n^2} \left\{2n\sigma_0^4 + \sum_{i\neq j} \left[{\EE}_{(\sigma_0^2,\alpha_0)}\left(\left[\ee^{-\alpha_0|i-j|/n}X(s_j) + \sqrt{\sigma_0^2\left(1-\ee^{-2\alpha_0|i-j|/n}\right)}T_{ij}\right]^2X(s_j)^2\right)- \sigma_0^4 \right] \right\} \\
&= \frac{1}{n^2} \Bigg\{2n\sigma_0^4 + \sum_{i\neq j} \Bigg[{\EE}_{(\sigma_0^2,\alpha_0)}\Bigg\{\ee^{-2\alpha_0|i-j|/n}X(s_j)^4 + \sigma_0^2\left(1-\ee^{-2\alpha_0|i-j|/n}\right)T_{ij}^2 X(s_j)^2 \\
&\quad + 2X(s_j)^3\cdot \sqrt{\sigma_0^2\left(1-2\ee^{-\alpha_0|i-j|/n}\right)}T_{ij} \Bigg\}- \sigma_0^4 \Bigg] \Bigg\} \\
&= \frac{1}{n^2} \Bigg\{2n\sigma_0^4 + \sum_{i\neq j} \Bigg[3\sigma_0^4\ee^{-2\alpha_0|i-j|/n} + \sigma_0^4\left(1-\ee^{-2\alpha_0|i-j|/n}\right) + 0 - \sigma_0^4 \Bigg] \Bigg\} \\
&= \frac{1}{n^2} \Big\{2n\sigma_0^4 + 2\sigma_0^4\sum_{i\neq j} \ee^{-2\alpha_0|i-j|/n} \Big\} \\
&= 2\sigma_0^4 \left\{\frac{1}{n} + \frac{2(n-1)\ee^{-2\alpha_0/n}-2n\ee^{-4\alpha_0/n}+2\ee^{-2\alpha_0(n+1)/n}}{n^2\left(1-\ee^{-2\alpha_0/n}\right)^2}\right\}.
\end{align*}
Therefore, as $n\to\infty$, we have ${\Var}_{(\sigma_0^2,\alpha_0)}(A_3/n) \to \frac{\sigma_0^4 \ee^{-2\alpha_0}}{\alpha_0^2} $.

For any small number $\epsilon \in (0,1)$, we can apply the one-sided Chebyshev's inequality (or Cantelli's inequality) to obtain that as $n\to\infty$,
\begin{align*}
& \pr \left(A_3/n \leq {\EE}_{(\sigma_0^2,\alpha_0)}(A_3/n)  (1- \epsilon) \right) \\
\leq{}&  \frac{{\Var}_{(\sigma_0^2,\alpha_0)}(A_3/n)}{{\Var}_{(\sigma_0^2,\alpha_0)}(A_3/n) + \epsilon^2 \left[{\EE}_{(\sigma_0^2,\alpha_0)}(A_3/n)\right]^2} \to \frac{\ee^{-2\alpha_0}/\alpha_0^2}{\ee^{-2\alpha_0}/\alpha_0^2 + \epsilon^2} < 1 .
\end{align*}
Therefore, for any $\epsilon \in (0,1)$, $\pr\left(A_3/n >(1- \epsilon)\sigma_0^2\right)>0$ for all sufficiently large $n$, which implies that $A_3/n$ is lower bounded as $n\to\infty$ in $P_{(\sigma_0^2,\alpha_0)}$-probability (or equivalently, $A_3/n$ does not converge to zero as $n\to\infty$ in $P_{(\sigma_0^2,\alpha_0)}$-probability).

Since $A_1\leq A_3$, $A_1/n$ is also upper bounded as $n\to\infty$ in $P_{(\sigma_0^2,\alpha_0)}$-probability. Since $A_1/n=A_3/n- \left[X(s_1)^2+X(s_n)^2\right]/n$ and $\left[X(s_1)^2+X(s_n)^2\right]/n\to 0$ as $n\to\infty$ in $P_{(\sigma_0^2,\alpha_0)}$-probability, we can see that $A_1/n$ is also lower bounded as $n\to\infty$ in $P_{(\sigma_0^2,\alpha_0)}$-probability.
\vspace{4mm}

\noindent (v) Since $u_*=(A_1-A_2)/(A_1/n)$, the conclusion follows from (iii) and (iv).

\vspace{4mm}

\noindent (vi) Since $v_*=(A_1-2A_2+A_3)/(A_1/n)$, the conclusion follows from (ii) and (iv).

\vspace{4mm}

\noindent (vii) Using the notation of $u_*$ and $v_*$ in Parts (v) and (vi), we have
\begin{align*}
& 1 - \frac{A_1A_3-A_2^2}{A_1(A_1-2A_2+A_3)} = \frac{(A_1-A_2)^2}{A_1(A_1-2A_2+A_3)} = \frac{u_*^2}{nv_*} .
\end{align*}
From Parts (v) and (vi), we have that $u_*^2/(nv_*)=O_p(n^{-1})$ as $n\to\infty$ in $P_{(\sigma_0^2,\alpha_0)}$-probability.

\vspace{4mm}

\noindent (viii) We have
\begin{align} \label{eq:3Asq.1}
& A_1 \ee^{-2\alpha/n}-2A_2 \ee^{-\alpha/n} +A_3 \nonumber \\
={}& A_1\left[(1- \ee^{-\alpha/n})-\frac{A_1-A_2}{A_1}\right]^2 + A_1+A_3-2A_2 - \frac{(A_1-A_2)^2}{A_1}
\nonumber \\
={}& A_1\left[(1- \ee^{-\alpha/n})-\frac{A_1-A_2}{A_1}\right]^2 + \frac{A_1A_3-A_2^2}{A_1} .
\end{align}
Now if we replace $1- \ee^{-\alpha/n}$ with $\alpha/n$ for all $\alpha\in [0, n^{1/6}]$, then the difference would be
\begin{align} \label{eq:3Asq.2}
& \left| \left(A_1 \ee^{-2\alpha/n}-2A_2 \ee^{-\alpha/n} +A_3\right) - \left[ A_1\left(\frac{\alpha}{n}-\frac{A_1-A_2}{A_1}\right)^2 + \frac{A_1A_3-A_2^2}{A_1}\right] \right|\nonumber \\
\leq{}& A_1\left|\left[(1- \ee^{-\alpha/n})-\frac{A_1-A_2}{A_1}\right]^2 - \left(\frac{\alpha}{n}-\frac{A_1-A_2}{A_1}\right)^2 \right|  \nonumber \\
={}& A_1 \left|1- \ee^{-\alpha/n} + \frac{\alpha}{n} + \frac{2(A_1-A_2)}{A_1}\right| \cdot \left|1- \ee^{-\alpha/n} - \frac{\alpha}{n} \right|  \nonumber \\
\stackrel{(i)}{\leq}{}& \left( A_1\frac{n^{1/6}}{n} + |A_1-A_2|\right) \frac{n^{1/3}}{n^2} ,
\end{align}
where (i) follows from the fact that $1- \ee^{-x}\leq x$ and $|x-(1- \ee^{-x})|\leq x^2/2$ for all $x>0$.
\eqref{eq:3Asq.2} implies that
\begin{align}
& \frac{\left| \left(A_1 \ee^{-2\alpha/n}-2A_2 \ee^{-\alpha/n} +A_3\right) - \left[ A_1\left(\frac{\alpha}{n}-\frac{A_1-A_2}{A_1}\right)^2 + \frac{A_1A_3-A_2^2}{A_1}\right] \right|}{A_1\left(\frac{\alpha}{n}-\frac{A_1-A_2}{A_1}\right)^2 + \frac{A_1A_3-A_2^2}{A_1}} \nonumber \\
\leq{}& \frac{n^{1/3}}{n^2} \cdot \frac{\left(\frac{A_1n^{1/6}}{n} + |A_1-A_2|\right)}{\frac{A_1A_3-A_2^2}{A_1}} . \nonumber
\end{align}
Using Parts (ii), (iii), (iv) and (vii) together with the definition of $\overline\alpha_n$, we observe that
\begin{align}
& \frac{n^{1/3}}{n^2} \cdot \frac{\left(\frac{A_1n^{1/6}}{n} + |A_1-A_2|\right)}{\frac{A_1A_3-A_2^2}{A_1}} \preceq \frac{n^{1/3}}{n^2} \cdot \frac{\left(n^{1/6} + \log^2 n\right)}{(A_1+A_3-2A_2)(1+O_p(n^{-1}))} \nonumber \\
&\preceq  \frac{n^{1/3}}{n^2} \cdot \frac{n^{1/6}}{1+O_p(n^{-1})} =  O_p\left(n^{-3/2}\right), \nonumber
\end{align}
as $n\to\infty$ in $P_{(\sigma_0^2,\alpha_0)}$-probability. Hence the conclusion follows.
\vspace{4mm}

\noindent (ix) For $\alpha\in [0,n^{1/6}]$, $\alpha/n\leq n^{-5/6}\to 0$ as $n\to\infty$. With the Taylor expansion of $1- \ee^{-x}$ around $x=0$, as $n\to\infty$ almost surely $P_{(\sigma_0^2,\alpha_0)}$,
\begin{align*}
& \sqrt{1- \ee^{-2\alpha/n}} = \sqrt{\frac{2\alpha}{n}\left[1+O(n^{-5/6})\right]} = \sqrt{\frac{2\alpha}{n}}\left[1+O(n^{-5/12})\right]
\end{align*}
and the $o(1)$ term is uniformly over all $\alpha\in [0,n^{1/6}]$.
\end{proof}

\vspace{8mm}

\begin{lemma} \label{lem:2limdiff}
Define a normalized log profile likelihood function
\begin{align} \label{eq:norm.logprofile}
\widetilde \Lcal_*(\alpha) &= \Lcal_n(\alpha^{-2\nu}\widetilde\theta_{\alpha}, \alpha) + \frac{n}{2}\log\left(\frac{A_1A_3-A_2^2}{A_1}\right) +\frac{1}{2}\log \frac{n}{2} \nonumber \\
&= -\frac{n}{2}\log \left(A_1 \ee^{-2\alpha/n}-2A_2 \ee^{-\alpha/n} +A_3\right) + \frac{1}{2}\log \left(1- \ee^{-2\alpha/n}\right) \nonumber \\
&\quad + \frac{n}{2}\log\left(\frac{A_1A_3-A_2^2}{A_1}\right) +\frac{1}{2}\log \frac{n}{2}.
\end{align}
$\widetilde \Lcal_*(\alpha)$ in \eqref{eq:norm.logprofile} is well defined for all sufficiently large $n$ in $P_{(\sigma_0^2,\alpha_0)}$-probability. Then, under the model setup of Theorem \ref{thm:OU1} and Assumptions \ref{prior.1}, \ref{prior.2}, and \ref{prior.3OU}, the integrals
\begin{align*}
\int_0^{\infty} \exp\left\{\widetilde \Lcal_*(\alpha)\right\} \pi(\theta_0|\alpha)\pi(\alpha) \ud \alpha, \quad \text{and } \int_0^{\infty} \exp\left\{\widetilde \Lcal_*(\alpha)\right\} \pi(\alpha) \ud \alpha
\end{align*}
are lower bounded by positive constants in $P_{(\sigma_0^2,\alpha_0)}$-probability. Furthermore, the following convergence relations hold
\begin{align}
&\int_0^{\infty} \left| \exp\left\{\widetilde \Lcal_*(\alpha)\right\} - \sqrt{\alpha} \exp\left\{-\frac{(\alpha-u_*)^2}{2v_*}\right\} \right|\pi(\theta_0|\alpha)\pi(\alpha) \ud \alpha \to 0, \label{eq:diff.2pi.OU} \\
&\int_0^{\infty} \left| \exp\left\{\widetilde \Lcal_*(\alpha)\right\} - \sqrt{\alpha} \exp\left\{-\frac{(\alpha-u_*)^2}{2v_*}\right\} \right| \pi(\alpha) \ud \alpha \to 0, \label{eq:diff.1pi.OU} \\
&\int_0^{\infty} \left|\widetilde \pi(\alpha|Y_n) - \pi_*(\alpha|Y_n) \right| \ud \alpha \to 0, \label{eq:2prolik.OU}
\end{align}
as $n\to\infty$ in $P_{(\sigma_0^2,\alpha_0)}$-probability, for $\widetilde \pi(\alpha|Y_n)$ given in Theorem \ref{thm:bvm2:joint} and $\pi_*(\alpha|Y_n)$ given in Theorem \ref{thm:OU1}.
\end{lemma}

\begin{proof}[Proof of Lemma \ref{lem:2limdiff}]
Based on Part (vii) of Lemma \ref{lem:OU.orders}, $(A_1A_3-A_2^2)/A_1>0$ as $n\to\infty$ in $P_{(\sigma_0^2,\alpha_0)}$-probability. Therefore, $\widetilde \Lcal_*(\alpha)$ in \eqref{eq:norm.logprofile} is well defined for all sufficiently large $n$ in $P_{(\sigma_0^2,\alpha_0)}$-probability.

We first prove the convergence in $P_{(\sigma_0^2,\alpha_0)}$-probability in \eqref{eq:diff.2pi.OU}, and that the integral $\int_0^{\infty} \exp\{\widetilde \Lcal_*(\alpha)\} \pi(\theta_0|\alpha)\pi(\alpha) \ud \alpha$ is lower bounded by positive constant in $P_{(\sigma_0^2,\alpha_0)}$-probability. Note that the only difference between \eqref{eq:diff.2pi.OU} and \eqref{eq:diff.1pi.OU} is that $\pi(\theta_0|\alpha)\pi(\alpha)$ is replaced by $\pi(\alpha)$. The integral condition \eqref{A3.1.OU} in Assumption \ref{prior.3OU} guarantees that in the following derivation, all $\pi(\theta_0|\alpha)\pi(\alpha)$ can be replaced by $\pi(\alpha)$. Therefore, in the derivation below, we will only prove for the integrals involving $\pi(\theta_0|\alpha)\pi(\alpha)$, and the proof of \eqref{eq:diff.1pi.OU} and lower boundedness of $\int_0^{\infty} \exp\{\widetilde \Lcal_*(\alpha)\} \pi(\alpha) \ud \alpha$ follow similarly.
\vspace{2mm}

\noindent \underline{Proof of \eqref{eq:diff.2pi.OU}:}
\vspace{2mm}

Define the following quantities
\begin{align*}
\widetilde {\numer}_1 & = \int_{0}^{n^{1/6}} \left| \exp\left\{\widetilde \Lcal_*(\alpha)\right\} - \sqrt{\alpha} \exp\left\{-\frac{(\alpha-u_*)^2}{2v_*}\right\} \right|\pi(\theta_0|\alpha)\pi(\alpha) \ud \alpha , \\
\widetilde {\numer}_2 & = \int_{n^{1/6}}^{\infty}\exp\left\{\widetilde \Lcal_*(\alpha)\right\} \pi(\theta_0|\alpha)\pi(\alpha) \ud \alpha , \\
\widetilde {\numer}_3 & = \int_{n^{1/6}}^{\infty} \sqrt{\alpha} \exp\left\{-\frac{(\alpha-u_*)^2}{2v_*}\right\} \pi(\theta_0|\alpha)\pi(\alpha) \ud \alpha , \\
\widetilde \denom & = \int_0^{\infty} \sqrt{\alpha} \exp\left\{-\frac{(\alpha-u_*)^2}{2v_*}\right\} \pi(\theta_0|\alpha)\pi(\alpha)\ud \alpha.
\end{align*}
We define an auxiliary ``variance" $\widetilde v_* = n(A_1A_3-A_2^2)/A_1^2$ which is positive as $n\to\infty$ in $P_{(\sigma_0^2,\alpha_0)}$-probability given Parts (i) and (vii) of Lemma \ref{lem:OU.orders}. Then, we have that uniformly for all $\alpha\in [0,n^{1/6}]$, as $n\to\infty$ in $P_{(\sigma_0^2,\alpha_0)}$-probability,
\begin{align} \label{eq:Lstar.diff1}
& \quad \left| \exp\left\{\widetilde \Lcal_*(\alpha)\right\} - \sqrt{\alpha} \exp\left\{-\frac{(\alpha-u_*)^2}{2\widetilde v_*}\right\} \right|  \nonumber \\
&= \left| \frac{[(A_1A_3-A_2^2)/A_1]^{n/2}\sqrt{\frac{n}{2}(1- \ee^{-2\alpha/n})}}{\left( A_1 \ee^{-2\alpha/n}-2A_2 \ee^{-\alpha/n} +A_3 \right)^{n/2}} - \sqrt{\alpha} \exp\left\{-\frac{(\alpha-u_*)^2}{2\widetilde v_*}\right\} \right| \nonumber \\
&\overset{(i)}{=}  \Bigg| \frac{[(A_1A_3-A_2^2)/A_1]^{n/2} \sqrt{\alpha}\left[1+O(n^{-5/12})\right]}{\left[ A_1\left(\frac{\alpha}{n}-\frac{A_1-A_2}{A_1}\right)^2 + \frac{A_1A_3-A_2^2}{A_1} \right]^{n/2}\left[1+O_p\left(n^{-3/2}\right)\right]^{n/2}} \nonumber\\
&\quad ~~ - \sqrt{\alpha} \exp\left\{-\frac{(\alpha-u_*)^2}{2\widetilde v_*}\right\} \Bigg|  \nonumber \\
&\overset{(ii)}{=} \Bigg| \sqrt{\alpha}\left[1+O_p(n^{-5/12})\right]\left[ 1+ \frac{1}{n} \frac{\left(\alpha-u_*\right)^2}{\widetilde v_*} \right]^{-n/2} - \sqrt{\alpha} \exp\left\{-\frac{(\alpha-u_*)^2}{2\widetilde v_*}\right\} \Bigg|  \nonumber\\
&\overset{(iii)}{\leq} O_p(n^{-5/12})\cdot \sqrt{\alpha}\left[ 1+ \frac{1}{n} \frac{\left(\alpha-u_*\right)^2}{\widetilde v_*} \right]^{-n/2} + \sqrt{\alpha}\exp\left\{-\frac{(\alpha-u_*)^2}{2\widetilde v_*}\right\} \nonumber\\
&\quad \times \left|\exp\left(\frac{n}{2}\left\{\frac{\left(\alpha-u_*\right)^2}{n\widetilde v_*} - \log \left[ 1+ \frac{\left(\alpha-u_*\right)^2}{n\widetilde v_*} \right]\right\}\right) - 1\right|\nonumber \\
&\overset{(iv)}{\leq} O_p(n^{-1/3}) + \sqrt{\alpha} \exp\left\{-\frac{(\alpha-u_*)^2}{2\widetilde v_*}\right\} \cdot \left[\exp\left(\frac{n}{2}\left\{\frac{\left(\alpha-u_*\right)^2}{n\widetilde v_*}\right\}^{11/6}\right) - 1\right] \nonumber \\
&\overset{(v)}{\leq} O_p(n^{-1/3}) + \sqrt{\alpha} \exp\left\{-\frac{(\alpha-u_*)^2}{2\widetilde v_*}\right\} \cdot \frac{n}{2} \left[\frac{\left(\alpha-u_*\right)^2}{n\widetilde v_*}\right]^{11/6} \nonumber \\
&\overset{(vi)}{\leq} O_p(n^{-1/3}) + O_p(n^{1/12} \cdot 1 \cdot n/2 \cdot n^{-11/9}) = O_p( n^{-5/36} )  .
\end{align}
In the derivations above, (i) follows from Lemma \ref{lem:OU.orders} (vii) and (viii); (ii) follows from the fact that $[1+O_p\left(n^{-3/2}\right)]^{-n/2}=1+O_p(n^{-1/2})$ and the definitions of $u_*$ and $\widetilde v_*$; (iii) follows from the triangle inequality; (iv) follows from Lemma \ref{lem:OU.orders} (v), (vi), and the fact that $\alpha\in [0,n^{1/6}]$, hence $(\alpha-u_*)^2/(2v_*)\preceq n^{1/3}$, and the inequality $0<x-\log(1+x)\leq x^{11/6}$ for all $x>0$; (v) follows from the inequality $\ee^x-1\leq 2x$ for $x\in (0,1)$ and for sufficiently large $n$; (vi) follows from a comparison of orders.

On the other hand, if we replace $\widetilde v_*$ with $v_*$, then Part (vii) of Lemma \ref{lem:OU.orders} implies that $\big(\widetilde v_*-v_*\big)/ v_* = O_p(n^{-1})$ as $n\to\infty$ in $P_{(\sigma_0^2,\alpha_0)}$-probability. Therefore, uniformly for all $\alpha\in [0,n^{1/6}]$, as $n\to\infty$ in $P_{(\sigma_0^2,\alpha_0)}$-probability,
\begin{align*}
& \left|\frac{(\alpha-u_*)^2(\widetilde v_*-v_*)}{ v_*\widetilde v_*}\right| = \left|\frac{(\alpha-u_*)^2 \cdot O_p(n^{-1})}{ v_*\left(1+O_p(n^{-1})\right)}\right| \leq O_p(n^{1/3} \cdot n^{-1}) = O_p(n^{-2/3}) ,
\end{align*}
and hence by $|\ee^x-1|\leq 2|x|$ for all $|x|\leq 1/2$,
\begin{align} \label{eq:Lstar.diff2}
& \quad \left| \sqrt{\alpha} \exp\left\{-\frac{(\alpha-u_*)^2}{2\widetilde v_*}\right\} - \sqrt{\alpha} \exp\left\{-\frac{(\alpha-u_*)^2}{2 v_*}\right\} \right|  \nonumber \\
&= \sqrt{\alpha} \exp\left\{-\frac{(\alpha-u_*)^2}{2 v_*}\right\} \left|\exp\left\{-\frac{(\alpha-u_*)^2(\widetilde v_*-v_*)}{2 v_*\widetilde v_*}\right\}  - 1 \right|  \nonumber \\
&\leq \sqrt{\alpha} \exp\left\{-\frac{(\alpha-u_*)^2}{2 v_*}\right\}  \cdot \left|\frac{(\alpha-u_*)^2(\widetilde v_*-v_*)}{ v_*\widetilde v_*}\right|  \nonumber \\
&\leq O_p\left( n^{1/12}\cdot 1 \cdot n^{-2/3} \right) = O_p(n^{-1/2}) .
\end{align}
We combine \eqref{eq:Lstar.diff1} and \eqref{eq:Lstar.diff2} with the triangle inequality to conclude that uniformly for all $\alpha\in [0,n^{1/6}]$, as $n\to\infty$ in $P_{(\sigma_0^2,\alpha_0)}$-probability,
\begin{align} \label{eq:Lstar.diff3}
& \quad \left| \exp\left\{\widetilde \Lcal_*(\alpha)\right\} - \sqrt{\alpha} \exp\left\{-\frac{(\alpha-u_*)^2}{2 v_*}\right\} \right| \leq O_p(n^{-5/36}) + O_p(n^{-1/2}) \leq O_p(n^{-5/36}).
\end{align}

As a result, we have that there exists a constant $C_1>0$ such that as $n\to\infty$ in $P_{(\sigma_0^2,\alpha_0)}$-probability,
\begin{align} \label{eq:2lim.numer1}
&\widetilde  {\numer}_1 \leq C_1 n^{-5/36} \int_0^{\infty} \pi(\theta_0|\alpha)\pi(\alpha) \ud \alpha \to 0 .
\end{align}
For $\widetilde {\numer}_2$, since $(A_1A_3-A_2^2)/A_1>0$ as $n\to\infty$ in $P_{(\sigma_0^2,\alpha_0)}$-probability, we have that
\begin{align} \label{eq:2lim.numer2}
\widetilde {\numer}_2 & = \int_{n^{1/6}}^{\infty}\exp\left\{\widetilde \Lcal_*(\alpha)\right\} \pi(\theta_0|\alpha)\pi(\alpha) \ud \alpha  \nonumber \\
&= \int_{n^{1/6}}^{\infty} \frac{[(A_1A_3-A_2^2)/A_1]^{n/2}\sqrt{\frac{n}{2}\left(1- \ee^{-2\alpha/n}\right)}}{\left( A_1 \ee^{-2\alpha/n}-2A_2 \ee^{-\alpha/n} +A_3 \right)^{n/2}} \pi(\theta_0|\alpha)\pi(\alpha) \ud \alpha \nonumber \\
&= \int_{n^{1/6}}^{\infty} \frac{[(A_1A_3-A_2^2)/A_1]^{n/2}\sqrt{\frac{n}{2}\left(1- \ee^{-2\alpha/n}\right)}}{\left\{ A_1\left[(1- \ee^{-\alpha/n})-\frac{A_1-A_2}{A_1}\right]^2 + (A_1A_3-A_2^2)/A_1 \right\}^{n/2}} \pi(\theta_0|\alpha)\pi(\alpha) \ud \alpha \nonumber \\
&\leq \int_{n^{1/6}}^{\infty} \sqrt{\frac{n}{2}\left(1- \ee^{-2\alpha/n}\right)} \pi(\theta_0|\alpha)\pi(\alpha) \ud \alpha \nonumber \\
&\leq \int_{n^{1/6}}^{\infty} \sqrt{\alpha} \pi(\theta_0|\alpha)\pi(\alpha) \ud \alpha \to 0,
\end{align}
as $n\to\infty$ in $P_{(\sigma_0^2,\alpha_0)}$-probability according to Assumption \ref{prior.3OU} since $\overline{\kappa}<1/6$.

For $\widetilde {\numer}_3$, similarly we have that as $n\to\infty$ in $P_{(\sigma_0^2,\alpha_0)}$-probability,
\begin{align} \label{eq:2lim.numer3}
\widetilde {\numer}_3 & = \int_{n^{1/6}}^{\infty} \sqrt{\alpha} \exp\left\{-\frac{(\alpha-u_*)^2}{2v_*}\right\} \pi(\theta_0|\alpha)\pi(\alpha) \ud \alpha \nonumber \\
&\leq \int_{n^{1/6}}^{\infty} \sqrt{\alpha} \pi(\theta_0|\alpha)\pi(\alpha) \ud \alpha \to 0.
\end{align}
Hence, \eqref{eq:diff.2pi.OU} follows by combining \eqref{eq:2lim.numer1}, \eqref{eq:2lim.numer2}, and \eqref{eq:2lim.numer3} using the triangle inequality.
\vspace{3mm}

\noindent \underline{Proof of the lower boundedness of $\int_0^{\infty} \exp\left\{\widetilde \Lcal_*(\alpha)\right\} \pi(\theta_0|\alpha)\pi(\alpha) \ud \alpha$:}
\vspace{2mm}

We first derive a lower bound for $\widetilde \denom$. By Lemma \ref{lem:OU.orders} (v), $|u_*|\leq C_5$ for some constant $C_5>0$ as $n\to\infty$ in $P_{(\sigma_0^2,\alpha_0)}$-probability. By Lemma \ref{lem:OU.orders} (vi), $v_*>C_6$ for some constant $C_6>0$ as $n\to\infty$ in $P_{(\sigma_0^2,\alpha_0)}$-probability. By Assumptions \ref{prior.1} and \ref{prior.3OU}, $\inf_{\alpha\in [1,2]}\pi(\theta_0|\alpha)\pi(\alpha)\geq C_7>0$ for some constant $C_7$. This implies that there exists a constant $C_8>0$, such that
\begin{align}\label{eq:2lim.denom}
\widetilde \denom & \geq \int_1^{2} \sqrt{\alpha} \exp\left\{-\frac{(\alpha-u_*)^2}{2v_*}\right\} \pi(\theta_0|\alpha)\pi(\alpha) \ud \alpha \geq C_7\int_1^{2} \exp\left\{-\frac{(\alpha-u_*)^2}{2v_*}\right\}\ud \alpha \nonumber \\
&\geq C_7\int_1^{2} \exp\left\{-\frac{(\alpha+C_5)^2}{2C_6}\right\} \ud \alpha \equiv C_8>0,
\end{align}
as $n\to\infty$ in $P_{(\sigma_0^2,\alpha_0)}$-probability.

Now given the convergence in \eqref{eq:diff.2pi.OU}, we have that as $n\to\infty$ in $P_{(\sigma_0^2,\alpha_0)}$-probability,
\begin{align*}
& \left| \int_0^{\infty} \exp\left\{\widetilde \Lcal_*(\alpha)\right\} \pi(\theta_0|\alpha)\pi(\alpha) \ud \alpha - \int_0^{\infty} \sqrt{\alpha} \exp\left\{-\frac{(\alpha-u_*)^2}{2v_*}\right\} \pi(\theta_0|\alpha)\pi(\alpha) \ud \alpha \right| \to 0.
\end{align*}
This and \eqref{eq:2lim.denom} together imply that
\begin{align} \label{eq:2lim.denom2}
& \quad \int_0^{\infty} \exp\left\{\widetilde \Lcal_*(\alpha)\right\} \pi(\theta_0|\alpha)\pi(\alpha) \ud \alpha \nonumber \\
&\geq \widetilde \denom - \left| \int_0^{\infty} \exp\left\{\widetilde \Lcal_*(\alpha)\right\} \pi(\theta_0|\alpha)\pi(\alpha) \ud \alpha - \int_0^{\infty} \sqrt{\alpha} \exp\left\{-\frac{(\alpha-u_*)^2}{2v_*}\right\} \pi(\theta_0|\alpha)\pi(\alpha) \ud \alpha \right| \nonumber \\
& > C_8/2,
\end{align}
as $n\to\infty$ in $P_{(\sigma_0^2,\alpha_0)}$-probability, which proves the lower boundedness.

We note that as stated at the beginning of this proof, proving the convergence in \eqref{eq:diff.1pi.OU} and the lower boundedness of $\int_0^{\infty} \exp\left\{\widetilde \Lcal_*(\alpha)\right\} \pi(\alpha) \ud \alpha$ follows exactly the same procedure as proving \eqref{eq:diff.2pi.OU} and the lower boundedness of $\int_0^{\infty} \exp\left\{\widetilde \Lcal_*(\alpha)\right\} \pi(\theta_0|\alpha)\pi(\alpha) \ud \alpha$ under Assumption \ref{prior.3OU}, and is therefore omitted.
\vspace{2mm}

\noindent \underline{Proof of \eqref{eq:2prolik.OU}:}
\vspace{2mm}

Based on the definitions of $\widetilde \pi(\alpha|Y_n)$ and $\pi_*(\alpha|Y_n)$, by Lemma \ref{lemma:intdiff}, the convergence in \eqref{eq:2prolik.OU} holds true if the following relation holds as $n\to\infty$, in $P_{(\sigma_0^2,\alpha_0)}$-probability,
\begin{align} \label{eq:OU.joint11}
\frac{\int_0^{\infty} \left| \exp\left\{\widetilde \Lcal_*(\alpha)\right\} - \sqrt{\alpha} \exp\left\{-\frac{(\alpha-u_*)^2}{2v_*}\right\} \right|\pi(\theta_0|\alpha)\pi(\alpha) \ud \alpha} {\int_0^{\infty} \sqrt{\alpha} \exp\left\{-\frac{(\alpha-u_*)^2}{2v_*}\right\} \ud \alpha} \rightarrow 0,
\end{align}
which follows from \eqref{eq:diff.2pi.OU} and \eqref{eq:2lim.denom}. Hence the proof for Lemma \ref{lemma:intdiff} is complete.
\end{proof}

\vspace{8mm}

\noindent \textbf{Proof of Theorem \ref{thm:OU1}}

\begin{proof}[Proof of Theorem \ref{thm:OU1}]

We first prove the convergence in \eqref{eq:OU.joint1}. The proof follows the same process in the proof of Theorem \ref{thm:bvm2:joint}, with some differences due to the new Assumption \ref{prior.3OU}. The conclusion of Theorem \ref{thm:bvm2:joint} is proved by showing \eqref{joint:theta2} and \eqref{joint:theta3}. We show them respectively under the new Assumption \ref{prior.3OU}. We notice that since $p=0$ in Theorem \ref{thm:OU1}, $n-p=n$ in \eqref{joint:theta2} and \eqref{joint:theta3}.
\vspace{3mm}

\noindent \underline{Proof of \eqref{joint:theta2}:}

Using the same notation as in the proof of Theorem \ref{thm:bvm2:joint}, we define ${\numer}_1$, ${\numer}_2$, ${\numer}_3$, and $\denom$ as in \eqref{numerator2} and \eqref{denominator1}. The first step of showing ${\numer}_1/\denom\to 0$ is exactly the same as in the proof of Theorem \ref{thm:bvm2:joint}, since this step only relies on Assumptions \ref{prior.1} and \ref{prior.2}, which are both assumed in Theorem \ref{thm:OU1} as well. The main differences lie in the next two steps of showing ${\numer}_2/\denom\to 0$ and ${\numer}_3/\denom\to 0$.

\vspace{3mm}

\noindent \underline{Proof of ${\numer}_2/\denom \to 0$:}

Using the upper bound of ${\numer}_2$ in \eqref{en:numer21}, together with the definition of $\denom$ in \eqref{denominator1}, we have that
\begin{align} \label{eq:numer2.OU}
&\frac{{\numer}_2}{\denom} \leq \frac{2\int_0^{\underline\alpha_n}  \ee^{\Lcal_n(\alpha^{-2\nu}\widetilde\theta_{\alpha},\alpha)} \pi(\alpha)  \ud \alpha + \frac{4\theta_0\sqrt{\pi}}{\sqrt{n}}\int_0^{\underline\alpha_n} \ee^{\Lcal_n(\alpha^{-2\nu}\widetilde\theta_{\alpha},\alpha)} \pi(\theta_0|\alpha)\pi(\alpha)  \ud \alpha} {\frac{2\theta_0\sqrt{\pi}}{\sqrt{n}}\int_0^{\infty} \ee^{\Lcal_n(\alpha^{-2\nu}\widetilde\theta_{\alpha},\alpha)} \pi(\theta_0|\alpha)\pi(\alpha) \ud \alpha} \nonumber \\
&= \frac{\sqrt{n} \int_0^{\underline\alpha_n}  \exp \big\{\widetilde \Lcal_*(\alpha) \big\} \pi(\alpha)  \ud \alpha}{\theta_0\sqrt{\pi} \int_0^{\infty} \exp \big\{\widetilde \Lcal_*(\alpha) \big\} \pi(\theta_0|\alpha)\pi(\alpha) \ud \alpha} +
\frac{2\int_0^{\underline\alpha_n} \exp \big\{\widetilde \Lcal_*(\alpha) \big\} \pi(\theta_0|\alpha)\pi(\alpha)  \ud \alpha} {\int_0^{\infty} \exp \big\{\widetilde \Lcal_*(\alpha) \big\} \pi(\theta_0|\alpha)\pi(\alpha) \ud \alpha},
\end{align}
where $\widetilde \Lcal_*(\alpha)$ is the normalized log profile likelihood defined in \eqref{eq:norm.logprofile}.

We now show the first term in \eqref{eq:numer2.OU} converges to zero in probability. For the numerator, by the definition of $\widetilde \Lcal_*(\alpha)$, since $(A_1A_3-A_2^2)/A_1>0$ as $n\to\infty$ in $P_{(\sigma_0^2,\alpha_0)}$-probability, we have that
\begin{align} \label{eq:numer2.11}
& \sqrt{n} \int_0^{\underline\alpha_n}  \exp \big\{\widetilde \Lcal_*(\alpha) \big\} \pi(\alpha)  \ud \alpha \nonumber \\
={}& \sqrt{n} \int_0^{\underline\alpha_n}  \frac{[(A_1A_3-A_2^2)/A_1]^{n/2}}{(A_1 \ee^{-2\alpha/n}-2A_2 \ee^{-\alpha/n}+A_3)^{n/2}} \sqrt{\frac{n}{2}(1- \ee^{-2\alpha/n})} \pi(\alpha)  \ud \alpha \nonumber \\
={}& \sqrt{n} \int_0^{\underline\alpha_n}  \frac{[(A_1A_3-A_2^2)/A_1]^{n/2}}{\left[A_1\left(1- \ee^{-\alpha/n}-\frac{A_1-A_2}{A_1}\right)^2+(A_1A_3-A_2^2)/A_1\right]^{n/2}} \sqrt{\frac{n}{2}(1- \ee^{-2\alpha/n})} \pi(\alpha)  \ud \alpha \nonumber \\
\leq{}& \sqrt{n} \int_0^{\underline\alpha_n}  1\cdot \sqrt{\frac{n}{2}(1- \ee^{-2\alpha/n})} \pi(\alpha)  \ud \alpha \nonumber \\
\leq{}& \sqrt{n} \int_0^{\underline\alpha_n} \sqrt{\alpha} \pi(\alpha)  \ud \alpha,
\end{align}
where in the last step, the first ratio in the integral is less than 1 and we have used $1- \ee^{-x}\leq x$ for all $x>0$. By \eqref{A3.1.OU} in Assumption \ref{prior.3OU}, we have that this upper bound goes to zero as $n\to \infty$. Therefore, $\sqrt{n} \int_0^{\underline\alpha_n}  \exp \big\{\widetilde \Lcal_*(\alpha) \big\} \pi(\alpha)  \ud \alpha \to 0$ as $n\to \infty$ in $P_{(\sigma_0^2,\alpha_0)}$-probability. Since the denominator $\theta_0\sqrt{\pi} \int_0^{\infty} \exp \big\{\widetilde \Lcal_*(\alpha) \big\} \pi(\theta_0|\alpha)\pi(\alpha) \ud \alpha$ is lower bounded by positive constant in $P_{(\sigma_0^2,\alpha_0)}$-probability according to Lemma \ref{lem:2limdiff} (in \eqref{eq:2lim.denom2}), we have that the first term in \eqref{eq:numer2.OU} converges to zero as $n\to\infty$ in $P_{(\sigma_0^2,\alpha_0)}$-probability.

We then show the second term in \eqref{eq:numer2.OU} converges to zero in probability. For the numerator, similar to \eqref{eq:numer2.11}, we have that
\begin{align}
& \int_0^{\underline\alpha_n}  \exp \big\{\widetilde \Lcal_*(\alpha) \big\} \pi(\theta_0|\alpha) \pi(\alpha)  \ud \alpha
\leq \int_0^{\underline\alpha_n} \sqrt{\alpha} \pi(\theta_0|\alpha) \pi(\alpha)  \ud \alpha, \nonumber
\end{align}
which converges to zero as $n\to\infty$ since $\underline\alpha_n\to 0$ as $n\to\infty$ and $\int_0^{\infty} \sqrt{\alpha} \pi(\theta_0|\alpha) \pi(\alpha)  \ud \alpha$ is finite according to Assumption \ref{prior.3OU}. Therefore, with the lower bounded denominator, the second term in \eqref{eq:numer2.OU} also converges to zero as $n\to\infty$ in $P_{(\sigma_0^2,\alpha_0)}$-probability. This together with \eqref{eq:numer2.OU} has shown that ${\numer}_2/\denom\to 0$ as $n\to\infty$ in $P_{(\sigma_0^2,\alpha_0)}$-probability.

\vspace{3mm}

\noindent \underline{Proof of ${\numer}_3/\denom \to 0$:}
Using the upper bound of ${\numer}_3$ in \eqref{en:numer31}, together with the definition of $\denom$ in \eqref{denominator1}, we have that
\begin{align} \label{eq:numer3.OU}
&\frac{{\numer}_3}{\denom} \leq \frac{2\int_{\overline\alpha_n}^{\infty}  \ee^{\Lcal_n(\alpha^{-2\nu}\widetilde\theta_{\alpha},\alpha)} \pi(\alpha)  \ud \alpha + \frac{4\theta_0\sqrt{\pi}}{\sqrt{n}}\int_{\overline\alpha_n}^{\infty} \ee^{\Lcal_n(\alpha^{-2\nu}\widetilde\theta_{\alpha},\alpha)} \pi(\theta_0|\alpha)\pi(\alpha)  \ud \alpha} {\frac{2\theta_0\sqrt{\pi}}{\sqrt{n}}\int_0^{\infty} \ee^{\Lcal_n(\alpha^{-2\nu}\widetilde\theta_{\alpha},\alpha)} \pi(\theta_0|\alpha)\pi(\alpha) \ud \alpha} \nonumber \\
&= \frac{\sqrt{n} \int_{\overline\alpha_n}^{\infty}  \exp \big\{\widetilde \Lcal_*(\alpha) \big\} \pi(\alpha)  \ud \alpha}{\theta_0\sqrt{\pi} \int_0^{\infty} \exp \big\{\widetilde \Lcal_*(\alpha) \big\} \pi(\theta_0|\alpha)\pi(\alpha) \ud \alpha} +
\frac{2\int_{\overline\alpha_n}^{\infty} \exp \big\{\widetilde \Lcal_*(\alpha) \big\} \pi(\theta_0|\alpha)\pi(\alpha)  \ud \alpha} {\int_0^{\infty} \exp \big\{\widetilde \Lcal_*(\alpha) \big\} \pi(\theta_0|\alpha)\pi(\alpha) \ud \alpha},
\end{align}
For both terms in \eqref{eq:numer3.OU}, the denominators are lower bounded by positive constants in $P_{(\sigma_0^2,\alpha_0)}$-probability by Lemma \ref{lem:2limdiff}. Using the same derivation as in \eqref{eq:numer2.11}, the numerator in the first term of \eqref{eq:numer3.OU} can be upper bounded by
\begin{align*}
& \sqrt{n} \int_{\overline\alpha_n}^{\infty}   \exp \big\{\widetilde \Lcal_*(\alpha) \big\} \pi(\alpha)  \ud \alpha
\leq \sqrt{n} \int_{\overline\alpha_n}^{\infty}  \sqrt{\alpha} \pi(\alpha)  \ud \alpha,
\end{align*}
which converges to zero as $n\to\infty$ by \eqref{A3.1.OU} in Assumption \ref{prior.3OU}. The numerator in the second term of \eqref{eq:numer3.OU} also converges to zero since $\overline\alpha_n\to\infty$ as $n\to\infty$ and $\int_0^{\infty} \sqrt{\alpha} \pi(\theta_0|\alpha) \pi(\alpha)  \ud \alpha$ is finite according to Assumption \ref{prior.3OU}. Therefore, it follows that ${\numer}_3/\denom\to 0$ as $n\to\infty$ in $P_{(\sigma_0^2,\alpha_0)}$-probability. Thus, the convergence in \eqref{joint:theta2} happens as $n\to\infty$ in $P_{(\sigma_0^2,\alpha_0)}$-probability.
\vspace{3mm}

\noindent \underline{Proof of \eqref{joint:theta3}:}

Compared to the proof of \eqref{joint:theta3} in the proof of Theorem \ref{thm:bvm2:joint}, the upper bounds in \eqref{joint:theta31} and \eqref{joint:theta32} still hold. We only need to show the convergence in \eqref{joint:theta33} and \eqref{joint:theta34} using the new Assumption \ref{prior.3OU}. In particular, using the definition of $\widetilde \Lcal_*(\alpha)$ in \eqref{eq:norm.logprofile}, we have
\begin{align*}
\int_0^{\underline\alpha_n} \widetilde \pi(\alpha|Y_n) \ud \alpha & = \frac{\int_0^{\underline\alpha_n} \exp \big\{\widetilde \Lcal_*(\alpha) \big\} \pi(\theta_0|\alpha)\pi(\alpha)  \ud \alpha} {\int_0^{\infty} \exp \big\{\widetilde \Lcal_*(\alpha) \big\} \pi(\theta_0|\alpha)\pi(\alpha) \ud \alpha}
\end{align*}
which converges to zero in $P_{(\sigma_0^2,\alpha_0)}$-probability as already shown above in the proof of ${\numer}_2/\denom \to 0$. Similarly, $\int_{\overline\alpha_n}^{\infty} \widetilde \pi(\alpha|Y_n)\ud \alpha \to 0$ in $P_{(\sigma_0^2,\alpha_0)}$-probability as shown in the proof of ${\numer}_3/\denom \to 0$. Therefore, the convergence in \eqref{joint:theta3} happens as $n\to\infty$ in $P_{(\sigma_0^2,\alpha_0)}$-probability. This completes the proof of the convergence in \eqref{eq:OU.joint1}.

For the proof of the convergence in \eqref{eq:OU.joint2}, we notice that
\begin{align*}
& \int_0^{\infty} \int_{\RR} \left|\frac{\sqrt{n}}{2\sqrt{\pi}\theta_0} \ee^{-\frac{n(\theta-\widetilde\theta_{\alpha_0})^2}{4\theta_0^2}} \cdot \widetilde \pi(\alpha|Y_n) - \frac{\sqrt{n}}{2\sqrt{\pi}\theta_0} \ee^{-\frac{n(\theta-\widetilde\theta_{\alpha_0})^2}{4\theta_0^2}} \cdot \pi_*(\alpha|Y_n)\right| \ud \theta \ud \alpha \\
={}& \int_0^{\infty} \left|\widetilde \pi(\alpha|Y_n) - \pi_*(\alpha|Y_n) \right| \ud \alpha \to 0,
\end{align*}
as $n\to\infty$ in $P_{(\sigma_0^2,\alpha_0)}$-probability, by \eqref{eq:2prolik.OU} of Lemma \ref{lem:2limdiff}. Then \eqref{eq:OU.joint2} follows from \eqref{eq:OU.joint1} and the triangle inequality.
\end{proof}

\vspace{5mm}

\subsection{Proof of Corollary \ref{cor:OU2}} \label{supsubsec:OU2}
\begin{proof}[Proof of Corollary \ref{cor:OU2}]
Recall that for Case (ii) in Section \ref{subsec:OU1} of the main text, we observe the 1-dimensional Ornstein-Uhlenbeck process with a constant regression term $\bbm_1(\cdot)\equiv 1$, so $Y(\cdot)=\beta_0 + X(\cdot)\sim \gp(0,\sigma_0^2 K_{\alpha_0,\nu})$ on the grid $s_i=i/n$, for $i=1,\ldots,n$, where $\beta_0$ denotes the true mean parameter. In Corollary \ref{cor:OU2}, we have defined $B_1 = \sum_{i=2}^{n-1} Y(s_i)$, $B_2  = \sum_{i=1}^{n} Y(s_i)$, and $A_1,A_2,A_3$ as in \eqref{eq:3A}.

We briefly explain the derivation of the expressions for $\widetilde\theta_{\alpha}$ and $\widetilde\Lcal_n(\alpha)$ in Corollary \ref{cor:OU2}. With $M_n=1_n$, using the expression of $R_{\alpha}$ in Section \ref{subsec:OU1}, it follows that
\begin{align*}
& M_n^\top R_{\alpha}^{-1} M_n = \frac{(n-2)\left(1-\ee^{-\alpha/n}\right)+2}{1+\ee^{-\alpha/n}}, \quad M_n^\top R_{\alpha}^{-1} Y_n = \frac{B_2 - B_1\ee^{-\alpha/n}}{1+\ee^{-\alpha/n}} .
\end{align*}
We then plug in these formulas to the expression of $\widetilde\theta_{\alpha}$ in \eqref{tildetheta1} and $\widetilde\Lcal_n(\alpha)$ in \eqref{def:prologlik} with $\Omega_{\beta}=0_{p\times p}$ to obtain the expressions for $\widetilde\theta_{\alpha}$ and $\widetilde\Lcal_n(\alpha)$ in Corollary \ref{cor:OU2}. We notice that the profile restricted log-likelihood $\widetilde\Lcal_n(\alpha)$ is defined up to an additive constant.

Similarly, we obtain the normal conditional posterior of $\beta$ in \eqref{eq:OU2.beta} of the main text, by plugging the formulas above to the conditional posterior of $\beta$ in \eqref{eq:beta.post} of the main text. The convergence in total variation norm of \eqref{eq:OU2.joint} follows directly from Theorem \ref{thm:bvm2:joint}, under Assumptions \ref{assump.m.func}, \ref{prior.1}, \ref{prior.2}, and \ref{prior.3}.

Next, we prove that the posterior of $\beta$ is inconsistent for $\beta_0$. We already know that the conditional posterior of $\beta$ is given by $\beta | Y_n,\theta,\alpha \sim \Ncal\left(\mu_n, v_n \right) $, where
\begin{align*}
& \mu_n=\frac{B_2-B_1\ee^{-\alpha/n}}{(n-2)(1-\ee^{-\alpha/n})+2}, \quad v_n=\frac{\theta\left(1+\ee^{-\alpha/n}\right)}{\left[(n-2)(1-\ee^{-\alpha/n})+2\right]\alpha} .
\end{align*}
Let $\Phi(x)=\int_{-\infty}^x \frac{1}{\sqrt{2\pi}} \ee^{-z^2/2}\ud z$ be the standard normal cumulative distribution function. For a given $\epsilon_0>0$ whose value will be chosen later, using the mean value theorem, we have that
\begin{align} \label{eq:Pi.beta.1}
&\quad ~ \Pi(|\beta-\beta_0| > \epsilon_0 | Y_n,\theta,\alpha)  =  1- \Pi(|\beta-\beta_0| \leq \epsilon_0 | Y_n,\theta,\alpha)  \nonumber \\
&= 1 - \left\{\Phi\left(\frac{\beta_0+\epsilon_0-\mu_n}{\sqrt{v_n}}\right) - \Phi\left(\frac{\beta_0-\epsilon_0-\mu_n}{\sqrt{v_n}}\right)\right\} \nonumber \\
&=  1 - \frac{2\epsilon_0}{\sqrt{2\pi v_n}} \exp\left(-\frac{x_1^2}{2v_n}\right) \nonumber \\
&\geq 1 -  \frac{2\epsilon_0}{\sqrt{2\pi v_n}},
\end{align}
for some value $x_1\in [\beta_0-\epsilon_0-\mu_n,\beta_0+\epsilon_0-\mu_n]$, where the inequality follows from the bound $\exp\left(-x_1^2/(2v_n)\right)\leq 1$.

Let $\Ecal_9=\{\alpha\in [\underline\alpha_n, \overline\alpha_n]\}$. Then under Assumptions \ref{assump.m.func}-\ref{prior.3}, Theorem \ref{thm:bvm2:joint} implies that $\left|\Pi(\Ecal_9^c|Y_n)-\widetilde\Pi(\Ecal_9^c|Y_n)\right|\to 0$ as $n\to\infty$ almost surely $P_{(\beta_0,\sigma_0^2,\alpha_0)}$. \eqref{joint:theta33} and \eqref{joint:theta34} in the proof of Theorem \ref{thm:bvm2:joint} imply that $\widetilde\Pi(\Ecal_9^c|Y_n)\to 0$ as $n\to\infty$ almost surely $P_{(\beta_0,\sigma_0^2,\alpha_0)}$. Therefore, $\Pi(\Ecal_9^c|Y_n)\to 0$ as $n\to\infty$ almost surely $P_{(\beta_0,\sigma_0^2,\alpha_0)}$. This implies that given any $\eta\in (0,1/4)$, any $\delta\in (0,1/4)$, there exist two numbers $0<\alpha_1<\alpha_2<\infty$ and a sufficiently large integer $N_{10}'>0$ ($\alpha_1,\alpha_2,N_{10}'$ are dependent on $\eta,\delta$), such that for all $n>N_{10}'$, $\Pr\left(\Pi(\Acal_{3n}^c~|~Y_n) \leq \delta/2 \right) >1- \eta/2$, where we let $\Acal_{3n}=\{\alpha\in [\alpha_1,\alpha_2]\}$.

We find the limit of $v_n$. For the $\alpha_1,\alpha_2$ above, it is clear that using Taylor series expansion for $\ee^{-x}$, we have that as $n\to\infty$,
\begin{align} \label{eq:sup.exp.alpha}
\sup_{\alpha\in [\alpha_1,\alpha_2]} n\big|1-\ee^{-\alpha/n}-\alpha/n\big| \to 0.
\end{align}

Therefore, for a given $\theta>0$, as $n\to\infty$,
\begin{align} \label{eq:vn.unif1}
\sup_{\alpha\in [\alpha_1,\alpha_2]} \left| v_n - \frac{2\theta}{\alpha(\alpha+2)} \right| \to 0.
\end{align}
\eqref{eq:Pi.beta.1} and \eqref{eq:vn.unif1} imply that by choosing $N_{10}'$ to be large, for all $n>N_{10}'$, on the event $\Acal_{3n}$, $v_n>\frac{\theta}{2\alpha(\alpha+2)}$, such that
\begin{align} \label{eq:Pi.beta.2}
&\quad ~ \Pi(|\beta-\beta_0| > \epsilon_0| Y_n,\theta,\alpha)  \nonumber \\
&\geq  1 - \frac{2\epsilon_0}{\sqrt{2\pi v_n}} > 1 - \frac{2\epsilon_0}{\sqrt{2\pi \cdot \frac{\theta}{2\alpha(\alpha+2)}}} = 1 - 2\epsilon_0\sqrt{\frac{\alpha_2(\alpha_2+2)}{\pi \theta}}.
\end{align}
Let $\Acal_{4n}= \big\{|\theta-\theta_0|\leq n^{-1/2}\log^2 n \big\}$. Under Assumptions \ref{assump.m.func}-\ref{prior.3}, Theorem \ref{thm:bvm2:joint} and Lemma \ref{lem:theta.alpha0} imply that $\Pi(\Acal_{4n}^c | Y_n)\to 0$ as $n\to\infty$ almost surely $P_{(\beta_0,\sigma_0^2,\alpha_0)}$. In other words, for any small $\eta \in (0,1/4)$, any small $\delta \in (0,1/4)$, there exists a large integer $N_{11}'$, such that for all $n>N_{11}'$, $\pr(\Pi(\Acal_{4n}^c| Y_n) \leq \delta/2) > 1-\eta/2$. Therefore, together with $\Pr\left(\Pi(\Acal_{3n}^c~|~Y_n) \leq \delta/2 \right) >1- \eta/2$ for all $n>N_{10}'$, we have that $\pr(\Pi(\Acal_{3n}^c \cup \Acal_{4n}^c | Y_n) \leq \delta) > 1-\eta$ for all $n>\max(N_{10}',N_{11}')$, which implies that $\Pi(\Acal_{3n}\cap \Acal_{4n}|Y_n)>1-\delta$ happens with $P_{(\beta_0,\sigma_0^2,\alpha_0)}$-probability at least $1-\eta$ for all $n>\max(N_{10}',N_{11}')$. On the event $\Pi(\Acal_{3n}\cap \Acal_{4n}|Y_n)>1-\delta$ for all $n>\max(N_{10}',N_{11}')$,
\begin{align} \label{eq:Pi.beta.3}
&\quad~ \Pi(|\beta-\beta_0| > \epsilon_0 | Y_n) = {\EE}_{\Pi(\ud\theta,\ud\alpha|Y_n)} \left[ \Pi(|\beta-\beta_0| > \epsilon_0 | Y_n,\theta,\alpha) \right ]  \nonumber \\
&= {\EE}_{\Pi(\ud\theta,\ud\alpha|Y_n)} \left[ \Pi(|\beta-\beta_0| > \epsilon_0 | Y_n,\theta,\alpha)\cdot  \Ical(\Acal_{3n}\cap \Acal_{4n}) \right ] \nonumber \\
&\quad ~ +  {\EE}_{\Pi(\ud\theta,\ud\alpha|Y_n)} \left[ \Pi(|\beta-\beta_0| > \epsilon_0 | Y_n,\theta,\alpha)  \cdot \Ical(\Acal_{3n}^c \cup \Acal_{4n}^c) \right ] \nonumber \\
& \stackrel{(i)}{\geq}  {\EE}_{\Pi(\ud\theta,\ud\alpha|Y_n)} \left[ \left\{1 - 2\epsilon_0\sqrt{\frac{\alpha_2(\alpha_2+2)}{\pi \theta_0/2}}\right\} \cdot \Ical(\Acal_{3n}\cap\Acal_{4n}) \right ] \nonumber \\
& = \left\{1 - 2\epsilon_0\sqrt{\frac{2\alpha_2(\alpha_2+2)}{\pi \theta_0}}\right\} \Pi(\Acal_{3n}\cap \Acal_{4n}|Y_n) \nonumber \\
&> (1-\delta)\left\{1 - 2\epsilon_0\sqrt{\frac{2\alpha_2(\alpha_2+2)}{\pi \theta_0}}\right\},
\end{align}
where ${\EE}_{\Pi(\ud\theta,\ud\alpha|Y_n)} $ denotes the posterior expectation with respect to $(\theta,\alpha)$; the inequality (i) follows because $\theta>\theta_0-n^{-1/2}\log^2 n>\theta_0/2$ on the event $\Acal_{4n}$ for $n>N_{11}'$ and the second expectation in the previous line is nonnegative.

On the right-hand side of \eqref{eq:Pi.beta.3}, we can set $\delta=1/2$ and $\epsilon_0 = \frac{1}{4}\sqrt{\frac{\pi \theta_0}{2\alpha_2(\alpha_2+2)}}$ ($\epsilon_0$ depends on $\alpha_2$ and hence depends on $\eta$), such that \eqref{eq:Pi.beta.3} leads to $\pr(\Pi(|\beta-\beta_0| > \epsilon_0 | Y_n) > 1/4) > 1-\eta$ for all $n>\max(N_{10}',N_{11}')$. The conclusion of Corollary \ref{cor:OU2} follows by taking $\epsilon_0 = \frac{1}{4}\sqrt{\frac{\pi \theta_0}{2\alpha_2(\alpha_2+2)}}$, $\delta_0=1/4$, and $N_2=\max(N_{10}',N_{11}')$.
\end{proof}

\vspace{8mm}

\section{Proof of Theorems in Section \ref{sec:PAE}}
\label{supsec:pae}

In this section, we present the proofs of Theorems \ref{thm:pae.micro}, \ref{thm:pae.main}, \ref{thm:suprate.OU}, and \ref{thm:vn.rate} in Section \ref{sec:PAE} of the main text.

\subsection{Proof of \eqref{eq:varBLUP.true} and Theorem \ref{thm:pae.micro}} \label{supsubsec:pae.micro}
\noindent \textbf{Derivation of $\vv_n(s^*;\sigma^2,\alpha)$ in \eqref{eq:varBLUP.true}:}
\vspace{2mm}

Define $b_{\alpha}(s^*)=\bbm(s^*) - M_n^\top R_{\alpha}^{-1} r_{\alpha}(s^*)$. First we recall that
$$\beta|Y_n,\sigma^2,\alpha\sim \Ncal\left(\big(M_n^\top R_{\alpha}^{-1} M_n + \Omega_{\beta}\big)^{-1} M_n^\top R_{\alpha}^{-1}Y_n, ~\sigma^2\big(M_n^\top R_{\alpha}^{-1} M_n + \Omega_{\beta}\big)^{-1}\right) . $$
Given $Y_n$, the GP predictive distribution for $\widetilde Y(s^*)$ is
\begin{align*}
\widetilde Y(s^*)|Y_n,\beta,\sigma^2,\alpha &\sim \Ncal \left(\widehat Y(s^*;\beta,\alpha), ~\sigma^2\left[1 - r_{\alpha}(s^*)^\top R_{\alpha}^{-1} r_{\alpha}(s^*)\right] \right),
\end{align*}
where
$$\widehat Y(s^*;\beta,\alpha) = \bbm(s^*)^\top \beta + r_{\alpha}(s^*)^\top R_{\alpha}^{-1} \left(Y_n - M_n \beta\right) = r_{\alpha}(s^*)^\top R_{\alpha}^{-1} Y_n + b_{\alpha}(s^*)^\top \beta.$$

Therefore, by the law of iterated expectation, we can integrate out $\beta$ and obtain that $Y(s^*)|Y_n,\sigma^2,\alpha$ still follows a normal distribution, whose mean is
\begin{align}\label{eq:form.e2.0}
& {\EE}\left\{\widetilde Y(s^*)|Y_n,\sigma^2,\alpha\right\} = {\EE}_{\beta|Y_n,\sigma^2,\alpha} {\EE}\left\{Y(s^*)|Y_n,\beta,\sigma^2,\alpha\right\} = {\EE}_{\beta|Y_n,\sigma^2,\alpha}\left\{\widehat Y(s^*;\beta,\alpha)\right\} \nonumber \\
={}& r_{\alpha}(s^*)^\top R_{\alpha}^{-1} Y_n + b_{\alpha}(s^*)^\top \big(M_n^\top R_{\alpha}^{-1} M_n + \Omega_{\beta}\big)^{-1} M_n^\top R_{\alpha}^{-1}Y_n ,
\end{align}
and by the law of total variance, the variance of $Y(s^*)|Y_n,\sigma^2,\alpha$ is
\begin{align}\label{eq:form.e2.1}
& {\Var}\left\{\widetilde Y(s^*)|Y_n,\sigma^2,\alpha\right\} \nonumber \\
={}& {\Var}_{\beta|Y_n,\sigma^2,\alpha} \left\{\widehat Y(s^*;\beta,\alpha) \right\} + {\EE}_{\beta|Y_n,\sigma^2,\alpha}\left\{ {\Var}[Y(s^*)|Y_n,\beta,\sigma^2,\alpha] \right\}\nonumber \\
={}& \sigma^2 b_{\alpha}(s^*)^\top \big(M_n^\top R_{\alpha}^{-1} M_n + \Omega_{\beta}\big)^{-1} b_{\alpha}(s^*)  + \sigma^2\left[1 - r_{\alpha}(s^*)^\top R_{\alpha}^{-1} r_{\alpha}(s^*)\right],
\end{align}
which has proved \eqref{eq:varBLUP.true}.

\vspace{5mm}

\begin{proof}[Proof of Theorem \ref{thm:pae.micro}]
The proof of Part (i) closely follow Theorem \ref{thm:bvm1:theta} for a given $\alpha>0$, and the proof of Part (ii) closely follow Theorem \ref{thm:bvm2:joint} for the joint posterior of $(\theta,\alpha)$. The two proofs are highly similar and we only show the proof of Part (ii) below, while the proof of Part (i) follows similarly.

We first use the reparameterization $\theta=\sigma^2\alpha^{2\nu}$ to replace $\sigma^2$ by $\theta$. By the definition of ${\vv}_n(s^*;\sigma^2,\alpha)$ in \eqref{eq:varBLUP.true}, we have the following decomposition of ratios:
\begin{align}
& \frac{{\vv}_n(s^*;\sigma^2,\alpha)}{{\vv}_n(s^*;\theta_0/\alpha^{2\nu},\alpha)} = \frac{{\vv}_n(s^*;\sigma^2,\alpha)}{{\vv}_n(s^*;\widetilde \theta_{\alpha_0}/\alpha^{2\nu},\alpha)} \cdot \frac{{\vv}_n(s^*;\widetilde \theta_{\alpha_0}/\alpha^{2\nu},\alpha)}{{\vv}_n(s^*;\theta_0/\alpha^{2\nu},\alpha)} . \nonumber
\end{align}
Using the formula \eqref{eq:varBLUP.true} for ${\vv}_n(s^*;\sigma^2,\alpha)$, we can see that for any $s^*\in \Scal \backslash \Scal_n$,
\begin{align} \label{eq:v.ratios}
& \frac{{\vv}_n(s^*;\sigma^2,\alpha)}{{\vv}_n(s^*;\widetilde \theta_{\alpha_0}/\alpha^{2\nu},\alpha)} = \frac{\sigma^2}{\widetilde \theta_{\alpha_0}/\alpha^{2\nu}} = \frac{\theta}{\widetilde\theta_{\alpha_0}}, \nonumber \\
& \frac{{\vv}_n(s^*;\widetilde \theta_{\alpha_0}/\alpha^{2\nu},\alpha)}{{\vv}_n(s^*;\theta_0/\alpha^{2\nu},\alpha)} =\frac{\widetilde\theta_{\alpha_0}/\alpha^{2\nu}}{\theta_0/\alpha^{2\nu}}= \frac{\widetilde\theta_{\alpha_0}}{\theta_0},
\end{align}

Recall that Lemma \ref{lem:theta.alpha0} has proved that for the event $\Ecal_4(\epsilon)=\big\{|\widetilde \theta_{\alpha_0} - \theta_0| < \epsilon\big\}$,
$$\pr\left\{\Ecal_4(5\theta_0n^{-1/2}\log n)^c\right\}\leq 3\exp(-4\log^2 n)$$
for all sufficiently large $n$. Let $\Ecal_8=\{ |\theta/ \widetilde\theta_{\alpha_0}-1 |> n^{-1/2}\log n \}$. Then by Theorem \ref{thm:bvm2:joint}, as $n\to\infty$, almost surely $P_{(\sigma_0^2,\alpha_0)}$,
\begin{align}\label{eq:vardiff2}
& \left|\Pi\left( \Ecal_8 |Y_n \right) - \int_0^{\infty} \int_{\Ecal_8} \frac{\sqrt{n}}{2\sqrt{\pi}\theta_0} \ee^{-\frac{n(\theta-\widetilde\theta_{\alpha_0})^2}{4\theta_0^2}} \cdot \widetilde \pi(\alpha|Y_n)\ud \theta \ud \alpha \right| \to 0.
\end{align}
(For Part (i), we simply use Theorem \ref{thm:bvm1:theta} instead and replace $\Pi\left( \Ecal_8 |Y_n \right)$ in \eqref{eq:vardiff2} by $\Pi\left( \Ecal_8 |Y_n,\alpha \right)$ and remove the integral over $\alpha$, similarly for the rest of the proof.)

On the event $\Ecal_4(5\theta_0n^{-1/2}\log n)\cap \Ecal_8$, for all sufficiently large $n$,
\begin{align*}
\left|\theta-\widetilde\theta_{\alpha_0}\right| > \widetilde\theta_{\alpha_0} n^{-1/2} \log n > (\theta_0-5\theta_0n^{-1/2}\log n)n^{-1/2}  \log n > (\theta_0 /2) n^{-1/2}  \log n.
\end{align*}
Using the normal tail inequality \eqref{normal.tail}, the integral in \eqref{eq:vardiff2} can be bounded by
\begin{align}\label{eq:norm.e81}
& \int_0^{\infty} \int_{\Ecal_8} \frac{\sqrt{n}}{2\sqrt{\pi}\theta_0} \ee^{-\frac{n(\theta-\widetilde\theta_{\alpha_0})^2}{4\theta_0^2}} \cdot \widetilde \pi(\alpha|Y_n)\ud \theta \ud \alpha \nonumber \\
\leq{}& \int_{\left|\theta-\widetilde\theta_{\alpha_0}\right| > \frac{\theta_0}{2} n^{-1/2}  \log n } \frac{\sqrt{n}}{2\sqrt{\pi}\theta_0} \ee^{-\frac{n(\theta-\widetilde\theta_{\alpha_0})^2}{4\theta_0^2}} \ud \theta \cdot  \int_0^{\infty} \widetilde \pi(\alpha|Y_n) \ud \alpha  \nonumber \\
\leq{}&  \exp\left(- \log^2 n/16\right) \rightarrow 0, \text{ as } n\to\infty.
\end{align}
Therefore, by combining \eqref{eq:v.ratios}, \eqref{eq:vardiff2} and \eqref{eq:norm.e81} and noticing that $\Ecal_4(5\theta_0n^{-1/2}\log n,\alpha)$ happens almost surely $P_{(\sigma_0^2,\alpha_0)}$ as $n\to\infty$ by the Borel-Cantelli lemma, we have that
\begin{align}\label{eq:pae.joint1.1}
& \Pi\left(\sup_{s^* \in \Scal \backslash \Scal_n} \left|\frac{{\vv}_n(s^*;\sigma^2,\alpha)}{{\vv}_n(s^*;\widetilde \theta_{\alpha_0}/\alpha^{2\nu},\alpha)} - 1\right| > n^{-1/2} \log n \Big |Y_n \right) \nonumber \\
={}& \Pi\left( \left|\frac{\theta}{\widetilde\theta_{\alpha_0}} - 1\right| > n^{-1/2} \log n \Big |Y_n \right) =  \Pi\left( \Ecal_8 \big |Y_n \right) \rightarrow 0, \text{ a.s. } P_{(\sigma_0^2,\alpha_0)}.
\end{align}
The relation of \eqref{eq:v.ratios} and the almost sure convergence property of $\Ecal_4(5\theta_0n^{-1/2}\log n)$ also implies that
\begin{align}\label{eq:pae.joint1.2}
& \Pi\left(\sup_{s^* \in \Scal \backslash \Scal_n} \left| \frac{{\vv}_n(s^*;\widetilde \theta_{\alpha_0}/\alpha^{2\nu},\alpha)}{{\vv}_n(s^*;\theta_0/\alpha^{2\nu},\alpha)} - 1 \right| > 5 n^{-1/2} \log n \Big |Y_n\right) \nonumber \\
& = \Pi\left( \left| \frac{\widetilde \theta_{\alpha_0}} {\theta_0} - 1 \right| > 5n^{-1/2} \log n \Big |Y_n \right) = 0, \text{ a.s. } P_{(\sigma_0^2,\alpha_0)}.
\end{align}
For $n$ sufficiently large, we have $5n^{-1/2}\log n<1/5$. Hence, $ | \widetilde \theta_{\alpha_0}/\theta_0 - 1 |< 1/5$ and $ \widetilde \theta_{\alpha_0}/\theta_0< 6/5$ as $n\to\infty$ almost surely $P_{(\sigma_0^2,\alpha_0)}$. We combine \eqref{eq:pae.joint1.1} and \eqref{eq:pae.joint1.2} to obtain that
\begin{align}
& \Pi\left( \sup_{s^* \in \Scal \backslash \Scal_n} \left| \frac{{\vv}_n(s^*;\sigma^2,\alpha)}{{\vv}_n(s^*;\theta_0/\alpha^{2\nu},\alpha)} - 1 \right| > 7 n^{-1/2} \log n \Big |Y_n \right) \nonumber \\
={}& \Pi\left( \sup_{s^* \in \Scal \backslash \Scal_n} \left| \frac{{\vv}_n(s^*;\sigma^2,\alpha)}{{\vv}_n(s^*;\widetilde \theta_{\alpha_0}/\alpha^{2\nu},\alpha)} \cdot \frac{{\vv}_n(s^*;\widetilde \theta_{\alpha_0}/\alpha^{2\nu},\alpha)}{{\vv}_n(s^*;\theta_0/\alpha^{2\nu},\alpha)} - 1 \right| > 7 n^{-1/2} \log n \Big |Y_n \right)  \nonumber \\
\leq{}&  \Pi\Bigg( \sup_{s^* \in \Scal \backslash \Scal_n} \Bigg\{\left| \frac{{\vv}_n(s^*;\sigma^2,\alpha)}{{\vv}_n(s^*;\widetilde \theta_{\alpha_0}/\alpha^{2\nu},\alpha)} -1 \right | \cdot \left|\frac{{\vv}_n(s^*; \widetilde \theta_{\alpha_0}/\alpha^{2\nu},\alpha)}{{\vv}_n(s^*;\theta_0/\alpha^{2\nu},\alpha)}\right| \nonumber \\
&\quad  + \left|\frac{{\vv}_n(s^*;\widetilde \theta_{\alpha_0}/\alpha^{2\nu},\alpha)}{{\vv}_n(s^*;\theta_0/\alpha^{2\nu},\alpha)} - 1 \right|\Bigg\} > 7 n^{-1/2} \log n \Big |Y_n \Bigg)  \nonumber \\
\leq{}&  \Pi\Bigg( \frac{6}{5}\sup_{s^* \in \Scal \backslash \Scal_n} \left| \frac{{\vv}_n(s^*;\sigma^2,\alpha)}{{\vv}_n(s^*;\widetilde \theta_{\alpha_0}/\alpha^{2\nu},\alpha)} -1 \right | \nonumber \\
&\quad + \sup_{s^* \in \Scal \backslash \Scal_n} \left|\frac{{\vv}_n(s^*;\widetilde \theta_{\alpha_0}/\alpha^{2\nu},\alpha)}{{\vv}_n(s^*;\theta_0/\alpha^{2\nu},\alpha)} - 1 \right| > 7 n^{-1/2} \log n \Big |Y_n \Bigg)  \nonumber \\
\leq{}& \Pi\left( \sup_{s^* \in \Scal \backslash \Scal_n} \left| \frac{{\vv}_n(s^*;\sigma^2, \alpha)}{{\vv}_n(s^*;\widetilde \theta_{\alpha_0}/\alpha^{2\nu},\alpha)} -1 \right | > n^{-1/2} \log n \Big |Y_n \right)  \nonumber \\
& ~~ + \Pi\left(\sup_{s^* \in \Scal \backslash \Scal_n} \left|\frac{{\vv}_n(s^*;\widetilde \theta_{\alpha_0}/\alpha^{2\nu},\alpha)}{{\vv}_n(s^*;\theta_0/\alpha^{2\nu},\alpha)} - 1 \right| > 5 n^{-1/2} \log n \Big |Y_n \right)  \nonumber \\
\rightarrow {}&  0, \text{ a.s. } P_{(\sigma_0^2,\alpha_0)}. \nonumber
\end{align}
Since $\sup_{s^* \in \Scal \backslash \Scal_n} \left| \frac{{\vv}_n(s^*;\sigma^2,\alpha)}{{\vv}_n(s^*;\theta_0/\alpha^{2\nu},\alpha)} - 1 \right| = |\theta/\theta_0-1|$, this has also proved that
\begin{align}\label{eq:theta.theta0.ratio7}
& \Pi\left(\left| \frac{\theta}{\theta_0} - 1 \right| > 7n^{-1/2} \log n \Big |Y_n\right) \rightarrow  0, \text{ a.s. } P_{(\sigma_0^2,\alpha_0)}.
\end{align}
This completes the proof.
\end{proof}

\vspace{8mm}

\subsection{Proof of Theorems \ref{thm:pae.main} and \ref{thm:suprate.OU}} \label{supsubsec:pae.main}
\begin{proof}[Proof of Theorem \ref{thm:pae.main}]

\noindent \underline{Proof of Part (i):}
\vspace{3mm}

First, we show the existence of the sequence $\varsigma_n(\alpha)$. Since the two Gaussian measures $\gp(0,(\theta_0/\alpha^{2\nu}) K_{\alpha,\nu})$ and $\gp(0,\sigma_0^2 K_{\alpha_0,\nu})$ are equivalent, by Assumption \ref{assump.dense1}, Equation (3.4) in \citet{Stein90a} implies that there exists a positive sequence $\varsigma_{1n}(\alpha) \to 0$ as $n\to\infty$, such that
\begin{align}
\sup_{s^* \in \Scal \backslash \Scal_n} \left | \frac{{\EE}_{(\sigma_0^2,\alpha_0)}\left\{e_n(s^*;\alpha)^2\right\}}{{\EE}_{(\theta_0/\alpha^{2\nu},\alpha)}\left\{e_n(s^*;\alpha)^2\right\}}
-1 \right| < \frac{1}{2} \varsigma_{1n}(\alpha). \nonumber
\end{align}
Notice that for a small $\epsilon\in (0,1/2)$, $\left|a/b-1\right|<\epsilon$ implies that $a/b\geq 1-\epsilon$ and hence $|b/a-1| \leq |a/b-1|/|a/b|\leq \epsilon/(1-\epsilon)< 2\epsilon$. Therefore, for sufficiently large $n$, $ \varsigma_{1n}(\alpha)<1/7$ and
\begin{align} \label{eq:stein93.1}
&\quad \sup_{s^* \in \Scal \backslash \Scal_n} \left | \frac {{\EE}_{(\theta_0/\alpha^{2\nu},\alpha)}\left\{e_n(s^*;\alpha)^2\right\}} {{\EE}_{(\sigma_0^2,\alpha_0)}\left\{e_n(s^*;\alpha)^2\right\}} -1 \right|  \leq \varsigma_{1n}(\alpha).
\end{align}
Theorem 1 and Lemma 2 of \citet{Stein90b} further imply that there exists a positive sequence $\varsigma_{2n}(\alpha) \to 0$ as $n\to\infty$, such that
\begin{align}\label{eq:stein93.2}
\sup_{s^* \in \Scal \backslash \Scal_n} \left | \frac{{\EE}_{(\theta_0/\alpha^{2\nu},\alpha)}\left\{e_n(s^*;\alpha)^2\right\}}
{{\EE}_{(\sigma_0^2,\alpha_0)}\left\{e_n(s^*;\alpha_0)^2\right\}} -1 \right| < \varsigma_{2n}(\alpha).
\end{align}
See our Lemma \ref{lem:stein90b} below for more details. Therefore, we can set $\varsigma_n(\alpha)=\max\{\varsigma_{1n}(\alpha),\varsigma_{2n}(\alpha)\}$ and $\varsigma_n(\alpha)\to 0$ as $n\to\infty$.

For abbreviation, let $\epsilon_{2n}(\alpha) = \max\left\{8n^{-1/2}\log n, \varsigma_n(\alpha) \right\}$. Then based on \eqref{eq:stein93.1} and Theorem \ref{thm:pae.micro}, we have that
\begin{align} \label{eq:cpae.cond1}
& \Pi\left( \sup_{s^* \in \Scal \backslash \Scal_n}\left|\frac{{\EE}_{(\sigma^2,\alpha)}\left\{e_n(s^*;\alpha)^2\right\}}
{{\EE}_{(\sigma_0^2,\alpha_0)}\left\{e_n(s^*;\alpha)^2\right\}} - 1 \right| > 2\epsilon_{2n}(\alpha) \Big |Y_n,\alpha \right) \nonumber \\
={}& \Pi\left( \sup_{s^* \in \Scal \backslash \Scal_n} \left|\frac{{\vv}_n(s^*;\theta/\alpha^{2\nu},\alpha)}{{\vv}_n(s^*;\theta_0/\alpha^{2\nu},\alpha)} \cdot \frac{{\EE}_{(\theta_0/\alpha^{2\nu},\alpha)}\left\{e_n(s^*;\alpha)^2\right\}}
{{\EE}_{(\sigma_0^2,\alpha_0)}\left\{e_n(s^*;\alpha)^2\right\}}  -  1 \right| > 2\epsilon_{2n}(\alpha) \Big |Y_n,\alpha \right) \nonumber \\
\leq{}& \Pi\left( \sup_{s^* \in \Scal \backslash \Scal_n} \left|\frac{{\EE}_{(\theta_0/\alpha^{2\nu},\alpha)}\left\{e_n(s^*;\alpha)^2\right\}}
{{\EE}_{(\sigma_0^2,\alpha_0)}\left\{e_n(s^*;\alpha)^2\right\}} \right| \cdot \left| \frac{{\vv}_n(s^*;\theta/\alpha^{2\nu},\alpha)}{{\vv}_n(s^*;\theta_0/\alpha^{2\nu},\alpha)} -  1 \right| > \epsilon_{2n}(\alpha) \Big |Y_n,\alpha \right) \nonumber \\
&\quad + \Pi\left( \sup_{s^* \in \Scal \backslash \Scal_n} \left| \frac{{\EE}_{(\theta_0/\alpha^{2\nu},\alpha)}\left\{e_n(s^*;\alpha)^2\right\}}
{{\EE}_{(\sigma_0^2,\alpha_0)}\left\{e_n(s^*;\alpha)^2\right\}} -  1 \right| \geq \epsilon_{2n}(\alpha) \Big |Y_n,\alpha \right) .
\end{align}
The second term on the right-hand side of \eqref{eq:cpae.cond1} is zero, due to \eqref{eq:stein93.1} and $\epsilon_{2n}(\alpha)\geq \varsigma_n(\alpha) \geq \varsigma_{1n}(\alpha)$. In the first term on the right-hand side of \eqref{eq:cpae.cond1}, using \eqref{eq:stein93.1} and the fact that $\varsigma_{1n}(\alpha)<1/7$ for sufficiently large $n$, we have from \eqref{eq:cpae.cond1} that
\begin{align*}
& \Pi\left( \sup_{s^* \in \Scal \backslash \Scal_n}\left|\frac{{\EE}_{(\sigma^2,\alpha)}\left\{e_n(s^*;\alpha)^2\right\}}
{{\EE}_{(\sigma_0^2,\alpha_0)}\left\{e_n(s^*;\alpha)^2\right\}} - 1 \right| > 2\epsilon_{2n}(\alpha) \Big |Y_n,\alpha \right) \nonumber \\
\leq{}& \Pi\left( \sup_{s^* \in \Scal \backslash \Scal_n} \left| \frac{{\vv}_n(s^*;\theta/\alpha^{2\nu},\alpha)}{{\vv}_n(s^*;\theta_0/\alpha^{2\nu},\alpha)} -  1 \right| > \frac{7}{8}\epsilon_{2n}(\alpha) \Big |Y_n,\alpha \right) \nonumber \\
\leq{}& \Pi\left( \sup_{s^* \in \Scal \backslash \Scal_n} \left| \frac{{\vv}_n(s^*;\theta/\alpha^{2\nu},\alpha)}{{\vv}_n(s^*;\theta_0/\alpha^{2\nu},\alpha)} -  1 \right| > 7n^{-1/2}\log n \Big |Y_n,\alpha \right) \rightarrow 0, \text{ a.s. } P_{(\sigma_0^2,\alpha_0)},
\end{align*}
following the result of Theorem \ref{thm:pae.micro} Part (i). This has proved the first convergence in Theorem \ref{thm:pae.main} Part (i). The proof of the second convergence in Theorem \ref{thm:pae.main} Part (i) is similar, by instead using \eqref{eq:stein93.2} and replacing all ${\EE}_{(\sigma_0^2,\alpha_0)}\left\{e_n(s^*;\alpha)^2\right\}$ in the display above by ${\EE}_{(\sigma_0^2,\alpha_0)}\left\{e_n(s^*;\alpha_0)^2\right\}$.
\vspace{4mm}

\noindent \underline{Proof of Part (ii):}
\vspace{3mm}

Let $\epsilon_{3n}=\max(8n^{-1/2}\log n, \varsigma_n)$. Let $\Ecal_9=\{\alpha\in [\underline\alpha_n, \overline \alpha_n]\}$. By Assumption \ref{assump.Mn}, for all sufficiently large $n$, on the event $\Ecal_9$,
\begin{align*}
\sup_{s^* \in \Scal \backslash \Scal_n} \left | \frac{{\EE}_{(\theta_0/\alpha^{2\nu},\alpha)}\left\{e_n(s^*;\alpha)^2\right\}}
{{\EE}_{(\sigma_0^2,\alpha_0)}\left\{e_n(s^*;\alpha)^2\right\}} -1 \right| \leq \varsigma_n < 1/7 , \\
\sup_{s^* \in \Scal \backslash \Scal_n} \left | \frac{{\EE}_{(\theta_0/\alpha^{2\nu},\alpha)}\left\{e_n(s^*;\alpha)^2\right\}}
{{\EE}_{(\sigma_0^2,\alpha_0)}\left\{e_n(s^*;\alpha_0)^2\right\}} -1 \right| \leq \varsigma_n < 1/7 .
\end{align*}
Therefore, we have that
\begin{align} \label{eq:Ecal9}
& \Pi\left( \sup_{s^* \in \Scal \backslash \Scal_n}\left|\frac{{\EE}_{(\sigma^2,\alpha)}\left\{e_n(s^*;\alpha)^2\right\}}
{{\EE}_{(\sigma_0^2,\alpha_0)}\left\{e_n(s^*;\alpha)^2\right\}}- 1 \right| > 2\epsilon_{3n} , \Ecal_9 \Big |Y_n \right) \nonumber \\
={}& \Pi\left( \sup_{s^* \in \Scal \backslash \Scal_n} \left|\frac{{\vv}_n(s^*;\theta/\alpha^{2\nu},\alpha)}{{\vv}_n(s^*;\theta_0/\alpha^{2\nu},\alpha)} \cdot \frac{{\EE}_{(\theta_0/\alpha^{2\nu},\alpha)}\left\{e_n(s^*;\alpha)^2\right\}}
{{\EE}_{(\sigma_0^2,\alpha_0)}\left\{e_n(s^*;\alpha)^2\right\}} -  1 \right| > 2\epsilon_{3n} , \Ecal_9 \Big |Y_n \right) \nonumber \\
\leq{}& \Pi\left( \sup_{s^* \in \Scal \backslash \Scal_n} \left| \frac{{\EE}_{(\theta_0/\alpha^{2\nu},\alpha)}\left\{e_n(s^*;\alpha)^2\right\}}
{{\EE}_{(\sigma_0^2,\alpha_0)}\left\{e_n(s^*;\alpha)^2\right\}} \right| \cdot \left| \frac{{\vv}_n(s^*;\theta/\alpha^{2\nu},\alpha)}{{\vv}_n(s^*;\theta_0/\alpha^{2\nu},\alpha)} -  1 \right| > \epsilon_{3n} , \Ecal_9 \Big |Y_n \right) \nonumber \\
&\quad + \Pi\left( \sup_{s^* \in \Scal \backslash \Scal_n} \left| \frac{{\EE}_{(\theta_0/\alpha^{2\nu},\alpha)}\left\{e_n(s^*;\alpha)^2\right\}}
{{\EE}_{(\sigma_0^2,\alpha_0)}\left\{e_n(s^*;\alpha)^2\right\}} -  1 \right| > \epsilon_{3n} , \Ecal_9 \Big |Y_n \right) \nonumber \\
\leq{}& \Pi\left( \sup_{s^* \in \Scal \backslash \Scal_n} \left| \frac{{\vv}_n(s^*;\theta/\alpha^{2\nu},\alpha)}{{\vv}_n(s^*;\theta_0/\alpha^{2\nu},\alpha)} -  1 \right| > 7 n^{-1/2}\log n, \Ecal_9 \Big |Y_n \right) \nonumber \\
&\quad + \Pi\left( \sup_{s^* \in \Scal \backslash \Scal_n} \left| \frac{{\EE}_{(\theta_0/\alpha^{2\nu},\alpha)}\left\{e_n(s^*;\alpha)^2\right\}}
{{\EE}_{(\sigma_0^2,\alpha_0)}\left\{e_n(s^*;\alpha)^2\right\}} -  1 \right| > \varsigma_n , \Ecal_9 \Big |Y_n \right) \nonumber \\
\leq{}& \Pi\left( \sup_{s^* \in \Scal \backslash \Scal_n} \left| \frac{{\vv}_n(s^*;\theta/\alpha^{2\nu},\alpha)}{{\vv}_n(s^*;\theta_0/\alpha^{2\nu},\alpha)} -  1 \right| > 7n^{-1/2}\log n \Big |Y_n \right) \rightarrow 0, \text{ a.s. } P_{(\sigma_0^2,\alpha_0)},
\end{align}
where the last convergence follows from Theorem \ref{thm:pae.micro} Part (ii).

On the other hand, for the event $\Ecal_9^c$, Theorem \ref{thm:bvm2:joint} implies that for the event
$$\Ecal_{10} = \left\{\sup_{s^* \in \Scal \backslash \Scal_n}\left| \frac{{\EE}_{(\sigma^2,\alpha)}\left\{e_n(s^*;\alpha)^2\right\}}
{{\EE}_{(\sigma_0^2,\alpha_0)}\left\{e_n(s^*;\alpha)^2\right\}} - 1 \right| > \max\left(16 n^{-1/2}\log n, 2\varsigma_n\right) \right\}\cap \Ecal_9^c,$$
as $n\to\infty$, almost surely $P_{(\sigma_0^2,\alpha_0)}$,
\begin{align}\label{eq:vardiff3}
& \left|\Pi\left( \Ecal_{10} |Y_n \right) - \int_{\Ecal_{10}} \frac{\sqrt{n}}{2\sqrt{\pi}\theta_0} \ee^{-\frac{n(\theta-\widetilde\theta_{\alpha_0})^2}{4\theta_0^2}} \cdot \widetilde \pi(\alpha|Y_n)\ud \theta \ud \alpha \right| \to 0.
\end{align}
But from \eqref{joint:theta33} and \eqref{joint:theta34} in the proof of Theorem \ref{thm:bvm2:joint}, it follows that as $n\to\infty$, almost surely $P_{(\sigma_0^2,\alpha_0)}$,
\begin{align}\label{eq:vardiff4}
& \int_{\Ecal_{10}} \frac{\sqrt{n}}{2\sqrt{\pi}\theta_0} \ee^{-\frac{n(\theta-\widetilde\theta_{\alpha_0})^2}{4\theta_0^2}} \cdot \widetilde \pi(\alpha|Y_n)\ud \theta \ud \alpha
\leq \int_{\Ecal_9^c} \widetilde \pi(\alpha|Y_n)\ud \theta \ud \alpha  \rightarrow 0.
\end{align}
Therefore, \eqref{eq:vardiff3} and \eqref{eq:vardiff4} imply that $\Pi\left( \Ecal_{10} |Y_n \right)\to 0$ almost surely $P_{(\sigma_0^2,\alpha_0)}$ as $n\to\infty$. The first convergence in Theorem \ref{thm:pae.main} Part (ii) follows by combining this with \eqref{eq:Ecal9}. The second convergence in Theorem \ref{thm:pae.main} Part (ii) follows from the similar argument as above by replacing all ${\EE}_{(\sigma_0^2,\alpha_0)}\left\{e_n(s^*;\alpha)^2\right\}$ by ${\EE}_{(\sigma_0^2,\alpha_0)}\left\{e_n(s^*;\alpha_0)^2\right\}$.
\end{proof}

\vspace{8mm}

Define $\kl(P_1,P_2)=\int \log (\ud P_1/\ud P_2) \ud P_1$ to be the Kullback-Leibler divergence between two measures $P_1$ and $P_2$, where $\ud P_1/\ud P_2$ is the Radon-Nickdym derivative of $P_1$ with respect to $P_2$. For two mean zero Gaussian processes with Mat\'ern covariance functions $\sigma_i^2 K_{\alpha_i,\nu}$ ($i=1,2$), let $P_{(\sigma_i^2,\alpha_i)}^{(n)}$ be the joint Gaussian distribution of the observations ${X(s_1),\ldots,X(s_n)}$. Then one can show that
$$\kl\left(P_{(\sigma_1^2,\alpha_1)}^{(n)},P_{(\sigma_2^2,\alpha_2)}^{(n)}\right)= \frac{1}{2}\left\{ \log \frac{|\sigma_2^2 R_{\alpha_2}|}{|\sigma_1^2 R_{\alpha_1}|} - n + \frac{\sigma_1^2}{\sigma_2^2}\tr\left(R_{\alpha_2}^{-1}R_{\alpha_1}\right)\right\}.$$
For $d\in\{1,2,3\}$, let us consider two equivalent Gaussian measures with Mat\'ern covariance functions $\sigma_0^2 K_{\alpha_0,\nu}$ and $\sigma^2 K_{\alpha,\nu}$, such that $\sigma_0^2\alpha_0^{2\nu}=\theta_0=\sigma^2\alpha^{2\nu}$. Let
\begin{align} \label{eq:rrn}
\rr_n(\alpha) &= \kl\left(P_{(\sigma_0^2,\alpha_0)}^{(n)},P_{(\sigma^2,\alpha)}^{(n)}\right) + \kl\left(P_{(\sigma^2,\alpha)}^{(n)},P_{(\sigma_0^2,\alpha_0)}^{(n)}\right) \nonumber \\
&= -n + \frac{\alpha^{2\nu}}{2\alpha_0^{2\nu}}\tr\left(R_{\alpha}^{-1}R_{\alpha_0}\right) + \frac{\alpha_0^{2\nu}}{2\alpha^{2\nu}}\tr\left(R_{\alpha_0}^{-1}R_{\alpha}\right).
\end{align}
Then due to the equivalence, for any given $\alpha>0$, under Assumption \ref{assump.dense1} that $\Scal_n$ is dense in $\Scal=[0,T]^d$ as $n\to\infty$, the sequence $\{\rr_n(\alpha)\}_{n=1}^{\infty}$ is increasing with $n$ to a finite limit $\rr(\alpha)=\lim_{n\to \infty} \rr_n(\alpha)$ (\citet{IbrRoz78}), which satisfies $\rr(\alpha)=\kl(P_{(\sigma_0^2,\alpha_0)},P_{(\sigma^2,\alpha)}) + \kl(P_{(\sigma^2,\alpha)},P_{(\sigma_0^2,\alpha_0)})$, where $\kl(P_{(\sigma_0^2,\alpha_0)},P_{(\sigma^2,\alpha)})$ and $\kl(P_{(\sigma^2,\alpha)},P_{(\sigma_0^2,\alpha_0)})$ are the limits of \\
$\kl(P_{(\sigma_0^2,\alpha_0)}^{(n)},P_{(\sigma^2,\alpha)}^{(n)})$ and $\kl(P_{(\sigma^2,\alpha)}^{(n)}, P_{(\sigma_0^2,\alpha_0)}^{(n)})$ as $n\to\infty$ (\citet{Kuletal87}); see Section 3 of \citet{Stein90b}.

The following lemma is a result from \citet{Stein90b}.
\begin{lemma} \label{lem:stein90b}
Suppose that $d\in\{1,2,3\}$, $\nu\in \RR^+$, and Assumption \ref{assump.dense1} holds. Consider two mean zero Gaussian processes with Mat\'ern covariance functions $\sigma_0^2 K_{\alpha_0,\nu}$ and $\sigma^2 K_{\alpha,\nu}$, where $\sigma_0^2\alpha_0^{2\nu}=\theta_0=\sigma^2\alpha^{2\nu}$ and $\alpha>0$ is given. If $\rr_n(\cdot)$ is defined as in \eqref{eq:rrn} and $\rr(\alpha)=\lim_{n\to \infty} \rr_n(\alpha)$, then as $n\to\infty$,
\begin{align}
\sup_{s^* \in \Scal \backslash \Scal_n} \left| \frac{{\EE}_{(\sigma_0^2, \alpha_0)}\big\{e_n(s^*;\alpha)^2\big\}}{{\EE}_{(\sigma^2, \alpha)}\big\{e_n(s^*;\alpha)^2\big\}} - 1 \right| \leq 2\sqrt{\rr(\alpha) -\rr_n(\alpha)} \to 0, \label{eq:stein90b.1} \\
\sup_{s^* \in \Scal \backslash \Scal_n} \left| \frac{{\EE}_{(\sigma^2, \alpha)}\big\{e_n(s^*;\alpha)^2\big\}}{{\EE}_{(\sigma_0^2, \alpha_0)}\big\{e_n(s^*;\alpha_0)^2\big\}} - 1 \right| \leq 4\sqrt{\rr(\alpha) -\rr_n(\alpha)} \to 0. \label{eq:stein90b.2}
\end{align}
\end{lemma}

\begin{proof}[Proof of Lemma \ref{lem:stein90b}]
Using similar notation to \citet{Stein90b}, we let
\begin{align*}
a_n(s^*;\alpha) &= \frac{{\EE}_{(\sigma_0^2, \alpha_0)}\big\{e_n(s^*;\alpha)^2\big\}}{{\EE}_{(\sigma_0^2, \alpha_0)}\big\{e_n(s^*;\alpha_0)^2\big\}} - 1, \quad \widetilde a_n(s^*;\alpha) = \frac{{\EE}_{(\sigma^2, \alpha)}\big\{e_n(s^*;\alpha_0)^2\big\}}{{\EE}_{(\sigma^2, \alpha)}\big\{e_n(s^*;\alpha)^2\big\}} - 1,\\
b_n(s^*;\alpha) &= \frac{{\EE}_{(\sigma^2, \alpha)}\big\{e_n(s^*;\alpha_0)^2\big\}}{{\EE}_{(\sigma_0^2, \alpha_0)}\big\{e_n(s^*;\alpha_0)^2\big\}} - 1 , \quad \widetilde b_n(s^*;\alpha) = \frac{{\EE}_{(\sigma_0^2, \alpha_0)}\big\{e_n(s^*;\alpha)^2\big\}}{{\EE}_{(\sigma^2, \alpha)}\big\{e_n(s^*;\alpha)^2\big\}} - 1 .
\end{align*}
In \citet{Stein90b}, their Theorem 1, Lemma 2 and the analysis in Section 3 imply that for every given $\alpha>0$, as $n\to\infty$,
\begin{align}
& 0\leq \sup_{s^* \in \Scal \backslash \Scal_n} \left[b_n(s^*;\alpha) + \widetilde b_n(s^*;\alpha)\right]  \leq 2[\rr(\alpha) -\rr_n(\alpha)], \label{eq:s90b.0} \\
& \sup_{s^* \in \Scal \backslash \Scal_n} \left|b_n(s^*;\alpha)\right|  \leq \sqrt{4[\rr(\alpha) -\rr_n(\alpha)]\max\{1,4[\rr(\alpha) -\rr_n(\alpha)]\}} \stackrel{(i)}{\leq} 2\sqrt{\rr(\alpha) -\rr_n(\alpha)}, \label{eq:s90b.1}  \\
& \text{and similarly } \sup_{s^* \in \Scal \backslash \Scal_n} \big| \widetilde b_n(s^*;\alpha) \big| \leq 2\sqrt{\rr(\alpha) -\rr_n(\alpha)}, \label{eq:s90b.2}
\end{align}
where (i) follows because $\rr_n(\alpha)$ increases to $\rr(\alpha)$ as $n\to\infty$. Therefore, \eqref{eq:stein90b.1} follows from \eqref{eq:s90b.2} and the definition of $\widetilde b_n(s^*;\alpha)$.

Using the relation
\begin{align*}
& \left[1+a_n(s^*;\alpha)\right]\left[1+\widetilde a_n(s^*;\alpha)\right] = \left[1+b_n(s^*;\alpha)\right]\left[1+\widetilde b_n(s^*;\alpha)\right],
\end{align*}
and the fact that $\widetilde a_n(s^*;\alpha)\geq 0$, we can obtain that
\begin{align} \label{eq:s90b.3}
\sup_{s^* \in \Scal \backslash \Scal_n} a_n(s^*;\alpha) & \leq \sup_{s^* \in \Scal \backslash \Scal_n} \left[b_n(s^*;\alpha) + \widetilde b_n(s^*;\alpha) + b_n(s^*;\alpha) \widetilde b_n(s^*;\alpha)\right] \nonumber \\
&\leq \sup_{s^* \in \Scal \backslash \Scal_n}\left|b_n(s^*;\alpha) + \widetilde b_n(s^*;\alpha)\right| + \sup_{s^* \in \Scal \backslash \Scal_n}\left|b_n(s^*;\alpha) \widetilde b_n(s^*;\alpha)\right| \nonumber \\
&\stackrel{(i)}{\leq} 2\left[\rr(\alpha) -\rr_n(\alpha)\right] + 4\left[\rr(\alpha) -\rr_n(\alpha)\right] \nonumber \\
&\leq 6\left[\rr(\alpha) -\rr_n(\alpha)\right],
\end{align}
where (i) follows from \eqref{eq:s90b.0}, \eqref{eq:s90b.1} and \eqref{eq:s90b.2}.

On the other hand, by the definition of $a_n(s^*;\alpha)$ and $\widetilde b_n(s^*;\alpha)$, we have
\begin{align*}
&\sup_{s^* \in \Scal \backslash \Scal_n} \left| \frac{{\EE}_{(\sigma^2, \alpha)}\big\{e_n(s^*;\alpha)^2\big\}}{{\EE}_{(\sigma_0^2, \alpha_0)}\big\{e_n(s^*;\alpha_0)^2\big\}} - 1 \right|
= \sup_{s^* \in \Scal \backslash \Scal_n} \left| \frac{a_n(s^*;\alpha)+1}{1+\widetilde b_n(s^*;\alpha)}-1\right| \\
&= \sup_{s^* \in \Scal \backslash \Scal_n} \left| \frac{a_n(s^*;\alpha) - \widetilde b_n(s^*;\alpha)}{1+\widetilde b_n(s^*;\alpha)}\right|
\stackrel{(i)}{\leq} \frac{3}{2}\sup_{s^* \in \Scal \backslash \Scal_n} a_n(s^*;\alpha) + \frac{3}{2}\sup_{s^* \in \Scal \backslash \Scal_n} \left|\widetilde b_n(s^*;\alpha)\right| \\
&\stackrel{(ii)}{\leq}  9\left[\rr(\alpha) -\rr_n(\alpha)\right] + 3\sqrt{\rr(\alpha) -\rr_n(\alpha)} \\
&\stackrel{(iii)}{\leq}  4\sqrt{\rr(\alpha) -\rr_n(\alpha)} ,
\end{align*}
where (i) follows from that $a_n(s^*;\alpha)\geq 0$ and for all sufficiently large $n$, $\rr(\alpha) -\rr_n(\alpha)<1/81$ so $\big|\widetilde b_n(s^*;\alpha)\big|\leq 1/3$ by \eqref{eq:s90b.2}; (ii) follows from \eqref{eq:s90b.2} and \eqref{eq:s90b.3}; and (iii) follows from $\sqrt{\rr(\alpha) -\rr_n(\alpha)}<1/9$ as $n\to\infty$. This has proved \eqref{eq:stein90b.2}.
\end{proof}

\vspace{5mm}

\begin{proof}[Proof of Theorem \ref{thm:suprate.OU}]
We verify Assumption \ref{assump.Mn} for this special case. We can calculate that
\begin{align} \label{eq:rn.OU}
\rr_n(\alpha) & = \frac{\alpha}{2\alpha_0}\tr\left(R_{\alpha}^{-1}R_{\alpha_0}\right) + \frac{\alpha_0}{2\alpha}\tr\left(R_{\alpha_0}^{-1}R_{\alpha}\right) -n \nonumber \\
&= \frac{\alpha}{2\alpha_0} \left[n + \frac{(n-1)\alpha}{\alpha_0}\frac{ \ee^{-\alpha/n}( \ee^{-\alpha/n}- \ee^{-\alpha_0/n})}{1- \ee^{-2\alpha/n}} \right] \nonumber\\
&\quad +\frac{\alpha_0}{2\alpha} \left[ n+ \frac{(n-1)\alpha_0}{\alpha}\frac{ \ee^{-\alpha_0/n}( \ee^{-\alpha_0/n}- \ee^{-\alpha/n})}{1- \ee^{-2\alpha_0/n}}\right]-n.
\end{align}
The Taylor series expansion of the first term in \eqref{eq:rn.OU} over all $\alpha\in [\underline\alpha_n,\overline\alpha_n]$ gives
\begin{align} \label{eq:Taylor1}
& \frac{\alpha}{2\alpha_0} \left[n + \frac{(n-1)\alpha}{\alpha_0}\frac{ \ee^{-\alpha/n}( \ee^{-\alpha/n}- \ee^{-\alpha_0/n})}{1- \ee^{-2\alpha/n}} \right] \nonumber \\
={}& \frac{n}{2} + \frac{(\alpha-\alpha_0)(\alpha+\alpha_0+2)}{4\alpha_0} + \frac{(\alpha_0^2-\alpha^2)(\alpha_0+3)}{12\alpha_0 n} \nonumber \\
& \quad + \frac{(\alpha^2-\alpha_0^2)(\alpha_0^2+4\alpha_0-\alpha^2)}{48\alpha_0 n^2} + O\left(\frac{1}{n^{5/2}}\right).
\end{align}
The order of the remainder is at most $O(n^{-5/2})$ since $\overline\alpha_n\preceq n^{0.02}$ and $\underline\alpha_n \succeq n^{-0.05}$.

By symmetry, for the second term in \eqref{eq:rn.OU}, we have
\begin{align} \label{eq:Taylor2}
& \frac{\alpha_0}{2\alpha} \left[n + \frac{(n-1)\alpha_0}{\alpha}\frac{ \ee^{-\alpha_0/n}( \ee^{-\alpha_0/n}- \ee^{-\alpha/n})}{1- \ee^{-2\alpha_0/n}} \right] \nonumber \\
={}& \frac{n}{2} + \frac{(\alpha_0-\alpha)(\alpha+\alpha_0+2)}{4\alpha} + \frac{(\alpha^2-\alpha_0^2)(\alpha+3)}{12\alpha n} \nonumber \\
& \quad + \frac{(\alpha_0^2-\alpha^2)(\alpha^2+4\alpha-\alpha_0^2)}{48\alpha n^2} + O\left(\frac{1}{n^{5/2}}\right).
\end{align}
Therefore, \eqref{eq:rn.OU}, \eqref{eq:Taylor1}, and \eqref{eq:Taylor2} together imply that
\begin{align}
\rr_n(\alpha) & = \frac{(\alpha-\alpha_0)^2(\alpha+\alpha_0+2)}{4\alpha\alpha_0} - \frac{(\alpha-\alpha_0)^2(\alpha+\alpha_0)}{4\alpha\alpha_0 n}  \nonumber \\
&\quad - \frac{(\alpha-\alpha_0)^2(\alpha+\alpha_0)^3}{48\alpha\alpha_0 n^2} + O\left(\frac{1}{n^{5/2}}\right), \nonumber
\end{align}
and
\begin{align}
\rr(\alpha) & = \lim_{n\to\infty}\rr_n(\alpha) =  \frac{(\alpha-\alpha_0)^2(\alpha+\alpha_0+2)}{4\alpha\alpha_0} .  \nonumber
\end{align}
Therefore, uniformly over all $\alpha\in [\underline\alpha_n,\overline\alpha_n]$,
\begin{align}
\rr(\alpha) - \rr_n(\alpha) & =  \frac{(\alpha-\alpha_0)^2(\alpha+\alpha_0)}{4\alpha\alpha_0 n}  + \frac{(\alpha-\alpha_0)^2(\alpha+\alpha_0)^3}{48\alpha\alpha_0 n^2} + O\left(\frac{1}{n^{5/2}}\right).
\end{align}
By \eqref{eq:stein90b.1} in Lemma \ref{lem:stein90b} and the uniformity over all $\alpha\in [\underline\alpha_n,\overline\alpha_n]$, we obtain that for sufficiently large $n$,
\begin{align} \label{eq:thm6.ineq1}
&\sup_{\alpha\in [\underline\alpha_n,\overline\alpha_n]} \sup_{s^* \in \Scal \backslash \Scal_n} \left[ \frac{{\EE}_{(\sigma_0^2, \alpha_0)}\big\{e_n(s^*;\alpha)^2\big\}}{{\EE}_{(\theta_0/\alpha^{2\nu}, \alpha)}\big\{e_n(s^*;\alpha)^2\big\}} - 1 \right]^2  \nonumber \\
&\leq \sup_{\alpha\in [\underline\alpha_n,\overline\alpha_n]} 4[\rr(\alpha) - \rr_n(\alpha)] \leq \sup_{\alpha\in [\underline\alpha_n,\overline\alpha_n]} \frac{2(\alpha-\alpha_0)^2(\alpha+\alpha_0)}{n\alpha\alpha_0 } \nonumber \\
&\leq \sup_{\alpha\in [\underline\alpha_n,\overline\alpha_n]} \frac{2(\alpha+\alpha_0)}{n}\cdot \left(\frac{\alpha}{\alpha_0} + \frac{\alpha_0}{\alpha} - 2\right) \leq \frac{2(\overline\alpha_n+\alpha_0)}{n}\cdot \max\left\{\frac{(\overline\alpha_n - \alpha_0)^2}{\overline\alpha_n \alpha_0}, \frac{(\underline\alpha_n - \alpha_0)^2}{\underline\alpha_n \alpha_0} \right\} \nonumber  \\
&\leq \frac{3\overline\alpha_n\max\left(\frac{\overline\alpha_n }{\alpha_0}, \frac{\alpha_0}{\underline\alpha_n} \right)}{n} \leq 3n^{-1} \overline\alpha_n \left(\frac{\overline\alpha_n }{\alpha_0} + \frac{\alpha_0}{\underline\alpha_n}\right) \leq 4n^{2\overline\kappa+\underline\kappa-1}.
\end{align}
Since ${\EE}_{(\sigma_0^2, \alpha_0)}\big\{e_n(s^*;\alpha)^2\big\} \geq {\EE}_{(\theta_0/\alpha^{2\nu},\alpha)}\big\{e_n(s^*;\alpha)^2\big\}$, it follows from \eqref{eq:thm6.ineq1} that
\begin{align} \label{eq:thm6.ineq1.1}
&\quad \sup_{\alpha \in [\underline\alpha_n, \overline\alpha_n]} \sup_{s^* \in \Scal \backslash \Scal_n} \left| \frac{{\EE}_{(\theta_0/\alpha^{2\nu},\alpha)}\big\{e_n(s^*;\alpha)^2\big\}}{{\EE}_{(\sigma_0^2,\alpha_0)}\big\{e_n(s^*;\alpha)^2\big\}} - 1 \right| \nonumber \\
&=\sup_{\alpha \in [\underline\alpha_n, \overline\alpha_n]} \sup_{s^* \in \Scal \backslash \Scal_n} \frac{ {\EE}_{(\sigma_0^2,\alpha_0)}\big\{e_n(s^*;\alpha)^2\big\}\Big / {\EE}_{(\theta_0/\alpha^{2\nu},\alpha)}\big\{e_n(s^*;\alpha)^2\big\} - 1 } {{\EE}_{(\sigma_0^2,\alpha_0)}\big\{e_n(s^*;\alpha)^2\big\}\Big/ {\EE}_{(\theta_0/\alpha^{2\nu},\alpha)}\big\{e_n(s^*;\alpha)^2\big\}} \nonumber  \\
&\leq \sup_{\alpha \in [\underline\alpha_n, \overline\alpha_n]} \sup_{s^* \in \Scal \backslash \Scal_n} \left[ \frac{{\EE}_{(\sigma_0^2,\alpha_0)}\big\{e_n(s^*;\alpha)^2\big\}}{{\EE}_{(\theta_0/\alpha^{2\nu},\alpha)}\big\{e_n(s^*;\alpha)^2\big\}} - 1 \right] \nonumber \\
&\leq 2n^{(\overline\kappa+\underline\kappa/2) - 1/2} .
\end{align}
From \eqref{eq:stein90b.2} in Lemma \ref{lem:stein90b}, we obtain that for sufficiently large $n$,
\begin{align}
& \sup_{\alpha\in [\underline\alpha_n,\overline\alpha_n]} \sup_{s^* \in \Scal \backslash \Scal_n} \left| \frac{{\EE}_{(\theta_0/\alpha^{2\nu}, \alpha)}\big\{e_n(s^*;\alpha)^2\big\}}{{\EE}_{(\sigma_0^2, \alpha_0)}\big\{e_n(s^*;\alpha_0)^2\big\}} - 1 \right|^2 \nonumber \\
&\leq \sup_{\alpha\in [\underline\alpha_n,\overline\alpha_n]} \frac{5(\alpha-\alpha_0)^2(\alpha+\alpha_0)}{n\alpha\alpha_0 } \leq \sup_{\alpha\in [\underline\alpha_n,\overline\alpha_n]} \frac{5(\alpha+\alpha_0)}{n}\cdot \left(\frac{\alpha}{\alpha_0} + \frac{\alpha_0}{\alpha} - 2\right) \nonumber \\
&\leq \frac{5(\overline\alpha_n+\alpha_0)}{n}\cdot \max\left\{\frac{(\overline\alpha_n - \alpha_0)^2}{\overline\alpha_n \alpha_0}, \frac{(\underline\alpha_n - \alpha_0)^2}{\underline\alpha_n \alpha_0} \right\} \leq \frac{6\overline\alpha_n\max\left(\frac{\overline\alpha_n }{\alpha_0}, \frac{\alpha_0}{\underline\alpha_n} \right)}{n} \leq 7n^{2\overkappa + \underkappa - 1} . \nonumber
\end{align}
Therefore, for sufficiently large $n$,
\begin{align} \label{eq:thm6.ineq2}
& \sup_{\alpha\in [\underline\alpha_n,\overline\alpha_n]} \sup_{s^* \in \Scal \backslash \Scal_n} \left| \frac{{\EE}_{(\theta_0/\alpha^{2\nu}, \alpha)}\big\{e_n(s^*;\alpha)^2\big\}}{{\EE}_{(\sigma_0^2, \alpha_0)}\big\{e_n(s^*;\alpha_0)^2\big\}} - 1 \right| \leq 3 n^{(\overline\kappa+\underline\kappa/2) - 1/2} .
\end{align}
Based on \eqref{eq:thm6.ineq1.1} and \eqref{eq:thm6.ineq2}, we conclude that Assumption \ref{assump.Mn} is satisfied with $\varsigma_n=3n^{- 1/2+(\overline\kappa+\underline\kappa/2)}$. Therefore, the posterior convergence rates of asymptotic efficiency in Theorem \ref{thm:pae.main} become $\max\left(16n^{-1/2}\log n, 2\varsigma_n\right) = 6n^{- 1/2+(\overline\kappa+\underline\kappa/2)}$ as $n\to\infty$. This completes the proof of Theorem \ref{thm:suprate.OU}.
\end{proof}

\vspace{8mm}

\subsection{Proof of Theorem \ref{thm:vn.rate}} \label{supsubsec:vn.rate}

We introduce some concepts from scattered data approximation. For technical details, we refer the readers to the book \citet{Wen05}. For a generic kernel function $K(\cdot,\cdot)$ on $\Scal$, we define the \textit{power function} (Chapter 11 of \citealt{Wen05}) as
\begin{align}\label{eq:power.func}
\sP(s;K,\Scal_n) &= \left\{ K(s,s) - K(\Scal_n,s)^\top K(\Scal_n,\Scal_n)^{-1} K(\Scal_n,s)\right\}^{1/2}, \quad \text{for any } s\in \Scal,
\end{align}
where $\Scal_n=\{s_1,\ldots,s_n\}$, $K(\Scal_n,s) = (K(s_1,s),\ldots,K(s_n,s))^\top$, and $K(\Scal_n,\Scal_n)$ is the $n\times n$ covariance matrix with entries $\{K(\Scal_n,\Scal_n)\}_{ij}=K(s_i,s_j)$ for $i,j=1,\ldots,n$. The power function $\sP_{K,\Scal_n}(s)$ plays an important role in error estimates of kriging interpolation. We cite the following results from \citet{Wen05}:
\begin{lemma}  \label{lem:error.est.power}
(\citealt{Wen05} Theorem 11.4) For any $f\in \Hcal_K$, let $f_n = (f(s_1),\ldots,f(s_n))^\top$. Then
\begin{align}
& \left|f(s) - f_n^\top K(\Scal_n,\Scal_n)^{-1}K(\Scal_n,s) \right| \leq \|f\|_{\Hcal_K} \sP(s;K,\Scal_n), \text{ for any } s\in \Scal.  \nonumber
\end{align}
\end{lemma}

\vspace{3mm}

\begin{proof}[Proof of Theorem \ref{thm:vn.rate}]
By Assumption \ref{assump.m.func} and Lemma \ref{lem:matern.sobolev}, $\|\bbm_j\|_{\Hcal_{\sigma^2 K_{\alpha,\nu}}}$ is finite for any $(\sigma^2,\alpha)\in \RR^+ \times \RR^+$ and for all $j=1,\ldots,p$.

Define $\bbm_{j,n}=(\bbm_j(s_1),\ldots,\bbm_j(s_n))^\top$, for $j=1,\ldots,p$. Then Lemma \ref{lem:error.est.power} shows that for any $\alpha>0$, any $s\in \Scal$ and each $j=1,\ldots,p$,
\begin{align}\label{eq.bbm.bias1}
\big|\bbm_j(s)-r_{\alpha}(s)^{\top} R_{\alpha}^{-1} \bbm_{j,n}\big|\leq  \|\bbm_j\|_{\Hcal_{(\theta_0/\alpha^{2\nu}) K_{\alpha,\nu}}} \sP(s; (\theta_0/\alpha^{2\nu}) K_{\alpha,\nu},\Scal_n) .
\end{align}
Therefore, using the definition of $b_{\alpha}(s)$ in \eqref{eq:varBLUP.true}, for any $s^*\in \Scal\backslash \Scal_n$,
\begin{align}\label{eq:m.sum1}
b_{\alpha}(s^*)^{\top} b_{\alpha}(s^*) & =  \sum_{j=1}^p \big| \bbm_j(s^*) - r_{\alpha}(s^*)^{\top} R_{\alpha}^{-1} \bbm_{j,n} \big|^2    \nonumber \\
& \leq \sP(s^*;(\theta_0/\alpha^{2\nu}) K_{\alpha,\nu},\Scal_n)^2 \sum_{j=1}^p \|\bbm_j\|^2_{\Hcal_{(\theta_0/\alpha^{2\nu}) K_{\alpha,\nu}}}.
\end{align}
By the inequality (B.4) and the subsequent argument in the proof of Theorem 2 in \citet{Wangetal19}, for any $\alpha>0$,
\begin{align} \label{eq:MRM.min}
& \lambda_{\min}\left(M_n^\top R_{\alpha}^{-1} M_n \right) \geq \max_{\Ical\subseteq \{1,\ldots,n\},|\Ical|=p} \lambda_{\min}\left(M_{\Ical}^\top M_{\Ical}\right) / p  = \underline \lambda(M_n,p).
\end{align}

Therefore, using the reparameterization $\theta=\sigma^2\alpha^{2\nu}$ and the definition of ${\vv}_n(s^*;\sigma^2,\alpha)$ in \eqref{eq:varBLUP.true}, we can combine \eqref{eq:m.sum1} and \eqref{eq:MRM.min} and obtain that for any $(\sigma^2,\alpha)\in \RR^+ \times \RR^+$,
\begin{align} \label{eq:vv0.1}
&\quad ~ {\vv}_n(s^*;\sigma^2,\alpha) \nonumber \\
&=\sigma^2\left\{1 - r_{\alpha}(s^*)^\top R_{\alpha}^{-1} r_{\alpha}(s^*)\right\} + \sigma^2 b_{\alpha}(s^*)^\top \big(M_n^\top R_{\alpha}^{-1} M_n + \Omega_{\beta}\big)^{-1} b_{\alpha}(s^*) \nonumber \\
&\leq \frac{\theta}{\theta_0} \sP(s^*;(\theta_0/\alpha^{2\nu}) K_{\alpha,\nu},\Scal_n)^2 + \frac{\theta}{\theta_0}\cdot \frac{\theta_0}{\alpha^{2\nu}} \cdot  \lambda_{\min}\left(M_n^\top R_{\alpha}^{-1} M_n \right)^{-1} b_{\alpha}(s^*)^\top b_{\alpha}(s^*) \nonumber \\
&\leq \frac{\theta}{\theta_0}\sP(s^*;(\theta_0/\alpha^{2\nu}) K_{\alpha,\nu},\Scal_n)^2 \Bigg\{ \underline \lambda(M_n,p)^{-1} \frac{\theta_0}{\alpha^{2\nu}}\sum_{j=1}^p \|\bbm_j\|_{\Hcal_{(\theta_0/\alpha^{2\nu}) K_{\alpha,\nu}}}^2 + 1 \Bigg\} .
\end{align}
Because $\Scal=[0,T]^d$ is compact and convex with positive Lebesgue measure, and $\Scal_n$ is dense in $\Scal$ by Assumption \ref{assump.dense1}, Theorem 5.14 of \citet{WuSch93} has shown that for a constant $C_{\vv,1}>0$ that depends on $\sigma_0^2,\alpha_0,\nu,d,T$ and for all sufficiently large $n$,
\begin{align}\label{eq:wusch93}
\sup_{s^*\in \Scal} \sP(s^*;\sigma_0^2 K_{\alpha_0,\nu},\Scal_n)\leq C_{\vv,1} h_{\Scal_n}^{\nu} .
\end{align}
Therefore, the upper bound for ${\vv}_n(s^*;\sigma_0^2,\alpha_0)$ in the first convergence of Theorem \ref{thm:vn.rate} in the main text immediately follows from the upper bounds in \eqref{eq:vv0.1} and \eqref{eq:wusch93}, by setting $\theta=\theta_0$ and $\alpha=\alpha_0$:
\begin{align*}
& \sup_{s^*\in \Scal} {\vv}_n(s^*;\sigma_0^2,\alpha_0) \leq C_{\vv,1} \left[C_{\bbm} \sigma_0^2 \underline \lambda(M_n,p)^{-1} + 1 \right] h_{\Scal_n}^{2\nu}.
\end{align*}
Now we turn to the second convergence in Theorem \ref{thm:vn.rate}, where $(\sigma^2,\alpha)$ is randomly drawn from the posterior distribution $\Pi(\cdot |Y_n)$. We notice that Assumption \ref{assump.Mn2} can be equivalently written as
\begin{align}\label{eq:varsigma.p2}
\sup_{\alpha\in [\underline\alpha_n,\overline\alpha_n]} \sup_{s^* \in \Scal\backslash\Scal_n} \left|\frac{\sP(s^*;(\theta_0/\alpha^{2\nu})K_{\alpha,\nu},\Scal_n)^2}
    {\sP(s^*;\sigma_0^2K_{\alpha_0,\nu},\Scal_n)^2} - 1 \right| \leq \tilde \varsigma_n,
\end{align}
for a deterministic sequence $\tilde \varsigma_n\to 0$ as $n\to\infty$. It is trivial to see that if $s^*\in \Scal_n$, then $\sP(s^*;(\theta_0/\alpha^{2\nu}) K_{\alpha,\nu},\Scal_n)=\sP(s^*;\sigma_0^2 K_{\alpha_0,\nu},\Scal_n)=0$.

We recall from the proof of Theorem \ref{thm:pae.main} Part (ii) that $\Ecal_9=\{\alpha\in [\underline\alpha_n, \overline\alpha_n]\}$, and that Theorem \ref{thm:bvm2:joint} and its proof implies that $\Pi(\Ecal_9^c|Y_n)\to 0$ as $n\to\infty$ almost surely $P_{(\sigma_0^2,\alpha_0)}$. This implies that given any $\eta\in (0,1)$, any $\delta\in (0,1)$, there exist two numbers $0<\alpha_1<\alpha_2<\infty$ and a sufficiently large integer $N_{12}'$ ($\alpha_1,\alpha_2,N_{12}'$ are dependent on $\eta,\delta$), such that for all $n>N_{12}'$, $\Pr\left(\Pi(\alpha\in [\alpha_1,\alpha_2]~|~Y_n) \leq 1-\delta/2\right) < \eta/2$.

On the other hand, in the proof of Theorem \ref{thm:pae.micro}, \eqref{eq:theta.theta0.ratio7} has shown that $\Pi(|\theta/\theta_0 - 1| > 7n^{-1/2}\log n ~|~Y_n)\to 0$ as $n\to\infty$ almost surely in $P_{(\sigma_0^2,\alpha_0)}$. This implies that for a sufficiently large $N_{13}'$, such that for all $n>N_{13}'$, $\Pr(\Pi(|\theta/\theta_0 - 1| \leq 7n^{-1/2}\log n ) \leq 1- \delta/2) < \eta/2$. Define the event
\begin{align}\label{eq:Ecal11}
\Ecal_{11} & = \left\{|\theta/\theta_0 - 1| \leq 7n^{-1/2}\log n, \text{ and } \alpha\in [\alpha_1,\alpha_2] \right\}.
\end{align}
Then for all $n>N_{14}'=\max(N_{12}',N_{13}')$,
\begin{align}\label{eq:Ecal11.2delta}
& \quad ~ \Pr\left(\Pi(\Ecal_{11}^c|Y_n) > \delta\right)  \nonumber \\
&\leq \Pr\left(\Pi(|\theta/\theta_0 - 1| > 7n^{-1/2}\log n, \text{ or } \alpha\notin [\alpha_1,\alpha_2]~|~Y_n) > \delta\right) \nonumber \\
&\leq \Pr\left(\Pi(|\theta/\theta_0 - 1| > 7n^{-1/2}\log n ~|~Y_n) > \delta/2, \text{ or } \Pi(\alpha\notin [\alpha_1,\alpha_2]~|~Y_n) > \delta/2 \right) \nonumber \\
&\leq \Pr\left(\Pi(|\theta/\theta_0 - 1| > 7n^{-1/2}\log n ~|~Y_n) > \delta/2 \right) + \Pr\left(\Pi(\alpha\notin [\alpha_1,\alpha_2]~|~Y_n) > \delta/2 \right) \nonumber \\
&<\frac{\eta}{2} + \frac{\eta}{2} = \eta,
\end{align}
which can be equivalently written as $\Pr\left(\Pi(\Ecal_{11}|Y_n) > 1- \delta\right) > 1-\eta$.

Therefore, we combine \eqref{eq:vv0.1}, \eqref{eq:wusch93}, \eqref{eq:varsigma.p2}, and the posterior convergence of $\theta$ to $\theta_0$ above together, and obtain that on the event $\Ecal_{11}$,  for all $n>N_{14}'$,
\begin{align} \label{eq:vv0.2}
&\quad~ \sup_{s^* \in \Scal} {\vv}_n(s^*;\sigma^2,\alpha) \nonumber\\
&\leq \frac{\theta}{\theta_0} \cdot \left[\sup_{s^*\in \Scal}  \sP(s^*;(\theta_0/\alpha^{2\nu}) K_{\alpha,\nu},\Scal_n)^2 \right]  \cdot \Bigg\{ \underline \lambda(M_n,p)^{-1} \frac{\theta_0}{\alpha^{2\nu}}\sum_{j=1}^p \|\bbm_j\|_{\Hcal_{(\theta_0/\alpha^{2\nu}) K_{\alpha,\nu}}}^2 + 1 \Bigg\}  \nonumber \\
&\stackrel{(i)}{\leq} \left(1 + 7n^{-1/2} \log n \right)\cdot \left[ \left(1 + \tilde \varsigma_n \right)\sup_{s^*\in \Scal} \sP(s^*;\sigma_0^2 K_{\alpha_0,\nu},\Scal_n)^2 \right] \nonumber\\
&\quad \times \Bigg[ \underline \lambda(M_n,p)^{-1} \frac{\theta_0}{\alpha_1^{2\nu}} \max\left\{\left(\frac{\alpha_2}{\alpha_0}\right)^{2\nu+d},1\right\} \sum_{j=1}^p \|\bbm_j\|_{\Hcal_{\sigma_0^2 K_{\alpha_0,\nu}}}^2 + 1 \Bigg] \nonumber\\
&\stackrel{(ii)}{\leq} \left(1 + 7n^{-1/2} \log n \right) \left(1 + \tilde \varsigma_n \right) C_{\vv,1}^2 h_{\Scal_n}^{2\nu}  \nonumber\\
&\quad \times \Bigg[ \underline \lambda(M_n,p)^{-1} \frac{\theta_0}{\alpha_1^{2\nu}} \max\left\{\left(\frac{\alpha_2}{\alpha_0}\right)^{2\nu+d},1\right\} c_2(\sigma_0,\alpha_0)^2 \sum_{j=1}^p\|\bbm_j\|_{\Wcal_2^{\nu+d/2}}^2 + 1 \Bigg] , \nonumber \\
&\stackrel{(iii)}{\leq} C_{\vv,2} \left[ C_{\vv,3} C_{\bbm}\underline\lambda(M_n,p)^{-1} + 1  \right] h_{\Scal_n}^{2\nu} ,
\end{align}
where the inequality (i) follows from $\alpha_1\leq \alpha\leq \alpha_2$ on $\Ecal_{11}$, the inequality \eqref{eq:varsigma.p2}, and the relation between the RKHS norms of $\Hcal_{(\theta_0/\alpha^{2\nu})K_{\alpha,\nu}}$ and $\Hcal_{\sigma_0^2K_{\alpha_0,\nu}}$ in Lemma \ref{lem:rkhs.ordering}; (ii) follows from \eqref{eq:wusch93} and the equivalence between the Mat\'ern RKHS norm and the Sobolev norm in Lemma \ref{lem:matern.sobolev}; in (iii), the constant $C_{\vv,2} \geq \sup_{n\geq 1} \left(1 + 7n^{-1/2} \log n \right) \left(1 + \tilde \varsigma_n \right) C_{\vv,1}^2$, which depends on $\sigma_0^2,\alpha_0,\nu,d,T$, and $C_{\vv,3}= (\theta_0/\alpha_1^{2\nu}) \max\big\{\left(\alpha_2/\alpha_0\right)^{2\nu+d},1\big\}c_2(\sigma_0,\alpha_0)^2$, which depends on $\eta,\delta,\sigma_0^2,\alpha_0,\nu,d,T$.

Finally, we combine \eqref{eq:Ecal11.2delta} and \eqref{eq:vv0.2} to conclude that for all $n>N_{14}'$,
\begin{align}
& \Pr\left(\Pi\left[\sup_{s^* \in \Scal} {\vv}_n(s^*;\sigma^2,\alpha) \leq C_{\vv,2} \left[C_{\vv,3} C_{\bbm} \underline \lambda(M_n,p)^{-1} + 1 \right] h_{\Scal_n}^{2\nu}~ \Big|~ Y_n\right] > 1-\delta \right) > 1-\eta . \nonumber
\end{align}
Setting $N_3=N_{14}'$ completes the proof.
\end{proof}

\vspace{8mm}

\section{Additional Simulation Results for Universal Kriging Model with Regression Terms} \label{sec:add.sim}
We present additional simulation results for the universal kriging model \eqref{eq:obs.model} with regression terms $Y(s)=\bbm(s)^\top\beta_0 + X(s)$ for $s\in \Scal$, and $X(\cdot)\sim \gp(0,\sigma_0^2 K_{\alpha_0,\nu})$. We consider three values of the smoothness parameter $\nu=1/2$, $\nu=1/4$ and $\nu=3/2$. We still set $\Scal=[0,1]^d$ and $\Scal_n$ to be the regular grid as in Section \ref{sec:simulation} for $d=1,2$. For the $d=1$ case, we let $\bbm(s)=\left(1,s,s^2,s^3\right)^\top$ for $s\in [0,1]$ and $\beta_0=(1,0.66,-1.5,1)^\top$. For $d=2$, we let $\bbm(s)=\left(1,s_1,s_2,s_1^2,s_1s_2,s_2^2\right)^\top$ and $\beta_0=(1,-1.5,-1.5,2,1,2)^\top$ for $s=(s_1,s_2)\in [0,1]^2$. The true covariance parameters are $\alpha_0=1,\sigma_0^2=2,\theta_0=\sigma_0^2\alpha_0^{2\nu}=2$ for $\nu=1/2,1/4,3/2$. We impose the noninformative improper prior $\pi(\beta|\sigma^2,\alpha)\propto 1$ on $\beta$, corresponding to $\Omega_{\beta}=0_{p\times p}$. The prior specification for $(\theta,\alpha)$ and the posterior sampling and estimation procedures are all the same as in Section \ref{sec:simulation}.

For $\nu=1/2$, we report the posterior means and variances of $(\theta,\alpha)$ from both the true posterior distribution and the limiting posterior from Theorem \ref{thm:bvm2:joint}, as well as the $W_2$ distance between these two distributions in Tables \ref{tab:W2.dim1.m} and \ref{tab:W2.dim2.m}. We have similar observation to Tables \ref{tab:W2.dim1} and \ref{tab:W2.dim2} for the model without regression terms in Section \ref{sec:simulation} of the main text. The marginal posterior of $\theta$ is close to the normal limiting distribution whose center is increasingly close to $\theta_0=2$ with a shrinking variance as $n$ increases. The marginal posterior of $\alpha$ maintains a large posterior variance. The approximation errors from the limiting marginal posterior distributions of $\theta$ and $\alpha$ decrease as $n$ increases. Figure \ref{fig:contour2} illustrates the convergence of posterior densities for the $d=1$ case, which shows similar convergence to that in Figure \ref{fig:contour1} in the main text.

\begin{table}[ht]
\caption{Parameter estimation and Wasserstein-2 distances between the true posterior and the limiting posteriors in Theorem \ref{thm:bvm2:joint} for the model with $\nu=1/2$, $d=1$ and with regression terms $\bbm(\cdot)^\top \beta$. $\EE(\cdot|Y_n)$, $\Var(\cdot|Y_n)$, $\widetilde \EE(\cdot|Y_n)$, and $\widetilde \Var(\cdot|Y_n)$ are the posterior means and variances under the true posterior, the limiting posterior in Theorem \ref{thm:bvm2:joint}, and the limiting posterior in Theorem \ref{thm:OU1}. The true parameter values are $\theta_0=2$ and $\alpha_0=1$. All numbers are averaged over 100 macro replications. The standard errors are in the parentheses.}
\label{tab:W2.dim1.m}
\centering
{
\footnotesize
\begin{tabular}{c|ccccc}
\hline
$d=1$ & $n=25$ & $n=50$ & $n=100$ & $n=200$ & $n=400$ \\
\hline
$\EE(\theta|Y_n)$ & 2.7152 (0.0826) & 2.3743 (0.0505) & 2.2113 (0.0333) & 2.0659 (0.0193) & 2.0334 (0.0130) \\
$\Var(\theta|Y_n)$ & 1.3074 (0.0800) & 0.3269 (0.0143) & 0.1162 (0.0036) & 0.0465 (0.0009) & 0.0214 (0.0003) \\
\hdashline
$\widetilde \EE(\theta|Y_n)$ & 1.9597 (0.0630) &  2.0529 (0.0436) & 2.0664 (0.0311) & 1.9983 (0.0188) & 2.0004 (0.0127) \\
$\widetilde \Var(\theta|Y_n)$ & 0.3204 (0.0007) & 0.1604 (0.0003) & 0.0802 (0.0002) & 0.0399 (0.0001) & 0.0200 (0.0000) \\
\hline
$\EE(\alpha|Y_n)$ &  9.4697 (0.3697) & 8.5853 (0.4230) & 8.2324 (0.4578) & 8.1489 (0.4035) & 7.5458 (0.3547) \\
$\Var(\alpha|Y_n)$& 67.3011 (4.6625) & 46.8428 (4.0312) & 39.2354 (3.5382) & 36.8218 (2.9984) & 31.6578 (2.4480) \\
\hdashline
$\widetilde \EE(\alpha|Y_n)$ &  8.5783 (0.3278) & 8.1409 (0.3872) & 8.0208 (0.4350) & 8.0678 (0.4043) & 7.5148 (0.3569) \\
$\widetilde \Var(\alpha|Y_n)$ & 56.5017 (4.0339) & 42.9290 (3.6248) & 37.2086 (3.3029) & 35.9076 (2.9028) & 31.5305 (2.4689) \\
\hline
\end{tabular}
\begin{tabular}{c|ccccc}
\hline
$d=1$ & $n=25$ & $n=50$ & $n=100$ & $n=200$ & $n=400$ \\
\hline
\multirow{2}{*}{$W_2\left(\Pi(\ud\theta|Y_n),\Ncal\left(\ud\theta\Big|\widetilde\theta_{\alpha_0}, \tfrac{2\theta_0^2}{n}\right) \right)$} & 0.9812  & 0.3834  & 0.1672  & 0.0753  & 0.0370  \\
& (0.0360) & (0.0132) & (0.0050) & (0.0018) & (0.0008) \\
\multirow{2}{*}{$W_2(\Pi(\ud\alpha|Y_n),\widetilde \Pi(\ud\alpha|Y_n))$} & 1.3446  & 0.7517  & 0.5161  & 0.4720  & 0.3998 \\
& (0.0679) & (0.0521) & (0.0380) & (0.0269) & (0.0203) \\
\hline
\end{tabular}
}
\end{table}

\begin{table}[ht]
\caption{Parameter estimation and Wasserstein-2 distances between the true posterior and the limiting posteriors in Theorem \ref{thm:bvm2:joint} for the model with $\nu=1/2$, $d=2$ and with regression terms $\bbm(\cdot)^\top \beta$. $\EE(\cdot|Y_n)$, $\Var(\cdot|Y_n)$, $\widetilde \EE(\cdot|Y_n)$, and $\widetilde \Var(\cdot|Y_n)$ are the posterior means and variances under the true posterior and the limiting posterior in Theorem \ref{thm:bvm2:joint}. The true parameter values are $\theta_0=2$ and $\alpha_0=1$. All numbers are averaged over 100 macro replications. The standard errors are in the parentheses.}
\label{tab:W2.dim2.m}
\centering
{
\footnotesize
\begin{tabular}{c|ccc}
\hline
$d=2$ & $n=10^2$ & $n=20^2$ & $n=30^2$ \\
\hline
$\EE(\theta|Y_n)$ & 2.0309 (0.0307) & 2.0139 (0.0146)& 1.9947 (0.0096)  \\
$\Var(\theta|Y_n)$ & 0.0884 (0.0026) & 0.0210 (0.0003) & 0.0090 (0.0001) \\
\hdashline
$\widetilde \EE(\theta|Y_n)$ & 2.0223 (0.0320) & 2.0099 (0.0146) & 1.9927 (0.0097)  \\
$\widetilde \Var(\theta|Y_n)$ & 0.0800 (0.0001) & 0.0200 (0.0000) & 0.0089 (0.0000) \\
\hdashline
$\EE(\alpha|Y_n)$ & 1.1007 (0.0179) & 1.0905 (0.0197) & 1.0981 (0.0252)  \\
$\Var(\alpha|Y_n)$ & 1.0745 (0.0441) & 1.0086 (0.0352) &  1.0276 (0.0578) \\
\hdashline
$\widetilde \EE(\alpha|Y_n)$ & 1.1028 (0.0179) & 1.0767 (0.0427) & 1.0871 (0.0240)  \\
$\widetilde \Var(\alpha|Y_n)$ &  1.0952 (0.0462) & 1.0019 (0.0444) & 1.0192 (0.0588) \\
\hline
$W_2\left(\Pi(\ud\theta|Y_n),\Ncal\left(\ud\theta\Big|\widetilde\theta_{\alpha_0}, \tfrac{2\theta_0^2}{n}\right) \right)$ & 0.0595 (0.0215) & 0.0167 (0.0055) & 0.0086 (0.0025) \\
$W_2(\Pi(\ud\alpha|Y_n),\widetilde \Pi(\ud\alpha|Y_n))$ & 0.1011 (0.0365) & 0.1020 (0.0459)  & 0.0963 (0.0458) \\
\hline
\end{tabular}
}
\end{table}

\begin{figure}
\centering
\includegraphics[width=0.84\textwidth]{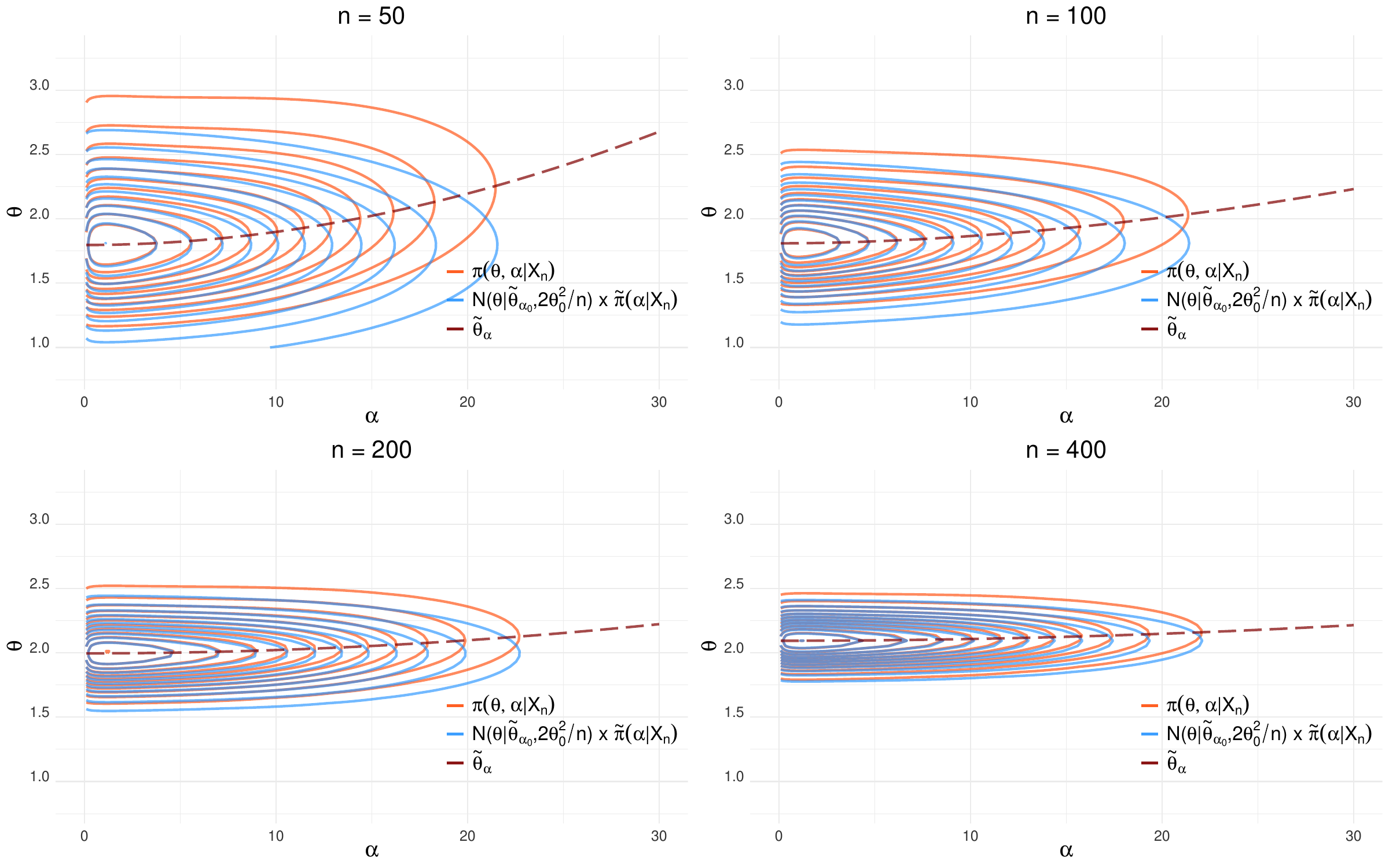}
\caption{Contour plots of the true joint posterior density $\pi(\theta,\alpha|Y_n)$ (in red) and the limiting posterior density $\mathcal{N}(\theta|\widetilde \theta_{\alpha_0}, 2\theta_0^2/n)\times \widetilde \pi(\alpha|Y_n)$ (in blue) in Theorem \ref{thm:bvm2:joint}, for the 1-d Ornstein-Uhlenbeck process with sample size $n=50,100,200,400$ in the model with regression terms $\bbm(\cdot)^\top\beta$. The dashed line is the ``ridge" REML $\widetilde\theta_{\alpha}$ (given in \eqref{tildetheta1}), the value of $\theta$ that maximizes the joint likelihood for each given $\alpha$. The true parameter values are $\theta_0=2$ and $\alpha_0=1$.}
\label{fig:contour2}
\end{figure}

Similar to Table \ref{tab:mse} in Section \ref{sec:simulation} of the main text, we further compute the asymptotic efficiency measure for the model with regression terms, using the relative error of GP predictive variance to the oracle predictive variance defined as
\begin{align}\label{eq:ratiodiff.m}
\mathsf{r}_{n}(s^*) = \left| \tfrac{\vv_n(s^*;\sigma^2,\alpha)}{\vv_n(s^*;\sigma_0^2,\alpha_0) } - 1 \right|,
\end{align}
over a large number of testing points $s^*$ from the Latin hypercube design, where $\vv_n(s^*;\sigma^2,\alpha)$ is given in \eqref{eq:varBLUP.true}. We again use $1000$ testing points in $\Scal=[0,1]$ for the $d=1$ case, and $2500$ testing points in $\Scal=[0,1]^2$ for the $d=2$ case. The posterior expectations of $\mathsf{r}_{n}(s^*)$ are reported in Table \ref{tab:mse.m}. We can see that for both $d=1$ and $d=2$ cases, the GP predictive variance based on a randomly drawn $(\sigma^2,\alpha)$ from the posterior has a decreasing relative error to the oracle predictive variance as $n$ increases.

\begin{table}[ht]
\caption{The posterior means of the ratio of predictive variance defined in \eqref{eq:ratiodiff.table} maximized over $2500$ testing points $s^*$ for the model with $\nu=1/2$ and with regression terms $\bbm(\cdot)^\top \beta$, averaged over 100 macro replications. The standard errors are in the parentheses.}
\label{tab:mse.m}
\centering
{
\footnotesize
\begin{tabular}{c|ccccc}
\hline
$d=1$ & $n=25$ & $n=50$ & $n=100$ & $n=200$ & $n=400$ \\
\hline
\multirow{2}{*}{$\EE\big[\max\limits_{s^*\in\Scal^*}\mathsf{r}_{n}(s^*)|Y_n\big]$} & 0.5452 & 0.3197 & 0.2055 & 0.1201 & 0.0795 \\
& (0.0520) & (0.0245) & (0.0142) & (0.0082) & (0.0055)  \\
\hline
\hline
$d=2$ & $n=10^2$ & $n=20^2$ & $n=30^2$ & \\
\hline
\multirow{2}{*}{$\EE\big[\max\limits_{s^*\in\Scal^*}\mathsf{r}_{n}(s^*)|Y_n\big]$} & 0.1458 & 0.0861 & 0.0696 & \\
& (0.0105) & (0.0054) & (0.0041) &   \\
\hline
\end{tabular}
}
\end{table}

For $\nu=1/4$, we summarize the estimation and prediction results of in Tables \ref{tab:W2.dim1.nu0.25}, \ref{tab:W2.dim2.nu0.25}, and \ref{tab:mse.nu0.25}. For $\nu=3/2$, we summarize the results in Tables \ref{tab:W2.dim1.nu1.5}, \ref{tab:W2.dim2.nu1.5}, and \ref{tab:mse.nu1.5}. All results are averaged over 100 macro simulations. In particular, Tables \ref{tab:W2.dim1.nu0.25}, \ref{tab:W2.dim2.nu0.25},  \ref{tab:W2.dim1.nu1.5}, and \ref{tab:W2.dim2.nu1.5} provide the estimation results for marginal posterior means, variances, and the $W_2$ distance to the limiting distribution for the parameters $\theta$ and $\alpha$, in $d=1$ and $d=2$ cases. Tables \ref{tab:mse.nu0.25} and \ref{tab:mse.nu1.5} provide the prediction results for the asymptotic efficiency measure $\mathsf{r}_{n}(s^*)$ defined in \eqref{eq:ratiodiff.m}.

Overall, the tables for $\nu=1/4$ and $\nu=3/2$ show similar trends as the tables for $\nu=1/2$. The marginal posterior distribution of $\theta$ becomes concentrated around the true value $\theta_0=2$ as $n$ increases in all cases, and the normal limiting distribution is accurate in approximation. The marginal posterior of $\alpha$ does not converge to the true value $\alpha_0=1$ with a non-shrinking variance. The asymptotic efficiency measure $\mathsf{r}_{n}(s^*)$ decreases quickly to zero as $n$ increases for all cases except the case of $\nu=3/2,d=1$, where $\mathsf{r}_{n}(s^*)$ seems to decrease slower with $n$.

\begin{table}[H]
\caption{Parameter estimation and Wasserstein-2 distances between the true posterior and the limiting posteriors in Theorem \ref{thm:bvm2:joint} for the model with $\nu=1/4$, $d=1$ and with regression terms $\bbm(\cdot)^\top \beta$. $\EE(\cdot|Y_n)$, $\Var(\cdot|Y_n)$, $\widetilde \EE(\cdot|Y_n)$, and $\widetilde \Var(\cdot|Y_n)$ are the posterior means and variances under the true posterior and the limiting posterior in Theorem \ref{thm:bvm2:joint}. The true parameter values are $\theta_0=2$ and $\alpha_0=1$. All numbers are averaged over 100 macro replications. The standard errors are in the parentheses.}
\label{tab:W2.dim1.nu0.25}
\centering
{
\footnotesize
\begin{tabular}{c|ccccc}
\hline
$d=1$ & $n=25$ & $n=50$ & $n=100$ & $n=200$ & $n=400$ \\
\hline
$\EE(\theta|Y_n)$ & 2.5894 (0.0795) & 2.3458 (0.0493) & 2.2013 (0.0332) & 2.0612 (0.0193) & 2.0331 (0.0129) \\
$\Var(\theta|Y_n)$ & 1.0144 (0.0620) & 0.3020 (0.0124) & 0.1143 (0.0035) & 0.0457 (0.0009) & 0.0215 (0.0003) \\
\hdashline
$\widetilde \EE(\theta|Y_n)$ & 1.9593 (0.0630) &  2.0576 (0.0437) & 2.0680 (0.0312) & 1.9980 (0.0187) & 2.0004 (0.0127) \\
$\widetilde \Var(\theta|Y_n)$ & 0.3204 (0.0007) & 0.1604 (0.0003) & 0.0802 (0.0002) & 0.0399 (0.0001) & 0.0200 (0.0000) \\
\hline
$\EE(\alpha|Y_n)$ &  10.6603 (0.2895) & 10.0808 (0.4566) & 9.2329 (0.4538) & 9.3252 (0.4351) & 8.7252 (0.3941) \\
$\Var(\alpha|Y_n)$& 96.3202 (4.4944) & 78.058 (6.4254) & 59.2596 (5.3221) & 56.6530 (4.9028) & 49.7329 (4.0853) \\
\hdashline
$\widetilde \EE(\alpha|Y_n)$ &  9.8786 (0.2713) & 9.6640 (0.4221) & 8.9950 (0.4364) & 9.2310 (0.4351) & 8.7241 (0.3947) \\
$\widetilde \Var(\alpha|Y_n)$ & 83.3761 (3.9361) & 71.4120 (5.7540) & 56.1845 (5.0395) & 55.3350 (4.7184) & 49.4396 (4.0132) \\
\hline
\end{tabular}
\begin{tabular}{c|ccccc}
\hline
$d=1$ & $n=25$ & $n=50$ & $n=100$ & $n=200$ & $n=400$ \\
\hline
\multirow{2}{*}{$W_2\left(\Pi(\ud\theta|Y_n),\Ncal\left(\ud\theta\Big|\widetilde\theta_{\alpha_0}, \tfrac{2\theta_0^2}{n}\right) \right)$} & 0.8028  & 0.3450  & 0.1562  & 0.0709  & 0.0366  \\
& (0.0291) & (0.0100) & (0.0041) & (0.0016) & (0.0008) \\
\multirow{2}{*}{$W_2(\Pi(\ud\alpha|Y_n),\widetilde \Pi(\ud\alpha|Y_n))$} & 1.4547  & 0.9177  & 0.6772  & 0.5813  & 0.5626 \\
& (0.0580) & (0.0667) & (0.0403) & (0.0336) & (0.0287) \\
\hline
\end{tabular}
}
\end{table}

\begin{table}[H]
\caption{Parameter estimation and Wasserstein-2 distances between the true posterior and the limiting posteriors in Theorem \ref{thm:bvm2:joint} for the model with $\nu=1/4$, $d=2$ and with regression terms $\bbm(\cdot)^\top \beta$. $\EE(\cdot|Y_n)$, $\Var(\cdot|Y_n)$, $\widetilde \EE(\cdot|Y_n)$, and $\widetilde \Var(\cdot|Y_n)$ are the posterior means and variances under the true posterior and the limiting posterior in Theorem \ref{thm:bvm2:joint}. The true parameter values are $\theta_0=2$ and $\alpha_0=1$. All numbers are averaged over 100 macro replications. The standard errors are in the parentheses.}
\label{tab:W2.dim2.nu0.25}
\centering
{
\footnotesize
\begin{tabular}{c|ccc}
\hline
$d=2$ & $n=10^2$ & $n=20^2$ & $n=30^2$ \\
\hline
$\EE(\theta|Y_n)$ & 2.0277 (0.0303) & 2.0138 (0.0146)& 1.9951 (0.0096)  \\
$\Var(\theta|Y_n)$ & 0.0880 (0.0026) & 0.0209 (0.0003) & 0.0089 (0.0001) \\
\hdashline
$\widetilde \EE(\theta|Y_n)$ & 2.0228 (0.0316) & 2.0104 (0.0146) & 1.9928 (0.0097)  \\
$\widetilde \Var(\theta|Y_n)$ & 0.0800 (0.0001) & 0.0200 (0.0000) & 0.0089 (0.0000) \\
\hdashline
$\EE(\alpha|Y_n)$ & 1.1063 (0.0134) & 1.1009 (0.0154) & 1.1035 (0.0196)  \\
$\Var(\alpha|Y_n)$ & 1.1027 (0.0328) & 1.0606 (0.0366) &  1.0937 (0.0519) \\
\hdashline
$\widetilde \EE(\alpha|Y_n)$ & 1.0986 (0.0125) & 1.0844 (0.0157) & 1.0903 (0.0186)  \\
$\widetilde \Var(\alpha|Y_n)$ &  1.0958 (0.0319) & 1.0411 (0.0392) & 1.0632 (0.0511) \\
\hline
$W_2\left(\Pi(\ud\theta|Y_n),\Ncal\left(\ud\theta\Big|\widetilde\theta_{\alpha_0}, \tfrac{2\theta_0^2}{n}\right) \right)$ & 0.0584 (0.0239) & 0.0169 (0.0053) & 0.0086 (0.0026) \\
$W_2(\Pi(\ud\alpha|Y_n),\widetilde \Pi(\ud\alpha|Y_n))$ & 0.1099 (0.0422) & 0.1075 (0.0475)  & 0.1037 (0.0433) \\
\hline
\end{tabular}
}
\end{table}

\begin{table}[H]
\caption{The posterior means of the ratio of predictive variance defined in \eqref{eq:ratiodiff.table} maximized over $2500$ testing points $s^*$ for the model with $\nu=1/4$ and with regression terms $\bbm(\cdot)^\top \beta$, averaged over 100 macro replications. The standard errors are in the parentheses.}
\label{tab:mse.nu0.25}
\centering
{
\footnotesize
\begin{tabular}{c|ccccc}
\hline
$d=1$ & $n=25$ & $n=50$ & $n=100$ & $n=200$ & $n=400$ \\
\hline
\multirow{2}{*}{$\EE\big[\max\limits_{s^*\in\Scal^*}\mathsf{r}_{n}(s^*)|Y_n\big]$} & 0.4872 & 0.3088 & 0.2019 & 0.1188 & 0.0794 \\
& (0.0434) & (0.0232) & (0.0141) & (0.0082) & (0.0055)  \\
\hline
\hline
$d=2$ & $n=10^2$ & $n=20^2$ & $n=30^2$ & \\
\hline
\multirow{2}{*}{$\EE\big[\max\limits_{s^*\in\Scal^*}\mathsf{r}_{n}(s^*)|Y_n\big]$} & 0.1480 & 0.0868 & 0.0697 & \\
& (0.0111) & (0.0053) & (0.0041) &   \\
\hline
\end{tabular}
}
\end{table}

\begin{table}[H]
\caption{Parameter estimation and Wasserstein-2 distances between the true posterior and the limiting posteriors in Theorem \ref{thm:bvm2:joint} for the model with $\nu=3/2$, $d=1$ and with regression terms $\bbm(\cdot)^\top \beta$. $\EE(\cdot|Y_n)$, $\Var(\cdot|Y_n)$, $\widetilde \EE(\cdot|Y_n)$, and $\widetilde \Var(\cdot|Y_n)$ are the posterior means and variances under the true posterior and the limiting posterior in Theorem \ref{thm:bvm2:joint}. The true parameter values are $\theta_0=2$ and $\alpha_0=1$. All numbers are averaged over 100 macro replications. The standard errors are in the parentheses.}
\label{tab:W2.dim1.nu1.5}
\centering
{
\footnotesize
\begin{tabular}{c|ccccc}
\hline
$d=1$ & $n=25$ & $n=50$ & $n=100$ & $n=200$ & $n=400$ \\
\hline
$\EE(\theta|Y_n)$ & 2.8495 (0.0913) & 2.3841 (0.0507) & 2.2177 (0.0344) & 2.0674 (0.0196) & 2.0305 (0.0142) \\
$\Var(\theta|Y_n)$ & 1.6724 (0.1238) & 0.3364 (0.0162) & 0.1167 (0.0038) & 0.0466 (0.0009) & 0.0215 (0.0003) \\
\hdashline
$\widetilde \EE(\theta|Y_n)$ & 1.9658 (0.0664) &  2.0504 (0.0427) & 2.0693 (0.0319) & 1.9972 (0.0189) & 1.9983 (0.0139) \\
$\widetilde \Var(\theta|Y_n)$ & 0.3204 (0.0007) & 0.1604 (0.0003) & 0.0802 (0.0002) & 0.0399 (0.0001) & 0.0200 (0.0000) \\
\hline
$\EE(\alpha|Y_n)$ &  6.0370 (0.3310) & 5.0376 (0.2576) & 4.8005 (0.3143) & 5.0102 (0.2890) & 4.2705 (0.2080) \\
$\Var(\alpha|Y_n)$& 17.9624 (1.5519) & 11.4256 (0.8790) & 9.3582 (0.7143) & 9.6294 (0.6663) & 7.9125 (0.5840) \\
\hdashline
$\widetilde \EE(\alpha|Y_n)$ &  5.5050 (0.2907) & 4.8298 (0.2484) & 4.7126 (0.3044) & 4.9523 (0.2834) & 4.2357 (0.2059) \\
$\widetilde \Var(\alpha|Y_n)$ & 15.6473 (1.3884) & 10.5893 (0.7763) & 9.0112 (0.6664) & 9.5040 (0.6710) & 7.8498 (0.5704) \\
\hline
\end{tabular}
\begin{tabular}{c|ccccc}
\hline
$d=1$ & $n=25$ & $n=50$ & $n=100$ & $n=200$ & $n=400$ \\
\hline
\multirow{2}{*}{$W_2\left(\Pi(\ud\theta|Y_n),\Ncal\left(\ud\theta\Big|\widetilde\theta_{\alpha_0}, \tfrac{2\theta_0^2}{n}\right) \right)$} & 1.1592  & 0.3972  & 0.1710  & 0.0781  & 0.0366  \\
& (0.0540) & (0.0155) & (0.0063) & (0.0021) & (0.0009) \\
\multirow{2}{*}{$W_2(\Pi(\ud\alpha|Y_n),\widetilde \Pi(\ud\alpha|Y_n))$} & 0.6743  & 0.3400  & 0.2608  & 0.2630  & 0.2205 \\
& (0.0486) & (0.0185) & (0.0193) & (0.0165) & (0.0124) \\
\hline
\end{tabular}
}
\end{table}

\begin{table}[H]
\caption{Parameter estimation and Wasserstein-2 distances between the true posterior and the limiting posteriors in Theorem \ref{thm:bvm2:joint} for the model with $\nu=3/2$, $d=2$ and with regression terms $\bbm(\cdot)^\top \beta$. $\EE(\cdot|Y_n)$, $\Var(\cdot|Y_n)$, $\widetilde \EE(\cdot|Y_n)$, and $\widetilde \Var(\cdot|Y_n)$ are the posterior means and variances under the true posterior and the limiting posterior in Theorem \ref{thm:bvm2:joint}. The true parameter values are $\theta_0=2$ and $\alpha_0=1$. All numbers are averaged over 100 macro replications. The standard errors are in the parentheses.}
\label{tab:W2.dim2.nu1.5}
\centering
{
\footnotesize
\begin{tabular}{c|ccc}
\hline
$d=2$ & $n=10^2$ & $n=20^2$ & $n=30^2$ \\
\hline
$\EE(\theta|Y_n)$ & 2.0504 (0.0315) & 2.0162 (0.0148)& 1.9956 (0.0096)  \\
$\Var(\theta|Y_n)$ & 0.0953 (0.0029) & 0.0211 (0.0003) & 0.0091 (0.0001) \\
\hdashline
$\widetilde \EE(\theta|Y_n)$ & 2.0293 (0.0328) & 2.0118 (0.0149) & 1.9936 (0.0097)  \\
$\widetilde \Var(\theta|Y_n)$ & 0.0800 (0.0001) & 0.0200 (0.0000) & 0.0089 (0.0000) \\
\hdashline
$\EE(\alpha|Y_n)$ & 1.1005 (0.0451) & 0.9758 (0.0304) & 1.0077 (0.0445)  \\
$\Var(\alpha|Y_n)$ & 0.8706 (0.0676) & 0.6304 (0.0363) &  0.6375 (0.0474) \\
\hdashline
$\widetilde \EE(\alpha|Y_n)$ & 1.1206 (0.0481) & 0.9722 (0.0306) & 0.9976 (0.0434)  \\
$\widetilde \Var(\alpha|Y_n)$ &  0.9128 (0.0737) & 0.6283 (0.0356) & 0.6353 (0.0479) \\
\hline
$W_2\left(\Pi(\ud\theta|Y_n),\Ncal\left(\ud\theta\Big|\widetilde\theta_{\alpha_0}, \tfrac{2\theta_0^2}{n}\right) \right)$ & 0.0732 (0.0327) & 0.0184 (0.0066) & 0.0102 (0.0046) \\
$W_2(\Pi(\ud\alpha|Y_n),\widetilde \Pi(\ud\alpha|Y_n))$ & 0.0801 (0.0470) & 0.0672 (0.0349)  & 0.0651 (0.0394) \\
\hline
\end{tabular}
}
\end{table}

\begin{table}[H]
\caption{The posterior means of the ratio of predictive variance defined in \eqref{eq:ratiodiff.table} maximized over $2500$ testing points $s^*$ for the model with $\nu=3/2$ and with regression terms $\bbm(\cdot)^\top \beta$, averaged over 100 macro replications. The standard errors are in the parentheses.}
\label{tab:mse.nu1.5}
\centering
{
\footnotesize
\begin{tabular}{c|ccccc}
\hline
$d=1$ & $n=25$ & $n=50$ & $n=100$ & $n=200$ & $n=400$ \\
\hline
\multirow{2}{*}{$\EE\big[\max\limits_{s^*\in\Scal^*}\mathsf{r}_{n}(s^*)|Y_n\big]$} & 0.8196 & 0.4218 & 0.3957 & 0.3152 & 0.2998 \\
& (0.5916) & (0.1615) & (0.2152) & (0.1874) & (0.1255)  \\
\hline
\hline
$d=2$ & $n=10^2$ & $n=20^2$ & $n=30^2$ & \\
\hline
\multirow{2}{*}{$\EE\big[\max\limits_{s^*\in\Scal^*}\mathsf{r}_{n}(s^*)|Y_n\big]$} &  0.1773 & 0.0935 & 0.0806 & \\
& (0.0122) & (0.0052) & (0.0083) &   \\
\hline
\end{tabular}
}
\end{table}

\bibliographystyle{Chicago}
\bibliography{bvm}

\begin{thebibliography}{}

\bibitem[\protect\citeauthoryear{Anderes}{Anderes}{2010}]{And10}
Anderes, E. (2010).
\newblock {On the consistent separation of scale and variance in Gaussian
  random fields}.
\newblock {\em {The Annals of Statistics}\/}~{\em 38\/}(2), 870--893.

\bibitem[\protect\citeauthoryear{Arnold}{Arnold}{2015}]{Arn15}
Arnold, B.~C. (2015).
\newblock {\em Pareto Distributions (2nd Edition, Chapman \& Hall/CRC
  Monographs on Statistics and Applied Probability)}.
\newblock CRC Press.

\bibitem[\protect\citeauthoryear{Bachoc, Bevilacqua, and Velandia}{Bachoc
  et~al.}{2019}]{Bacetal19}
Bachoc, F., M.~Bevilacqua, and D.~Velandia (2019).
\newblock {Composite likelihood estimation for a Gaussian process under fixed
  domain asymptotics}.
\newblock {\em Journal of Multivariate Analysis\/}~{\em 174}, 104534.

\bibitem[\protect\citeauthoryear{Bachoc and Lagnoux}{Bachoc and
  Lagnoux}{2020}]{BacLag20}
Bachoc, F. and A.~Lagnoux (2020).
\newblock {Fixed-domain asymptotic properties of maximum composite likelihood
  estimators for Gaussian processes}.
\newblock {\em Journal of Statistical Planning and Inference\/}~{\em 209},
  62--75.

\bibitem[\protect\citeauthoryear{Banerjee, Gelfand, Finley, and Sang}{Banerjee
  et~al.}{2008}]{Banetal08}
Banerjee, S., A.~E. Gelfand, A.~O. Finley, and H.~Sang (2008).
\newblock {Gaussian predictive process models for large spatial data sets}.
\newblock {\em Journal of the Royal Statistical Society: Series B (Statistical
  Methodology)\/}~{\em 70\/}(4), 825--848.

\bibitem[\protect\citeauthoryear{Berger, De~Oliveria, and Sans\'o}{Berger
  et~al.}{2001}]{Beretal01}
Berger, J.~O., V.~De~Oliveria, and B.~Sans\'o (2001).
\newblock {Objective Bayesian analysis of spatially correlated data}.
\newblock {\em {Journal of the American Statistical Association}\/}~{\em
  96\/}(456), 1361--1374.

\bibitem[\protect\citeauthoryear{Bevilacqua, Faouzi, Furrer, and
  Porcu}{Bevilacqua et~al.}{2019}]{Bevetal19}
Bevilacqua, M., T.~Faouzi, R.~Furrer, and E.~Porcu (2019).
\newblock {Estimation and prediction using generalized Wendland covariance
  functions under fixed domain asymptotics}.
\newblock {\em {The Annals of Statistics}\/}~{\em 47\/}(2), 828--856.

\bibitem[\protect\citeauthoryear{Bickel and Kleijn}{Bickel and
  Kleijn}{2012}]{BicKle12}
Bickel, P.~J. and B.~J.~K. Kleijn (2012).
\newblock {The semiparametric Bernstein von Mises theorem}.
\newblock {\em {The Annals of Statistics}\/}~{\em 40\/}(1), 206--237.

\bibitem[\protect\citeauthoryear{Bochkina and Green}{Bochkina and
  Green}{2014}]{BocGre14}
Bochkina, N.~A. and P.~J. Green (2014).
\newblock {The Bernstein-von Mises theorem and nonregular models}.
\newblock {\em {The Annals of Statistics}\/}~{\em 42\/}(5), 1850--1878.

\bibitem[\protect\citeauthoryear{Brazauskas}{Brazauskas}{2002}]{Bra02}
Brazauskas, V. (2002).
\newblock Fisher information matrix for the feller–pareto distribution.
\newblock {\em Statistics and Probability Letters\/}~{\em 59}, 159--167.

\bibitem[\protect\citeauthoryear{Chae and Walker}{Chae and
  Walker}{2020}]{ChaWal20}
Chae, M. and S.~G. Walker (2020).
\newblock {Wasserstein upper bounds of the total variation for smooth
  densities}.
\newblock {\em Statistics and Probability Letters\/}~{\em 163}, 1--6.

\bibitem[\protect\citeauthoryear{Chang, Huang, and Ing}{Chang
  et~al.}{2014}]{Chaetal14}
Chang, C.-H., H.-C. Huang, and C.-K. Ing (2014).
\newblock {Asymptotic theory of generalized information criterion for
  geostatistical regression model selection}.
\newblock {\em {The Annals of Statistics}\/}~{\em 42\/}(6), 2441--2468.

\bibitem[\protect\citeauthoryear{Chang, Huang, and Ing}{Chang
  et~al.}{2017}]{Chaetal17}
Chang, C.-H., H.-C. Huang, and C.-K. Ing (2017).
\newblock {Mixed domain asymptotics for a stochastic process model with time
  trend and measurement error}.
\newblock {\em Bernoulli\/}~{\em 23\/}(1), 159--190.

\bibitem[\protect\citeauthoryear{Chen, Simpson, and Ying}{Chen
  et~al.}{2000}]{Chenetal00}
Chen, H.-S., D.~G. Simpson, and Z.~Ying (2000).
\newblock {Infill asymptotics for a stochastic process model with measurement
  error}.
\newblock {\em Statistica Sinica\/}~{\em 10}, 141--156.

\bibitem[\protect\citeauthoryear{Chen, Christensen, and Tamer}{Chen
  et~al.}{2018}]{Chenetal18}
Chen, X., T.~M. Christensen, and E.~Tamer (2018).
\newblock {Monte Carlo confidence sets for identified sets}.
\newblock {\em Econometrica\/}~{\em 86\/}(6), 1965--2018.

\bibitem[\protect\citeauthoryear{Chernozhukov and Hong}{Chernozhukov and
  Hong}{2004}]{CheHon04}
Chernozhukov, V. and H.~Hong (2004).
\newblock {Likelihood estimation and inference in a class of nonregular
  econometric models}.
\newblock {\em Econometrica\/}~{\em 72\/}(5), 1445--1480.

\bibitem[\protect\citeauthoryear{Cressie}{Cressie}{1993}]{Cre93}
Cressie, N. (1993).
\newblock {\em Statistics for Spatial Data}.
\newblock Wiley, New York.

\bibitem[\protect\citeauthoryear{Crowder}{Crowder}{1976}]{Cro76}
Crowder, M.~J. (1976).
\newblock {Maximum likelihood estimation for dependent observations}.
\newblock {\em Journal of the Royal Statistical Society: Series B (Statistical
  Methodology)\/}~{\em 38\/}(1), 45--53.

\bibitem[\protect\citeauthoryear{Datta, Banerjee, Finley, and Gelfand}{Datta
  et~al.}{2016}]{Datetal16}
Datta, A., S.~Banerjee, A.~O. Finley, and A.~E. Gelfand (2016).
\newblock {Hierarchical nearest-neighbor Gaussian process models for large
  geostatistical datasets}.
\newblock {\em Journal of the American Statistical Association\/}~{\em
  111\/}(514), 800--812.

\bibitem[\protect\citeauthoryear{De~Oliveira, Kedem, and Short}{De~Oliveira
  et~al.}{1997}]{DeO97}
De~Oliveira, V., B.~Kedem, and D.~A. Short (1997).
\newblock {Bayesian prediction of transformed Gaussian random fields}.
\newblock {\em {Journal of the American Statistical Association}\/}~{\em
  92\/}(440), 1422--1433.

\bibitem[\protect\citeauthoryear{Devroye, Mehrabian, and Reddad}{Devroye
  et~al.}{2018}]{Devetal18}
Devroye, L., A.~Mehrabian, and T.~Reddad (2018).
\newblock {The total variation distance between high-dimensional Gaussians}.
\newblock {\em arXiv preprint arXiv:1810.08693\/}.

\bibitem[\protect\citeauthoryear{Du, Zhang, and Mandrekar}{Du
  et~al.}{2009}]{Duetal09}
Du, J., H.~Zhang, and V.~S. Mandrekar (2009).
\newblock {Fixed-domain asymptotic properties of tapered maximum likelihood
  estimators}.
\newblock {\em {The Annals of Statistics}\/}~{\em 37\/}(6A), 3330--3361.

\bibitem[\protect\citeauthoryear{Fuglstad, Simpson, Lindgren, and Rue}{Fuglstad
  et~al.}{2019}]{Fugetal19}
Fuglstad, G.-A., D.~Simpson, F.~Lindgren, and H.~Rue (2019).
\newblock {Constructing priors that penalize the complexity of Gaussian random
  fields}.
\newblock {\em {Journal of the American Statistical Association}\/}~{\em 114},
  445--452.

\bibitem[\protect\citeauthoryear{Ghosal and van~der Vaart}{Ghosal and van~der
  Vaart}{2017}]{GhoVan17}
Ghosal, S. and A.~W. van~der Vaart (2017).
\newblock {\em Fundamentals of Nonparametric Bayesian Inference}.
\newblock Cambridge University Press.

\bibitem[\protect\citeauthoryear{Gneiting}{Gneiting}{2002}]{Gne02}
Gneiting, T. (2002).
\newblock Compactly supported correlation functions.
\newblock {\em Journal of Multivariate Analysis\/}~{\em 83\/}(2), 493--508.

\bibitem[\protect\citeauthoryear{Gu and Anderson}{Gu and
  Anderson}{2018}]{GuAnd18}
Gu, M. and K.~Anderson (2018).
\newblock {Calibration of imperfect mathematical models by multiple sources of
  data with measurement bias}.
\newblock {\em arXiv preprint arXiv:1810.11664\/}.

\bibitem[\protect\citeauthoryear{Gu, Wang, and Berger}{Gu
  et~al.}{2018}]{Guetal18}
Gu, M., X.~Wang, and J.~O. Berger (2018).
\newblock {Robust Gaussian stochastic process emulation}.
\newblock {\em {The Annals of Statistics}\/}~{\em 46\/}(6A), 3038--3066.

\bibitem[\protect\citeauthoryear{Guhaniyogi, Li, Savitsky, and
  Srivastava}{Guhaniyogi et~al.}{2022}]{Guhetal17}
Guhaniyogi, R., C.~Li, T.~D. Savitsky, and S.~Srivastava (2022).
\newblock {Distributed Bayesian inference in massive spatial data}.
\newblock {\em Statistical Science\/}, (forthcoming).

\bibitem[\protect\citeauthoryear{Gustafson}{Gustafson}{2014}]{Gus14}
Gustafson, P. (2014).
\newblock {Bayesian inference in partially identified models: Is the shape of
  the posterior distribution useful?}
\newblock {\em Electronic Journal of Statistics\/}~{\em 8}, 476--496.

\bibitem[\protect\citeauthoryear{Gustafson}{Gustafson}{2015}]{Gus15}
Gustafson, P. (2015).
\newblock {\em {Bayesian inference for partially identified models: Exploring
  the limits of limited data}}.
\newblock CRC Press, New York.

\bibitem[\protect\citeauthoryear{Handcock and Stein}{Handcock and
  Stein}{1993}]{HanSte93}
Handcock, M.~S. and M.~L. Stein (1993).
\newblock {A Bayesian analysis of kriging}.
\newblock {\em Technometrics\/}~{\em 35\/}(4), 403--410.

\bibitem[\protect\citeauthoryear{Heaton, Datta, Finley, Furrer, Guinness,
  Guhaniyogi, Gerber, Gramacy, Hammerling, Katzfuss, Lindgren, Nychka, Sun, and
  Zammit-Mangion}{Heaton et~al.}{2019}]{Heatonetal18}
Heaton, J.~H., A.~Datta, A.~O. Finley, R.~Furrer, J.~Guinness, R.~Guhaniyogi,
  F.~Gerber, R.~B. Gramacy, D.~Hammerling, M.~Katzfuss, F.~Lindgren, D.~W.
  Nychka, F.~Sun, and A.~Zammit-Mangion (2019).
\newblock A case study competition among methods for analyzing large spatial
  data.
\newblock {\em Journal of Agricultural, Biological and Environmental
  Statistics\/}~{\em 24}, 398–--425.

\bibitem[\protect\citeauthoryear{Horn and Johnson}{Horn and
  Johnson}{1985}]{HorJoh85}
Horn, R.~A. and C.~R. Johnson (1985).
\newblock {\em Matrix Analysis}.
\newblock Cambrige University Press.

\bibitem[\protect\citeauthoryear{Hsu, Kakade, and Zhang}{Hsu
  et~al.}{2012}]{Hsuetal12}
Hsu, D., S.~M. Kakade, and T.~Zhang (2012).
\newblock A tail inequality for quadratic forms of subgaussian random vectors.
\newblock {\em Electronic Communications in Probability\/}~{\em 17\/}(52),
  1--6.

\bibitem[\protect\citeauthoryear{Ibragimov and Rozanov}{Ibragimov and
  Rozanov}{1978}]{IbrRoz78}
Ibragimov, I.~A. and Y.~A. Rozanov (1978).
\newblock {\em Gaussian Random Processes (translated by A. B. Aries)}.
\newblock Springer, New York.

\bibitem[\protect\citeauthoryear{Jiang}{Jiang}{2017}]{Jiang17}
Jiang, W. (2017).
\newblock {On limiting distribution of quasi-posteriors under partial
  identification}.
\newblock {\em Econometrics and Statistics\/}~{\em 3\/}(C), 60--72.

\bibitem[\protect\citeauthoryear{Jiang and Li}{Jiang and Li}{2019}]{JiangLi19}
Jiang, W. and C.~Li (2019).
\newblock {On Bayesian oracle properties}.
\newblock {\em Bayesian Analysis\/}~{\em 14\/}(1), 235--260.

\bibitem[\protect\citeauthoryear{Jun, Pinkse, and Wan}{Jun
  et~al.}{2015}]{Junetal15}
Jun, S.~J., J.~Pinkse, and Y.~Wan (2015).
\newblock {Classical Laplace estimation for cube root-n-consistent estimators:
  Improved convergence rates and rate-adaptive inference}.
\newblock {\em Journal of Econometrics\/}~{\em 187\/}(1), 201--216.

\bibitem[\protect\citeauthoryear{Kanagawa, Hennig, Sejdinovic, and
  Sriperumbudur}{Kanagawa et~al.}{2018}]{Kanetal18}
Kanagawa, M., P.~Hennig, D.~Sejdinovic, and B.~K. Sriperumbudur (2018).
\newblock {Gaussian processes and kernel methods: A review on connections and
  equivalences}.
\newblock {\em arXiv preprint arXiv:1807.02582\/}.

\bibitem[\protect\citeauthoryear{Kaufman, Schervish, and Nychka}{Kaufman
  et~al.}{2008}]{Kauetal08}
Kaufman, C.~G., M.~J. Schervish, and D.~W. Nychka (2008).
\newblock Covariance tapering for likelihood-based estimation in large spatial
  data sets.
\newblock {\em {Journal of the American Statistical Association}\/}~{\em
  103\/}(484), 1545--1555.

\bibitem[\protect\citeauthoryear{Kaufman and Shaby}{Kaufman and
  Shaby}{2013}]{KauSha13}
Kaufman, C.~G. and B.~A. Shaby (2013).
\newblock {The role of the range parameter for estimation and prediction in
  geostatistics}.
\newblock {\em Biometrika\/}~{\em 100\/}(2), 473--484.

\bibitem[\protect\citeauthoryear{Kennedy and O'Hagan}{Kennedy and
  O'Hagan}{2001}]{KenOha01}
Kennedy, M.~C. and A.~O'Hagan (2001).
\newblock Bayesian calibration of computer models.
\newblock {\em Journal of the Royal Statistical Society: Series B (Statistical
  Methodology)\/}~{\em 63\/}(3), 425--464.

\bibitem[\protect\citeauthoryear{Kleijn and Knapik}{Kleijn and
  Knapik}{2012}]{KleKna12}
Kleijn, B. J.~K. and B.~Knapik (2012).
\newblock {Semiparametric posterior limits under local asymptotic
  exponentiality}.
\newblock {\em arXiv preprint: arXiv 1210.6204v3\/}.

\bibitem[\protect\citeauthoryear{Kreh}{Kreh}{2012}]{Kre12}
Kreh, M. (2012).
\newblock {\em Bessel Functions. Lecture Notes, Penn State - G\"ottingen Summer
  School on Number Theory}.

\bibitem[\protect\citeauthoryear{Kullback, Keegel, and Kullback}{Kullback
  et~al.}{1987}]{Kuletal87}
Kullback, S., J.~C. Keegel, and J.~H. Kullback (1987).
\newblock {\em Topics in Statistical Information Theory. Lecture Notes in
  Statist.}, Volume~42.
\newblock Springer, New York.

\bibitem[\protect\citeauthoryear{Laurent and Massart}{Laurent and
  Massart}{2000}]{LauMas00}
Laurent, B. and P.~Massart (2000).
\newblock Adaptive estimation of a quadratic functional by model selection.
\newblock {\em Annals of Statistics\/}~{\em 28\/}(5), 1302--1338.

\bibitem[\protect\citeauthoryear{Lehmann and Casella}{Lehmann and
  Casella}{1998}]{LehCas98}
Lehmann, E.~L. and G.~Casella (1998).
\newblock {\em Theory of Point Estimation}.
\newblock Springer-Verlag New York.

\bibitem[\protect\citeauthoryear{Li, Srivastava, and Dunson}{Li
  et~al.}{2017}]{Lietal17}
Li, C., S.~Srivastava, and D.~B. Dunson (2017).
\newblock Simple, scalable and accurate posterior interval estimation.
\newblock {\em Biometrika\/}~{\em 104\/}(3), 665--680.

\bibitem[\protect\citeauthoryear{Loh}{Loh}{2005}]{Loh05}
Loh, W.-L. (2005).
\newblock {Fixed-domain asymptotics for a subclass of Mat\'ern-type Gaussian
  random fields}.
\newblock {\em {The Annals of Statistics}\/}~{\em 33\/}(5), 2344--2394.

\bibitem[\protect\citeauthoryear{Loh}{Loh}{2015}]{Loh15}
Loh, W.-L. (2015).
\newblock {Estimating the smoothness of a Gaussian random field from
  irregularly spaced data via higher-order quadratic variations}.
\newblock {\em {The Annals of Statistics}\/}~{\em 43\/}(6), 2766--2794.

\bibitem[\protect\citeauthoryear{Loh, Sun, and Wen}{Loh
  et~al.}{2021}]{Lohetal20}
Loh, W.-L., S.~Sun, and J.~Wen (2021).
\newblock {On fixed-domain asymptotics, parameter estimation and isotropic
  Gaussian random fields with Mat\'ern covariance functions}.
\newblock {\em The Annals of Statistics\/}~{\em 49\/}(6), 3127--3152.

\bibitem[\protect\citeauthoryear{Manski}{Manski}{2003}]{Man03}
Manski, C. (2003).
\newblock {\em Partial Identification of Probability Distributions}.
\newblock Springer Verlag, New York.

\bibitem[\protect\citeauthoryear{Mardia and Marshall}{Mardia and
  Marshall}{1984}]{MarMar84}
Mardia, K.~V. and R.~J. Marshall (1984).
\newblock {Maximum likelihood estimation of models for residual covariance in
  spatial statistics}.
\newblock {\em Biometrika\/}~{\em 71\/}(1), 135--146.

\bibitem[\protect\citeauthoryear{Moon and Schorfheide}{Moon and
  Schorfheide}{2012}]{MonSch12}
Moon, H.~R. and F.~Schorfheide (2012).
\newblock {Bayesian and frequentist inference in partially identified models}.
\newblock {\em Econometrica\/}~{\em 80\/}(2), 755--782.

\bibitem[\protect\citeauthoryear{Peruzzi, Banerjee, and Finley}{Peruzzi
  et~al.}{2022}]{Peretal21}
Peruzzi, M., S.~Banerjee, and A.~O. Finley (2022).
\newblock {Highly scalable Bayesian geostatistical modeling via meshed Gaussian
  processes on partitioned domains}.
\newblock {\em Journal of the American Statistical Association\/}~{\em
  117\/}(538), 969--982.

\bibitem[\protect\citeauthoryear{Putter and Young}{Putter and
  Young}{2001}]{PutYou01}
Putter, H. and G.~A. Young (2001).
\newblock {On the effect of covariance function estimation on the accuracy of
  kriging predictors}.
\newblock {\em Bernoulli\/}~{\em 7\/}(3), 421--438.

\bibitem[\protect\citeauthoryear{Rasmussen and Williams}{Rasmussen and
  Williams}{2006}]{RasWil06}
Rasmussen, C.~E. and C.~K. Williams (2006).
\newblock {\em Gaussian Process for Machine Learning}.
\newblock MIT press.

\bibitem[\protect\citeauthoryear{Ritter}{Ritter}{2000}]{Rit00}
Ritter, K. (2000).
\newblock {\em Average-case Analysis of Numerical Problems}.
\newblock Springer.

\bibitem[\protect\citeauthoryear{Sang and Huang}{Sang and
  Huang}{2012}]{SanHua12}
Sang, H. and J.~Z. Huang (2012).
\newblock A full scale approximation of covariance functions for large spatial
  data sets.
\newblock {\em Journal of the Royal Statistical Society: Series B (Statistical
  Methodology)\/}~{\em 74\/}(1), 111--132.

\bibitem[\protect\citeauthoryear{Sang, Un, and Huang}{Sang
  et~al.}{2011}]{Sanetal11}
Sang, H., M.~Un, and J.~Z. Huang (2011).
\newblock Covariance approximation for large multivariate spatial data sets
  with an application to multiple climate model errors.
\newblock {\em Annals of Applied Statistics\/}~{\em 5\/}(4), 2519--2548.

\bibitem[\protect\citeauthoryear{Schwartz}{Schwartz}{1965}]{Sch65}
Schwartz, L. (1965).
\newblock On {B}ayes procedures.
\newblock {\em Z. Wahrscheinlichkeitstheorie verw. Geb.\/}~{\em 4}, 10 -- 26.

\bibitem[\protect\citeauthoryear{Shaby and Ruppert}{Shaby and
  Ruppert}{2012}]{ShaRup12}
Shaby, B. and D.~Ruppert (2012).
\newblock Tapered covariance: {B}ayesian estimation and asymptotics.
\newblock {\em Journal of Computational and Graphical Statistics\/}~{\em
  21\/}(2), 433--452.

\bibitem[\protect\citeauthoryear{Shen}{Shen}{2002}]{Shen02}
Shen, X. (2002).
\newblock {Asymptotic normality of semiparametric and nonparametric posterior
  distributions}.
\newblock {\em {Journal of the American Statistical Association}\/}~{\em
  97\/}(457), 222--235.

\bibitem[\protect\citeauthoryear{Stein}{Stein}{1988}]{Stein88}
Stein, M.~L. (1988).
\newblock {Asymptotically efficient prediction of a random field with a
  misspecified covariance function}.
\newblock {\em {The Annals of Statistics}\/}~{\em 16\/}(1), 55--63.

\bibitem[\protect\citeauthoryear{Stein}{Stein}{1990a}]{Stein90a}
Stein, M.~L. (1990a).
\newblock {{\GG{1}} Uniform asymptotic optimality of linear predictions of a
  random field using an incorrect second-order structure}.
\newblock {\em {The Annals of Statistics}\/}~{\em 18\/}(2), 850--872.

\bibitem[\protect\citeauthoryear{Stein}{Stein}{1990b}]{Stein90b}
Stein, M.~L. (1990b).
\newblock {{\GG{2}} Bounds on the efficiency of linear predictions using an
  incorrect covariance function}.
\newblock {\em {The Annals of Statistics}\/}~{\em 18\/}(3), 1116--1138.

\bibitem[\protect\citeauthoryear{Stein}{Stein}{1990c}]{Stein90c}
Stein, M.~L. (1990c).
\newblock {{{\GG{3}}A comparison of generalized cross validation and modified
  maximum likelihood for estimating the parameters of a stochastic process}}.
\newblock {\em {The Annals of Statistics}\/}~{\em 18\/}(3), 1139--1157.

\bibitem[\protect\citeauthoryear{Stein}{Stein}{1993}]{Stein93}
Stein, M.~L. (1993).
\newblock {A simple condition for asymptotic optimality of linear predictions
  of random fields}.
\newblock {\em Statistics and Probability Letters\/}~{\em 17}, 399--404.

\bibitem[\protect\citeauthoryear{Stein}{Stein}{1997}]{Stein97}
Stein, M.~L. (1997).
\newblock {Efficiency of linear predictors for periodic processes using an
  incorrect covariance function}.
\newblock {\em Journal of Statistical Planning and Inference\/}~{\em 58\/}(2),
  321--331.

\bibitem[\protect\citeauthoryear{Stein}{Stein}{1999a}]{Stein99a}
Stein, M.~L. (1999a).
\newblock {\em {\GG{1}} Interpolation for Spatial Data: Some Theory for
  Kriging}.
\newblock Springer, New York.

\bibitem[\protect\citeauthoryear{Stein}{Stein}{1999b}]{Stein99b}
Stein, M.~L. (1999b).
\newblock {{\GG{2}} Predicting random fields with increasing dense
  observations}.
\newblock {\em The Annals of Applied Probability\/}~{\em 9\/}(1), 242--273.

\bibitem[\protect\citeauthoryear{Sun, Miao, Duan, Ashouri, Sorooshian, and
  Hsu}{Sun et~al.}{2018}]{Sunetal18}
Sun, Q., C.~Miao, Q.~Duan, H.~Ashouri, S.~Sorooshian, and K.~L. Hsu (2018).
\newblock A review of global precipitation data sets: data sources, estimation,
  and intercomparisons.
\newblock {\em Review of Geophysics\/}~{\em 56}, 79--107.

\bibitem[\protect\citeauthoryear{Tamer}{Tamer}{2010}]{Tam10}
Tamer, E. (2010).
\newblock {Partial identification in econometrics}.
\newblock {\em Annual Review of Economics\/}~{\em 3}, 167--195.

\bibitem[\protect\citeauthoryear{Tang, Zhang, and Banerjee}{Tang
  et~al.}{2021}]{Tanetal19}
Tang, W., L.~Zhang, and S.~Banerjee (2021).
\newblock {On identifiability and consistency of the nugget in Gaussian spatial
  process models}.
\newblock {\em Journal of the Royal Statistical Society: Series B (Statistical
  Methodology)\/}~{\em 83\/}(5), 1044--1070.

\bibitem[\protect\citeauthoryear{Tuo and Wang}{Tuo and Wang}{2020}]{TuoWang20}
Tuo, R. and W.~Wang (2020).
\newblock {Kriging prediction with isotropic Mat\'ern correlations: Robustness
  and experimental designs}.
\newblock {\em Journal of Machine Learning Research\/}~{\em 21}, 1--38.

\bibitem[\protect\citeauthoryear{van~der Vaart}{van~der Vaart}{1998}]{Van98}
van~der Vaart, A.~W. (1998).
\newblock {\em Asymptotic Statistics}.
\newblock Cambridge University Press, Cambridge.

\bibitem[\protect\citeauthoryear{van~der Vaart and van Zanten}{van~der Vaart
  and van Zanten}{2008}]{VarZan08a}
van~der Vaart, A.~W. and J.~H. van Zanten (2008).
\newblock {Rates of contraction of posterior distributions based on Gaussian
  process priors}.
\newblock {\em The Annals of Statistics\/}~{\em 36\/}(3), 1435--1463.

\bibitem[\protect\citeauthoryear{van~der Vaart and van Zanten}{van~der Vaart
  and van Zanten}{2009}]{VarZan09}
van~der Vaart, A.~W. and J.~H. van Zanten (2009).
\newblock {Adaptive Bayesian estimation using a Gaussian random field with
  inverse gamma bandwidth}.
\newblock {\em {The Annals of Statistics}\/}~{\em 37}, 2655--2675.

\bibitem[\protect\citeauthoryear{van~der Vaart and van Zanten}{van~der Vaart
  and van Zanten}{2011}]{VarZan11}
van~der Vaart, A.~W. and J.~H. van Zanten (2011).
\newblock {Information rates of nonparametric Gaussian process methods}.
\newblock {\em Journal of Machine Learning Research\/}~{\em 12\/}(Jun),
  2095--2119.

\bibitem[\protect\citeauthoryear{Velandia, Bachoc, Bevilacqua, and
  Gendre}{Velandia et~al.}{2017}]{Veletal17}
Velandia, D., F.~Bachoc, M.~Bevilacqua, and X.~Gendre (2017).
\newblock {Maximum likelihood estimation for a bivariate Gaussian process under
  fixed domain asymptotics}.
\newblock {\em Electronic Journal of Statistics\/}~{\em 11\/}(2), 2978--3007.

\bibitem[\protect\citeauthoryear{Villani}{Villani}{2008}]{Vil08}
Villani, C. (2008).
\newblock {\em Optimal Transport: Old and New}.
\newblock Springer.

\bibitem[\protect\citeauthoryear{Wang and Loh}{Wang and Loh}{2011}]{WangLoh11}
Wang, D. and W.-L. Loh (2011).
\newblock {On fixed-domain asymptotics and covariance tapering in Gaussian
  random field models}.
\newblock {\em Electronic Journal of Statistics\/}~{\em 5}, 238--269.

\bibitem[\protect\citeauthoryear{Wang, Tuo, and Wu}{Wang
  et~al.}{2019}]{Wangetal19}
Wang, W., R.~Tuo, and C.~F.~J. Wu (2019).
\newblock {On prediction properties of kriging: Uniform error bounds and
  robustness}.
\newblock {\em {Journal of the American Statistical Association}\/}~{\em
  115\/}(530), 920--930.

\bibitem[\protect\citeauthoryear{Wendland}{Wendland}{2005}]{Wen05}
Wendland, H. (2005).
\newblock {\em Scattered Data Approximation}.
\newblock Cambridge University Press.

\bibitem[\protect\citeauthoryear{Wu and Schaback}{Wu and
  Schaback}{1993}]{WuSch93}
Wu, Z. and R.~Schaback (1993).
\newblock Local error estimates for radial basis function interpolation of
  scattered data.
\newblock {\em IMA Journal of Numerical Analysis\/}~{\em 13\/}(1), 13--27.

\bibitem[\protect\citeauthoryear{Wynne, Briol, and Girolami}{Wynne
  et~al.}{2021}]{Wynetal21}
Wynne, G., F.-X. Briol, and M.~Girolami (2021).
\newblock {Convergence guarantees for Gaussian process approximations under
  several observation models}.
\newblock {\em Journal of Machine Learning Research\/}~(123), 1--40.

\bibitem[\protect\citeauthoryear{Yakowitz and Szidarovszky}{Yakowitz and
  Szidarovszky}{1985}]{YakSzi85}
Yakowitz, S. and F.~Szidarovszky (1985).
\newblock {A comparison of kriging with nonparametric regression methods}.
\newblock {\em Journal of Multivariate Analysis\/}~{\em 16\/}(1), 23--35.

\bibitem[\protect\citeauthoryear{Yang and Tokdar}{Yang and
  Tokdar}{2015}]{YanTok15}
Yang, Y. and S.~T. Tokdar (2015).
\newblock {Minimax-optimal nonparametric regression in high dimensions}.
\newblock {\em {The Annals of Statistics}\/}~{\em 43\/}(2), 652--674.

\bibitem[\protect\citeauthoryear{Ying}{Ying}{1991}]{Ying91}
Ying, Z. (1991).
\newblock {Asymptotic properties of a maximum likelihood estimator with data
  from a Gaussian process}.
\newblock {\em Journal of Multivariate Analysis\/}~{\em 36}, 280--296.

\bibitem[\protect\citeauthoryear{Ying}{Ying}{1993}]{Ying93}
Ying, Z. (1993).
\newblock {Maximum likelihood estimation of parameters under a spatial sampling
  scheme}.
\newblock {\em {The Annals of Statistics}\/}~{\em 21\/}(3), 1567--1590.

\bibitem[\protect\citeauthoryear{Zhang}{Zhang}{2004}]{Zhang04}
Zhang, H. (2004).
\newblock {Inconsistent estimation and asymptotically equal interpolations in
  model-based geostatistics}.
\newblock {\em {Journal of the American Statistical Association}\/}~{\em
  99\/}(465), 250--261.

\bibitem[\protect\citeauthoryear{Zhang and Zimmerman}{Zhang and
  Zimmerman}{2005}]{ZhaZim05}
Zhang, H. and D.~L. Zimmerman (2005).
\newblock {Towards reconciling two asymptotic frameworks in spatial
  statistics}.
\newblock {\em Biometrika\/}~{\em 92\/}(4), 921--936.

\end{thebibliography}

\end{document}